\documentclass[12pt]{amsbook}
\usepackage[all,cmtip]{xy}

\usepackage{a4wide}
\usepackage{pxfonts}
\usepackage{wasysym}
\usepackage{amssymb,amsxtra,mathrsfs}
\usepackage{eucal}
\usepackage{setspace,upgreek}
\usepackage{mathtools}
\usepackage{url}
\usepackage[hidelinks]{hyperref}

\newcommand{\mf}[1]{\mathfrak{#1}}
\newcommand{\mc}[1]{\mathcal{#1}}

\newcommand{\ms}[1]{\mathsf{#1}}

\newcommand{\ov}{\overline}
\newcommand{\ra}{\right\rangle}
\newcommand{\lr}[2]{\left\langle #1,#2\ra}
\newcommand{\ro}[1]
{\overset{#1}{\blacklozenge}}

\newcommand{\F}[2]{\mathfrak{F}_{#1}^{#2}}

\newcommand{\cc}[2]{\mathcal{C}_{#1}
\left(#2\right)}

\newcommand{\up}{\upharpoonright}

\newcommand{\ep}{\varepsilon} 

\newcommand{\cu}{\overset{K}{=}}

\newcommand{\n}{\mathfrak{N}}

\newcommand{\ze}{\mathbf{0}} 
\newcommand{\un}{\mathbf{1}}

\newcommand{\p}{\mathcal{P}}
\newcommand{\A}{\mathcal{A}}

\newcommand{\B}{\mathcal{B}}
 
\newcommand{\N}{\mathbb{N}}
\newcommand{\K}{\mathbb{K}}
\newcommand{\R}{\mathbb {R}}
\newcommand{\C}{\mathbb {C}}

\newtheorem{theorem}{Theorem}
\newtheorem{corollary}[theorem]{Corollary}
\newtheorem{definition}[theorem]{Definition}
\newtheorem{lemma}[theorem]{Lemma}
\newtheorem{proposition}[theorem]{Proposition}
\theoremstyle{definition}

\newtheorem{claim}[theorem]{Claim}
\newtheorem{remark}[theorem]{Remark}

\newtheorem{assumptions}[theorem]{Assumptions}

\newtheorem{convention}[theorem]{Convention}

\newtheorem{notation}[theorem]{Notation}

\numberwithin{section}{chapter}
\numberwithin{theorem}{section}
\numberwithin{equation}{section}

\begin{document}

\title
{Use of bundles of locally convex
spaces
in problems of convergence
of semigroups of 
operators
defined on different Banach
spaces.
Applications to spectral
stability problems.}

\author{Benedetto Silvestri}

\email{bndsilvestri@proton.me}

\address{440 W Kelly Ave Apt 6 Jackson, WY 83001-8544 USA (via express couriers signature required)}

\urladdr{\url{https://arxiv.org/a/silvestri_b_1}}

\urladdr{https://orcid.org/0000-0002-8042-4932}



\begin{abstract}
In this work we   construct certain general bundles $<\mathfrak{M},\rho,X>$ 
and $<\mathfrak{B},\eta,X>$ of Hausdorff locally convex spaces associated with a given 
Banach bundle $<\mathfrak{E},\pi,X>$. Then we   present conditions ensuring the existence of 
bounded sections $\mathcal{U}\in \Gamma^{x_{\infty}}(\rho)$ and $\mathcal{P}\in \Gamma^{x_{\infty}}(\eta)$ 
both continuous at a point $x_{\infty}\in X$, such that $\mathcal{U}(x)$ 
is a $C_{0}-$semigroup of contractions on $\mathfrak{E}_{x}$ and $\mathcal{P}(x)$ 
is a spectral projector of the infinitesimal generator of the semigroup $\mathcal{U}(x)$, for every $x\in X$.
\end{abstract}

\maketitle

\flushbottom
\tableofcontents
\chapter{Main Structures and statement of the problem}
\label{05250734}
\section*{Introduction}
\label{intro}
This work consists of three parts of which the present represents the first one.
We construct
certain general bundles
$\langle\mathfrak{M},\rho,X\rangle$
and
$\langle\mathfrak{B},\eta,X\rangle$
of
Hausdorff locally convex spaces
associated with
a given
Banach bundle
$\langle\mathfrak{E},\pi,X\rangle$.
Then we present
conditions
ensuring
the 
existence
of
bounded
sections
$\mathcal{U}\in
\Gamma^{x_{\infty}}(\rho)$
and
$\mathcal{P}\in
\Gamma^{x_{\infty}}(\eta)$
both
continuous
at a point
$x_{\infty}\in X$,
such that
$\mathcal{U}(x)$ is 
a
$C_{0}-$semigroup
of contractions
on 
$\mathfrak{E}_{x}$
and
$\mathcal{P}(x)$
is a spectral projector
of the infinitesimal
generator
of the semigroup
$\mathcal{U}(x)$,
for every $x\in X$.
\par 
Here
$\mf{W}
\coloneqq
\lr{\mf{M}}{\rho,X}$
and
$\lr{\mf{B}}{\eta,X}$
are
special kind
of
bundles
of
Hausdorff locally
convex spaces
(bundle of $\Omega-$spaces \cite{gie})
while
$\mf{V}
\coloneqq
\lr{\mf{E}}{\pi,X}$
is 
a 
suitable
Banach 
bundle
such that
the common
base space
$X$
is a completely regular topological space
and the filter of neighbourhoods of $x_{\infty}$ 
admits a countable basis\footnote{in particular $X$ a metric space and $x_{\infty}$ any point of $X$.}.
Moreover
for all $x\in X$
the
stalk 
$\mf{M}_{x}
\coloneqq
\overset{-1}{\rho}(x)$
is a topological
subspace
of the space
$\cc{c}{\R^{+},\mc{L}_{S_{x}}(\mf{E}_{x})}$
with the topology of
compact convergence,
of all 
continuous maps defined
on $\R^{+}$
and
with values in 
$\mc{L}_{S_{x}}(\mf{E}_{x})$,
and
the
stalk
$\mf{B}_{x}
\coloneqq
\overset{-1}{\eta}(x)$
is a topological
subspace
of 
$\mc{L}_{S_{x}}(\mf{E}_{x})$.
Here
$\mf{E}_{x}
\coloneqq
\overset{-1}{\pi}(x)$,
while
$\mc{L}_{S_{x}}(\mf{E}_{x})$,
is
the space, 
of all 
linear
bounded 
maps
on
$\mf{E}_{x}$
with the topology of
uniform convergence
over the subsets
of
$S_{x}\subset 
Bounded(\mf{E}_{x})$
which 
depends,
for all $x\in X$, 
on 
the same 
subspace $\mc{E}\subseteq\Gamma(\pi)$.
Here
$\rho:\mf{M}\to X$,
$\eta:\mf{B}\to X$,
and
$\pi:\mf{E}\to X$
are 
the projection maps
of the respecive bundles,
$\Gamma^{x_{\infty}}(\rho)$
is the set of all bounded sections of $\mf{W}$ continuous at $x_{\infty}$
with respect to the topology on the bundle space $\mf{M}$ and 
$\Gamma(\pi)$ is the set of all bounded continuous sections of $\mf{V}$.
\par
\emph
{An essential factor 
is that
the continuity
at $x_{\infty}$
of
$\mc{U}$ 
and
$\mc{P}$ 
derives
by 
a
sort
of continuity
at the same point
of the section
$\mc{T}$
of the 
graphs
of the
infinitesimal
generators
of the semigroups in the range of
$\mc{U}$},
where
the sort of continuity 
has to be understood in 
the following sense.
For every $x\in X$
let 
$\mc{T}(x)$
be
the graph
of  
the infinitesimal 
generator
$T_{x}$
of the semigroup
$\mc{U}(x)$,
then
\begin{equation}
\label{19240703}
\begin{cases}
\mc{T}(x_{\infty})
=
\{
\phi(x_{\infty})
\,\vert\,
\phi
\in
\Phi
\}
\\
\Phi
\subseteq
\Gamma^{x_{\infty}}(\pi_{\ms{E}^{\oplus}})
\\
(\forall x\in X)
(\forall \phi\in\Phi)
(\phi(x)\in\mc{T}(x)),
\end{cases}
\end{equation}
where
$\Gamma^{x_{\infty}}(\pi_{\ms{E}^{\oplus}})$
is the set of all bounded 
sections
of the 
direct 
sum
of 
bundles
$\mf{V}\oplus\mf{V}$
which are
continuous
at $x_{\infty}$.
\par
Hence
for any
$v\in Dom(T_{x_{\infty}})$
there exists
a bounded
section
$\phi$
of
$\mf{V}\oplus\mf{V}$
such that
\begin{equation}
\label{16490703}
\begin{cases}
(v,T_{x_{\infty}}v)
=
\lim_{x\to x_{\infty}}
(\phi_{1}(x),\phi_{2}(x))
\\
(\phi_{1}(x),\phi_{2}(x))
\in Graph(T_{x}),
\forall x\in X-\{x_{\infty}\},
\end{cases}
\end{equation}
where the limit 
is with respect to the topology
on the bundle
space of
$\mf{V}\oplus\mf{V}$\footnote{
Later we shall see that
the topology 
on the bundle space of
$\mf{V}\oplus\mf{V}$
will be constructed
in order
to ensure that
the 
limit in \eqref{16490703}
is equivalent to say that
$v=
\lim_{x\to x_{\infty}}
\phi_{1}(x)$
and
$T_{x_{\infty}}v=
\lim_{x\to x_{\infty}}
\phi_{2}(x)$,
both limits
with respect to the topology
on the bundle space
$\mf{E}$.}.
\par 
\emph
{The main strategy
for obtaining the continuity at $x_{\infty}$
of $\mc{U}$ and $\mc{P}$,
it is to correlate
the topologies on $\mf{M}$ and $\mf{B}$,
with the topology on $\mf{E}$.}
Thus it is clear 
that the 
construction
of 
the right
structures 
has a prominent
role.
\par
It is well-known 
the relative
freedom of choice
of the 
topology on
the
bundle
space 
of 
any 
bundle of $\Omega-$spaces.
More exactly fixed a suitable linear space say $G$ of bounded sections 
there exists always a topology on the bundle space such that 
all the maps in $G$ are continuous. Moreover if $X$ is compact 
one can find a topology such that $G$ is the whole space of bounded continuous sections
\cite[Thm. $5.9$]{gie}.
This freedom of choice
allows the construction
of examples of the above-mentioned correlations of topologies.
\par 
From the following simple result Cor. \ref{28111707}
and without entering in the definition
of the topology of a bundle of $\Omega-$space, 
we can recognize the power of determining 
the right set $\Gamma(\zeta)$
of continuous sections of a general 
bundle $\lr{\mf{Q}}{\zeta,X}$
of $\Omega-$space.
Let
$f\in\prod_{x\in X}^{b}\mf{Q}_{x}$
be
any
bounded
section
and
$x_{\infty}\in X$
such that
there exists
a section
$\sigma\in\Gamma(\zeta)$
such that
$\sigma(x_{\infty})=f(x_{\infty})$.
Then
\begin{equation}
\label{21040503}
f\in\Gamma^{x_{\infty}}(\zeta)
\Leftrightarrow
(\forall j\in J)
(\lim_{z\to x_{\infty}}
\nu_{j}^{z}(f(z)-\sigma(z))=0),
\end{equation}
where $J$ is a set
such that
$\{\nu_{j}^{z}\,\vert\, j\in J\}$
is a directed fundamental 
set of seminorms
of the locally convex space
$\mf{Q}_{z}\coloneqq\overset{-1}{\zeta}(z)$
for all $z\in X$.
About 
the problem
of establishing
if
there are continuous bounded sections
intersecting 
$f$
in $x_{\infty}$,
we can use an important result
of the theory of Banach bundles,
stating that any Banach bundle
over a locally
compact base space
is full, namely
for any point of the bundle 
space there exists a section
passing on it.
While for more general bundles
of $\Omega-$spaces
we can use the above described freedom.
\par 
\emph
{The criterium we used for
determining the correlations
between
$\mf{M}$ 
(resp.$\mf{B}$)
and
$\mf{E}$ 
is that 
of extending
to a general bundle of $\Omega-$spaces
two properties of the topology
of the space $\cc{c}{Y,\mc{L}_{s}(Z)}$}.
\par
Here
$Z$ is a normed space,
$S$
is a set of bounded subsets
of $Z$,
$\mc{L}_{s}(Z)$
is 
the space
of all linear continuous
maps on $Z$
with the pointwise
topology, 
finally
$\cc{c}{Y,\mc{L}_{s}(Z)}$
is the space of all continuous
maps
on a topological space $Y$
with values in 
$\mc{L}_{s}(Z)$
with the topology
of uniform convergence 
over the compact subsets
of $Y$.
\par
In order to simplify the notation
we here shall consider $Z$ 
as a Banach space
and take $\mc{L}_{S}(Z)=B_{s}(Z)$,
i.e. the space of all bounded linear 
operators on $Z$ with the strong operator
topology.
\par
Let $X$ be a compact space
\begin{equation*}
\begin{aligned}
\mc{M}
\coloneqq
&
\{
F\in\cc{b}{X,\cc{c}{Y,B_{s}(Z)}}
\,\vert\,
(\forall K\in Comp(Y))
\\
&
(C(F,K)\coloneqq
\sup_{(x,s)\in X\times K}
\|F(x)(s)\|_{B(Z)}<\infty)
\}
\\
\ms{M}_{x}
&
\coloneqq\ov{\{
F(x)\,\vert\, F\in\mc{M}
\}}
\end{aligned}
\end{equation*}
Let $\mf{V}\coloneqq\lr{\mf{E}}{\pi,X}$ denote 
the trivial bundle with constant stalk 
$Z$
so
$\Gamma(\pi)
\simeq\cc{b}{X,Z}$,
set
\begin{equation}
\label{18252606}
\begin{cases}
\A_{x}
\coloneqq\{\mu_{(v,x)}^{K}
\,\vert\,
K\in Comp(Y),
v\in\Gamma(\pi)
\},
\\
\mu_{(v,x)}^{K}:
\ms{M}_{x}\ni G\mapsto
\sup_{s\in K}
\|G(s)v(x)\|,
\\
\ms{M}
\coloneqq\{
\lr{\ms{M}_{x}}{\A_{x}}
\}_{x\in X}.
\end{cases}
\end{equation}
Then
by using 
Lemma \ref{22312406}
and
\cite[Thm. $5.9$]{gie}
we can construct
a bundle of $\Omega-$spaces
say
$\mf{V}(\ms{M},\mc{M})$
whose
stalk
at $x$
is the locally convex space
$\lr{\ms{M}_{x}}{\A_{x}}$
and
whose 
space of bounded continuous
sections
$\Gamma(\pi_{\ms{M}})$
is such that
$\Gamma(\pi_{\ms{M}})\simeq\mc{M}$.
\par
Let
$f\in\prod_{x\in\ X}\ms{M}_{x}$
be
such that
$(\forall K\in Comp(Y))
(\sup_{(x,s)\in X\times K}
\|f(x)(s)\|_{B(Z)}<\infty)$
then according to \textbf{Thm. \ref{22372406}} we obtain that
$(1)
\Leftrightarrow
(2)
\Leftrightarrow
(3)$
with
\begin{enumerate}
\item
$(\forall K\in Comp(Y))
(\forall v\in\Gamma(\pi))$
\begin{equation*}
(\lim_{x\to x_{\infty}}
\sup_{s\in K}
\|
f(x)(s)v(x)
-
f(x_{\infty})(s)v(x)
\|
=0);
\end{equation*}
\item
$f\in\Gamma^{x_{\infty}}(\pi_{\ms{M}})$;
\item
$f:X\to\cc{c}{Y,B_{s}(Z)}$
continuous at $x_{\infty}$.
\end{enumerate}
Moreover
if 
$Y$ is locally compact
for all $t\in Y$
\begin{equation}
\label{17150703}
\Gamma(\pi_{\ms{M}})_{t}
\bullet
\Gamma(\pi)
\subseteq
\Gamma(\pi).
\end{equation}
\emph
{
Therefore we constructed
two bundles
$\mf{V}$ 
and
$\mf{V}(\ms{M},\mc{M})$
whose topologies
are 
$(I)$
stalkwise 
related
by 
$\{\A_{x}\}_{x\in X}$
in \eqref{18252606}
and
for which
hold 
$(1)\Leftrightarrow(2)$
and
$(II)$
globally related by
\eqref{17150703}.
Finally
$\Gamma^{x_{\infty}}(\pi_{\ms{M}})$
coincides
with
the subset
of all maps
$f:X\to\cc{c}{Y,B_{s}(Z)}$
continuous at $x_{\infty}$
such that
$(\forall K\in Comp(Y))
(\sup_{(x,s)\in X\times K}
\|f(x)(s)\|_{B(Z)}<\infty)$.}
The extension at general bundles
of the property
$(I)$
leads to the concept
of 
\textbf{$\left(\Theta,\mc{E}\right)-$structure},
provided in 
\textbf{Def. \ref{10282712}} see Lemma \ref{15482712},
while the generalization of the property
$(II)$
leads to the concept of compatible
$\left(\Theta,\mc{E}\right)-$structure,
given in Def. \ref{10282712}.
\par
A similar and more important
global correlation
between 
$\mf{M}$ 
and 
$\mf{E}$,
this time
for the case in which 
the topology
on
each stalk
$\mf{M}_{x}$
is that of
the pointwise
convergence
instead
of the compact convergence,
is that
encoded in 
\eqref{18112502}
in
the definition
of 
invariant
$\left(\Theta,\mc{E},\mu\right)-$structures
provided in 
Def. \ref{17161902}.
This closes the 
discussion
about
the relationship
between 
the topologies
on
$\mf{M}$
and 
$\mf{E}$,
in particular
between 
those 
on
$\mf{B}$
and 
$\mf{E}$
\footnote{
Indeed it is sufficient
to take
$Y=\{pt\}$
i.e. one point space.}
\par
Briefly we recall what here
has to be understood as 
a classical stability problem
in order to understand how
to generalize it through the language of bundles.
The classical
stability
problem
could be so described.
Fixed a Banach space
$Z$ 
find 
a 
sequence
$\{S_{n}:D_{n}\subseteq Z\to Z\}$
of possibly unbounded linear operators
in $Z$ 
and 
a 
sequence
$\{P_{n}\}\subset B(Z)$
where
$P_{n}$
is a spectral projector
of $S_{n}$
for $n\in\N$,
such that
\begin{description}
\item[$(A)$]
whenever there exists
an
operator
$S:D\subset Z\to Z$
such that
$S=\lim_{n\to\infty}S_{n}$
with respect to 
a suitable topology
or in any other generalized sense,
\item[$(B)$]
then
there exists
a spectral
projector
$P\in B(Z)$
of $S$
such that
$P=\lim_{n\to\infty}P_{n}$
with respect to the strong operator
topology.
\end{description}
Here
a spectral projector
of an operator $S$
in a Banach space
is a continuous
projector
associated with
a closed 
$S-$invariant
subspace
$Z_{0}$
such that
$
\sigma(S\up Z_{0})
\subset
\sigma(S)
$,
where
$\sigma(T)$
is the spectrum of the operator
$T$.
\par 
In \cite[Ch $IV$]{kato}
one finds many stability theorems
in which the limit in $(A)$ has to be understood
with respect to the metric
induced by the so called gap between
the corresponding closed graphs.

\par
Additional stability theorems,
even for 
operators
defined
in different
spaces, are available.
They have been obtained by using
the concept
of 
\emph
{Transition Operators}
introduced by Victor I. Burenkov,
see for expample
\cite{bl1},
\cite{bl2}
and
\cite{bll}.
Instead to their stability theorems
Massimo Lanza de Cristoforis
and 
Pier Domenico Lamberti
employed functional analytic approaches,
see
for examples
\cite{l1},
\cite{l2},
\cite{ll}.
\par 
If we try
to generalize
the 
classical
stability
problem
to the case
in which
$Z$
is replaced
by
any sequence
$\{Z_{n}\}$
of
Banach spaces
and
$S_{n}$
is defined
in $Z_{n}$
for all $n$,
then we would face
the following difficulty.
How can we adapt
the  definition
of the gap
given by Kato
to the case
of a sequence
of
different
spaces?
More in general in which 
sense has to be understood
the convergence of operators defined in different
spaces.
\par 
A first step
toward the
generalization 
to the case of different spaces
of the classical stability 
problem 
is the following
result
Thomas G. Kurtz \cite{kurtz}.
\begin{theorem}
[$2.1.$ of \cite{kurtz}]
\label{16250603}
For 
each $n$, let $U_{n}(t)$ be a strongly
continuous contraction semigroup
defined on $L_{n}$ 
with the infinitesimal
operator $A_{n}$.
Let
$A=ex-\lim_{n\to\infty}A_{n}$.
Then there exists a strongly continuous
semigroup $U(t)$ on $L$
such that
$\lim_{n\to\infty}U_{n}(t)Q_{n}f=U(t)f$
for all $f\in L$ and $t\in\R^{+}$
if and only if
the domain $D(A)$ is dense
and
the range
$R(\lambda_{0}-A)$ of $\lambda_{0}-A$
is dense in $L$ for some $\lambda_{0}>0$.
If the above conditions hold 
$A$ is the infinitesimal generator
of $U$ and
we have
\begin{equation}
\label{19100603}
\lim_{n\to\infty}\sup_{0\leq s\leq t}
\|
U_{n}(s)Q_{n}f-
Q_{n}U(s)f
\|_{n}=0,
\end{equation}
for every $f\in L$
and $t\in\R^{+}$.
\end{theorem}
Here
$\lr{L}{\|\cdot\|}$
is a Banach space,
$\{\lr{L_{n}}{\|\cdot\|_{n}}\}_{n\in\N}$
is a sequence of Banach spaces,
$\{Q_{n}\in B(L,L_{n})\}_{n\in\N}$
such that
$\lim_{n\to\infty}\|Q_{n}f\|_{n}=\|f\|$
for all $f\in L$.
Let
$f\in L$
and
$\{f_{n}\}_{n\in\N}$
such that
$f_{n}\in L_{n}$ for every $n\in\N$,
thus he set
\footnote{
Notice the strong
similarity
of \eqref{16590603}
with \eqref{21040503}.}
\begin{equation}
\label{16590603}
f=\lim_{n\to\infty}f_{n}
\Leftrightarrow
\lim_{n\to\infty}
\|f_{n}-Q_{n}f\|_{n}=0.
\end{equation}
Moreover
if
$A_{n}:Dom(A_{n})\subseteq L_{n}\to L_{n}$
he defined
\begin{equation}
\label{16500603}
\begin{cases}
Graph(ex-lim_{n\to\infty}A_{n})
\coloneqq
\{
\lim_{n\in\N}
s_{0}(n)
\,\vert\,
s_{0}\in\Phi_{0}
\}
\\
\begin{aligned}
&
\Phi_{0}
\coloneqq
\{
(f_{n},A_{n}f_{n})_{n\in\N}
\in
(Z\times Z)^{\N}
\,\vert\,
\\
&
(\forall n\in\N)
(f_{n}\in Dom(A_{n}))
\wedge
(\exists\,
\lim_{n\in\N}(
f_{n},A_{n}f_{n}))
\},
\end{aligned}
\end{cases}
\tag{Gr}
\end{equation}
where 
$(f,g)=\lim_{n\in\N}
(f_{n},A_{n}f_{n}))
$
iff
$f=\lim_{n\in\N}f_{n}$
and
$g=\lim_{n\in\N}A_{n}f_{n}$
and all these limits
are those
defined in \eqref{16590603}.
Whenever
$Graph(ex-lim_{n\to\infty}A_{n})$
is a graph in $L$ Kurtz denoted by
$ex-lim_{n\to\infty}A_{n}$
the corresponding operator
in $L$.
\par
The Kurtz's approach did not make
use of the bundle theory,
and, 
except when 
imposing
stronger
assumptions, 
it 
cannot be 
implemented in terms of bundles of $\Omega-$spaces.
\par
What follows results
fundamental
for understanding the 
strategy behind this work.
\eqref{21040503}
essentially generalizes
\eqref{16590603}.
More importantly
\emph
{if 
the topology
on $\mf{M}$
and
that
on
$\mf{E}$
are related 
by a 
$\left(\Theta,\mc{E}\right)-$structure
(for a very simple model see \eqref{21310603})
then the convergence
\eqref{21040503}
essentially generalizes
the convergence 
\eqref{19100603}
of the sequence of semigroups
$\{U_{n}\}_{n\in\N}$
to the semigroup
$U$}.\footnote{
Indeed 
if we set
assume
that there exists for every $n\in\N$
$S_{n}\in B(L_{n},L)$
such that
$S_{n} Q_{n}= Id$
then
\eqref{19100603}
would become
\begin{equation}
\label{17520603}
(\forall t\in\R^{+})
(\forall f\in L)
(\lim_{n\to\infty}\sup_{0\leq s\leq t}
\|
(U_{n}(s)
-
Q_{n}U(s)S_{n})
Q_{n}f
\|_{n}=0).
\end{equation}
Moreover
let
$\lr{\mf{M}}{\rho,X}$
and 
$\lr{\mf{E}}{\pi,X}$
be
set as
in the beginning
and assume
that
$\{\nu_{(K,v)}^{z}
\,\vert\,(K,v)\in Comp(Y),v\in\mc{E}
\}$
is 
a fundamental set of seminorms
on $\mf{M}_{z}$ for every $z\in X$,
where
$\mc{E}\subseteq\Gamma(\pi)$.
Finally
assume that
for all
$K\in Comp(Y)$, $v\in\mc{E}$
and for all
$z\in X$
and
$f^{z}\in\mf{M}_{z}$
\begin{equation}
\label{21310603}
\nu_{(K,v)}^{z}(f^{z})
\coloneqq
\sup_{s\in K}
\|
f^{z}(s)
v(z)
\|_{z}.
\end{equation}
Thus
\eqref{21040503}
would read:
if there exists $\sigma\in\Gamma(\rho)$
such that
$\sigma(x_{\infty})=F(x_{\infty})$
then
\begin{equation}
\label{17510603}
F\in\Gamma^{x_{\infty}}(\rho)
\Leftrightarrow
(\forall K\in Comp(Y))
(\forall v\in\mc{E})
(\lim_{z\to x_{\infty}}
\sup_{s\in K}
\|(F(z)-\sigma(z))v(z)\|_{z}
=0).
\end{equation}
Therefore 
by setting 
$X$ the Alexandroff compactification
of $\N$, $x_{\infty}=\infty$
and for all 
$n\in\N$
\begin{equation}
\label{22270603}
\begin{cases}
\mf{E}_{n}
\coloneqq
L_{n},\,
\mf{E}_{\infty}\coloneqq L
\\
\mf{M}_{n}\coloneqq
\cc{c}{\R^{+},B_{s}(L_{n})}
\\
\mf{M}_{\infty}\coloneqq
\cc{c}{\R^{+},B_{s}(L)}
\\
\mc{E}
\coloneqq
\left\{
Qf
\,\vert\, f\in L
\right\},
\end{cases}
\end{equation}
if there exist conditions 
under which we can obtain that
\begin{equation}
\label{21420603}
\begin{cases}
\left\{
Qf
\,\vert\, f\in L
\right\}
\subseteq
\Gamma(\pi)
\\
\left\{
QVS
\,\vert\,
V\in\ms{U}(L)
\right\}
\subseteq
\Gamma(\rho),
\end{cases}
\end{equation}
where
$
(Qf)(n)
\coloneqq
Q_{n}f
$,
$(Qf)(\infty)
\coloneqq
f$,
while
$(QVS)(n)
\coloneqq
Q_{n}V S_{n}$,
$(QVS)(\infty)
\coloneqq
V$,
for all $n\in\N$
and
$\ms{U}(L)$,
is the set of all 
$C_{0}-$semigroup
on
$L$,
then
by
\eqref{17510603}
and
\eqref{17520603}
follows
that
$$
\mc{U}
\in
\Gamma^{\infty}(\rho),
$$
where
$\mc{U}(n)\coloneqq U_{n}$
and 
$\mc{U}(\infty)\coloneqq U$.}
\par
We used the word ``essentially'' 
due to the difficulty
to build a couple
of Kurtz' bundles, namely
two
bundles 
of $\Omega-$spaces
$\lr{\mf{E}}{\pi,X}$
and
$\lr{\mf{M}}{\rho,X}$
such that
$X$ is 
the Alexandroff compactification
of $\N$
and
\eqref{22270603},
\eqref{21420603} 
hold.
In any case it is possible under 
strong assumptions, see 
Section \ref{17572301}.
Despite the difficulty of constructing Kurtz's bundles,
since the above remark we opted to investigate to which extent
the Kurtz's Thm. \ref{16250603}
can be extended in the framework
of bundles of $\Omega-$spaces,
by using the concept of $\left(\Theta,\mc{E}\right)-$structure.
\par
It is now clear that, 
in the way of extending the Kurtz's Theorem,
we replace the 
sequence of Banach 
spaces
$\{L_{n}\}_{n\in\N\cup\{\infty\}}$
where
$L_{\infty}\coloneqq L$,
with 
a 
Banach
bundle
$\mf{E}$,
while we replace the sequence 
$\{\cc{c}{\R^{+},B_{s}(L_{n)}}
\}_{n\in\N\cup\{\infty\}}$
by the bundle
of $\Omega-$spaces $\mf{M}$.
Hence 
the Kurtz' convergences
\eqref{19100603}
and
\eqref{16590603}
will be replaced by the 
\emph
{convergences of sections
on the bundles spaces
$\mf{M}$
and
$\mf{E}$}
respectively.
In this view definition
\eqref{16500603}
has to be replaced
by that
of
Pre-Graph section
Def.
\ref{16161212bis}
(essentially
\eqref{19240703}),
while
the case in which
$Graph(ex-lim_{n\to\infty}A_{n})$
is a graph in $L$
with 
that
of
Graph section
Def.
\ref{12432110bis}.
Hence it arises as a natural question
which topology
has to be selected
for
the
bundle space
of
$\mf{V}\oplus\mf{V}$.
\par
An essential tool
used in the definition
of 
$Graph(ex-lim_{n\to\infty}A_{n})$
in \eqref{16500603}
is that of convergence
of a sequence $(f_{n},A_{n}f_{n})$
in the direct sum
of the spaces
$L_{n}\oplus L_{n}$,
given by construction
as the convergence
of both the sequences in
$L_{n}$
in the meaning of
\eqref{16590603}.
\par
It is exactly this factorization the property
which we want to preserve when selecting the 
right topology
on the bundle space
of
$\mf{V}\oplus\mf{V}$.
\par
It is a well-known result
the
solution of this problem
in the special case of Banach bundles.
We  generalize this result for a finite
direct sum of general bundles of $\Omega-$spaces,
by constructing
in
\textbf{Thm. \ref{16322110}}
a directed family of seminorms
on the direct sum of
Hausdorff locally convex spaces
that generates the product topology.
\par
This result along with 
\textbf{Lemma \ref{19420111}}
allow
to define the direct sum of
(full) bundles
of $\Omega-$spaces as given in
Def.
\ref{14442410}
\par
\emph
{The result that the topology
on each stalk is the product topology, encoded in \eqref{18260603},
the choice provided in
\eqref{18270603}
of a set that will become a subset of bounded continuous sections
of the direct sum of bundles
and the general
convergence
criterium
in \eqref{21040503},
allow to show the
claimed factorization
property
in 
\textbf{Cor. \ref{17571212}}.
Namely any continuous map from $X$ at values in 
the direct sum $\bigoplus_{i=1}^{n}\mf{E}_{i}$ of bundles 
is continuous at a point 
if and only if all its $n$ components 
are continuous at the same point.} 
\par
In \textbf{Thm. \ref{17301812b}}
we resolve the claim of extending the Kurtz's result to the setting of bundles of $\Omega-$spaces.
More exacly we construct an element of the set
$\Delta_{\Theta}\lr{\mf{V},\mf{W}}
{\mc{E},X,\R^{+}}$ \textbf{Def. \ref{19490412bis}}.
Roughly and limited to singletons we have that the singleton 
$\{\lr{\mc{T}}{x_{\infty},\Phi}\}$ belongs to 
$\Delta_{\Theta}\lr{\mf{V},\mf{W}}
{\mc{E},X,\R^{+}}$
iff
$\mc{T}(x)$
is
the graph 
of  
the infinitesimal 
generator
$T_{x}$
of a
$C_{0}-$semigroup
$\mc{U}(x)$
on $\mf{E}_{x}$,
for all $x\in X$,
\eqref{19240703} holds true
and
\begin{equation}
\label{05211905}
\mc{U}\in\Gamma^{x_{\infty}}(\rho).
\end{equation}
Thus, according to the discussed way of extending the Kurtz' theorem,
to find
an
element
in the
set
$\Delta_{\Theta}\lr{\mf{V},\mf{W}}{\mc{E},X,\R^{+}}$
means
to find
an extension
of Thm. 
\ref{16250603}.
\par
Finally let us outline how the main result of the entire work
\textbf{Thm. \ref{13020103}}
extends the classical stability problem at operators defined in different spaces. 
It provides
the 
existence
of
an
element
$\lr{\mc{T}}{\Phi,x_{\infty}}$
whose singleton 
belongs to the intersection of the set
$\Delta_{\Theta}\lr{\mf{V},\mf{W}}{\mc{E},X,\R^{+}}$
with the set
$\Delta
\lr{\mf{V},\mf{D}}
{\Theta,\mc{E}}$
which ammounts to what follows.
There exists $\mc{U}$ satisfying \eqref{05211905} 
and there exists a section
\begin{equation}
\label{17321002}
\mc{P}
\in
\Gamma^{x_{\infty}}(\eta),
\end{equation}
satisfying \eqref{20020803} with $T_{x}$ 
the infinitesimal 
generator
of the 
$C_{0}-$semigroup
$\mc{U}(x)$
for all $x\in X$.
Actually the result is stronger since it establishes that 
$\mc{P}(x)$ is a \emph{spectral projector} of $T_{x}$ for all $x\in X$.
\par
Roughly speaking and limited to singletons we have what follows,
see \textbf{Def. \ref{15312011bis}} for the precise and general definition.
Given a
$\left(\Theta,\mc{E}\right)-$structure
$\lr{\mf{V},\mf{D}}{X,\{pt\}}$
and denoted
$\mf{D}
\coloneqq
\lr{\mf{B}}{\eta,X}$,
we have that the singleton of $\lr{\mc{T}}{\Phi,x_{\infty}}$ belongs to 
$\Delta
\lr{\mf{V},\mf{D}}
{\Theta,\mc{E}}$
iff
for all $x\in X$
the set
$\mc{T}(x)$
is a graph
in $\mf{E}_{x}$,
\eqref{19240703} holds true
and there exists 
$\mc{P}
\in
\Gamma^{x_{\infty}}(\eta)$
such that
$\mc{P}(x)$
is a projector 
on $\mf{E}_{x}$
for all $x\in X$
and 
\begin{equation}
\label{20020803}
\mc{P}(x)
T_{x}
\subseteq
T_{x}
\mc{P}(x),
\end{equation}
where
$T_{x}$ is the 
operator
in $\mf{E}_{x}$
whose
graph
is
$\mc{T}(x)$.
\par
In others words
$\lr{\mc{T}}{\Phi,x_{\infty}}
\in
\Delta
\lr{\mf{V},\mf{D}}
{\Theta,\mc{E}}$
iff
$\mc{T}$
is a 
section
of graphs
in $\mf{E}$
continuous
at $x_{\infty}$
in the sense of
\eqref{16490703}
and such that
there exists
a 
section
$\mc{P}$
of 
projectors
on $\mf{E}$
continuous
at $x_{\infty}$
such that
\emph
{$\mc{P}$
commutes
with
$\mc{T}$}
in the meaning
of
\eqref{20020803}.
\par
Notice that \eqref{20020803}
is satisfied by any element of the resolution of the identity of 
a spectral operator \cite[Def. 18.2.1]{ds}.
Moreover whenever
$T_{x}$
is the infinitesimal generator
of a
$C_{0}-$semigroup
$\mc{W}_{T}(x)$
of contractions
on $\mf{E}_{x}$,
the most important
case in this work,
it results
that
\eqref{20020803}
is the property
satisfied
by all
the 
spectral projectors
of
the form
$$
\mc{P}(x)
\coloneqq
-
\frac{1}{2\pi i}
\int_{\Gamma}
R(-T_{x};\zeta)\,
d\zeta,
$$
where
$R(-T_{x};\zeta)$
is the resolvent
map
of the 
operator
$-T_{x}$
and $\Gamma$
is a suitable 
closed
curve 
on the complex plane.
Hence
we can consider
the commutation
in
\eqref{20020803}
as the 
defining
property
of what we here consider
as 
the
interesting
bundle
$\mc{P}$
of projectors
associated with
$\mc{T}$.
Therefore
\emph
{as \eqref{05211905} represents the extension of the Kurtz's theorem 
so \eqref{17321002} 
realizes our initial claim to extend in the framework of bundles of $\Omega-$spaces 
the classical stability problem. 
Moreover the two solutions $\mc{U}$ and $\mc{P}$ are correlated 
since $\mc{P}(x)$ is a spectral projector of the infinitesimal generator $T_{x}$
of the semigroup $\mc{U}(x)$ for all $x\in X$, in particular \eqref{20020803} holds true.}
\clearpage
The
main results of this work are the
following ones
\begin{enumerate}
\item
Construction of a suitable
directed fundamental set
of seminorms of the
topological
direct
sum of a finite family of
Hausdorff locally convex spaces,
and construction of 
$\mc{E}^{\oplus}$
satisfying
$FM(3)-FM(4)$
with respect
to 
$\ms{E}^{\oplus}$
(Thm. \ref{16322110} and Lemma \ref{19420111});
\item
Factorization property of the convergence in any direct sum of bundles of $\Omega-$spaces
(Cor. \ref{17571212});
\item
Characterization
of 
sections
of
$\mf{W}$
continuous
at a point
when
$\lr{\mf{V},\mf{W}}{X,Y}$
is
a
$\left(\Theta,\mc{E}\right)-$structure,
(Lemma \ref{15482712});
\item
Construction
of a
$\left(\Theta,\mc{E}\right)-$structure
$\lr{\mf{V},\mf{W}}{X,Y}$
and characterization
of a subset
of $\Gamma^{x_{\infty}}(\rho)$
when
$\mf{V}$
is trivial,
(Thm. \ref{22372406});
\item
Construction
of an element in the
set
$\Delta_{\Theta}\lr{\mf{V},\mf{W}}
{\mc{E},X,\R^{+}}$,
(\textbf{Thm. \ref{17301812b}},
Cor. \ref{21343012}
\item
Conditions yielding the 
bounded equicontinuity of which in
hypothesis $(ii)$
of
Thm. \ref{17301812b}
(Cor. \ref{21343012});
\item
Conditions yielding the 
\eqref{18470109}
(Prp.\ref{18390901});
\item
The technical
Lemma \ref{11011501}
and
Thm. \ref{15251401};
\item
\textbf{Thm. \ref{15332203}}
Cor. \ref{15111901}
and 
Cor. \ref{18491004};
\item
$\mc{K}-$Uniform Convergence
\textbf{Thm. \ref{10581004}};
\item
Consequence
of
being
an
$\lr{\nu,\eta}{E,Z,T}$
invariant set
$V$
with respect to
$\mc{F}$
(Prp. \ref{18310302});
\item
Construction
of
a set
$\Delta_{\Theta}\lr{\mf{V},\mf{D},\mf{W}}
{\mc{E},X,\R^{+}}$
by using
an
$\lr{\nu,\eta}{\mf{G},K(\Gamma),\R^{+}}$
invariant set
$V$
with respect to
$\{\ov{F}_{T}\}$
(Cor. \ref{13511102});
\item
A bundle
version
of the Lebesgue theorem
for
a
$\mu-$related
couple
$\lr{\mf{V}}{\mf{Z}}$
(Thm. \ref{15101701});
\item
Technical 
Lemma \ref{12151902}
and 
Lemma \ref{14452602};
\item
Cor. \ref{15262502};
\item
Construction of
a section of spectral
projectors
continuous at a point,
given
a section of semigroups
continuous at
the same point
(Cor. \ref{21152602});
\item
The Main result of the entire work
namely the construction
of
an element in the set
$\Delta\lr{\mf{V},\mf{D}}{\Theta,\mc{E}}$
(\textbf{Thm. \ref{13020103}}).
\end{enumerate}
The main structures defined in this work
are the following ones
\begin{enumerate}
\item
Direct sum of full bundles of
$\Omega-$spaces
(Def. \ref{14442410});
\item
(Invariant)
$\left(\Theta,\mc{E}\right)-$structure
$\lr{\mf{V},\mf{W}}{X,Y}$,
(Def. \ref{10282712});
\item
Graph section
$\lr{\mc{T}}{x_{\infty},\Phi}$,
(Def. \ref{12432110bis});
\item
$\Delta\lr{\mf{V},\mf{D}}
{\Theta,\mc{E}}$,
(Def. \ref{15312011bis});
\item
$\Delta_{\Theta}\lr{\mf{V},\mf{W}}
{\mc{E},X,\R^{+}}$,
(Def. \ref{19490412bis});
\item
$\Delta_{\Theta}\lr{\mf{V},\mf{D},\mf{W}}
{\mc{E},X,\R^{+}}$;
(Def. \ref{18550612bis});
\item
$\lr{\mf{V},\mf{W}}{X,\R^{+}}$
with the 
Laplace duality property,
(Def. \ref{16401812b});
\item
$\ms{U}-$Spaces
(Def. \ref{14302503});
\item
The locally convex space
$\mf{G}$
(Def. \ref{10221801});
\item
$\lr{\nu,\eta}{E,Z,T}$
invariant set
$V$
with respect to
$\mc{F}$
(Def. \ref{20030202});
\item
$\mu-$related
couple
$\lr{\mf{V}}{\mf{Z}}$
(Def. \ref{15492502});
\item
(Invariant)
$\left(\Theta,\mc{E},\mu\right)-$
structure
$\lr{\mf{V},\mf{Q}}{X,Y}$
(Def. \ref{17161902});
\item
$\left(\Theta,\mc{E}\right)-$
structure
$\lr{\mf{V},\mf{V}(\ms{M},\Gamma(\xi))}{X,Y}$
underlying a
$\left(\Theta,\mc{E},\mu\right)-$
structure
$\lr{\mf{V},\mf{Q}}{X,Y}$
(Def. \ref{18072802}).
\end{enumerate}
\section{Notation}
\label{notat}
For any two sets $X,Y$ we let $Y^{X}$ denote the set of maps defined on $X$ and at values in $Y$.
Let $Graph(X\times Y)$ denote the set of subsets of $X\times Y$ which are graphs,
while for any map $f$ let $Graph(f)$ denote its graph.
If $\mc{B}$ is a base of a filter on $X$, we let $\mf{F}_{\mc{B}}^{X}$ denote the filter on $X$
generated by the base $\mc{B}$.
If $S$
is any set
then
$\p_{\omega}(S)$
denotes the 
set
of all 
finite subsets
of $S$.
If $\tau$ is any topology on $X$ and $x\in X$, then 
$\mc{I}_{x}^{\tau}$ denotes the filter of neighbourhoods of $x$ 
of the topological space $\lr{X}{\tau}$.
Let u.s.c. mean upper semicontinuous.
All vector spaces are assumed to be over
$\K\in\{\R,\C\}$,
Hlcs stands for 
Hausdorff locally convex spaces.
We say
that
$
\ms{V}
\coloneqq
\{\lr{V_{x}}{\A_{x}}\}_{x\in X}
$
is a 
\emph{nice}
family
of Hlcs
if
$\{V_{x}\}_{x\in X}$
is a family
of Hlcs
and there exists a set $J$ for which
$\forall x\in X$ the set
$
\A_{x}
\coloneqq
\{
\mu_{j}^{x}
\}_{j\in J}
$
is a 
directed
\footnote{
I.e.
$
(\forall j_{1},j_{2}\in J)
(\exists\,j\in J)
(\mu_{j_{1}}^{x},\mu_{j_{1}}^{x}
\leq
\mu_{j}^{x})
$
with
the 
order
relation
of
pointwise
order
on
$
\R^{V_{x}}
$.
}
family of seminorms
on
$V_{x}$ generating the 
locally convex topology on it.
For any 
family
of seminorms
$K$
on a vector space $V$
we call
the 
directed
family
of 
seminorms
associated with
$K$
the 
set
$
\{
\sup F
\,\vert\,
F\in\p_{\omega}(\Gamma)
\}
$
with
the 
order
relation
of
pointwise
order
on
$\R^{V}$.
fss stands for ``fundamental set of seminorms''. 
Given
two locally convex spaces (lcs)
$E$
and
$F$
we denote
by
$\mc{L}(E,F)$
the linear
space of all 
linear and continuous maps
on $E$ with values in $F$,
and set
$
\mc{L}(E)
\coloneqq
\mc{L}(E,E)
$,
moreover
let
$\Pr(E)
\coloneqq
\{P\in\mc{L}(E)\vert P\circ P =P\}$
denote
the set of all continuous projectors on $E$.
Let
$S$
be a 
set of
bounded
subsets
of a lcs
$E$,
thus
$
\mc{L}_{S}(E)
$
denotes
the 
lcs
whose 
underlining
linear
space
is
$\mc{L}(E)$
and whose
locally convex topology
is that
of 
uniform convergence
over
the subsets
in $S$.
When $E$ is a normed space
and $S$ is the set of all finite
parts of $E$, then
$\mc{L}_{S}(E)$ 
will be denoted
by $B_{s}(E)$,
while
$B(E)$
denotes
$\mc{L}(E)$ 
with the usual 
norm topology.
Let
$\{E_{i}\}_{i\in I}$
a 
family of lcs. 
Then
we 
denote
by
$\tau_{0}$,
$\tau_{b}$,
$\tau_{l}$,
$
\tau_{\mf{l}}
$
the 
topology
on
$\bigoplus_{i\in I}E_{i}$
induced 
by the product 
topology
on
$\prod_{i\in I}E_{i}$,
that
induced 
by the box
topology
on
$\prod_{i\in I}E_{i}$
(see \cite{jar}),
the 
direct sum
topology,
Ch. $4$, $\S 3$ of
\cite{jar}
and
the 
lc-direct sum topology
Ch. $6$, $\S 6$ of
\cite{jar}
respectively.
\par
Let
$X,Y$
be 
two 
topological spaces
then $Comp(X)$ is the set of all compact subsets of $X$,
while $\cc{}{X,Y}$
is
the 
set
of 
all continuous maps
on $X$ valued in $Y$,
while
$\cc{c}{X,Y}$
is the topological 
space
of 
all continuous maps
on $X$ valued in $Y$
with the topology of
uniform convergence
over the compact subsets
of $Y$.
If $Y$ is a uniform
space
then
$\mc{C}^{b}(X,Y)$
is the space
of all bounded
maps
in
$\cc{}{X,Y}$,
while
$\mc{C}_{c}^{b}(X,Y)
=
\cc{c}{X,Y}\cap
\mc{C}^{b}(X,Y)$.
If 
$E$ is a lcs then
$\cc{c}{X,E}$
is a lcs,
while
if
$E$ is a Hlcs 
and
$Comp(X)$
is a covering
of $X$,
for example if $X$
is a locally compact space,
then
$\cc{c}{X,E}$
is a Hlcs.
Let 
$Y$ be a locally compact space,
$\mu\in Radon(Y)$
and $E\in Hlcs$,
then
$\mf{L}_{1}(Y,E,\mu)$
denotes the linear space
of all scalarly essentially
$\mu-$integrable maps
$f:Y\to E$
such that
its integral
belongs to $E$,
see 
\cite[Ch. $6$]{IntBourb},
while
$Meas(Y,E,\mu)$
denotes
the linear space of all 
$\mu-$measurable maps
$f:Y\to E$.
Let $E$ be a topological vector space,
and $\lr{\mc{L}(E)}{\tau}$ the 
topological vector space whose underlying linear space is 
$\mc{L}(E)$ provided by the topology $\tau$. 
Thus
$\ms{U}(\lr{\mc{L}(E)}{\tau})$
is the set of all
continuous 
semigroup 
morphisms
defined on $\R^{+}$
and with values
in $\lr{\mc{L}(E)}{\tau}$. 
Moreover if $\|\cdot\|$
is any seminorm
on $\mc{L}(E)$
(not necessarly continuous
with respect to $\tau$)
we set
$\ms{U}_{\|\cdot\|}
(\lr{\mc{L}(E)}{\tau})$
as
the 
subset
of
all $U\in\ms{U}(\lr{\mc{L}(E)}{\tau})$
such that
$\|U(s)\|\leq 1$, for all $s\in\R^{+}$.
Let
$\ms{U}_{is}
(\lr{\mc{L}(E)}{\tau})$
be 
the 
subset
of all
$U\in\ms{U}(\lr{\mc{L}(E)}{\tau})$
such that there exists a fundamental set of seminorms $K$ 
on $E$ such that
$U(s)$ is an isometry
with respect to any
element in $K$,
for all $s\in\R^{+}$.
We use throughout this work
the notation of \cite{gie}
and often when referring to
Banach bundles
those of \cite{fell}.
In particular
$
\lr{\lr{\mf{E}}{\tau}}{p,X,\n}
$
or simply
$
\lr{\mf{E}}{p,X}
$,
whenever 
$\tau$
and
$\n$
are 
known,
is 
a bundle
of $\Omega-$spaces
($1.5.$ of \cite{gie}),
where
we denote
by
$\tau$
the topology
on $\mf{E}$
while
with
$
\n
\coloneqq
\{
\nu_{j}\,\vert\,
j\in J
\}
$
the directed
set of 
seminorms
on
$\mf{E}$
($1.3.$ of \cite{gie}).
Thus
we 
set
$
\n_{x}
\coloneqq
\{
\nu_{j}^{x}
\,\vert\,
j\in J
\}
$
with
$\nu_{j}^{x}
\coloneqq
\nu_{j}\up\mf{E}_{x}$
and
$\mf{E}_{x}
\coloneqq
\overset{-1}{p}(x)$,
for all $x\in X$ and $j\in J$.
Moreover
for any $U\subseteq X$
we call $\Gamma_{U}(p)$
the space of bounded continuous sections
of
$\lr{\lr{\mf{E}}{\tau}}{p,X,\n}$
on $U$
defined by 
$$
\Gamma_{U}(p)
\coloneqq
\cc{}{U,\mf{E}}
\bigcap
\prod_{x\in U}^{b}
\lr{\mf{E}_{x}}{\n_{x}}
$$
where
$$
\prod_{x\in U}^{b}
\lr{\mf{E}_{x}}{\n_{x}}
\coloneqq
\bigl\{
\sigma\in
\prod_{x\in U}
\mf{E}_{x}
\,\vert\,
(\forall j\in J)
(\sup_{x\in U}
\nu_{j}^{x}
(\sigma(x))
<\infty)
\bigr\}.
$$
Let $U\subseteq X$
and $x\in U$
set
$$
\Gamma_{U}^{x}(p)
\coloneqq
\bigl\{
f\in
\prod_{x\in U}^{b}
\lr{\mf{E}_{x}}{\n_{x}}
\,\vert\,
f\text{ is continuous at $x$}
\bigr\}.
$$
So
$\Gamma_{U}(p)
=
\bigcap_{x\in U}
\Gamma_{U}^{x}(p)$.
We set
$\Gamma(p)
\coloneqq
\Gamma_{X}(p)$
and 
$\Gamma^{x}(p)
\coloneqq
\Gamma_{X}^{x}(p)$
for any $x\in X$.
The definition of 
trivial
bundle
of $\Omega-$spaces
is given
in
$1.8.$ of \cite{gie}.
Whenever we mention the properties $FM(3),FM(4)$ 
we always mean those provided in 
\cite[$\S 5$]{gie} and recalled in Def. \ref{17471910s}.
If 
$
\mf{A}
\coloneqq
\lr{\lr{\mf{B}}{\tau}}
{\xi,X,\n}
$
is 
a 
bundle of $\Omega-$spaces,
$x\in X$
and
$Q,S$ are subsets of
$\prod_{z\in X}\mf{B}_{z}$,
we set
\begin{equation}
\label{15012602}
\begin{aligned}
Q_{S}^{x}
&\coloneqq
\{
H\in Q\,\vert\,
(\exists\,F\in S)
(H(x)=F(x))
\},
\\
Q_{\diamond}^{x}
&\coloneqq
Q_{\Gamma(\xi)}^{x},
\\
\Gamma_{S}^{x}(\xi)
&\coloneqq
(\Gamma^{x}(\xi))_{S}^{x},
\\
\Gamma_{\diamond}^{x}(\xi)
&\coloneqq
(\Gamma^{x}(\xi))_{\diamond}^{x}.
\end{aligned}
\end{equation}
\section{Continuous sections of bundles of $\Omega-$spaces}
In this section we provide simple but helpful results concerning 
convergence in bundles of $\Omega-$spaces and more specifically 
characterizations of the continuity of sections at a certain point.
\begin{proposition}
\label{28111555}
Let
$\mf{V}
=
\lr{\lr{\mf{E}}{\tau}}{\pi,X,\n}$
be
a bundle
of $\Omega-$spaces
where
$
\n
\coloneqq
\{
\nu_{j}\,\vert\,
j\in J
\}
$.
Moreover
let
$b\in\mf{E}$
and
$\{b_{\alpha}\}_{\alpha\in D}$
a net in $\mf{E}$.
Then
$(1)
\Leftarrow
(2)
\Leftarrow
(3)
\Leftrightarrow
(4)$
where
\begin{enumerate}
\item
$\lim_{\alpha\in D}b_{\alpha}=b$;
\item
$
(\exists\,U\in Op(X)\,\vert\, U\ni\pi(b))
(\exists\,\sigma\in\Gamma_{U}(\pi))
(\sigma\circ\pi(b)=b)
$
such that
$ 
\lim_{\alpha\in D}\pi(b_{\alpha})
=\pi(b)
$
and
$
(\forall j\in J)
(\lim_{\alpha\in D}
\nu_{j}(b_{\alpha}-\sigma(\pi(b_{\alpha})))
=0)
$;
\item
$
(\exists\,U'\in Op(X)\,\vert\, U'\ni\pi(b))
(\exists\,\sigma'\in\Gamma_{U}(\pi)\,\vert\,
\sigma'\circ\pi(b)=b)
$
and
$
(\forall U\in Op(X)\,\vert\, U\ni\pi(b))
(\forall\sigma\in\Gamma_{U}(\pi)\,\vert\,
\sigma\circ\pi(b)=b)
$
we have
$ 
\lim_{\alpha\in D}\pi(b_{\alpha})
=\pi(b)
$
and
$
(\forall j\in J)
(\lim_{\alpha\in D}
\nu_{j}(b_{\alpha}-\sigma(\pi(b_{\alpha})))
=0)
$;
\item
$
(\exists\,U'\in Op(X)\,\vert\, U'\ni\pi(b))
(\exists\,\sigma'\in\Gamma_{U}(\pi))
(\sigma'\circ\pi(b)=b)
$
and
$\lim_{\alpha\in D}b_{\alpha}=b$.
\end{enumerate}
Moreover
if
$\mf{V}$
is locally full
then
$(1)\Leftrightarrow(4)$.
\end{proposition}
\begin{proof}
Clearly $(3)\Rightarrow(2)$.
$(2)$
is equivalent
to say that
$
(\exists\,U\in Op(X)\,\vert\, U\ni\pi(b))
(\exists\,\sigma\in\Gamma_{U}(\pi))
(\sigma\circ\pi(b)=b)
$
such that
$
(\forall V\in Op(X)\,\vert\, 
\pi(b)\in
V\subseteq U)
(\exists\,\alpha(V)\in D)
(\forall\alpha\geq\alpha(V))
(\pi(b_{\alpha})
\in V)
$
and
$
(\forall j\in J)
(\forall\ep>0)
(\exists\,\alpha(V)\in D)
(\forall\alpha\geq\alpha(j,\ep))
(\nu_{j}(b_{\alpha}-\sigma(\pi(b_{\alpha})))
<\ep)
$.
Set
$
\alpha(V,j,\ep)\in D
$
such that
$
\alpha(V,j,\ep)
\geq
\alpha(V),
\alpha(j,\ep)
$
which there exists
$D$ being directed,
thus
we have
$
(\forall V\in Op(X)\,\vert\, 
\pi(b)\in
V\subseteq U)
(\forall j\in J)
(\forall\ep>0)
(\exists\,\alpha(V,j,\ep)\in D)
$
such that
$(\forall\alpha\geq\alpha(V,j,\ep))
(\nu_{j}(b_{\alpha}-\sigma(\pi(b_{\alpha})))
<\ep)$
and
$\pi(b_{\alpha})
\in V$.
Thus
$(1)$
follows
by
applying $1.5.\,VII$
of \cite{gie}.
Finally
by
applying $1.5.\,VII$
of \cite{gie}
$(4)$
(respectively
$(1)$
if
$\mf{V}$
is locally full)
is equivalent
to
$
(\exists\,U'\in Op(X)\,\vert\, U'\ni\pi(b))
(\exists\,\sigma'\in\Gamma_{U}(\pi))
(\sigma'\circ\pi(b)=b)
$
and
$
(\forall\sigma\in\Gamma_{U}(\pi)\,\vert\,
\sigma\circ\pi(b)=b)
(\forall j\in J)
(\forall\ep>0)
(\forall V\in Op(X)\,\vert\, 
\pi(b)\in
V\subseteq U)
(\exists\,\ov{\alpha}\in D)
(\forall\alpha\geq\ov{\alpha})
$
we have
$
\pi(b_{\alpha})
\in V
$
and
$
\nu_{j}(b_{\alpha}-\sigma(\pi(b_{\alpha})))
<\ep
$
which is
$(3)$.
\end{proof}
\begin{theorem}
\label{15380512}
Let
$\mf{V}=
\lr{\lr{\mf{E}}{\tau}}{\pi,X,\n}$
be
a bundle
of $\Omega-$spaces,
$W\subseteq X$
and 
indicate
$
\n=
\{
\nu_{j}
\,\vert\,
j\in J
\}
$.
Moreover
let
$f\in\mf{E}^{W}$,
$x_{\infty}\in W$.
Then
$
(1)
\Leftarrow
(2)
\Leftrightarrow
(3)
\Leftarrow
(4)
\Leftrightarrow
(5)
\Leftrightarrow
(6)
$
where
\begin{enumerate}
\item
$f$ is continuous in $x_{\infty}$;
\item
$
(\exists\,U\in Op(X)\,\vert\, U\ni x_{\infty})
(\exists\,\sigma\in\Gamma_{U}(\pi))
(\sigma(x_{\infty})
=f(x_{\infty}))$
such that
$
\nu_{j}\circ(f-\sigma\circ\pi\circ f):
W\cap U\to\R
$
and
$\pi\circ f:W\to X$
are 
continuous
in
$x_{\infty}$
for all
$j\in J$;
\item
$\pi\circ f:W\to X$
is
continuous
in
$x_{\infty}$
and
$
(\exists\,U\in Op(X)\,\vert\, U\ni x_{\infty})
(\exists\,\sigma\in\Gamma_{U}(\pi))
(\sigma(x_{\infty})
=f(x_{\infty}))$
such that
$$
(\forall j\in J)
(\lim_{
y\to x_{\infty},y\in W\cap U}
\nu_{j}(f(y)-\sigma
\circ\pi\circ f(y))=0);
$$
\item
$
(\exists\,U'\in Op(X)\,\vert\, U'\ni x_{\infty})
(\exists\,\sigma'\in\Gamma_{U}(\pi))
(\sigma'(x_{\infty})
=f(x_{\infty}))$
and
$
(\forall U\in Op(X)\,\vert\, U\ni x_{\infty})
(\forall\sigma\in\Gamma_{U}(\pi)
\,\vert\,
\sigma(x_{\infty})
=f(x_{\infty}))$
we have
$
\nu_{j}\circ(f-\sigma):
W\cap U\to\R
$
and
$\pi\circ f:W\to X$
are 
continuous
in
$x_{\infty}$
for all
$j\in J$;
\item
$\pi\circ f:W\to X$
is
continuous
in
$x_{\infty}$
and
$
(\exists\,U'\in Op(X)\,\vert\, U'\ni x_{\infty})
(\exists\,\sigma'\in\Gamma_{U}(\pi))
(\sigma'(x_{\infty})
=f(x_{\infty}))$
and
$
(\forall U\in Op(X)\,\vert\, U\ni x_{\infty})
(\forall\sigma\in\Gamma_{U}(\pi)
\,\vert\,
\sigma(x_{\infty})
=f(x_{\infty}))$
we have
$$
(\forall j\in J)
(\lim_{
y\to x_{\infty},y\in W\cap U}
\nu_{j}(f(y)-\sigma\circ\pi\circ f(y))=0);
$$
\item
$
(\exists\,U'\in Op(X)\,\vert\, U'\ni x_{\infty})
(\exists\,\sigma'\in\Gamma_{U}(\pi))
(\sigma'(x_{\infty})
=f(x_{\infty}))$
and
$f$ is continuous at
$x_{\infty}$.
\end{enumerate}
Moreover
if 
$\mf{V}$
is locally full
then
$(1)\Leftrightarrow(6)$
and if
it is full we can choose
$U=X$ 
and
$U'=X$. 
\end{theorem}
\begin{proof}
$(1)$
is equivalent
to say that
for each
net
$
\{x_{\alpha}\}_{\alpha\in D}
\subset
W
$
such that
$\lim_{\alpha\in D}x_{\alpha}= x_{\infty}$
in $W$,
we have
$\lim_{\alpha\in D}f(x_{\alpha})
=f( x_{\infty})$
in $\mf{E}$.
Similarly
$(2)$
is equivalent
to say
that
for each
net
$
\{x_{\alpha}\}_{\alpha\in D}
\subset
W
$
such that
$\lim_{\alpha\in D}x_{\alpha}= x_{\infty}$
in $W$,
we have
$
\lim_{\alpha\in D}\pi\circ f(x_{\alpha})
=
\pi\circ f(x_{\infty})
$
and
$
(\forall j\in J)
(\lim_{\alpha\in D}
\nu_{j}\circ(f-\sigma\circ\pi\circ f)
(x_{\alpha})
=
\nu_{j}\circ(f-\sigma\circ\pi\circ f)
(x_{\infty}))
$.
Thus
$(1)
\Leftarrow
(2)$
follows
by the corresponding
one
in
Prp.
\ref{28111555}
with the positions
$
(\forall\alpha\in D)
(b_{\alpha}\coloneqq f(x_{\alpha}))
$
and
$b\coloneqq f(x_{\infty})$.
Similarly
$(1)\Leftarrow(5)$
follows
by
$(1)\Leftarrow(3)$
of
Prp.
\ref{28111555}.
Finally
$(5)\Rightarrow(6)$
follows by 
$(5)\Rightarrow(1)$,
while
if $(6)$ is true then
$\pi\circ f$ is continuous
at $x_{\infty}$
indeed $\pi$ is continuous,
then $(5)$
follows by the implication
$(4)\Rightarrow(3)$
of
Prp.
\ref{28111555}
with the positions
$
(\forall\alpha\in D)
(b_{\alpha}\coloneqq f(x_{\alpha}))
$
and
$b\coloneqq f(x_{\infty})$.
\end{proof}
\begin{corollary}
\label{28111707}
Let
$\mf{V}
=
\lr{\lr{\mf{E}}{\tau}}{\pi,X,\n}$
be
a bundle
of $\Omega-$spaces,
$W\subseteq X$
and indicate
$
\n=
\{
\nu_{j}\,\vert\,
j\in J
\}
$.
Moreover
let
$f\in\prod_{x\in W}\mf{E}_{x}$
and
$x_{\infty}\in W$.
Then
$
(1)
\Leftarrow
(2)
\Leftrightarrow
(3)
\Leftarrow
(4)
\Leftrightarrow
(5)
\Leftrightarrow
(6)
$
where
\begin{enumerate}
\item
$f$ is continuous in $x_{\infty}$;
\item
$(\exists\,U\in Op(X)\,\vert\, U\ni x_{\infty})
(\exists\,\sigma\in\Gamma_{U}(\pi))
(\sigma(x_{\infty})
=f(x_{\infty}))$
such that
$
\nu_{j}\circ(f-\sigma):
W\cap U\to\R
$
is
continuous
in
$x_{\infty}$
for all
$j\in J$;
\item
$(\exists\,U\in Op(X)\,\vert\, U\ni x_{\infty})
(\exists\,\sigma\in\Gamma_{U}(\pi))
(\sigma(x_{\infty})
=f(x_{\infty}))$
such that
$$
(\forall j\in J)
(\lim_{
y\to x_{\infty},y\in W\cap U}
\nu_{j}(f(y)-\sigma(y))=0);
$$
\item
$(\exists\,U'\in Op(X)\,\vert\, U'\ni x_{\infty})
(\exists\,\sigma'\in\Gamma_{U}(\pi))
(\sigma'(x_{\infty})
=f(x_{\infty}))$
and
$
(\forall U\in Op(X)\,\vert\, U\ni x_{\infty})
(\forall\sigma\in\Gamma_{U}(\pi)
\,\vert\,
\sigma(x_{\infty})
=f(x_{\infty}))$
we have
that
$
\nu_{j}\circ(f-\sigma):
W\cap U\to\R
$
is
continuous
in
$x_{\infty}$
for all
$j\in J$;
\item
$(\exists\,U'\in Op(X)\,\vert\, U'\ni x_{\infty})
(\exists\,\sigma'\in\Gamma_{U}(\pi))
(\sigma'(x_{\infty})
=f(x_{\infty}))$
and
$
(\forall U\in Op(X)\,\vert\, U\ni x_{\infty})
(\forall\sigma\in\Gamma_{U}(\pi)
\,\vert\,
\sigma(x_{\infty})
=f(x_{\infty}))$
we have
$$
(\forall j\in J)
(\lim_{
y\to x_{\infty},y\in W\cap U}
\nu_{j}(f(y)-\sigma(y))=0).
$$
\item
$(\exists\,U'\in Op(X)\,\vert\, U'\ni x_{\infty})
(\exists\,\sigma'\in\Gamma_{U}(\pi))
(\sigma'(x_{\infty})
=f(x_{\infty}))$
and
$f$ is continuous at $x_{\infty}$
\end{enumerate}
If
$\mf{V}$
is locally full
then
$(1)
\Leftrightarrow
(6)$
and
if
it
is
full
we can choose
$U=X$ and
$U'=X$.
\end{corollary}
\begin{proof}
By Thm. \ref{15380512}
and 
$\pi\circ f=Id$.
\end{proof}
\begin{proposition}
\label{16572003}
Let
$\mf{V}$
be
full
and
such that
there exists
a linear
space $E$
such that
for all
$x\in X$
there exists
a linear subspace
$E_{x}\subseteq E$
such that
$\mf{E}_{x}
=
\{x\}\times E_{x}$,
and that
\footnote{
An example 
is when
$\mf{V}$
is the trivial bundle.}
$$
\{
\mf{t}_{v}:X\ni x
\mapsto(x,v)\in\mf{E}_{x}
\,\vert\, v\in
\bigcap_{x\in X}
E_{x}
\}
\subset
\Gamma(\pi),
$$
If
$f_{0}\in\prod_{x\in X}E_{x}$
and
$f\in\prod_{x\in X}\mf{E}_{x}$
such that
$f(x)=(x,f_{0}(x))$
for all $x\in X$
and
$f_{0}(x_{\infty})\in
\bigcap_{x\in X}
E_{x}$,
then
$(1)
\Leftrightarrow(2)
\Leftrightarrow(3)$,
where
\begin{enumerate}
\item
$f$ is continuous at $x_{\infty}$
\item
$(\exists\,U\in Op(X)\,\vert\, U\ni x_{\infty})
(\exists\,\sigma\in\cc{b}{U,E})
(\sigma(x_{\infty})
=f(x_{\infty}))$
such that
for all $j\in J$
$$
\lim_{z\to x_{\infty},z\in W\cap U}
\nu_{j}^{z}(f(z)-\sigma(z))=0;
$$
\item
for all $j\in J$
$$
\lim_{z\to x_{\infty},z\in W\cap U}
\nu_{j}^{z}((z,f_{0}(z))-
(z,f(x_{\infty})))=0.
$$
\end{enumerate}
\end{proposition}
\begin{proof}
By
Cor.
\ref{28111707}
$(1)\Leftrightarrow(2)$.
Let
$(3)$
hold
then $(2)$ is true
by setting
$\sigma=\mf{t}_{f(x_{\infty})}\up U$.
Let
$(2)$
hold
then
$
\nu_{j}^{z}((z,f_{0}(z))-
(z,f(x_{\infty})))
\leq
\nu_{j}^{z}((z,f_{0}(z))-
\sigma(z))
+
\nu_{j}^{z}(\sigma(z)-
\mf{t}_{f(x_{\infty})}(z))$,
thus
$(3)$
follows
by
$(2)$
and 
by Cor. \ref{28111707}
applied to the continuous
map
$\mf{t}_{f(x_{\infty})}\up U$.
\end{proof}
\begin{corollary}
\label{21492812}
Let
$\mf{V}
\coloneqq
\lr{\lr{\mf{E}}{\tau}}{\pi,X,\n}$
be
a 
bundle
of $\Omega-$spaces,
$W\subseteq X$
and indicate
$
\n=
\{
\nu_{j}\,\vert\,
j\in J
\}
$.
Moreover
let
$f,g\in\prod_{x\in W}\mf{E}_{x}$
and
$x_{\infty}\in W$.
Then
if
$\mf{V}$
locally
full
or
$\nu_{j}$
is continuous
$\forall j\in J$,
then
$
(1)
\rightarrow
(2)
$
where
\begin{enumerate}
\item
$f(x_{\infty})=g(x_{\infty})$
and
$f$ and $g$
are
continuous in $x_{\infty}$;
\item
$
(\exists\,U\in Op(X)\,\vert\,
x_{\infty}\in U)
$
such that
$$
(\forall j\in J)
(\lim_{
y\to x_{\infty},y\in W\cap U}
\nu_{j}(f(y)-g(y))=0).
$$
\end{enumerate}
Moreover
if 
$\mf{V}$
is full 
we can choose
$U=X$.
\end{corollary}
\begin{proof}
The statement is trivial
in the case of continuiuty
of all the $\nu_{j}$.
Whereas
if
$\mf{V}$
is
locally full
by
$(1)\rightarrow(5)$
of
Cor. \ref{28111707}
we have
$(\exists\,U\in Op(X))
(\exists\,\sigma\in\Gamma_{U}(\pi))
(\sigma(x_{\infty})
=f(x_{\infty})
=g(x_{\infty}))$
such that
$$
(\forall j\in J)
(\lim_{y\to x_{\infty},y\in W\cap U}
\nu_{j}(f(y)-\sigma(y))
=
\lim_{
y\to x_{\infty},y\in W\cap U}
\nu_{j}(g(y)-\sigma(y))
=0).
$$
Therefore
$$
\lim_{y\to x_{\infty},y\in W\cap U}
\nu_{j}(f(y)-g(y))
\leq
\lim_{y\to x_{\infty},y\in W\cap U}
\nu_{j}(f(y)-\sigma(y))
+
\lim_{y\to x_{\infty},y\in W\cap U}
\nu_{j}(g(y)-\sigma(y))
=0.
$$
\end{proof}
\begin{corollary}
\label{281117010}
Let
$
\lr{\lr{\mf{E}}{\tau}}{\pi,X,\n}
$
be
a bundle
of $\Omega-$spaces,
$W\in Op(X)$
and indicate
$
\n=
\{
\nu_{j}\,\vert\,
j\in J
\}
$.
Moreover
let
$f\in\prod_{x\in W}^{b}\mf{E}_{x}$.
Then
$
(1)
\Leftarrow
(2)
\Leftrightarrow
(3)
\Leftarrow
(4)
\Leftrightarrow
(5)
$
where
\begin{enumerate}
\item
$f\in\Gamma_{W}(\pi)$;
\item
$$
(\forall x\in\ W)
(\exists\,U_{x}\in Op(X)\,\vert\, U_{x}\ni x)
(\exists\,\sigma_{x}\in\Gamma_{U_{x}}(\pi))
(\sigma_{x}(x)
=f(x))
$$
such that
$
\nu_{j}
\circ(f-\sigma_{x})
$
is
continuous
in
$x$,
$\forall j\in J$;
\item
$$
(\forall x\in\ W)
(\exists\,U_{x}\in Op(X)\,\vert\, U_{x}\ni x)
(\exists\,\sigma_{x}\in\Gamma_{U_{x}}(\pi))
(\sigma_{x}(x)
=f(x))
$$
such that
$
(\forall j\in J)
(\lim_{
y\to x,y\in W\cap U_{x}}
\nu_{j}(f(y)-\sigma_{x}(y))
=
0)$;
\item
$$
(\forall x\in\ W)
(\exists\,U_{x}'\in Op(X)\,\vert\, U_{x}'\ni x)
(\exists\,\sigma_{x}'\in\Gamma_{U_{x}}(\pi))
(\sigma_{x}'(x)
=f(x))
$$
and
$$
(\forall\,U_{x}\in Op(X)\,\vert\, U_{x}\ni x)
(\forall\sigma_{x}\in\Gamma_{U_{x}}(\pi)
\,\vert\,
\sigma_{x}(x)
=f(x))
$$
we have
that
$
\nu_{j}
\circ(f-\sigma_{x})
$
is
continuous
in
$x$
for all
$x\in W$
and
$j\in J$;
\item
$$
(\forall x\in\ W)
(\exists\,U_{x}'\in Op(X)\,\vert\, U_{x}'\ni x)
(\exists\,\sigma_{x}'\in\Gamma_{U_{x}}(\pi))
(\sigma_{x}'(x)
=f(x))
$$
and
$$
(\forall x\in\ W)
(\forall\,U_{x}\in Op(X)\,\vert\, U_{x}\ni x)
(\forall\sigma_{x}\in\Gamma_{U_{x}}(\pi)
\,\vert\,
\sigma_{x}(x)
=f(x))
$$
we have
$(\forall j\in J)
(\lim_{
y\to x,y\in W\cap U_{x}}
\nu_{j}(f(y)-\sigma_{x}(y))
=
0)$.
\end{enumerate}
\end{corollary}
\begin{proof}
By Cor.
\ref{28111707}.
\end{proof}
\section{Direct Sum  of Bundles of $\Omega-$spaces}
\label{directsum}
The aim of this section is to extend in Def. \ref{14442410}
the standard construction of direct sum of Banach bundles
to bundles of $\Omega-$spaces. 
In order to do this in Thm. \ref{16322110}
we find a suitable directed set of seminorms inducing the product topology
on the direct sum of a finite family of locally convex spaces.
Then since Lemma \ref{19420111} we can apply the general construction given in Def. \ref{17471910Ba}
to the objects defined in Def. \ref{12422110}.
Finally the factorization property of the convergence in any direct sum of bundles of $\Omega-$spaces
presented in Cor. \ref{17571212}, shows that our definition extends the product topology and  
more in general it extends the usual definition of direct sum of Banach bundles.
\begin{theorem}
\label{16322110}
Let
$\{\lr{E_{i}}{\nu_{i}}_{i=1}\}^{n}$
be a family of lcs where
$\nu_{i}=\{\nu_{i,l_{i}}\,\vert\, l_{i}\in L_{i}\}$
is a 
fundamental
directed
set
of 
seminorms of $E_{i}$.
Let us set
for all $i=1,...,n$,
$l_{i}\in L_{i}$
and
$
\rho\in
\prod_{i=1}^{n}
L_{i}
$
$$
\begin{cases}
\hat{\nu}_{i,l_{i}}
\coloneqq
\nu_{i,l_{i}}
\circ
\Pr_{i}
\\
\hat{\mu_{\rho}}
\coloneqq
\sum_{k=1}^{n}
\hat{\nu}_{k,\rho_{k}},
\end{cases}
$$
where
$
\Pr_{i}:\prod_{k=1}^{n}E_{k}
\ni
x
\mapsto
x_{i}
\in E_{i}
$.
\par
Then
$
\hat{\mu}
\coloneqq
\{\hat{\mu}_{\rho}
\,\vert\,\rho\in\prod_{i=1}^{n}L_{i}\}
$
is 
a
directed
set
of 
seminorms on
$
\bigoplus_{i=1}^{n}
E_{i}
$.
Moreover 
by
setting
$$
\begin{cases}
\B(\ze)
\coloneqq
\{
W_{\ep}^{\rho}
\,\vert\,
\ep,\rho\in\prod_{i=1}^{n}
L_{i}
\}
\\
W_{\ep}^{\rho}
\coloneqq
\{
x\in
\bigoplus_{i=1}^{n}
E_{i}
\,\vert\,
\hat{\mu}_{\rho}(x)
<\ep
\},
\end{cases}
$$
we have
that
$\B(\ze)$
is a base of the filter of the neighbourhoods of $\ze$
with respect to the unique locally convex topology $\tau$
on $\bigoplus_{i=1}^{n}E_{i}$ generated by $\hat{\mu}$.
In other words
$$
\F{\B(\ze)}{\bigoplus_{i=1}^{n}E_{i}}
=
\mc{I}_{\ze}^{\tau}.
$$
Finally we have
$
\tau
=
\tau_{0}
=
\tau_{b}
=
\tau_{l}
=
\tau_{\mf{l}}
$.
\end{theorem}
\begin{proof}
Only in this proof
we set
$I\coloneqq\{1,...,n\}$,
$L\coloneqq\prod_{i\in I}L_{i}$
and
$
E^{\oplus}
\coloneqq
\bigoplus_{i=1}^{n}
E_{i}
$.
Due to the fact that
$n<\infty$
we know that 
$\prod_{i=1}^{n}E_{i}=E^{\oplus}$
so by
\cite{jar}
$\S 4.3.$
the 
set
$
\{\prod_{i=1}^{n}U_{i}\,\vert\, U_{i}\in\mf{U}_{i}\}
$
is a $\ze-$basis for the box topology
on $E^{\oplus}$
if $\mf{U}_{i}$ 
is a $\ze-$basis for the topology
on $E_{i}$.
Moreover $\nu_{i}$ is directed so by $II.3$ 
of \cite{BourTVS}
we can choose
\begin{equation}
\label{17442110}
\begin{aligned}
\mf{U}_{i}
&=
\{
V(\nu_{i,l_{i}},\ep)
\,\vert\,
\ep>0,
l_{i}\in L_{i}
\},
\\
V(\nu_{i,l_{i}},\ep)
&\coloneqq
\{x_{i}\in E_{i}\,\vert\,\nu_{i,l_{i}}(x_{i})<\ep\}.
\end{aligned}
\end{equation}
Thus
if we set
\begin{equation}
\label{18052110}
\begin{cases}
\B_{1}(\ze)
\coloneqq
\{
U_{\eta}^{\rho}\,\vert\,\eta\in(\R_{0}^{+})^{n},
\rho\in L
\},
\\
U_{\eta}^{\rho}
\coloneqq
\{
x\in E^{\oplus}
\,\vert\,
(\forall i\in I)
(\hat{\nu}_{i,\rho_{i}}(x)<\eta_{i})
\};
\end{cases}
\end{equation}
then
$\B_{1}(\ze)$
is a $\ze-$basis for the topology
$\tau_{0}$.
Moreover
$
U_{\ep}^{\rho}
=
\bigcap_{i=1}^{n}
V(\hat{\nu}_{i,\rho_{i}}\eta_{i})
$
so
if we set
$$
\mc{G}(\ze)
\coloneqq
\bigl\{
\bigcap_{s\in M}
V(\hat{\nu}_{s}\ep_{M}(s))
\,\vert\,
M\in
\p_{\omega}
\bigl(\bigcup_{i\in I}\{i\}\times L_{i}\bigr),\,
\ep_{M}:M\to\R_{0}^{+}
\bigr\},
$$
then
by \eqref{18052110}
$
\B_{1}(\ze)
\subseteq
\mc{G}(\ze)
$.
Moreover
by applying
$II.3$
of \cite{BourTVS},
$\mc{G}(\ze)$
is a basis of a filter 
thus
$$
\F{\B_{1}(\ze)}{E^{\oplus}}
\subseteq
\F{\mc{G}(\ze)}{E^{\oplus}}.
$$
Now for all
$M
\in
\p_{\omega}
\left(\bigcup_{i\in I}\{i\}\times L_{i}\right)
$
we have
$
M=\bigcup_{i\in I}M_{i}
$
with
$
M_{i}\coloneqq
M\cap(\{i\}\times L_{i})
=
\{i\}\times Q_{i}
$
for some
$Q_{i}\in\p_{\omega}(L_{i})$.
Hence
$
\forall M
\in
\p_{\omega}
\left(\bigcup_{i\in I}\{i\}\times L_{i}\right)
$
and $\forall\ep_{M}:M\to\R_{0}^{+}$
$$
\begin{aligned}
T\coloneqq
\bigcap_{s\in M}
V(\hat{\nu}_{s},\ep_{M}(s))
&
=
\bigcap_{i\in I}
\bigcap_{s\in M_{i}}
V(\hat{\nu}_{s},\ep_{M}(s))
\\
&
=
\bigcap_{i\in I}
\bigcap_{l_{i}\in Q_{i}}
\{
x\in E^{\oplus}
\,\vert\,
x_{i}\in V(\nu_{i,l_{i}},\ep_{M}(i,l_{i}))
\\
&
=
\bigcap_{i\in I}
\bigl\{
x\in E^{\oplus}
\,\vert\,
x_{i}\in 
\bigcap_{l_{i}\in Q_{i}}
V(\nu_{i,l_{i}},\ep_{M}(i,l_{i}))
\bigr\}.
\end{aligned}
$$
Moreover we know
that
$\mf{U}_{i}$
is a basis of a filter
on $E_{i}$ thus
for any $i\in I$
there exists
$\lambda_{i}>0$
and
$k_{i}\in L_{i}$
such that
$$
V(\nu_{i,k_{i}},\lambda_{i})
\subseteq
\bigcap_{l_{i}\in Q_{i}}
V(\nu_{i,l_{i}},\ep_{M}(i,l_{i})),
$$
hence
$$
\begin{aligned}
\mc{G}(\ze)
\ni
T
&
\supseteq
\bigcap_{i\in I}
\{
x\in E^{\oplus}
\,\vert\,
x_{i}\in 
V(\nu_{i,k_{i}},\lambda_{i})
\}
\\
&
=
\bigcap_{i\in I}
V(\hat{\nu}_{i,k_{i}},\lambda_{i})
\in
\B_{1}(\ze).
\end{aligned}
$$
Therefore
by a well-known property of filters
$
\F{\mc{G}(\ze)}{E^{\oplus}}
\subseteq
\F{\B_{1}(\ze)}{E^{\oplus}}
$
then
\begin{equation}
\label{18062110}
\F{\mc{G}(\ze)}{E^{\oplus}}
=
\F{\B_{1}(\ze)}{E^{\oplus}}.
\end{equation}
By applying
$II.3$
of \cite{BourTVS}
we know
that 
$\F{\mc{G}(\ze)}{E^{\oplus}}$
is the $\ze-$neighbourhood's
filter with respect to the
locally convex topology 
generated by the family of seminorms
$
\{\nu_{s}\,\vert\, 
s\in
\bigcup_{i\in I}\{i\}\times L_{i}
\}
$
thus
by
\eqref{18052110}
and
\eqref{18062110}
\begin{equation}
\label{18112110}
\bigl
\{\nu_{s}\,\vert\, 
s\in
\bigcup_{i\in I}\{i\}\times L_{i}
\bigr\}
\text{ is a fss for $\tau_{0}$}.
\end{equation}
\par
Now
$\hat{\mu}$
is a set of seminorms on $E^{\oplus}$.
Let
$\rho^{1},\rho^{2}\in L$
then 
by the hypothesis that
$\nu_{i}$
is directed,
for all $i\in I$
there exists
$\rho_{i}\in L_{i}$
such that
$\rho_{i}\geq\rho^{1},\rho^{2}$
thus
$\hat{\mu}_{\rho}
\geq
\hat{\mu}_{\rho^{1}},
\hat{\mu}_{\rho^{2}}
$,
hence
$\hat{\mu}$
is directed.
Therefore
setting
$$
\begin{cases}
\B(\ze)
\coloneqq
\{
W_{\ep}^{\rho}
\,\vert\,
\ep>0,
\rho\in L
\}
\\
W_{\ep}^{\rho}
\coloneqq
\{
x\in E^{\oplus}
\,\vert\,
\hat{\mu}_{\rho}(x)
<\ep
\}
\end{cases}
$$
by 
applying
$II.3$
of \cite{BourTVS}
\begin{equation}
\label{18242110}
\B(\ze)
\text{ is the $\ze-$basis
for the topology
generated by
$\hat{\mu}$.
}
\end{equation}
Now
$
(\forall(k,l_{k})\in 
\bigcup_{i\in I}\{i\}\times L_{i})
(\exists\,\rho\in L)
(\hat{\nu}_{k,l_{k}}\leq a\hat{\mu}_{\rho})
$
indeed 
keep any $\rho$
s.t. $\rho(k)=l_{k}$.
While
$
(\forall\rho\in I)
(m\in\N)
(\exists\,s_{1},...,s_{m}\in
\bigcup_{i\in I}\{i\}\times L_{i})
(\exists\,a>0)
(\hat{\mu}_{\rho}\leq a
\sup_{r}\hat{\nu}_{s_{r}})
$
indeed
it is sufficient to set
$m=n$,
$
a=n
$
and
$
s_{i}
=
(i,\rho_{i})
$
for all $i\in I$.
Therefore
by applying
Cor.
$1$
$II.7$
of
\cite{BourTVS}
and
by
\eqref{18242110}
and
\eqref{18112110}
we have that
$\hat{\mu}$
is a directed fss
for the topology
$\tau_{0}$
hence
the part of the statement
concerning 
$\tau_{0}$ follows.
By
Prop.
$2$,
$\S 3$,
Ch $4$
of
\cite{jar}
we know
that
$\tau_{0}
=
\tau_{b}
=
\tau_{l}
$.
Finally
$
\tau_{\mf{l}}
=
\tau_{l}
$
by the fact that
$
\tau_{\mf{l}}
$
is the finest 
locally convex
topology
among those
which are
coarser
than
$\tau_{l}$,
$\S 6$,
Ch $6$
of
\cite{jar},
and the just now shown
fact that
$\tau_{l}$
is 
locally convex
being
equal
to $\tau_{0}$
which 
is generated by
$\hat{\mu}$.
\end{proof}
\begin{notation}
In the remaining of the present section \ref{directsum} 
we let 
$\{\mf{V}_{i}\}_{i=1}^{n}$
be a family of \textbf{full}
bundles of $\Omega-$spaces.
Here
$\mf{V}_{i}
=
\lr{\lr{\mf{E}_{i}}{\tau_{i}}}
{\pi_{i},X,\n_{i}}$,
$\n_{i}
=
\{
\nu_{i,l_{i}}
\,\vert\,
l_{i}\in L_{i}
\}
$
moreover 
$
\n_{i}^{x}
\coloneqq
\{
\nu_{i,l_{i}}^{x}
\,\vert\,
l_{i}\in L_{i}
\}
$,
with 
$\nu_{i,l_{i}}^{x}
\coloneqq
\nu_{i,l_{i}}\up\left(\mf{E}_{i}\right)_{x}$
and 
$
\left(\mf{E}_{i}\right)_{x}
\coloneqq
\overset{-1}{\pi_{i}}(x)$
for all
$i=1,...,n$ and $x\in X$.
\end{notation}
\begin{definition}
\label{12422110}
Define
\begin{enumerate}
\item
$\ms{E}_{x}^{\oplus}
\coloneqq
\bigoplus_{i=1}^{n}\left(\mf{E}_{i}\right)_{x}
$;
\item
$
\mf{n}_{x}^{\oplus}
\coloneqq
\{
\hat{\mu}_{\rho}^{x}
\,\vert\,
\rho
\in
\prod_{i=1}^{n}
L_{i}
\}
$,
where
\begin{equation}
\label{18260603}
\hat{\mu}_{\rho}^{x}
=
\sum_{i=1}^{n}
\hat{\nu}_{i,\rho_{i}}^{x};
\end{equation}
\item
$\ms{E}^{\oplus}
\coloneqq
\left\{
\lr{\ms{E}_{x}^{\oplus}}
{\mf{n}_{x}^{\oplus}}
\right\}_{x\in X}$;
\item
$\mc{E}^{\oplus}$
is the 
linear subspace
of 
$
\prod_{x\in X}
\ms{E}_{x}^{\oplus}
$
generated
by the following
set
\begin{equation}
\label{18270603}
\bigcup_{i=1}^{n}
\tilde{\Gamma}(\pi_{i}).
\end{equation}
\end{enumerate}
Here
$
\Pr_{i}^{x}:
\ms{E}_{x}^{\oplus}
\ni
x
\mapsto
x(i)
\in
\left(\mf{E}_{i}\right)_{x}
$
while
$
\hat{\nu}_{i,\rho_{i}}^{x}
=
\nu_{i,\rho_{i}}^{x}
\circ
\Pr_{i}^{x}
$
and
$
I_{i}^{x}:
\left(\mf{E}_{i}\right)_{x}
\to
\ms{E}_{x}^{\oplus}
$
is the canonical inclusion,
i.e.
$
\Pr_{j}^{x}
\circ
I_{i}^{x}
=
\delta_{i,j}
\,
Id^{x}
$,
finally
$
\tilde{\Gamma}(\pi_{i})
\coloneqq
\{
\tilde{f}
\,\vert\,
f\in
\Gamma(\pi_{i})
\}
$,
with
$
\tilde{f}(x)
\coloneqq
I_{i}^{x}
(f(x))
$.
\end{definition}
Notice that
$
\{
\lr{
\left(\mf{E}_{i}\right)_{x}}{\n_{i}^{x}}
\}_{i=1}^{n}
$
for all $x\in X$
is a family of 
Hlcs
where
$\n_{i}^{x}$
is a 
directed family of
seminorms defining
the topology on
$
\left(\mf{E}_{i}\right)_{x}
$,
for all 
$i=1,\dots,n$.
\begin{lemma}
\label{19420111}
$\mc{E}^{\oplus}$
satisfies
$FM(3)-FM(4)$
with respect
to 
$\ms{E}^{\oplus}$.
\end{lemma}
\begin{proof}
$
I_{i}^{x}
$
is a bijective
map
onto
its
range
whose
inverse
is
$
\Pr_{i}^{x}
\up
Range(I_{i}^{x})
$.
Moreover
by definition
of the product 
topology
$\Pr_{i}^{x}$
is continuous
with respect 
to the 
topology
$\tau_{0}^{i}$
on
$Range(I_{i}^{x})$
induced by 
$\tau_{0}$
\cite[Ch.1]{BourGT},
while
$I_{i}^{x}$
is continuous
with respect 
to
$\tau_{0}^{i}$
by 
\cite[$\S$ 4.3 Pr.1]{jar}
and the definition
of 
$\tau_{l}$.
Hence
by 
Thm. 
\ref{16322110}
$
I_{i}^{x}
$
is 
an isomorphism
of the
tvs's
$
\lr{\left(\mf{E}_{i}\right)_{x}}{\n_{i}^{x}}
$
and
$
I_{i}^{x}
\left(
\left(\mf{E}_{i}\right)_{x}
\right)
$
as subspace
of
$
\lr{\ms{E}_{x}^{\oplus}}
{\mf{n}_{x}^{\oplus}}
$.
Since
\cite[1.5.III]{gie}
and
\cite[1.6.viii]{gie}\footnote{
which ensures that
the locally convex topology
on
$
\left(\mf{E}_{i}\right)_{x}
$
generated by the set of seminorms
$\n_{i}^{x}$
is exactly
the topology
induced
on
it
by the topology
$\tau_{i}$
on $\mf{E}_{i}$,
for all $i$
and $x\in X$.}
we deduce that
$
\{
\sigma(x)\,\vert\,
\sigma\in\Gamma(\pi_{i})
\}
$
is dense
in
$
\lr{\left(\mf{E}_{i}\right)_{x}}{\n_{i}^{x}}
$.
Therefore
$\forall i=1,...,n$
and
$\forall x\in X$
\begin{equation}
\label{22170111}
\{
I_{i}^{x}(\sigma(x))\,\vert\,
\sigma\in\Gamma(\pi_{i})
\}
\text{
is dense
in
$
I_{i}^{x}
\left(
\left(\mf{E}_{i}\right)_{x}
\right)
$.
}
\end{equation}
where
$
I_{i}^{x}
\left(
\left(\mf{E}_{i}\right)_{x}
\right)
$
has to be intended 
as topological vector
subspace
of
$
\lr{\ms{E}_{x}^{\oplus}}
{\mf{n}_{x}^{\oplus}}
$.
So 
by the continuity
of the 
sum on
$
\lr{\ms{E}_{x}^{\oplus}}
{\mf{n}_{x}^{\oplus}}
$
and
the fact
that
$\ms{E}_{x}^{\oplus}$
is generated
as linear
space
by 
the set
$
\bigcup_{i=1}^{n}
I_{i}^{x}
\left(
\left(
\mf{E}_{i}\right)_{x}\right)
$
we can state 
$\forall x\in X$
that
\begin{equation}
\label{22050111}
\{
F(x)
\,\vert\,
F
\in
\mc{E}^{\oplus}
\}
\text{ 
is dense in 
$
\lr{\ms{E}_{x}^{\oplus}}
{\mf{n}_{x}^{\oplus}}$.
}
\end{equation}
Namely
by
\eqref{22170111}
$$
(\forall v\in\mf{E}^{\oplus})
(\forall i=1,...,n)
(\exists\,
\{
\sigma_{\alpha_{i}}
\}_{\alpha_{i}\in D_{i}}
\text{ net }
\subset
\Gamma(\pi_{i})
)
$$
such that
$$
\begin{aligned}
v
&
=
\sum_{i=1}^{n}I_{i}^{x}(\Pr_{i}^{x}(v))
=
\sum_{i=1}^{n}
\lim_{\alpha_{i}\in D_{i}}
I_{i}^{x}(\sigma_{\alpha_{i}}(x))
\\
&
=
\sum_{i=1}^{n}
\lim_{\alpha\in D}w_{\alpha}^{i}(x)
=
\lim_{\alpha\in D}
\sum_{i=1}^{n}
w_{\alpha}^{i}(x)
\\
&
=
\lim_{\alpha\in D}
\sum_{i=1}^{n}
I_{i}^{x}(\sigma_{\alpha(i)}(x)),
\end{aligned}
$$
where
$
D
\coloneqq
\prod_{i=1}^{n}
D_{i}
$
while
$
w_{\alpha}^{i}(x)
\coloneqq
I_{i}^{x}(\sigma_{\alpha(i)}(x))
$
for all
$\alpha\in D$.
Moreover
$\forall\alpha\in D$
$$
\bigl(
X\ni x\mapsto
\sum_{i=1}^{n}
I_{i}^{x}(\sigma_{\alpha(i)}(x))
\bigr)
\in
\mc{E}^{\oplus}
$$
then
\eqref{22050111}
and
$FM(3)$
follow.
\par
Finally
$FM(4)$
follows
by 
\cite[1.6.iii]{gie}
applied 
to
any $\sigma_{i}\in\Gamma(\pi_{i})$
for
all
$i=1,...,n$ 
indeed
$\forall\sigma_{i}\in\Gamma(\pi_{i})$
$$
\hat{
\nu_{i,\rho_{i}}^{x}
}
(\tilde{\sigma}_{i}(x))
=
\nu_{i,\rho_{i}}^{x}
\circ
\Pr_{i}^{x}
\circ
I_{i}^{x}
\circ
\sigma_{i}(x)
=
\nu_{i,\rho_{i}}^{x}
\circ
\sigma_{i}(x).
$$
\end{proof}
Now we are able to extend to bundles of $\Omega-$spaces,
the standard construction of direct sum of Banach bundles. 
Namely by Thm. \ref{16322110} we know that
$\mf{n}_{x}^{\oplus}$
is a directed set of seminorms
on $\ms{E}_{x}^{\oplus}$
inducing on
$\ms{E}_{x}^{\oplus}$
the product topology,
thus since Lemma \ref{19420111} we can apply Def. \ref{17471910Ba} 
and set the following
\begin{definition}
\label{14442410}
We call bundle direct sum
of the family $\{\mf{V}_{i}\}_{i=1}^{n}$
the following bundle of $\Omega-$spaces
\begin{equation*}
\bigoplus_{i=1}^{n}
\mf{V}_{i}
\coloneqq
\mf{V}(\ms{E}^{\oplus},\mc{E}^{\oplus}).
\end{equation*}
\end{definition}
\begin{remark}
\label{21262110}
By 
Def.
\ref{17471910Ba}
and Def.
\ref{14442410}
$$
\bigoplus_{i=1}^{n}
\mf{V}_{i}
=
\lr{\lr{\mf{E}(\ms{E}^{\oplus})}
{\tau(\ms{E}^{\oplus},\mc{E}^{\oplus})}}
{\pi_{\ms{E}^{\oplus}},X,\mf{n}^{\oplus}}
$$
where
\begin{enumerate}
\item
$
\mf{E}(\ms{E}^{\oplus})
\coloneqq
\bigcup_{x\in X}
\{x\}
\times
\ms{E}_{x}^{\oplus}
$,
$
\pi_{\ms{E}^{\oplus}}:
\mf{E}(\ms{E}^{\oplus})
\ni
(x,v)
\mapsto x
\in X
$.
\item
$
\mf{n}^{\oplus}
=
\{
\hat{\mu}_{\rho}:
\,\vert\,
\rho
\in
\prod_{i=1}^{n}
L_{i}
\}
$,
with
$
\hat{\mu}_{\rho}:
\mf{E}(\ms{E}^{\oplus})
\ni
(x,v)
\mapsto
\hat{\mu}_{\rho}^{x}(v)
$;
\item
$\tau(\ms{E}^{\oplus},\mc{E}^{\oplus})$
is
the topology
on
$\mf{E}(\ms{E}^{\oplus})$
such that
for all
$
(x,v)\in
\mf{E}(\ms{E}^{\oplus})
$
$$
\mc{I}_{(x,v)}^{\mf{E}(\ms{E}^{\oplus})}
\coloneqq
\F{\B^{\oplus}((x,v))}{\mf{E}(\ms{E}^{\oplus})}.
$$
Here we recall that
$\F{\B^{\oplus}((x,v))}{\mf{E}(\ms{E}^{\oplus})}$
is the filter 
on
$\mf{E}(\ms{E}^{\oplus})$
generated
by the following
base of filters 
\begin{alignat*}{1}
\B^{\oplus}((x,v))
\coloneqq
\bigl\{
T_{\ms{E}^{\oplus}}(U,\sigma,\ep,\rho)
&
\,\vert\,
U\in Open(X),
\sigma\in\mc{E}^{\oplus},
\ep>0,
\rho
\in
\prod_{i=1}^{n}
L_{i}
\\
&
\,\vert\,
x\in U,
\hat{\mu}_{\rho}^{x}(v-\sigma(x))
<\ep
\bigr\},
\end{alignat*}
where
$$
T_{\ms{E}^{\oplus}}(U,\sigma,\ep,\rho)
\coloneqq
\left\{
(y,w)\in\mf{E}(\ms{E}^{\oplus})
\,\vert\,
y\in U,
\hat{\mu}_{\rho}^{y}(w-\sigma(y))
<\ep
\right\}.
$$
\end{enumerate}
\end{remark}
In what follows we state the factorization property of convergence
which proves that our construction of bundle direct sum of a 
family of bundles of $\Omega-$spaces,
extends the standard definition provided in the Banach bundle case.
\begin{corollary}
\label{17571212}
Let $f:X\to\mf{E}(\ms{E}^{\oplus})$ and $x\in X$.
Thus
$f$ is continuous in $x$
if and only if
$f_{0}^{i}:X\to\mf{E}_{i}$
is continuous
in $x$
for all
$i=1,...,n$,
where
$
f_{0}:X\to
\bigcup_{z\in X}
\ms{E}_{z}^{\oplus}
$
such that
$\forall z\in X$
$
f(z)=(z,f_{0}(z))
$
and
$$
f_{0}^{i}(z)
\coloneqq
\Pr_{i}^{\pi_{\ms{E}^{\oplus}}
(f(z))}
\circ
f_{0}(z).
$$
In particular
$f\in\Gamma(\pi_{\ms{E}^{\oplus}})$
if and only if
$\left(
X\ni z\mapsto
\Pr_{i}^{z}\circ f_{0}(z)
\in(\mf{E}_{i})_{z}
\right)
\in\Gamma(\pi_{i})$,
for all
$i=1,...,n$.
\end{corollary}
\begin{proof}
Since the definition of $\mc{E}^{\oplus}$, the request that all the bundles 
in the family
$\{\mf{V}_{i}\}_{i=1}^{n}$
are full and the fact that 
$\mc{E}^{\oplus}$ is linearly isomorphic to a subspace of $\Gamma(\pi_{\ms{E}^{\oplus}})$
we obtain that, 
when applied to the bundle direct sum of the family $\{\mf{V}_{i}\}_{i=1}^{n}$,
the first part of $(6)$ in Thm. \ref{15380512}
is satisfied by global sections
belonging to $\mc{E}^{\oplus}$.
Therefore the statement follows since 
$(5)
\Leftrightarrow
(6)$
in
Thm. 
\ref{15380512}.
\end{proof}
\begin{convention}
\label{16392601}
By construction
we have that
$\Gamma(\pi_{\ms{E}^{\oplus}})
\subset
\prod_{x\in X}
\{x\}\times
\ms{E}_{x}^{\oplus}$.
In what follows, except 
contrary mention,
we convein
to consider 
with abuse of language
in the obvious manner
$$
\Gamma(\pi_{\ms{E}^{\oplus}})
\subset
\prod_{x\in X}
\bigoplus_{i=1}^{n}\left(\mf{E}_{i}\right)_{x}.
$$
Similarly
for
$\Gamma^{x}(\pi_{\ms{E}^{\oplus}})$
for any $x\in X$.
Moreover
in the case
in which
for any $i=1,...,n$
we have
$\mf{V}_{i}
=\mf{V}(\ms{E_{i}},\mc{E_{i}})$,
with obvious
meaning of the symbols
we consider
$$
\Gamma(\pi_{\ms{E}^{\oplus}})
\subset
\prod_{x\in X}
\bigoplus_{i=1}^{n}\left(\ms{E}_{i}\right)_{x}.
$$
\end{convention}
\section{$\left(\Theta,\mc{E}\right)-$structure}
\label{01121953}
In Def. \ref{10282712} we define the concept of $\left(\Theta,\mc{E}\right)-$structure.
In Lemma \ref{15482712} and Cor. \ref{11121102}
we characterize basic properties of this structure.
In Thm. \ref{22372406} we construct the $\left(\Theta,\mc{E}\right)-$structure
described in Introduction and provide a set of continuous sections
which serves as a model to build the general definition. 
Finally in Prp. \ref{20492003} we provide a characterization of continuous sections
related to a suitable $\left(\Theta,\mc{E}\right)-$structure.
In order to construct the structure provided in Def. \ref{10282712} we need 
a sequence of steps starting with the following
\begin{definition}
\label{17471910}
$\lr{X}{\ms{E},\mc{S}}$
is a 
map system
if
\begin{enumerate}
\item
$X$
is a set;
\item
$
\ms{E}
=
\{
\lr{\ms{E}_{x}}{\n_{x}}
\}_{x\in X}
$
is
a nice
family of 
Hlcs
with
$
\n_{x}
\coloneqq
\{\nu^{x}_{j}\,\vert\, j\in J\}
$
for all $x\in X$;
\item
$
(\exists L\ne\emptyset)
(\mc{S}
=
\{S_{x}\}_{x\in X})
$
where
$
S_{x}
\coloneqq
\{
B_{l}^{x}
\,\vert\,
l\in L
\}
\subseteq
Bounded(\ms{E}_{x})
$
and
$
\bigcup_{l\in L}
B_{l}^{x}
$
is total
in 
$\ms{E}_{x}$
for all $x\in X$.
\end{enumerate}
\end{definition}
\begin{definition}
\label{17471910A}
We say that
$
\ms{M}
$
is 
a
map pre-bundle relative to 
$\lr{X,Y}{\ms{E},\mc{S}}$
if
\begin{enumerate}
\item
$\lr{X}{\ms{E},\mc{S}}$
is a map system;
\item
$
\ms{M}
=
\{
\lr{\ms{M}_{x}}{\mf{R}_{x}}
\}_{x\in X}
$
is a nice
family of Hlcs;
\item
$Y$ is a Hausdorff
topological space
and
$\forall x\in X$
\begin{equation*}
\begin{aligned}
\ms{M}_{x}
&
\subseteq
\cc{}{Y,\mc{L}_{S_{x}}(\ms{E}_{x})};
\\
\mf{R}_{x}
&
=
\bigl
\{
\sup_{(K,j,l)\in\mc{O}}
q_{(K,j,l)}^{x}
\up
\ms{M}_{x}
\,\vert\,
\mc{O}\in\p_{\omega}\left(Comp(Y)
\times J\times L\right)
\bigr\}.
\end{aligned}
\end{equation*}
\end{enumerate}
Here we recall that
$\p_{\omega}(A)$
is the set of all finite
parts of the set $A$,
$\mc{L}_{S_{x}}(\ms{E}_{x})$,
for all $x\in X$,
is the 
$lcs$
of all continuous linear maps
$\mc{L}(\ms{E}_{x})$
on $\ms{E}_{x}$
with the topology
of uniform
convergence
over
the 
sets
in $S_{x}$,
hence
its
topology
is generated by the following
set of 
seminorms
\begin{equation}
\label{21532606}
\bigl\{
p_{j,l}^{x}:
\mc{L}(\ms{E}_{x})
\ni
\phi
\mapsto
\sup_{v\in B_{l}^{x}}
\nu_{j}^{x}(\phi(v))
\,\vert\,
l\in L,
j\in J
\bigr
\}.
\end{equation}
Thus 
by
the totality hypothesis
and
by
\cite[Prop. $3$, $III.15$]{BourTVS}
$\mc{L}_{S_{x}}(\ms{E}_{x})$
is Hausdorff.
Finally
for all
$(K,j,l)
\in
Comp(Y)
\times
J
\times
L$
we set
\begin{equation}
\label{21552606}
q_{(K,j,l)}^{x}:
\cc{c}{Y,\mc{L}_{S_{x}}(\ms{E}_{x})}
\ni
f
\mapsto
\sup_{t\in K}
p_{j,l}^{x}
(f(t)).
\end{equation}
\end{definition}
\begin{remark}
By the fact that
$\{t\}$ is compact for all $t\in Y$
we have that
$\bigcup_{K\in Comp(Y)}K=Y$
thus by
the shown fact that
$\mc{L}_{S_{x}}(\ms{E}_{x})$
is Hausdorff
we deduce by 
\cite[Prp. $(1)$, $\S 1.2$, Ch $10$]{BourGT}
that
$\cc{c}{Y,\mc{L}_{S_{x}}(\ms{E}_{x})}$
is Hausdorff.
Moreover
by
\cite[Def. $(1)$, $\S 1.1$, Ch $10$]{BourGT}
and by the fact that
\eqref{21532606}
is a fss
on
$\mc{L}_{S_{x}}(\ms{E}_{x})$,
we can deduce that
$\left\{
\sup_{(K,j,l)\in\mc{O}}
q_{(K,j,l)}^{x}
\,\vert\,
\mc{O}
\in\p_{\omega}
\left(
Comp(Y)
\times
J
\times
L
\right)
\right\}$
is a directed fss
on
$\cc{c}{Y,\mc{L}_{S_{x}}(\ms{E}_{x})}$.
Hence
$\lr{\ms{M}_{x}}{\mf{R}_{x}}$
is a topological vector
subspace
of
$\cc{c}{Y,\mc{L}_{S_{x}}(\ms{E}_{x})}$
so it is 
Hausdorff,
hence
by the construction
of $\mf{R}_{x}$
we 
can state that
$\{
\lr{\ms{M}_{x}}{\mf{R}_{x}}
\}_{x\in X}$
is a nice
family of Hlcs
in agreement with request 
$(2)$
in Def.
\ref{17471910A}.
\end{remark}
Next we provide the explicit form of $\mf{V}(\ms{M},\mc{M})$.
\begin{remark}
\label{17471910B}
Let
$
\ms{M}
=
\{
\lr{\ms{M}_{x}}{\mf{R}_{x}}
\}_{x\in X}
$
be
a
map 
pre-bundle
relative
to
$
\lr{X,Y}
{\ms{E}=\lr{\ms{E}_{x}}{\n_{x}}_{x\in X},
\mc{S}}
$,
moreover
let
$\mc{M}$
satisfy
$FM(3)-FM(4)$
with respect
to 
$
\ms{M}
$.
Let
us denote
$
\n_{x}
=
\{
\nu_{j}^{x}
\,\vert\,
j\in J
\}
$
for all
$x\in X$
and
use
the notation
in Def.
\ref{17471910A}.
Thus
for the 
bundle
$\mf{V}(\ms{M},\mc{M})$
generated
by the couple
$\lr{\ms{M}}{\mc{M}}$
we have
\begin{enumerate}
\item
$
\mf{V}(\ms{M},\mc{M})
=
\lr{\lr{\mf{E}(\ms{M})}
{\tau(\ms{M},\mc{M})}
}{\pi_{\ms{M}},X,\mf{R}}
$;
\item
$
\mf{E}(\ms{M})
\coloneqq
\bigcup_{x\in X}
\{x\}
\times
\ms{M}_{x}
$,
$
\pi_{\ms{M}}
:\mf{E}(\ms{M})
\ni
(x,f)
\mapsto x
\in X
$;
\item
$
\mf{R}
=
\left\{
\sup_{(K,j,l)\in\mc{O}}
q_{(K,j,l)}
\,\vert\,
\mc{O}
\in\p_{\omega}
\left(
Comp(Y)
\times
J
\times
L
\right)
\right\}$,
with
$
q_{(K,j,l)}:
\mf{E}(\ms{M})
\ni
(x,f)
\mapsto
q_{(K,j,l)}^{x}(f)$;
\item
$\tau(\ms{M},\mc{M})$
is the topology
on
$
\mf{E}(\ms{M})
$
such that
for all
$(x,f)\in\mf{E}(\ms{M})$
$$
\mc{I}_{(x,f)}^{\mf{E}(\ms{M})}
\coloneqq
\F{\B_{\ms{M}}((x,f))}{\mf{E}(\ms{M})}
$$
is the neighbourhood's
filter
of
$(x,f)$
with respect to it.
Here
$\F{\B_{\ms{M}}((x,f))}{\mf{E}(\ms{M})}$
is the filter
on
$\mf{E}(\ms{M})$
generated
by the following
filter's base
\begin{alignat*}{1}
&
\B_{\ms{M}}((x,f))
\coloneqq
\{
T_{\ms{M}}
\left(
U,\sigma,\ep,
\mc{O}
\right)
\,\vert\,
U\in Open(X),
\sigma\in\mc{M},
\ep>0,\\
&
\mc{O}
\in\p_{\omega}
\left(
Comp(Y)
\times
J
\times
L
\right)
\,\vert\,
x\in U,
\sup_{(K,j,l)\in\mc{O}}
q_{(K,j,l)}^{x}
(f-\sigma(x))
<\ep
\},
\end{alignat*}
where
$
\forall
U\in Open(X),
\sigma\in\mc{M},
\ep>0
$
and
$\forall\mc{O}
\in\p_{\omega}
\left(
Comp(Y)
\times
J
\times
L
\right)$
$$
T_{\ms{M}}
\left(
U,\sigma,\ep,
\mc{O}
\right)
\coloneqq
\bigl
\{
(y,g)\in\mf{E}(\ms{M})
\,\vert\,
y\in U,
\sup_{(K,j,l)\in\mc{O}}
q_{(K,j,l)}^{y}
(g-\sigma(y))
<\ep
\bigr
\}.
$$
\end{enumerate}
\end{remark}
\begin{remark}
\label{18070512}
Let
$
\ms{M}
=
\{
\lr{\ms{M}_{x}}{\mf{R}_{x}}
\}_{x\in X}
$
be
a
map 
pre-bundle
relative
to
$
\lr{X,Y}
{\ms{E}=\lr{\ms{E}_{x}}{\n_{x}}_{x\in X},
\mc{S}}
$,
moreover
let
$\mc{M}$
satisfy
$FM(3)-FM(4)$
with respect
to 
$
\ms{M}
$.
Thus by 
Rmk.
\ref{15412610}
$
\forall
U\in Open(X),
\sigma\in\mc{M},
\ep>0
$
and
$\forall
\mc{O}\in\p_{\omega}\left(Comp(Y)
\times J\times L\right)$
$$
T_{\ms{M}}
(U,\sigma,\ep,
\mc{O})
=
\bigcup_{y\in U}
B_{\ms{M}_{y},
\mc{O},\ep}(\sigma(y))
$$
where
for all 
$
s
\in
\ms{M}_{y}
$
$$
B_{\ms{M}_{y},
\mc{O},\ep}(s)
\coloneqq
\bigl
\{
(y,f)\in\mf{E}(\ms{M})_{y}
\,\vert\,
\sup_{(K,j,l)\in\mc{O}}
q_{(K,j,l)}^{y}
\left(f-s\right)
<\ep
\bigr
\}.
$$
\end{remark}
By applying 
Rmk. \ref{17150312}
we have the following
\begin{remark}
\label{16422010}
Let
$\ms{M}$
be
a
map 
pre-bundle
relative
to
$
\lr{X,Y}{\ms{E},\mc{S}}
$,
moreover
let
$\mc{M}$
satisfy
$FM(3)-FM(4)$
with respect
to 
$\ms{M}$.
Thus
\begin{enumerate}
\item
$\mf{V}(\ms{M},\mc{M})$
is 
a bundle
of $\Omega-$spaces;
\item
with the notation
of Def.
\ref{17471910B}
$\mf{V}(\ms{M},\mc{M})$
is such that
\begin{enumerate}
\item
$\lr{\mf{E}(\ms{M})_{x}}
{\tau(\ms{M},\mc{M})}$
as topological vector space
is isomorphic
to
$\lr{\ms{M}_{x}}{\mf{R}_{x}}$
for all
$x\in X$;
\item
$\mc{M}$
is canonically
isomorphic to a linear subspace
of 
$\Gamma(\pi_{\ms{M}})$
and
if 
$X$ is compact and $\mc{M}$ is a function module, then
$
\mc{M}
\simeq
\Gamma(\pi_{\ms{M}})
$.
\end{enumerate}
\end{enumerate}
\end{remark}
In Def. \ref{10282712}
we generalize the topology of uniform convergence
to bundles 
$\lr{\mf{M}}{\rho,X}$
of 
$\Omega-$spaces,
where
$\{\mf{M}_{x}\}_{x\in X}$
is 
a
map pre-bundle relative 
to 
$\lr{X,Y}{\{\mf{E}_{x}\}_{x\in X},\mc{S}}$
and
$\lr{\mf{E}}{\pi,X}$
is 
a
bundle 
of 
$\Omega-$spaces.
The aim is to correlate
the topology 
on
$\mf{M}$
with 
that 
on
$\mf{E}$
in order to extend
the 
correlation
established 
in the introduction
for the trivial
bundle case.
\begin{definition}
\label{15531102}
$$
(\bullet):
\prod_{x\in X}(\mf{E}_{x})^{\mf{E}_{x}}
\times
\prod_{x\in X}\mf{E}_{x}
\to
\prod_{x\in X}\mf{E}_{x}
$$
such that
for all 
$F\in
\prod_{x\in X}(\mf{E}_{x})^{\mf{E}_{x}},
v\in\prod_{x\in X}\mf{E}_{x}$
we have
$$
(F\bullet w)(x)
\coloneqq
F(x)(w(x)).
$$
\end{definition}
\begin{definition}
[\textbf{
$\left(\Theta,\mc{E}\right)-$structures}]
\label{10282712}
We say
that
$\lr{\mf{V},\mf{W}}{X,Y}$
is
a
$\left(\Theta,\mc{E}\right)-$structure
if
\begin{enumerate}
\item
$
\mf{V}
\coloneqq
\lr{\lr{\mf{E}}{\tau}}
{\pi,X,\n}
$
is a 
bundle of $\Omega-$spaces;
\item
$\mc{E}\subseteq\Gamma(\pi)$;
\item
$\Theta\subseteq
\prod_{x\in X}
Bounded(\mf{E}_{x})$;
\item
$\forall B\in\Theta$
\begin{enumerate}
\item
$\ms{D}(B,\mc{E})\ne\emptyset$;
\item
$\bigcup_{B\in\Theta}\mc{B}_{B}^{x}$
is total
in $\mf{E}_{x}$
for all $x\in X$;
\end{enumerate}
\item
$\mf{W}
\coloneqq
\lr{\lr{\mf{M}}{\gamma}}{\rho,X,\mf{R}}$
is a
bundle of $\Omega-$spaces
such that
$\{\lr{\mf{M}_{x}}{\mf{R}_{x}}\}_{x\in X}$
is 
a
map pre-bundle relative 
to 
$\lr{X,Y}{\{\lr{\mf{E}_{x}}{\n_{x}}
\}_{x\in X},\mc{S}}$.
\end{enumerate}
Here
$\mc{S}
\coloneqq
\{S_{x}\}_{x\in X}$
and
$(\forall B\in\Theta)(\forall x\in X)$
\begin{equation}
\label{11232712}
\boxed{
\begin{cases}
\ms{D}(B,\mc{E})
\coloneqq
\mc{E}
\cap
\left(\prod_{x\in X}B_{x}\right)
\\
\mc{B}_{B}^{x}
\coloneqq
\{v(x)\,\vert\, v\in\ms{D}(B,\mc{E})\}
\}
\\
S_{x}
\coloneqq
\{\mc{B}_{B}^{x}
\,\vert\, B\in\Theta\}.
\end{cases}
}
\end{equation}
Moreover
$\lr{\mf{V},\mf{W}}{X,Y}$
is
an
invariant $\left(\Theta,\mc{E}\right)-$structure
if it is
a
$\left(\Theta,\mc{E}\right)-$structure
such that
\begin{equation}
\label{18112502pre}
\bigl
\{
F\in\prod_{z\in X}^{b}
\mf{M}_{z}
\,\vert\,
(\forall t\in Y)
(F_{t}
\bullet
\mc{E}(\Theta)
\subseteq
\Gamma(\pi))
\bigr
\}
=
\Gamma(\rho).
\end{equation}
Finally
$\lr{\mf{V},\mf{W}}{X,Y}$
is
a
\emph
{compatible 
$\left(\Theta,\mc{E}\right)-$structure}
if
it is a
$\left(\Theta,\mc{E}\right)-$structure
such that
for all $t\in Y$
\begin{equation}
\label{18011202}
\Gamma(\rho)_{t}
\bullet
\mc{E}(\Theta)
\subseteq
\Gamma(\pi).
\end{equation}
Here
$$
\mc{E}(\Theta)
\coloneqq
\bigcup_{B\in\Theta}
\ms{D}(B,\mc{E}),
$$
and
$S_{t}
\coloneqq
\{
F_{t}\,\vert\,
F\in S
\}$
and
$F_{t}
\in\prod_{x\in X}
\mc{L}(\mf{E}_{x})$
such that
$F_{t}(x)\coloneqq F(x)(t)$,
for all
$S\subseteq\prod_{x\in X}
\mc{L}(\mf{E}_{x})^{Y}$
$t\in Y$,
and
$F\in S$.
\end{definition}
\begin{remark}
\label{22442606}
Let
$\lr{\mf{V},\mf{W}}{X,Y}$
be
a
$\left(\Theta,\mc{E}\right)-$structure.
Then for all $x\in X$
\begin{equation}
\label{15073006}
\mf{R}_{x}=
\{
\sup_{(K,j,B)\in O}
q_{(K,j,B)}^{x}\up\mf{M}_{x}
\,\vert\,
O\in\p_{\omega}
(Comp(Y)\times J\times \Theta)
\}
\end{equation}
where
by using the notation of Def. \ref{10282712}
we set
$\n=\{\nu_{j}^{x}\,\vert\, j\in J\}$
and
for all
$K\in Comp(Y)$,$
j\in J$,
$B\in\Theta$
\begin{equation}
\label{22542906}
q_{(K,j,B)}^{x}:
\cc{c}{Y,\mc{L}_{S_{x}}(\mf{E}_{x})}
\ni f_{x}
\mapsto
\sup_{t\in K}\sup_{v\in\ms{D}(B,\mc{E})}
\nu_{j}^{x}\left(f_{x}(t)v(x)\right).
\end{equation}
\end{remark}
\begin{remark}
\label{21022406}
Let
$\mf{V}
\coloneqq
\lr{\lr{\mf{E}}{\tau}}
{\pi,X,\n}$
be
a
bundle, 
$
\ms{M}
=
\{
\lr{\ms{M}_{x}}{\mf{R}_{x}}
\}_{x\in X}
$
a
map 
pre-bundle
relative
to
$\lr{X,Y}{\{\lr{\mf{E}_{x}}{\n_{x}}
\}_{x\in X},\mc{S}}$
and
$\mc{M}$
satisfy
$FM(3)-FM(4)$
with respect
to 
$
\ms{M}
$.
Then
Rmk. \ref{17471910B}
allows us to
construct
$\mf{W}$
satisfying
the condition
$(5)$ 
in Def. \ref{10282712}.
\end{remark}
The following 
characterization of
$\mc{U}\in\Gamma_{U}^{x_{\infty}}(\rho)$
will be basic in the sequel.
\begin{lemma}
\label{15482712}
Let
$\lr{\mf{V},\mf{W}}{X,Y}$
be
a
$\left(\Theta,\mc{E}\right)-$structure,
$x_{\infty}\in W\subseteq X$
and
$\mc{U}\in\prod_{x\in W}^{b}\mf{M}_{x}$.
By using the notation in
Def.
\ref{10282712}
we have
$
(1)
\Leftarrow
(2)
\Leftarrow
(3)
\Leftrightarrow
(4)
$
moreover
if
$\mf{W}$
is locally
full
$(1)
\Leftrightarrow
(2)
\Leftrightarrow
(3)
\Leftrightarrow
(4)
$,
finally
if 
$\mf{W}$
is full
we can choose
$U=X$
in $(2)$
and
$U'=X$
in
$(3)$
and 
$(4)$.
Here
\begin{enumerate}
\item
$\mc{U}\in\Gamma_{W}^{x_{\infty}}(\rho)$;
\item
$
(\exists\,U\in Op(X)\,\vert\, U\ni x_{\infty})
(\exists\,F\in\Gamma_{U}(\rho))
(F(x_{\infty})
=\mc{U}(x_{\infty}))
$
such that
$
(\forall j\in J)
(\forall K\in Comp(Y))
(\forall B\in\Theta)
$
\begin{equation}
\label{02022912}
\boxed{
\lim_{z\to x_{\infty},z\in W\cap U}
\sup_{t\in K}
\sup_{v\in\ms{D}(B,\mc{E})}
\nu_{j}\left(\mc{U}(z)(t)v(z)-F(z)(t)v(z)
\right)=0;}
\end{equation}
\item
$
(\exists\,U'\in Op(X)\,\vert\, U'\ni x_{\infty})
(\exists\,F\in\Gamma_{U'}(\rho))
(F(x_{\infty})
=\mc{U}(x_{\infty}))
$
and
$
(\forall U\in Op(X)\,\vert\, U\ni x_{\infty})
(\forall\,F\in\Gamma_{U}(\rho)
\,\vert\,
F(x_{\infty})
=\mc{U}(x_{\infty}))
$
we have 
\eqref{02022912}
$(\forall j\in J)
(\forall K\in Comp(Y))
(\forall B\in\Theta)$;
\item
$
(\exists\,U'\in Op(X)\,\vert\, U'\ni x_{\infty})
(\exists\,F\in\Gamma_{U'}(\rho))
(F(x_{\infty})
=\mc{U}(x_{\infty}))
$
and
$\mc{U}\in\Gamma_{W}^{x_{\infty}}(\rho)$.
\end{enumerate}
\end{lemma}
\begin{proof}
Since Cor. \ref{28111707}
and 
Def. \ref{17471910A}.
\end{proof}
\begin{corollary}
\label{11121102}
Let us assume the hypotheses
of Lemma \ref{15482712}
and that
$\mf{W}$
is
full.
Moreover
let
$B\in\Theta$
and
$v\in
\ms{D}(B,\mc{E})$.
Then
$(1)
\Rightarrow
(2)$,
where
\begin{enumerate}
\item
$\mc{U}\in\Gamma_{W}^{x_{\infty}}(\rho)$
and
$\exists\,F\in\Gamma(\rho)$
such that
$F(x_{\infty})
=\mc{U}(x_{\infty})$
and
$(\forall t\in Y)
(F(\cdot)(t)
\bullet 
v\in\Gamma(\pi))$;
\item
$(\forall t\in X)
(\mc{U}(\cdot)(t)\bullet v\in
\Gamma_{W}^{x_{\infty}}(\pi))$.
\end{enumerate}
\end{corollary}
\begin{proof}
By the position
$(1)$
and
by the implication
$(1)
\Rightarrow
(3)$
of Lemma \ref{15482712}
and by
the fact that
the union of all compact subsets
of $Y$ is $Y$, being locally compact,
we deduce that
$(\exists\,F\in\Gamma(\rho))
(F(x_{\infty})
=\mc{U}(x_{\infty}))$
such that
$(\forall j\in J)(\forall t\in Y)
(\forall B\in\Theta)$
and $\forall v\in\ms{D}(B,\mc{E})$
\begin{equation*}
\begin{cases}
\lim_{z\to x_{\infty},z\in W}
\nu_{j}
\left(
\mc{U}(z)(t)v(z)-F(z)(t)v(z)
\right)
=0,
\\
F(\cdot)(t)\bullet v
\in\Gamma(\pi).
\end{cases}
\end{equation*}
Thus the statement follows
by implication
$(3)\Rightarrow(1)$ 
of Cor. \ref{28111707}.
\end{proof}
Let us conclude this section
with two results 
constructing 
a
$\left(\Theta,\mc{E}\right)-$structure
and
describing
$\Gamma^{x_{\infty}}(\rho)$
when
$\mf{V}$
is trivial.
\begin{lemma}
\label{22312406}
Let $Z$ be a normed space
$X,Y$ be two topological spaces.
Set for all $x\in X$ 
and $v\in\cc{b}{X,Z}$
\begin{equation*}
\begin{cases}
\begin{aligned}
\mc{M}
\coloneqq
&
\{
F\in\cc{b}{X,\cc{c}{Y,\mc{L}_{s}(Z)}}
\,\vert\,
(\forall K\in Comp(Y))
\\
&
(C(F,K)\coloneqq
\sup_{(x,s)\in X\times K}
\|F(x)(s)\|_{B(Z)}<\infty)
\},
\end{aligned}
\\
\ms{M}_{x}
\coloneqq\ov{\{
F(x)\,\vert\, F\in\mc{M}
\}},
\\
\mu_{(v,x)}^{K}:
\ms{M}_{x}\ni G\mapsto
\sup_{s\in K}
\|G(s)v(x)\|,
\\
\A_{x}
\coloneqq\{\mu_{(w,x)}^{K}
\,\vert\,
K\in Comp(Y),
w\in\cc{b}{X,Z}
\},
\\
\ms{M}
\coloneqq\{
\lr{\ms{M}_{x}}{\A_{x}}\}_{x\in X}.
\end{cases}
\end{equation*}
closure in $\cc{c}{Y,B_{s}(Z)}$.
Then
$\mc{M}$ satisfies $FM3-FM4$ with respect
to $\ms{M}$
\end{lemma}
\begin{proof}
$FM(3)$ is true by construction,
let $v\in\cc{b}{X,Z}$,
$K\in Comp(Y)$,
$F\in\mc{M}$,
then
$$
\sup_{x\in X}
\mu_{(v,x)}^{K}(F(x))
\leq
\sup_{(x,s)\in X\times K}
\|F(x)(s)\|_{B(Z)}
\sup_{x\in X}
\|v(x)\|
<\infty.
$$
For all $x,x_{0}\in X$
\begin{equation}
\label{15472506}
\mu_{(v,x)}^{K}(F(x))
\leq
C
\|v(x)-v(x_{0})\|
+
\sup_{s\in K}
\|F(x)(s)v(x_{0})\|,
\end{equation}
where
$C
\coloneqq
\sup_{(x,s)\in X\times K}
\|F(x)(s)\|_{B(Z)}$.
Moreover
the map
$\cc{c}{Y,B_{s}(Z)}
\ni
f
\mapsto
\sup_{s\in K}\|f(s)w\|
\in
\R^{+}$,
for all $w\in Z$
is a continuous seminorm,
hence by the continuity of $F$
also the map
$X\ni x\mapsto 
\sup_{s\in K}\|F(x)(s)w\|
\in
\R^{+}$
is continuous.
So
by \eqref{15472506}
$\varlimsup_{x\to x_{0}}
\mu_{(v,x)}^{K}(F(x))
\leq
\sup_{s\in K}
\|F(x_{0})(s)v(x_{0})\|
=
\mu_{(v,x_{0})}^{K}(F(x_{0}))$,
and by 
\cite[$(15)$, $\S 5.6$]{BourGT}
we have
$$
\varlimsup_{x\to x_{0}}
\mu_{(v,x)}^{K}(F(x))
=
\mu_{(v,x_{0})}^{K}(F(x_{0})).
$$
Therefore
by 
\cite[$(13)$, $\S 5.6$]{BourGT},
\cite[Prp. $3$, $\S 6.2$]{BourGT},
and
the fact that 
any map $g$
is $u.s.c.$ 
at a point 
iff $-g$ 
is $l.s.c.$,
we can state that
$X\ni x\mapsto \mu_{(v,x)}^{K}(F(x))$ 
is 
$u.s.c.$ 
at $x_{0}$ for all $x_{0}\in X$, 
hence 
it is $u.s.c.$,
which is the $FM(4)$ condition.
\end{proof}
\begin{remark}
\label{05210948}
Let
$\mf{V}
\coloneqq
\lr{\lr{\mf{E}}{\tau}}
{\pi,X,\n}
$
be a bundle of $\Omega-$spaces
and
$\mc{E}\subseteq\prod_{x\in X}\mf{E}_{x}$.
Set for all $v\in\prod_{x\in X}\mf{E}_{x}$
\begin{equation*}
\begin{cases}
B_{v}:X\ni x
\mapsto
\{v(x)\},
\\
\Theta
\coloneqq
\left\{
B_{w}
\,\vert\,
w\in\mc{E}
\right\}
\end{cases}
\end{equation*}
Thus
$\Theta\subset
\prod_{x\in X}
Bounded(\mf{E}_{x})$
and
$\forall v\in\mc{E}$
\begin{equation*}
\mc{E}
\cap
\prod_{x\in X}
B_{v}(x)
=
\{v\}.
\end{equation*}
Therefore
for all $v\in\mc{E}$,
and
for all $x\in X$
with the notation of Def. \ref{10282712}
\begin{equation*}
\begin{cases}
\ms{D}(B_{v},\mc{E})
=
\{v\},
\\
\mc{B}_{B_{v}}^{x}
=\{v(x)\},\,
\\
S_{x}=\{\{w(x)\}
\,\vert\, w\in\mc{E}\},
\\
\mc{E}(\Theta)
=
\mc{E}.
\end{cases}
\end{equation*}
\end{remark}
By Lemma \ref{22312406}
and 
Def. \ref{17471910Ba}
we can construct
the
bundle
$\mf{V}(\ms{M},\mc{M})$
generated
by
the couple
$\lr{\ms{M}}{\mc{M}}$.
In 
the following result
we construct
a $\left(\Theta,\mc{E}\right)-$structure
and describe a subset
of $\Gamma^{x_{\infty}}(\rho)$.
\begin{theorem}
\label{22372406}
Let 
us assume the notation
and hypotheses of Lemma \ref{22312406},
let
$\mf{V}$
be the trivial
Banach bundle
with constant stalk $Z$
and set
$\Theta
\coloneqq
\left\{
B_{v}
\,\vert\,
v\in\cc{b}{X,Z}
\right\}$.
Then
\begin{enumerate}
\item
$\lr{\mf{V},\mf{V}(\ms{M},\mc{M})}
{X,Y}$
is a 
$(\Theta,\cc{b}{X,Z})-$
structure,
moreover
if $X$ is compact 
and $Y$ is locally compact
then
it is compatible;
\item
Let
$f\in\prod_{x\in\ X}\ms{M}_{x}$
be
such that
$\sup_{(x,s)\in X\times K}
\|f(x)(s)\|_{B(Z)}<\infty$
for all
$K\in Comp(Y)$
then
$
(a)
\Leftrightarrow
(b)
\Leftrightarrow
(c)
\Leftrightarrow
(d)$,
where
\begin{enumerate}
\item
$f\in\Gamma^{x_{\infty}}(\pi_{\ms{M}})$;
\item
$(\forall K\in Comp(Y))
(\forall v\in\cc{b}{X,Z})$
\begin{equation*}
\lim_{x\to x_{\infty}}
\sup_{s\in K}
\|
f(x)(s)v(x)
-
f(x_{\infty})(s)v(x)
\|
=0
\end{equation*}
\item
$f:X\to\cc{c}{Y,B_{s}(Z)}$
continuous at $x_{\infty}$;
\item
$(\forall K\in Comp(Y))
(\forall w\in Z)$
\begin{equation*}
\lim_{x\to x_{\infty}}
\sup_{s\in K}
\|
f(x)(s)w
-
f(x_{\infty})(s)w
\|
=0.
\end{equation*}
\end{enumerate}
\end{enumerate}
\end{theorem}
\begin{proof}
By Rmk. \ref{21022406} and 
Lemma \ref{22312406}
we have that $(5)$ of Def. \ref{10282712}
follows.
$\Gamma(\pi)\simeq\cc{b}{X,Z}$
hence
by Rmk. \ref{05210948}
the other requests of Def. \ref{10282712} follow.
Thus the first sentence of
statement $(1)$.
If $X$ is compact 
by 
Lemma \ref{22312406}
and
Rmk. \ref{17150312}
follows
that 
$\mc{M}\simeq\Gamma(\pi_{\ms{M}})$,
moreover by Rmk. \ref{05210948}
we have 
$\mc{E}(\Theta)=\mc{E}$
and finally
$\mc{E}\doteq
\Gamma(\pi)\simeq\cc{b}{X,Z}
$.
Hence 
the second sentence of statement
$(1)$ follows if we
show that
$\mc{M}_{t}
\bullet
\cc{b}{X,Z}
\subseteq
\cc{b}{X,Z}$.
To this end fix 
$v\in\cc{b}{X,Z}$,
$F\in\mc{M}$,
$s\in Y$
and $K_{s}$ a compact neighbourhood
of $s$, which there exists by the hypothesis
that $Y$ is locally compact.
Then we have for all $x,x_{0}\in X$
\begin{equation}
\begin{aligned}
\label{16102506}
&
\|
F(x)(s)v(x)-F(x_{\infty})(s)v(x_{0})
\|
\leq
\\
&
C(F,K_{s})
\|v(x)-v(x_{0})\|
+
\|\left(F(x)(s)-
F(x_{0})(s)
\right)v(x_{0})\|
\end{aligned}
\end{equation}
By considering
that
$F\in\cc{b}{X,\cc{c}{Y,B_{s}(Z)}}$
and that
$s\in K_{s}$
we have that
$\lim_{x\to x_{0}}
\|(F(x)(s)-F(x_{0})(s))v(x_{0})\|
=0$.
Hence
by \eqref{16102506}
we deduce that
$F_{s}\bullet v$
is continuous at $x_{0}$,
so continuous on $X$, in particular
$X$ being compact it is also
$\|\cdot\|_{Z}-$bounded.
Thus
$F_{s}\bullet v\in\cc{b}{X,Z}$
and the second sentence of the statement
follows.
\par
Fix
$f\in\prod_{x\in\ X}\ms{M}_{x}$
such that
$(\forall K\in Comp(Y))
(C(f,K)\coloneqq
\sup_{(x,s)\in X\times K}
\|f(x)(s)\|_{B(Z)}<\infty)$.
$(a)\Leftrightarrow(b)$
follows by Lemma \ref{15482712},
the fact that
$\mc{M}\subseteq\Gamma(\pi_{\ms{M}})$
by Rmk. \ref{17150312},
and
by 
$(H:X\ni x\mapsto f(x_{\infty})\in
\cc{c}{Y,B_{s}(Z)})\in\mc{M}$,
indeed
$H$
it is 
bounded and continuous
being constant, moreover
$\sup_{(x,s)\in X\times K}
\|H(x)(s)\|_{B(Z)}
=
\sup_{s\in K}
\|f(x_{\infty})(s)\|_{B(Z)}
<\infty$,
for all $K\in Comp(Y)$.
$(b)\Rightarrow(d)$
follows by the fact that
$(X\ni x\mapsto w\in Z)\in\cc{b}{X,Z}$,
and
$(c)\Leftrightarrow(d)$
is trivial.
For all $K\in Comp(Y)$,
$x\in X$ and $s\in K$
\begin{equation*}
\begin{aligned}
\|
(f(x)(s)
-
f(x_{\infty})(s))
v(x)
\|
&
\leq
\\
\|
f(x)(s)v(x)
-
f(x_{\infty})(s)v(x_{\infty})
\|
+
\|
f(x_{\infty})(s)v(x_{\infty})
-
f(x_{\infty})(s)v(x)
\|
&
\leq
\\
\|f(x)(s)(v(x)-v(x_{\infty}))\|
+
\|(f(x)(s)-f(x_{\infty})(s))v(x_{\infty})\|
+
\|f(x_{\infty})(s)(v(x_{\infty})-v(x))\|
&
\leq
\\
\left(
\|f(x)(s)\|
+
\|f(x_{\infty})(s)\|
\right)
\|v(x_{\infty})-v(x)\|
+
\|(f(x)(s)-f(x_{\infty})(s))v(x_{\infty})\|
&
\leq
\\
2C(f,K)\|v(x_{\infty})-v(x)\|
+
\|(f(x)(s)-f(x_{\infty})(s))v(x_{\infty})\|.
\end{aligned}
\end{equation*}
Hence 
$(d)$ 
implies
$(b)$.
\end{proof}
\begin{definition}
\label{21031238}
Let
$\lr{\mf{V},\mf{W}}{X,Y}$
be
a
$\left(\Theta,\mc{E}\right)-$structure,
$Y_{0}\subset Y$
and
$\mc{V}\in\prod_{x\in X}\mf{M}_{x}$.
We say that
$\mc{V}$ is equicontinuous
on $Y_{0}$
iff
$(\forall j\in J)
(\exists a>0)
(\exists\,j_{1}\in J)
(\forall z\in X)
(\forall v_{z}\in\mf{E}_{z})
$
\begin{equation}
\label{15422103}
\sup_{t\in Y_{0}}
\nu_{j}\left(\mc{V}(z)(t)v_{z}\right)
\leq
a
\nu_{j_{1}}(v_{z}).
\end{equation}
$\mc{V}$ is 
equicontinuous
iff
it is equicontinuous
on $Y$.
$\mc{V}$ is 
pointwise
equicontinuous
iff 
it is equicontinuous
on every point
of
$Y$
and
compactly
equicontinuous
iff 
it is equicontinuous
on every compact of
$Y$.
\end{definition}
Note that
in case $\mf{V}$ is trivial
with costant stalk $E$
then
$\mc{V}$ is equicontinuous
on $Y_{0}$
if and only if
it is equicontinuous
in the standard sense\footnote{See for instance \cite[Def $1$, $\S 2.1$, Ch. $10$]{BourGT}.}
the following set of maps
$\{
\mc{V}_{0}(z)(t)
\in\mc{L}(E)
\,\vert\,
(z,t)\in X\times Y_{0}
\}$,
where
$\mc{V}_{0}
\in
\left(\mc{L}(E)^{Y}
\right)^{X}$
such that
$\mc{V}(z)=(z,\mc{V}_{0}(z))$
for all $z\in X$.
\begin{proposition}
\label{20492003}
Let
$\mf{V}$
be trivial with 
costant stalk
$E$,
$A^{0}\in Bounded(E)$,
$x_{\infty}\in X$
and
\begin{equation}
\label{16062103}
\begin{cases}
\mc{E}_{0}
\subseteq
\cc{b}{X,E}
\\
\mc{E}_{0}
\text{ equicontinuous 
set
at $x_{\infty}$}
\\
\{(X\ni x\mapsto a\in E)\,\vert\, a\in A^{0}\}
\subset
\mc{E}_{0}.
\end{cases}
\end{equation}
Moreover let
$\lr{\mf{V},\mf{W}}{X,Y}$
be
a
$\left(\Theta,\mc{E}\right)-$structure
such that
for all $x\in X$
$$
\mf{M}_{x}
=
\cc{c}{Y,\mc{L}_{S_{x}}(\{x\}\times E)}.
$$
and
$$
\begin{cases}
\mc{E}
=\prod_{x\in X}
\{x\}\times\mc{E}_{0}
\\
\Theta
=
\{
B_{A^{0}}
\}
\end{cases}
$$
where
$B_{A^{0}}(x)
\coloneqq
\{x\}\times A^{0}$,
then
\begin{equation}
\label{23412003}
\begin{cases}
S_{x}
=
\{x\}\times A^{0},\forall x\in X
\\
\mf{M}_{x}
\simeq
\{x\}\times 
\cc{c}{Y,\mc{L}_{A^{0}}(E)}.
\\
\mf{M}
=
\bigcup_{x\in X}
\mf{M}_{x}
\simeq
\bigcup_{x\in X}
\{x\}
\times
\cc{c}{Y,\mc{L}_{A^{0}}(E)}
\\
\prod_{x\in X}
\mf{M}_{x}
\simeq
\prod_{x\in X}
\{x\}
\times
\cc{c}{Y,\mc{L}_{A^{0}}(E)}
\simeq
\cc{c}{Y,\mc{L}_{A^{0}}(E)}^{X}.
\end{cases}
\end{equation}
If
$\mf{W}$ is full
and 
$$
\{
X\ni x\mapsto
\mf{t}_{f}(x)
=(x,f)\in\mf{M}_{x}
\,\vert\, 
f
\in
\cc{c}{Y,\mc{L}_{A^{0}}(E)}
\}
\subset
\Gamma(\rho),
$$
then
for all
$\mc{V}\in\prod_{x\in X}^{b}\mf{M}_{x}$,
$(1)\Rightarrow(2)$
and
$(3)\Leftrightarrow(4)$,
where
\begin{enumerate}
\item
$\mc{V}\in\Gamma^{x_{\infty}}(\rho)$
\item
$\mc{V}_{0}
\in\cc{}{X,\cc{c}{Y,\mc{L}_{A^{0}}(E)}}$,
\item
$\mc{V}$ 
is
compactly
equicontinuous
and
$\mc{V}\in\Gamma^{x_{\infty}}(\rho)$
\item
$\mc{V}$
is
compactly
equicontinuous
and
$\mc{V}_{0}
\in\cc{}{X,\cc{c}{Y,\mc{L}_{A^{0}}(E)}}$.
\end{enumerate}
Here
in $(2)-(4)$
we consider
the
isomorphism
$\prod_{x\in X}
\mf{M}_{x}
\simeq
\cc{c}{Y,\mc{L}_{A^{0}}(E)}^{X}$,
and
set
$\mc{V}_{0}
\in
\cc{c}{Y,\mc{L}_{A^{0}}(E)}^{X}$
such that
$\mc{V}(x)=(x,\mc{V}_{0}(x))$
for all $x\in X$.
\end{proposition}
\begin{proof}
For all $x\in X$
by \eqref{11232712}
$\mc{B}_{B_{A^{0}}}^{x}
=
\{(x,v_{0}(x))
\,\vert\,
v_{0}\in\mc{E}_{0},
v_{0}(X)\subseteq A^{0}\}
$
so
$\mc{B}_{B_{A^{0}}}^{x}\subseteq A^{0}$.
Moreover
by construction
$(X\ni x\mapsto a\in E)
\in\mc{E}_{0}$
for all
$a\in A^{0}$,
thus
$\mc{B}_{B_{A^{0}}}^{x}= A^{0}$.
Thus
the first equality in
\eqref{23412003}
follows,
the others are trivial.
By Prp.
\ref{16572003}
\begin{equation}
\label{15102103}
(1)
\Leftrightarrow
\lim_{z\to x_{\infty}}
\sup_{t\in K}
\sup_{v_{0}\in
\mc{E}_{0}
\cap B_{A^{0}}
}
\nu_{j}\left((
\mc{V}_{0}(z)(t)-
\mc{V}_{0}(x_{\infty})(t))v_{0}(z)\right)
=0.
\end{equation}
Moreover 
by construction we deduce
that
$\{(X\ni x\mapsto a\in E)\,\vert\, a\in A^{0}\}
\subset
\mc{E}_{0}
\cap B_{A^{0}}$,
so
$(2)$
follows
by $(1)$
and \eqref{15102103}.
Let $v_{0}\in\mc{E}_{0}$
then for all
$z\in X$ and $t\in Y$
\begin{alignat}{1}
\label{15352103}
(\mc{V}(z)(t)-
\mc{V}(x_{\infty})(t))
v_{0}(z)
&=
\mc{V}(z)(t)(v_{0}(z)-v_{0}(x_{\infty}))
+
\notag\\
(\mc{V}(z)(t)-
\mc{V}(x_{\infty})(t))
v_{0}(x_{\infty})
&
+
\mc{V}(x_{\infty})(t)
(v_{0}(z)-v_{0}(x_{\infty})).
\end{alignat}
Moreover
by 
the hypothesis of equicontinuity
at $x_{\infty}$
of the set
$\mc{E}_{0}$,
for all $j\in J$
\begin{equation}
\label{15542103}
\lim_{z\to x_{\infty}}
\sup_{v_{0}\in\mc{E}_{0}}
\nu_{j}(v_{0}(z)-v_{0}(x_{\infty}))
=0.
\end{equation}
By
\eqref{15352103}
and
\eqref{15422103}
for all $j\in J$
there exists $j_{1}\in J$
and $a>0$
such that
for all $z\in X$
\begin{alignat}{1}
\label{15562103}
&
\sup_{t\in K}
\sup_{v_{0}\in
\mc{E}_{0}
\cap B_{A^{0}}
}
\nu_{j}\left((
\mc{V}_{0}(z)(t)-
\mc{V}_{0}(x_{\infty})(t))v_{0}(z)\right)
\leq
\notag\\
&
2a
\sup_{v_{0}\in
\mc{E}_{0}
\cap B_{A^{0}}
}
\nu_{j_{1}}
\left(v_{0}(z)-v_{0}(x_{\infty})\right)
+
\notag\\
&
\sup_{t\in K}
\sup_{v_{0}\in
\mc{E}_{0}
\cap B_{A^{0}}
}
\nu_{j}\left(
\mc{V}(z)(t)-
\mc{V}(x_{\infty})(t)\right)
v_{0}(x_{\infty}).
\end{alignat}
Therefore
by 
\eqref{15562103},
\eqref{15542103}
and by $(4)$
follows
$$
\lim_{z\to x_{\infty}}
\sup_{t\in K}
\sup_{v_{0}\in
\mc{E}_{0}
\cap B_{A^{0}}
}
\nu_{j}\left((
\mc{V}_{0}(z)(t)-
\mc{V}_{0}(x_{\infty})(t))v_{0}(z)\right)
=0.
$$
Hence
$(1)$
follows
by
\eqref{15102103}.
\end{proof}
\section{Main Claim}
In this section we state in a precise way the claims outlined in Introduction.
The main Claim \ref{19520412bis} which essentially establishes the existence of $\mc{T}$ and $\mc{P}$
satisfying \eqref{19240703}, \eqref{17321002} and \eqref{20020803}. 
The auxiliary Claim \ref{20090412bis} which provides $\mc{U}$ satisfying 
\eqref{05211905}
and  then Claim \ref{20290412bis}.
Prp. \ref{20340412bis} provides the main properties of those realizations of the main claim
obtained combining realizations of the two auxiliary ones.
We anticipate that 
Thm. \ref{13020103}
resolves the main claim in this fashion.
In what follows when dealing with bundle direct sums we use 
the notation provided in Rmk. \ref{21262110}.
\begin{definition}
\label{12432110bis}
Let 
$\mf{V}_{i}
\coloneqq
\lr{\lr{\mf{E}_{i}}{\tau_{i}}}
{\pi_{i},X,\n_{i}}
$
be a full bundle of $\Omega-$spaces
for any $i=1,2$.
Then we call set of graph sections relative to $\mf{V}_{1}$ and $\mf{V}_{2}$
the set $\ms{Gr}(\mf{V}_{1},\mf{V}_{2})$
of the elements $\lr{\mc{T}}{x_{\infty},\Phi}$
such that 
\begin{enumerate}
\item
$
\mc{T}
\in
\prod_{x\in X}
Graph(\left(\mf{E}_{1}\right)_{x}
\times
\left(\mf{E}_{2}\right)_{x})
$;
\item
$x_{\infty}\in X$;
\item
$\Phi$
is a linear subspace of 
$\Gamma^{x_{\infty}}(\pi_{\ms{E}^{\oplus}})$;
\item
$(\forall x\in X)
(\forall \phi\in\Phi)
(\phi(x)\in\mc{T}(x))$
\item
Asymptotic Graph
\begin{equation}
\label{19350512bis}
\boxed{
\left\{
\phi(x_{\infty})
\,\vert\,
\phi
\in
\Phi
\right\}
=
\mc{T}(x_{\infty}).
}
\end{equation}
\end{enumerate}
\end{definition}
\begin{definition}
\label{16161212bis}
Let 
$
\mf{V}_{i}
\coloneqq
\lr{\lr{\mf{E}_{i}}{\tau_{i}}}
{\pi_{i},X,\n_{i}}
$
be a full bundle of $\Omega-$spaces for any $i=1,2$.
Then we call set of pregraph sections relative to $\mf{V}_{1}$ and $\mf{V}_{2}$
the set $Pregraph\left(\mf{V}_{1},\mf{V}_{2}\right)$
of the elements $\lr{\mc{T}_{0}}{x_{\infty},\Phi}$
such that 
\begin{enumerate}
\item
$x_{\infty}\in X$;
\item
$
\mc{T}_{0}
\in
\prod_{x\in X-\{x_{\infty}\}}
Graph(\left(\mf{E}_{1}\right)_{x}
\times
\left(\mf{E}_{2}\right)_{x})
$;
\item
$\Phi$
is a linear subspace of 
$\Gamma^{x_{\infty}}(\pi_{\ms{E}^{\oplus}})$;
\item
$(\forall x\in X-\{x_{\infty}\})
(\forall \phi\in\Phi)
(\phi(x)\in\mc{T}_{0}(x))$.
\end{enumerate}
\end{definition}
We shall see in 
Lemma \ref{13001512b}
that 
it is possible to
construct
from
any suitable 
pregraph
section
$
\lr{\mc{T}_{0}}{x_{\infty},\Phi}
$
a corresponding
graph
section
$\lr{\mc{T}}{x_{\infty},\Phi}$
such that $\mc{T}$ extends $\mc{T}_{0}$,
while
$\mc{T}(x_{\infty})$
is defined by 
\eqref{19350512bis}.
To this end 
it is sufficient
to show that
$\mc{T}(x_{\infty})
\in
Graph((\mf{E}_{1})_{x_{\infty}}
\times(\mf{E}_{2})_{x_{\infty}})$.
\begin{remark}
\label{23112110bis}
The request that any
$\phi\in\Phi$ 
is a section
continuous in $x_{\infty}$
implies
that
$$
\begin{cases}
\left\{
\lim_{z\to x_{\infty}}
\phi(z)
\,\vert\,
\phi\in\Phi
\right\}
=
\mc{T}(x_{\infty})
\in 
Graph(\left(\mf{E}_{1}\right)_{x_{\infty}}
\times
\left(\mf{E}_{2}\right)_{x_{\infty}})
\\
\text{with}
\\
\phi(z)\in
\mc{T}(z)
\in 
Graph(\left(\mf{E}_{1}\right)_{z}
\times
\left(\mf{E}_{2}\right)_{z}),\,
\forall z\in X-\{x_{\infty}\},
\end{cases}
$$
which
justifies
the name
of 
asymptotic graph 
given
to
\eqref{19350512bis}.
Moreover
by
setting
$X\ni z\mapsto
\phi_{i}(z)
\coloneqq
\Pr_{i}^{z}(\phi(z))$
we have
by 
Cor.
\ref{17571212}
for all $i=1,2$
\begin{equation}
\begin{cases}
\label{09382412bis}
\left\{
\lim_{z\to x_{\infty}}
\phi_{i}(z)
\,\vert\,
\phi\in\Phi
\right\}
=
\Pr_{i}^{x_{\infty}}(\mc{T}(x_{\infty}))
\\
\text{with}
\\
\phi(z)\in
\mc{T}(z)
\in 
Graph(\left(\mf{E}_{1}\right)_{z}
\times
\left(\mf{E}_{2}\right)_{z}),\,
\forall z\in X-\{x_{\infty}\}.
\end{cases}
\end{equation}
Finally
for 
$i=1,2$
by 
Cor.
\ref{17571212}
and
Cor.
\ref{28111707}
we have
$(1_{i})
\Leftrightarrow
(2_{i})$
\begin{description}
\item[$(1_{i})$]
$(\exists\,
\sigma\in\Gamma(\pi))
(\sigma(x_{\infty})
=\phi_{i}(x_{\infty}))$
such that
$$
(\forall j\in J)
(\lim_{z\to x_{\infty}}
\nu_{j}(\phi_{i}(z)-\sigma(z))=0);
$$
\item[$(2_{i})$]
$(\forall\sigma\in\Gamma(\pi)
\,\vert\,
\sigma(x_{\infty})
=\phi_{i}(x_{\infty}))$
we have
$$
(\forall j\in J)
(\lim_{
z\to x_{\infty}}
\nu_{j}(\phi_{i}(z)-\sigma(z))=0).
$$
\end{description}
\end{remark}
\begin{definition}
\label{15312011bis}
Let $\lr{\mf{V},\mf{D}}{X,\{pt\}}$ be
a $\left(\Theta,\mc{E}\right)-$structure
such that $\mf{V}$ is full and let us denote
$\mf{D}
\coloneqq
\lr{\lr{\mf{B}}{\gamma}}
{\eta,X,\mf{L}}$.
Thus
$\Omega\,\in
\Delta
\lr{\mf{V},\mf{D}}
{\Theta,\mc{E}}$
if
\begin{enumerate}
\item
$
\Omega
\subseteq
\ms{Gr}(\mf{V},\mf{V})
$;
\item
Section of projectors associated with 
$\lr{\mc{T}}{x_{\infty},\Phi}$:
$
\forall
\lr{\mc{T}}{x_{\infty},\Phi}
\in
\Omega
$
\begin{equation}
\label{18270303}
\boxed{
(\exists\,
\mc{P}
\in
\Gamma^{x_{\infty}}(\eta)
\cap
\prod_{x\in X}
\Pr(\mf{E}_{x}))
\left(\forall x\in X\right)
\left(\mc{P}(x)
T_{x}
\subseteq
T_{x}
\mc{P}(x)\right).
}
\end{equation}
\end{enumerate}
Here
$
T_{x}:
D_{x}
\subseteq
\mf{E}_{x}
\to
\mf{E}_{x}
$
is the map
such that
$\mc{T}(x)=Graph(T_{x})$,
for all
$x\in X$.
\end{definition}
\begin{claim}
[MAIN]
\label{19520412bis}
Under the assumptions
in Def.
\ref{15312011bis},
possibly
with
$\lr{\mf{V},\mf{D}}{X,\{pt\}}$
invariant,
find elements
in the set
$$
\Delta\lr{\mf{V},\mf{D}}
{\Theta,\mc{E}}.
$$
\end{claim}
\begin{definition}
\label{19490412bis}
Let
$\lr{\mf{V},\mf{W}}{X,\R^{+}}$
be
a
$\left(\Theta,\mc{E}\right)-$structure.
Let us denote
$\mf{V}\coloneqq
\lr{\lr{\mf{E}}{\tau}}
{\pi,X,\n}$
and
$\mf{W}\coloneqq
\lr{\lr{\mf{M}}{\gamma}}{\rho,X,\mf{R}}$.
We require that $\mf{V}$ is full,
$\{\mf{E}_{x}\}_{x\in X}$
is a 
family 
of sequentially complete
Hlcs
and
$\ms{U}(\mc{L}_{S_{x}}(\mf{E}_{x}))
\subset
\mf{M}_{x}$
for all $x\in X$.
Then
$\Omega\in
\Delta_{\Theta}\lr{\mf{V},\mf{W}}
{\mc{E},X,\R^{+}}$
iff
\begin{enumerate}
\item
$
\Omega
\subseteq
\ms{Gr}(\mf{V},\mf{V})$;
\item
Section of semigroups associated with 
$\lr{\mc{T}}{x_{\infty},\Phi}$:
$\forall
\lr{\mc{T}}{x_{\infty},\Phi}
\in
\Omega
$
$$
\boxed{
\exists\,
\mc{U}_{\lr{\mc{T}}{x_{\infty},\Phi}}
\in
\Gamma^{x_{\infty}}(\rho)
}
$$
such that
$\forall x\in X$
\begin{enumerate}
\item
$
\mc{U}_{\lr{\mc{T}}{x_{\infty},\Phi}}(x)
$
is an equicontinuous
$(C_{0})-$semigroup
on
$\mf{E}_{x}$;
\item
$
(\forall x\in X)
(\mc{T}(x)=Graph(R_{x}))
$.
\end{enumerate}
\end{enumerate}
Here
$R_{x}$
is the infinitesimal generator
of the semigroup
$
\mc{U}_{\lr{\mc{T}}{x_{\infty},\Phi}}(x)
\in
\cc{c}{\R^{+},\mc{L}_{S_{x}}
(\mf{E}_{x})}
$.
\end{definition}
\begin{claim}[S]
\label{20090412bis}
Under the assumptions
in Def.
\ref{19490412bis},
possibly
with
$\lr{\mf{V},\mf{W}}{X,\R^{+}}$
compatible,
find elements
in the set
$$
\Delta_{\Theta}\lr{\mf{V},\mf{W}}
{\mc{E},X,\R^{+}}.
$$
\end{claim}
\begin{remark}
\label{20430412bis}
Notice
that 
$
\forall
\lr{\mc{T}}{x_{\infty},\Phi}
\in
\Omega
$
there exists
only one
semigroup section
associated with it.
Moreover
$\mc{U}_{\lr{\mc{T}}{x_{\infty},\Phi}}$
is characterized
by
any
of the
equivalent 
conditions in Lemma \ref{15482712}
with
$U=X$ and $Y=\R^{+}$.
\end{remark}
\begin{definition}
\label{18550612bis}
Let
$\lr{\mf{V},\mf{W}}{X,\R^{+}}$
be
a
$\left(\Theta,\mc{E}\right)-$structure
and
$\lr{\mf{V},\mf{D}}{X,\{pt\}}$
be
a
$\left(\Theta,\mc{E}\right)-$structure.
Let us denote
$
\mf{V}
\coloneqq
\lr{\lr{\mf{E}}{\tau}}
{\pi,X,\n}
$,
$\mf{D}
\coloneqq
\lr{\lr{\mf{B}}{\gamma}}
{\eta,X,\mf{L}}$
and
$\mf{W}
\coloneqq
\lr{\lr{\mf{M}}{\gamma}}{\rho,X,\mf{R}}$.
We require that $\mf{V}$ is full,
$\{\mf{E}_{x}\}_{x\in X}$
is a 
family 
of sequentially complete
Hlcs
and
$\ms{U}(\mc{L}_{S_{x}}(\mf{E}_{x}))
\subset
\mf{M}_{x}$
for all $x\in X$.
Then
$\Psi
\in
\Delta_{\Theta}\lr{\mf{V},\mf{D},\mf{W}}
{\mc{E},X,\R^{+}}
$
iff
\begin{enumerate}
\item
$\Psi\subseteq
\bigcup_{z\in X}
\Gamma^{z}(\rho)
$;
\item
$
(\forall\mc{U}\in\Psi)
(\forall x\in X)
$
$
(\mc{U}(x)
$
is an equicontinuous
$(C_{0})-$semigroup
on
$\mf{E}_{x})$;
\item
Section of projectors associated with $\mc{U}$:
$(\forall z\in X)
(\forall\mc{U}
\in\Psi\cap\Gamma^{z}(\rho))$
\begin{equation}
\label{16221411bis}
(\exists\,
\mc{P}
\in
\Gamma^{z}(\eta)
\cap
\prod_{y\in X}
\Pr(\mf{E}_{y}))
(\forall x\in X)
(\mc{P}(x)H_{x}
\subseteq
H_{x}
\mc{P}(x)).
\end{equation}
\end{enumerate}
Here
$H_{x}$
is the infinitesimal generator
of the semigroup
$
\mc{U}(x)
\in
\cc{c}{\R^{+},\mc{L}_{S_{x}}
(\mf{E}_{x})}$
for all $x\in X$.
\end{definition}
\begin{claim}
[S-P]
\label{20290412bis}
Under the assumptions
in Def.
\ref{18550612bis},
possibly with
$\lr{\mf{V},\mf{W}}{X,\R^{+}}$
compatible
and
$\lr{\mf{V},\mf{D}}{X,\{pt\}}$
invariant,
find elements
in the set
$\Delta_{\Theta}\lr{\mf{V},\mf{D},\mf{W}}
{\mc{E},X,\R^{+}}.$
\end{claim}
Claims
\ref{20090412bis}
and
\ref{20290412bis}
can be used to solve the main claim
\ref{19520412bis} 
indeed
\begin{proposition}
\label{20340412bis}
Under the notation and request in Def. \ref{18550612bis}
assume that
\begin{enumerate}
\item
$
\Omega\in
\Delta_{\Theta}\lr{\mf{V},\mf{W}}
{\mc{E},X,\R^{+}}
$;
\item
$
\Psi
\in
\Delta_{\Theta}\lr{\mf{V},\mf{D},\mf{W}}
{\mc{E},X,\R^{+}}
$;
\item
$(\forall
\lr{\mc{T}}{x_{\infty},\Phi}
\in\Omega)
(
\mc{U}_{\lr{\mc{T}}{x_{\infty},\Phi}}
\in
\Psi)$.
\end{enumerate}
Thus
$\Omega\,\in\Delta\lr{\mf{V},\mf{D}}{\Theta,\mc{E}}$,
namely $\Omega$ satisfies the claim \ref{19520412bis}.
Moreover
$$
\left(
\forall
\lr{\mc{T}}{x_{\infty},\Phi}
\in
\Omega
\right)
\left(
\exists\,
\mc{P}
\in
\Gamma^{x_{\infty}}(\eta)
\right)
\left(
\exists\,
\mc{U}
\in
\Gamma^{x_{\infty}}(\rho)
\right)
$$
\begin{enumerate}
\item
$
\mc{U}(x)
$
is an equicontinuous
$(C_{0})-$semigroup
on
$\mf{E}_{x}$,
for all 
$x\in X$;
\item
$
\left(\forall x\in X\right)
\left(\mc{P}(x)
\in
\Pr(\mf{E}_{x})
\right)
$;
\item
$
(\forall x\in X)
(\mc{T}(x)=Graph(R_{x}))
$;
\item
$
\forall x\in X
$
$$
\mc{P}(x)R_{x}
\subseteq
R_{x}
\mc{P}(x).
$$
\end{enumerate}
Here
$R_{x}$
is the infinitesimal generator
of the semigroup
$
\mc{U}(x)
\in
\cc{c}{\R^{+},\mc{L}_{S_{x}}(\mf{E}_{x})}$,
for all $x\in X$.
\end{proposition}
We conclude this chapter by anticipating that 
Thm. \ref{17301812b} 
resolves Claim \ref{20090412bis}
while 
Thm. \ref{13020103}
resolves Claims \ref{19520412bis} and \ref{20290412bis}. 
\section*{Appendix}
Excluding Def. \ref{18111610} which is ours,
in this appendix we provide some of those definitions 
essentially present in \cite{gie} we need in the work
and some simple results concerning them. 
In this section $X$ is a topological space.
\begin{definition}
[$FM(3)-FM(4)$ in $\S 5$ of \cite{gie}]
\label{17471910s}
Let
$
\ms{V}
\coloneqq
\{\lr{V_{x}}{\A_{x}}\}_{x\in X}
$
be
a nice
family
of Hlcs
with
$
\A_{x}
\coloneqq
\{\mu_{j}^{x}
\}_{j\in J}
$
for all $x\in X$.
We say
that
$\mc{G}$
satisfies
$FM(3)-FM(4)$
with respect
to 
$\ms{V}$
if 
$\mc{G}$
is a linear subspace of
$\prod_{x\in X}^{b}
\lr{V_{x}}{\A_{x}}$
and
\begin{description}
\item[$FM(3)$]
$
\{f(x)\,\vert\, f\in \mc{G}\}
$
is dense
in $V_{x}$
for all
$x\in X$;
\item[$FM(4)$]
$
X\ni x
\mapsto
\mu_{j}^{x}(f(x))
$
is u.s.c.
$\forall j\in J$
and
$\forall f\in\mc{G}$.
\end{description}
\end{definition}
We introduce a stronger
condition namely
we say
that
$\mc{G}$
satisfies
$FM(3^{*})-FM(4)$
with respect
to 
$\ms{V}$
if
$FM(3^{*})$
and
$FM(4)$
hold
where
\begin{equation}
\label{15012310}
(\forall x\in X)
(\{f(x)\,\vert\, f\in \mc{G}\}
=
V_{x}).
\tag*{$FM(3^{*})$}
\end{equation}
\begin{definition}
\label{18111610}
Let
$
\ms{V}'
\coloneqq
\{\lr{V_{x}}{\A_{x}'}\}_{x\in X}
$
be
a 
family
of Hlcs
where
$
\A_{x}'
\coloneqq
\{\mu_{j_{x}}^{x}
\}_{j_{x}\in J_{x}}
$
is a directed family
of seminorms
on
$V_{x}$
generating 
the locally convex topology
on it,
for all $x\in X$.
Then
we set
$$
\begin{cases}
J
\coloneqq
\prod_{x\in X}J_{x};
\\
\mu_{j}^{x}
\coloneqq
\mu_{j(x)}^{x},\,
\forall x\in X, j\in J;
\\
\A_{x}
\coloneqq
\{\mu_{j}^{x}\}_{j\in J},\,
\forall x\in X.
\end{cases}
$$
\end{definition}
Clearly the range of $\A_{x}$ equals that of $\A_{x}'$ and $\A_{x}$ is directed. 
Thus $\ms{V}\coloneqq\{\lr{V_{x}}{\A_{x}}\}_{x\in X}$ is a nice family of Hlcs, called the
\emph
{nice family of Hlcs associated with $\ms{V}'$}.
\begin{definition}
[Essentially $\S 5.2$, $\S 5.3$ and Prp. 5.8 of \cite{gie}]
\label{17471910Ba}
Let
$
\ms{E}
=
\{
\lr{\ms{E}_{x}}{\n_{x}}
\}_{x\in X}
$
be
a nice
family of 
Hlcs
with
$
\n_{x}
\coloneqq
\{\nu^{x}_{j}\,\vert\, j\in J\}
$
for all $x\in X$.
Moreover
let
$\mc{E}$
satisfy
$FM(3)-FM(4)$
with respect
to 
$
\ms{E}
$.
Since \cite[Prp. 5.8]{gie} we can define
$$
\mf{V}(\ms{E},\mc{E})
$$
to be 
the
bundle
generated
by
$\lr{\ms{E}}{\mc{E}}$,
if
\begin{enumerate}
\item
$
\mf{V}(\ms{E},\mc{E})
=
\lr{\lr{\mf{E}(\ms{E})}
{\tau(\ms{E},\mc{E})}}
{\pi_{\ms{E}},X,\n}
$;
\item
$
\mf{E}(\ms{E})
\coloneqq
\bigcup_{x\in X}
\{x\}
\times
\ms{E}_{x}
$,
$
\pi_{\ms{E}}:\mf{E}(\ms{E})
\ni
(x,v)
\mapsto x
\in X
$.
\item
$
\n=\{\nu_{j}\,\vert\, j\in J\}
$,
with
$
\nu_{j}:\mf{E}(\ms{E})
\ni
(x,v)
\mapsto
\nu_{j}^{x}(v)
$;
\item
$
\tau(\ms{E},\mc{E})
$
is
the topology
on
$\mf{E}(\ms{E})$\footnote{
By applying \cite[$\S 5.3.$]{gie}
and \cite[Ch.1]{BourGT}
we know
that this topology exists.}
such that
for all
$
(x,v)\in
\mf{E}(\ms{E})
$
$$
\mc{I}_{(x,v)}^{\tau(\ms{E},\mc{E})}
\coloneqq
\F{\B_{\ms{E}}((x,v))}{\mf{E}(\ms{E})}.
$$
Here we recall that 
$\mc{I}_{(x,v)}^{\tau(\ms{E},\mc{E})}$
is the filter of neighbourhoods 
of
$(x,v)$
with respect to the topology $\tau(\ms{E},\mc{E})$,
while
$\F{\B((x,v))}{\mf{E}(\ms{E})}$
is the filter 
on
$\mf{E}(\ms{E})$
generated
by the following
base of filters
\begin{alignat*}{1}
\B_{\ms{E}}((x,v))
\coloneqq
\{
T_{\ms{E}}(U,\sigma,\ep,j)
&
\,\vert\,
U\in Open(X),
\sigma\in\mc{E},
\ep>0,
j\in J,
\\
&
U\ni x,
\nu_{j}^{x}(v-\sigma(x))
<\ep
\},
\end{alignat*}
where
\begin{equation}
\label{14332310a}
T_{\ms{E}}(U,\sigma,\ep,j)
\coloneqq
\left\{
(y,w)\in\mf{E}(\ms{E})
\,\vert\,
y\in U,
\nu_{j}^{y}(w-\sigma(y))
<\ep
\right\}.
\end{equation}
\end{enumerate}
\end{definition}
$\mc{E}$ is canonically isomorphic to a linear subspace of $\Gamma(\pi_{\ms{E}})$ indeed
\begin{remark}
\label{17150312}
Let
$
\ms{E}
=
\{
\lr{\ms{E}_{x}}{\n_{x}}
\}_{x\in X}
$
be
a nice
family of 
Hlcs
with
$
\n_{x}
\coloneqq
\{\nu^{x}_{j}\,\vert\, j\in J\}
$
for all $x\in X$.
Moreover
let
$\mc{E}$
satisfy
$FM(3)-FM(4)$
with respect
to 
$
\ms{E}
$,
and
$\mf{V}(\ms{E},\mc{E})$
be 
the
bundle
generated
by
the couple
$
\lr{\ms{E}}{\mc{E}}
$.
Thus according to \cite[Prps. 5.8, 5.9]{gie}
we have that
\begin{enumerate}
\item
$
\mf{V}(\ms{E},\mc{E})
$
is 
a bundle
of $\Omega-$spaces;
\item
$\mf{V}(\ms{E},\mc{E})$
is
such that
\begin{enumerate}
\item
$\lr{\mf{E}(\ms{E})_{x}}
{\tau(\ms{E},\mc{E})}$
as topological vector space
is isomorphic
to
$\lr{\ms{E}_{x}}{\n_{x}}$
for all
$
x\in X
$;
\item
$
\mc{E}
$
is 
canonically
isomorphic
\footnote{
I.e.
$\sigma\leftrightarrow f$
iff
$
\sigma(x)
=(x,
f(x))
$
}
to
a linear subspace
of 
$\Gamma(\pi_{\ms{E}})$
and
if 
$X$ is compact
and
$\mc{E}$
is a function module
see \cite[$\S$ 5.1]{gie},
then
$
\mc{E}
\simeq
\Gamma(\pi_{\ms{E}})
$.
\end{enumerate}
\end{enumerate}
\end{remark}
\begin{remark}
\label{15412610}
Let
$\ms{E}$
be a nice 
family of Hlcs
and let $\mc{E}$
satisfy $FM(3-4)$
with respect to
$\ms{E}$.
Thus
for all
$U\in Open(X)$,
$\sigma\in\mc{E}$,
$\ep>0$,
$j\in J$
$$
T_{\ms{E}}(U,\sigma,\ep,j)
=
\bigcup_{y\in U}
B_{\ms{E}_{y},j,\ep}(\sigma(y))
$$
where
for all 
$
s
\in
\ms{E}_{y}
$
$$
B_{\ms{E}_{y},j,\ep}(s)
\coloneqq
\left\{
(y,w)
\in\mf{E}(\ms{E})_{y}
\,\vert\,
\nu_{j}^{y}
\left(w-s\right)
<\ep
\right\}.
$$
\end{remark}
\begin{definition}
[Essentially $\S 1.5(II)$ and $\S 1.5(vii)$ of \cite{gie}]
\label{12592310}
Let 
$
\mf{P}
\coloneqq
\lr{\lr{\mf{E}}{\tau}}{p,X,\n}
$
be a 
locally full bundle
of $\Omega-$spaces,
and let us denote
$
\n
\coloneqq
\{
\nu_{j}
\,\vert\,
j\in J
\}
$.
Set
$$
\begin{cases}
\mc{K}^{loc}
\coloneqq
\prod_{\alpha\in\mf{E}}
\mc{K}_{\alpha}^{loc}
\\
\mc{K}_{\alpha}^{loc}
\coloneqq
\left\{
(U,\sigma_{U})
\,\vert\,
U\in Op(X),
\sigma_{U}\in\Gamma_{U}(p)
\,\vert\,
p(\alpha)\in U,
\sigma_{U}(p(\alpha))=\alpha
\right\}.
\end{cases}
$$
Moreover
$\forall\alpha\in\mf{E}$
and
$\forall\mf{l}\in\mc{K}^{loc}$
set
$$
\begin{cases}
\B_{\mf{l}}^{loc}(\alpha)
\coloneqq
\left\{
T^{loc}
(V,\mf{l}_{2}(\alpha),\ep,j)
\,\vert\,
V\in Op(X),
\ep>0,
j\in J
\,\vert\,
p(\alpha)
\in V
\subseteq
\mf{l}_{1}(\alpha)
\right\},
\\
T^{loc}(U,\sigma_{U},\ep,j)
\coloneqq
\left\{
\beta\in\mf{E}
\,\vert\,
p(\beta)\in U,
\nu_{j}(\beta-\sigma_{U}(p(\beta)))
<\ep
\right\},
\end{cases}
$$
$
(\forall U\in Op(X))
(\forall j\in J)
(\forall\ep>0)
(\forall\sigma_{U}\in\Gamma_{U}(p))
$.
\par
If 
$
\mf{P}
$
is 
a full bundle
then we can set
$$
\begin{cases}
\mc{K}
\coloneqq
\prod_{\alpha\in\mf{E}}
\mc{K}_{\alpha}
\\
\mc{K}_{\alpha}
\coloneqq
\left\{
(U,\sigma)
\,\vert\,
U\in Op(X),
\sigma\in\Gamma(p)
\,\vert\,
p(\alpha)\in U,
\sigma(p(\alpha))=\alpha
\right\}.
\end{cases}
$$
Moreover
$\forall\alpha\in\mf{E}$
and
$\forall\mf{l}\in\mc{K}$
set
$$
\begin{cases}
\B_{\mf{l}}(\alpha)
\coloneqq
\left\{
T(V,\mf{l}_{2}(\alpha),\ep,j)
\,\vert\,
V\in Op(X),
\ep>0,
j\in J
\,\vert\,
p(\alpha)
\in V
\subseteq
\mf{l}_{1}(\alpha)
\right\},
\\
T(U,\sigma,\ep,j)
\coloneqq
T^{loc}(U,\sigma\up U,\ep,j),
\end{cases}
$$
$
(\forall U\in Op(X))
(\forall j\in J)
(\forall\ep>0)
(\forall\sigma\in\Gamma(p))
$.
Any
set
$
T^{loc}(U,\sigma\up U,\ep,j)
$
for a fixed $\ep>0$
is called
$\ep-$tube.
\end{definition}
\begin{remark}
\label{15512610}
Notice that
$
(\forall U\in Op(X))
(\forall j\in J)
(\forall\ep>0)
(\forall\sigma_{U}\in\Gamma_{U}(p))
$
$$
T^{loc}(U,\sigma_{U},\ep,j)
=
\bigcup_{y\in U}
B_{\mf{E}_{y},j,\ep}(\sigma_{U}(y))
$$
where
for all
$
\gamma
\in
\mf{E}_{y}
$
$$
B_{\mf{E}_{y},j,\ep}(\gamma)
\coloneqq
\left\{
\beta\in\mf{E}_{y}
\,\vert\,
\nu_{j}^{y}\left(\beta-\gamma\right)
<\ep
\right\}.
$$
\end{remark}
\begin{corollary}
[
Neighbourhood's filter 
$\mc{I}_{\alpha}^{\tau}$
]
\label{11372310}
Let 
$
\mf{P}
\coloneqq
\lr{\lr{\mf{E}}{\tau}}{p,X,\n}
$
be a 
bundle
of $\Omega-$spaces
\begin{enumerate}
\item
if
$\mf{P}$
is
locally full 
$\forall\alpha\in\mf{E}$
and
$\forall\mf{l}\in\mc{K}^{loc}$
the set
$\B_{\mf{l}}^{loc}(\alpha)$
is a basis of a filter 
moreover
$$
\F{\B_{\mf{l}}^{loc}(\alpha)}{\mf{E}}
=
\mc{I}_{\alpha}^{\tau};
$$
\item
if
$\mf{P}$
is full 
or
locally full
over a completely regular space
then
$\forall\alpha\in\mf{E}$
and
$\forall\mf{l}\in\mc{K}$
the set
$\B_{\mf{l}}(\alpha)$
is a basis of a filter 
moreover
$$
\F{\B_{\mf{l}}(\alpha)}{\mf{E}}
=
\mc{I}_{\alpha}^{\tau}.
$$
\end{enumerate}
Here
$\mc{I}_{\alpha}^{\tau}$
is the 
neighbourhood's filter 
of $\alpha$
in the topological space
$\lr{\mf{E}}{\tau}$.
\end{corollary}
\begin{proof}
Statement $(1)$
follows
by
applying 
\cite[$\S 1.5.(vii)$]{gie},
while
statement $(2)$
follows
by
statement $(1)$
and the fact that
for all $U\in Op(X)$
and
$\sigma\in\Gamma(p)$
we have
$
\sigma\up U
\in\Gamma_{U}(p)
$.
\end{proof}
In what follows let
$
\ms{E}
\coloneqq
\{\lr{\ms{E}_{x}}{\n_{x}}\}_{x\in X}
$
be
a nice
family
of Hlcs
with
$
\n_{x}
\coloneqq
\{
\nu_{j}^{x}
\}_{j\in J}
$
for all $x\in X$.
Let
$
\mc{E}
$
satisfy
$FM(3^{*})-FM(4)$
with respect
to 
$\ms{E}$.
\begin{definition}
\label{12312410}
Set
$$
\begin{cases}
\mc{K}^{\mc{E}}
\coloneqq
\prod_{(x,v)\in\mf{E}}
\mc{K}_{(x,v)}^{\mc{E}}
\\
\mc{K}_{(x,v)}^{\mc{E}}
\coloneqq
\left\{
(U,f)
\,\vert\,
U\in Op(X),
f
\in\mc{E}
\,\vert\,
x\in U,
f(x)=v
\right\}.
\end{cases}
$$
Moreover
$
\forall(x,v)\in\mf{E}(\ms{E})
$
and
$\forall\mf{l}\in\mc{K}^{\mc{E}}$
define
\begin{equation}
\label{15022310}
\B_{\mf{l}}^{\mc{E}}((x,v))
=
\left\{
T_{\ms{E}}(V,\mf{l}_{2}((x,v)),\ep,j)
\,\vert\,
\ep>0,
j\in J,
V\in Op(X),
x\in V
\subseteq
\mf{l}_{1}((x,v))
\right\}.
\end{equation}
\end{definition}
\begin{corollary}
[
Neighbourhood's filter 
$\mc{I}_{(x,v)}^{\tau(\ms{E},\mc{E})}$
]
\label{14272310}
Then
$\mf{V}(\ms{E},\mc{E})$
is a full bundle of $\Omega-$spaces
and
$\forall(x,v)\in\mf{E}(\ms{E})$
$$
\F{\B_{\mf{l}}^{\mc{E}}((x,v))}
{\mf{E}(\ms{E})}
=
\mc{I}_{(x,v)}^{\tau(\ms{E},\mc{E})}.
$$
\end{corollary}
\begin{proof}
By Thm. $5.9.$ of \cite{gie}
$\mc{E}$
and 
$\Gamma(p_{1})$
are
canonically
isomorphic
as 
linear spaces,
so
$\mf{V}(\ms{E},\mc{E})$
is full
by
\ref{15012310}.
The
statement
hence
follows
by statement
$(2)$
of
Cor.
\ref{11372310}.
\end{proof}
The following 
corollaries
provide
conditions
under which the topologies
over two bundle spaces
are equal.
\begin{corollary}
\label{15582210}
Let 
$
\lr{\lr{\mf{E}}{\tau_{k}}}{p_{k},X,\n_{k}}
$
be a 
full
bundle of $\Omega-$spaces
or
a 
locally full bundle
over 
a completely regular space
$X$,
for $k=1,2$.
If  $p_{1}=p_{2}$ 
and
$
\Gamma(p_{1})
=
\Gamma(p_{2})
$
then
$
\tau_{1} 
=
\tau_{2}
$.
\begin{proof}
By statement $(2)$
of
Cor.
\ref{11372310}.
\end{proof}
\end{corollary}
\begin{corollary}
\label{15592210}
Let 
$
\mf{P}_{2}
\coloneqq
\lr{\lr{E}{\tau_{2}}}{p_{2},X,\n_{2}}$
be a bundle of $\Omega-$spaces
such that
$\pi_{\ms{E}}=p_{2}$.
Thus
if the following 
conditions
are satisfied
\begin{enumerate}
\item
$X$ is compact,
\item
$
\mc{E}
$
and
$\Gamma(p_{2})$
are
canonically
isomorphic
as 
linear spaces,
\end{enumerate}
then
$\tau(\ms{E},\mc{E})=\tau_{2}$.
\begin{proof}
By Thm. $5.9.$ of \cite{gie}
$\mc{E}$
and 
$
\Gamma(\pi_{\ms{E}})
$
are
canonically
isomorphic
as 
linear spaces
if $X$ is compact,
so
$\Gamma(\pi_{\ms{E}})=\Gamma(p_{2})$.
Moreover 
\ref{15012310}
and
the
shown fact that
$\mc{E}$
and 
$\Gamma(\pi_{\ms{E}})$
are
canonically
isomorphic
ensure
that
$\mf{V}(\ms{E},\mc{E})$
is 
a full bundle,
thus
it is so
$\mf{P}_{2}$
by the 
equality
$\Gamma(\pi_{\ms{E}})=\Gamma(p_{2})$.
Hence
the statement
follows
by 
Cor.
\ref{15582210}.
\end{proof}
\end{corollary}

\chapter{Semigroup Approximation}
\label{05251221}
\section*{Introduction}
The present part of the work is dedicated to establish
\textbf{Thm. \ref{17301812b}} and its corollaries.
This result resolves the claim of extending the Kurtz's result 
to the setting of bundles of $\Omega-$spaces.
More exacly we construct an element of the set
$\Delta_{\Theta}\lr{\mf{V},\mf{W}}
{\mc{E},X,\R^{+}}$.
Roughly 
$\lr{\mc{T}}{x_{\infty},\Phi}\in
\Delta_{\Theta}\lr{\mf{V},\mf{W}}
{\mc{E},X,\R^{+}}$
iff
$\mc{T}(x)$
is
the graph 
of  
the infinitesimal 
generator
$T_{x}$
of a
$C_{0}-$semigroup
$\mc{U}(x)$
on $\mf{E}_{x}$,
for all $x\in X$,
\eqref{19240703} holds true
and
\begin{equation}
\label{05211741}
\mc{U}\in\Gamma^{x_{\infty}}(\rho).
\end{equation}
Thus, according to 
the way of extending
the Kurtz' theorem
which we intend to 
perform in this work
and outlined in the Introduction of 
Chapter \ref{05250734},
to find
an
element
in 
$\Delta_{\Theta}\lr{\mf{V},\mf{W}}
{\mc{E},X,\R^{+}}$
means
to find
an extension
of Thm. 
\ref{16250603}.
\par
There are two strong hypothesis
to be satisfied in 
Thm. \ref{17301812b}.
In constructing a model
for
hypothesis $(ii)$
one 
obtains
\textbf{Cor. \ref{21343012}},
while we establish 
\textbf{Cor. \ref{12080805}} 
and 
\textbf{Thm. \ref{10581004}}
as an application of the stategy developped to 
ensure hypothesis $(i)$.
Among the two hypothesis, $(i)$ is the most difficult one 
to realize.
It is the assumption that the
$\left(\Theta,\mc{E}\right)-$structure
$\lr{\mf{V},\mf{W}}{X,\R^{+}}$
has
the
Laplace duality property
defined in
\textbf{Def. \ref{16401812b}}.
\par 
Roughly
speaking
the full 
Laplace duality
property
means that
the
natural
action
of
$\prod_{x\in X}
\mc{L}(\mf{E}_{x})
$
over
$\prod_{x\in X}\mf{E}_{x}$,
induces,
by restriction,
an action
over
$\Gamma(\pi)$
of
the
Laplace trasform
of 
$\Gamma(\rho)$.
More exactly
\begin{equation}
\label{Laplace}
\tag{$\ms{LD}$}
(\forall\lambda>0)
\left(\mf{L}(\Gamma(\rho))
(\cdot)(\lambda)
\bullet
\Gamma(\pi)
\subseteq
\Gamma(\pi)\right),
\end{equation}
where
\begin{equation*}
\mf{L}(F)(x)(\lambda)
\coloneqq
\int_{0}^{\infty}
e^{-\lambda s}
F(x)(s)\,ds
\doteq
\int_{\R^{+}}
F(x)(s)\,d\mu_{\lambda}(s).
\end{equation*}
The implicit assumption
is that 
for all $x\in X$
and 
$\lambda>0$
$$
\mf{M}_{x}
\subseteq
\mf{L}_{1}(\R^{+},
\mc{L}_{S_{x}}(\mf{E}_{x}),\mu_{\lambda}),
$$ 
where
$\mu_{\lambda}$
is the 
Laplace measure 
associated with $\lambda$
and
$\mf{L}_{1}(\R^{+},
\mc{L}_{S_{x}}(\mf{E}_{x}),\mu_{\lambda})$
is the space
of all $\mu_{\lambda}-$integrable
maps
with values
in
the locally convex space
$\mc{L}_{S_{x}}(\mf{E}_{x})$.
We provide reasonable conditions
ensuring the above inclusion
in Prop. \ref{18390901}.
\par
In section \ref{13320803} we investigate 
a strategy for constructing sets
having the full Laplace duality property,
result achieved in \textbf{Cor. \ref{18491004}}.
Although in section \ref{13320803}
we  worked in a wide generality,
here we  present the applications of interest 
for the present introduction.
\par
Firstly we note that
by construction
$$
\Gamma(\pi)
\subset
\prod_{x\in X}\mf{E}_{x},
$$
hence the natural duality action
to consider
over
$\Gamma(\pi)$
is the restriction
on it
of
the standard action\footnote{
namely
$
(B,v)
\mapsto
B(v)$.}
of
$$
\mc{L}\bigl(\prod_{x\in X}\mf{E}_{x}\bigr).
$$
Secondly
we note that
the Laplace 
duality property
is described
in terms
of the 
action
restricted
over
$\Gamma(\pi)$
of a 
subspace
of
$\prod_{x\in X}
\mf{L}_{1}(\R^{+},
\mc{L}_{S_{x}}(\mf{E}_{x});\mu_{\lambda})$.
\par
Therefore
the idea
is to construct what we call in \textbf{Def. \ref{14302503}}
a $\ms{U}-$Space, which is essentially 
a couple formed by a locally convex space
$\mf{G}$ and a linear map $\Psi$
such that
\begin{equation}
\label{18250803}
\begin{aligned}
\mf{G}
&\subset
\mc{L}\left(\prod_{x\in X}\mf{E}_{x}\right)
\,
\text{ as linear spaces}
\\
\Psi(\mf{L}_{1}(\R^{+},\mf{G},\mu_{\lambda}))
&\subseteq
\prod_{x\in X}
\mf{L}_{1}(\R^{+},
\mc{L}_{S_{x}}(\mf{E}_{x});\mu_{\lambda}),
\end{aligned}
\end{equation}
and most
importantly
such that
the following relation
between the two
actions
holds
for all
$\ov{F}
\in\mf{L}_{1}(\R^{+},\mf{G},\mu_{\lambda})$,
$x\in X$,
$\lambda>0$
and
$v\in\Gamma(\pi)$
\begin{equation}
\label{18100803}
\lr{\int
\Psi(\ov{F})(x)(s)
\,d\mu_{\lambda}(s)}{v(x)}_{x}
=
\lr{\int\ov{F}(s)\,d\mu_{\lambda}(s)}{v}(x).
\end{equation}
Here
$\mf{L}_{1}(\R^{+},\mf{G},\mu_{\lambda})$
is the space of all
$\mu_{\lambda}-$integrable
maps
on $\R^{+}$
and at values in the locally convex space
$\mf{G}$,
while for any linear space $E$ 
we denote by 
$\lr{\cdot}{\cdot}:
End(E)
\times
E
\to
E$
the standard duality.
In \textbf{Cor. \ref{15111901}} we prove the existence of
a $\ms{U}-$Space whose topology 
we assemble in \textbf{Def. \ref{10221801}}
as the final one with respect to a suitable set of linear
continuous maps.
\par
\emph
{Precisely because of \eqref{18100803}
we can interpret \eqref{Laplace}
as a duality problem}.
More exactly
if
$\exists\,\mc{F}\subset
\bigcap_{\lambda>0}
\mf{L}_{1}(\R^{+},\mf{G},\mu_{\lambda})$
such that
$\Psi(\mc{F})=\Gamma(\rho)$
then
\begin{equation}
\label{17280803a}
\ms{LD}
\Leftrightarrow
(\forall\lambda>0)
(\lr{\mc{A}_{\lambda}}{\Gamma(\pi)}
\subseteq\Gamma(\pi)),
\end{equation}
where for all $\lambda>0$
\begin{equation}
\label{17280803b}
\mc{A}_{\lambda}
\coloneqq
\bigl\{
\int\ov{F}(s)\,d\mu_{\lambda}(s)
\,\vert\,\ov{F}\in\mc{F}
\bigr\}
\subset
\mc{L}\bigl(\prod_{x\in X}\mf{E}_{x}\bigr).
\end{equation}
\emph
{There are two advantage 
of decoding
the problem of
finding
the full
Laplace duality
property
into 
the problem of invariance
\eqref{17280803a}.
Firstly \eqref{17280803a} is an example of 
a classical problem of invariance 
of a subspace 
of a linear topological space
for the standard action of a subspace of
the space of all linear continuous
operators on it. 
Secondly the relatively simple space 
$\mf{L}_{1}(\R^{+},\mf{G},\mu_{\lambda})$
appears in
\eqref{17280803a}
through $\A_{\lambda}$
while
the subspace $\Gamma(\rho)$
of the much more involved space
$\prod_{x\in X}
\bigcap_{\lambda>0}
\mf{L}_{1}(\R^{+},
\mc{L}_{S_{x}}(\mf{E}_{x});\mu_{\lambda})$
appears in
\eqref{Laplace}.}
\par
The crucial idea behind the construction of
the space $\mf{G}$ performed in Def. \ref{10221801}
is the use of the concept of locally convex final topology.
Indeed the defining characteristic of this topology allows
in Lemma \ref{11011501},
to
ensure
that
for all $v\in\Gamma(\pi)$
the evaluation map
\begin{equation}
\label{18050803}
\mf{G}\ni
A\mapsto A v
\in
\prod_{x\in X}\mf{E}_{x}\,
\text{ is continuous.}
\end{equation}
And
\eqref{18100803}
is essentially a consequence of \eqref{18050803}
attained through the two steps
\textbf{Thm. \ref{16291501}} and \textbf{Thm. \ref{15332203}}. 
Although we are mainly
interested to the equality
\eqref{18100803},
there is an
important result
strictly 
determined by 
the 
locally convex final
topology 
on $\mf{G}$.
Namely
\textbf{Thm. \ref{15251401}}
ensures
that
holds
the second inclusion
in
\eqref{18250803}
and that
for all
$\ov{F}\in\mf{L}_{1}
(\R^{+},\mf{G},\mu_{\lambda})$
$$
\int\Pr_{x}(\Psi(\ov{F}))(s)
\,d\mu_{\lambda}(s)
=
\Pr_{x}
\circ
\bigl(
\int\ov{F}(s)\,d\mu_{\lambda}(s)
\bigr)
\circ\imath_{x}.
$$
\section{General Approximation Theorem I}
\label{11542812}
This section is devoted to the proof of the main Thm. \ref{17301812b}.
\begin{notation}
\label{15411512b}
We assume the notation in 
section \ref{notat} 
and that all the vector spaces are over $\C$.
Moreover we let $lcp$ stand for the set of locally compact spaces.\footnote{We 
implicitly consider all the sets involved in this work as elements of a fixed Universe say $V$.
So the set of all the models of a given structure say $S$, has to be understood as the subset 
of those elements of $V$ satisfying the request defining $S$.}
For any set $A$ we let $\mc{P}(A)$ be the set of all subsets of $A$.
If $Y$ is a topological space and 
$Z$ is topological vector spaces
we let $\cc{cs}{Y,Z}$ denote 
the linear space of all continuous maps
$f:Y\to Z$
with compact support.
For any
$
\mf{V}
\coloneqq
\lr{\lr{\mf{E}}{\tau}}
{\pi,X,\n}
$
full bundle 
of $\Omega-$spaces
and
any
$\lr{\mc{T}_{0}}{x_{\infty},\Phi}
\in
Pregraph
\left(\left(\mf{V},\mf{V}\right)\right)$,
set
$X_{0}\coloneqq X-\{x_{\infty}\}$,
and for any $\phi\in\Phi$
$\phi_{i}(x)\coloneqq
\Pr_{i}^{x}(\phi(x))$
for all $x\in X$
and $i=1,2$.
Moreover
let us denote
by
$T_{x}$ the
operator in $\mf{E}_{x}$
such that
$Graph(T_{x})=\mc{T}_{0}(x)$,
for all $x\in X_{0}$,
while
$
\mc{T}
\in
\prod_{x\in X}
Graph(\mf{E}_{x}
\times\mf{E}_{x})
$
so that
$$
\begin{cases}
\mc{T}\up X-\{x_{\infty}\}
\coloneqq
\mc{T}_{0}
\\
\mc{T}(x_{\infty})
\coloneqq
\{\phi(x_{\infty})\,\vert\,\phi\in\Phi\},
\end{cases}
$$
in addition set
$$
D(T_{x_{\infty}})
\coloneqq
\Pr_{1}^{x_{\infty}}(\mc{T}(x_{\infty}))
=
\{\phi_{1}(x_{\infty})
\,\vert\,\phi\in\Phi\}.
$$
Finally
for any map
$F:A\to B$ 
set
$\mc{R}(F)\coloneqq F(A)$
the range of $F$.
\end{notation}
\begin{remark}
\label{15401512b}
Let 
$
\mf{V}
\coloneqq
\lr{\lr{\mf{E}}{\tau}}
{\pi,X,\n}
$
be a full bundle of $\Omega-$spaces
and
$\lr{\mc{T}_{0}}{x_{\infty},\Phi}
\in
Pregraph
\left(\left(\mf{V},\mf{V}\right)\right)$.
By
Cor. \ref{17571212}
$\forall\phi\in\Phi$
\begin{equation}
\label{15211512b}
\begin{cases}
\phi_{i}\in\Gamma^{x_{\infty}}(\pi), i=1,2\\
(\forall x\in X_{0})
(\phi_{2}(x)=T_{x}\phi_{1}(x)).
\end{cases}
\end{equation}
\end{remark}
\begin{lemma}
\label{13001512b}
Let 
$
\mf{V}
\coloneqq
\lr{\lr{\mf{E}}{\tau}}
{\pi,X,\n}
$
be a full
bundle of $\Omega-$spaces,
where
$\n\coloneqq\{\nu_{j}\,\vert\, j\in J\}$.
Moreover
$\lr{\mc{T}_{0}}{x_{\infty},\Phi}
\in
Pregraph
\left(\left(\mf{V},\mf{V}\right)\right)$.
If
for all
$x\in X_{0}$,
$v_{x}\in Dom(T_{x})$,
$\lambda>0$
and
$j\in J$
we have 
$\nu_{j}((\lambda- T_{x})v_{x})
\geq
\lambda\nu_{j}(v_{x})$
and
$D(T_{x_{\infty}})$
is dense in $\mf{E}_{x_{\infty}}$,
then
$$
\lr{\mc{T}}{x_{\infty},\Phi}
\in
\ms{Gr}(\mf{V},\mf{V})
$$
Moreover
the following
\begin{equation}
\label{172212312b}
T_{x_{\infty}}:
D(T_{x_{\infty}})
\ni
\phi_{1}(x_{\infty})
\mapsto
\phi_{2}(x_{\infty})
\end{equation}
is a well-defined linear
operator in $\mf{E}_{x_{\infty}}$
such that
$
Graph(T_{x_{\infty}})
=
\mc{T}(x_{\infty})
$
and
$\forall v_{x_{\infty}}\in 
Dom(T_{x_{\infty}})$,
$\forall\lambda>0$
and
$\forall j\in J$
we have 
$$
\nu_{j}((\lambda- T_{x_{\infty}})
v_{x_{\infty}})
\geq
\lambda\nu_{j}(v_{x_{\infty}}).
$$
\end{lemma}
\begin{proof}
Clearly
$
\mc{T}(x_{\infty})
\in
Graph(\mf{E}_{x_{\infty}}
\times\mf{E}_{x_{\infty}})
$
if and only if
$\phi_{1}(x_{\infty})
=\ze_{x_{\infty}}
$
implies
$\phi_{2}(x_{\infty})=\ze_{x_{\infty}}$,
$
\forall\phi\in\Phi
$,
moreover 
denoting
by
$T_{x_{\infty}}$
the corresponding
operator
we have
that
$T_{x_{\infty}}:
D(T_{x_{\infty}})\to\mf{E}_{x_{\infty}}$
is a linear operator.
Any real 
map $F$ defined on a topological space
is l.s.c.
at a point
iff
$-F$ is u.s.c.
at the same point,
see
\cite[$\S 6.2.$ Ch.$4$]{BourGT},
thus
by
\cite[Prop. $3$ $\S 6.2.$ Ch.$4$]{BourGT}
and
\cite[(13),$\S 5.6.$ Ch.$4$]{BourGT}
$F:X\to\R$ is u.s.c. in $a\in X$
iff
$\varlimsup_{x\to a}F(x)=F(a)$.
Moreover
by 
\cite[$\S 6.2.$ Ch.$4$]{BourGT}
we know that 
$F:X\to\ov{\R}$ is l.s.c.
at $a$
iff
$F$ is continuous
at $a$
providing 
$\ov{\R}$ with the 
following
topology
$\{\emptyset,[-\infty,\infty],]a,\infty[
\,\vert\, a\in\R\}$.
Thus
for any
map
$\sigma:Y\to X$
continuous
at $b$ such that
$\sigma(b)=a$
we have that
$F\circ\sigma$
is l.s.c.
at $a$.
Hence
because
$(-F)\circ\sigma=-(F\circ\sigma)$
we can state
that if
$F:X\to\ov{\R}$ is u.s.c.
at $a$
then
for any
map
$\sigma:Y\to X$
continuous
at $b$ such that
$\sigma(b)=a$
we have that
$F\circ\sigma$
is u.s.c.
at $a$.
Therefore
by using 
\cite[$1.6.(ii)$]{gie}
we have
$\forall\sigma\in\Gamma^{x_{\infty}}(\pi)$
and
$\forall j\in J$
\begin{equation}
\label{16051512b}
\nu_{j}(\sigma(x_{\infty}))
=
\varlimsup_{x\to x_{\infty}}
\nu_{j}(\sigma(x)).
\end{equation}
Let
$\psi\in\Phi$
such that
$\psi_{1}(x_{\infty})=\ze_{x_{\infty}}$
thus $\forall\phi\in\Phi$,
$\forall\lambda>0$,
$\forall x\in X_{0}$
and
$\forall j\in J$
we have
by \eqref{16051512b}
and \eqref{15211512b}
\begin{alignat}{2}
\label{16401512b}
\nu_{j}\left(
\lambda\phi_{1}(x_{\infty})-
\phi_{2}(x_{\infty})-
\lambda\psi_{2}(x_{\infty})
\right)
&
=
\notag
\\
\varlimsup_{x\to x_{\infty}}
\nu_{j}\left((\lambda-T_{x})
(\phi_{1}(x)+\lambda\psi_{1}(x))\right)
&
\geq
\notag
\\
\varlimsup_{x\to x_{\infty}}
\lambda
\nu_{j}\left(\phi_{1}(x)+
\lambda\psi_{1}(x))\right)
&
=
\lambda
\nu_{j}(\phi_{1}(x_{\infty})),
\end{alignat}
where,
the inequality 
comes
by
\cite[Prop. $11$ $\S 5.6.$ Ch.$4$]{BourGT})
and
by the 
hypothesis
$
\nu_{j}\left((\lambda-T_{x})
(\phi_{1}(x)+\lambda\psi_{1}(x))\right)
\geq
\lambda
\nu_{j}\left((\phi_{1}(x)+
\lambda\psi_{1}(x))\right)
$
for all $x\in X_{0}$.
Now 
$
\lim_{\lambda\to\infty}v/\lambda
=
\ze_{x_{\infty}}
$
for any $v\in\mf{E}_{x_{\infty}}$,
hence
by the fact that
$\nu_{j}^{x_{\infty}}
\coloneqq
\nu_{j}\up\mf{E}_{x_{\infty}}$
is a continuous seminorm
and
by
\eqref{16401512b}
$
(\forall j\in J)
(\forall\phi\in\Phi)
$
\begin{equation}
\label{17011512b}
\nu_{j}\left(
\phi_{1}(x_{\infty})-\psi_{2}(x_{\infty})
\right)
=
\lim_{\lambda\to\infty}
\frac{\nu_{j}\left(
\lambda\phi_{1}(x_{\infty})-
\phi_{2}(x_{\infty})-
\lambda\psi_{2}(x_{\infty})
\right)}{\lambda}
\geq
\nu_{j}(\phi_{1}(x_{\infty})).
\end{equation}
By hypothesis
$
D(T_{x_{\infty}})
=
\{\phi_{1}(x_{\infty})
\,\vert\,\phi\in\mc{T}(x_{\infty})\}
$
is dense in
$\mf{E}_{x_{\infty}}$
thus
$\nu_{j}(\psi_{2}(x_{\infty}))=0$
for all
$j\in J$.
Indeed
let $j\in J$
and
$v\in\mf{E}_{x_{\infty}}$
thus
$\exists\,\{\phi^{\alpha}\}_{\alpha\in D}$
net 
in $\Phi$
such that
$
\lim_{\alpha\in D}
\phi_{1}^{\alpha}(x_{\infty})
=
v
$
in
$\mf{E}_{x_{\infty}}$.
So by the continuity
of $\nu_{j}^{x_{\infty}}$
and by
\eqref{17011512b}
we have
$\forall v\in\mf{E}_{x_{\infty}}$
$$
\nu_{j}\left(
v
-\psi_{2}(x_{\infty})
\right)
=
\lim_{\alpha\in D}
\nu_{j}\left(
\phi_{1}^{\alpha}(x_{\infty})
-\psi_{2}(x_{\infty})
\right)
\geq
\lim_{\alpha\in D}
\nu_{j}\left(
\phi_{1}^{\alpha}(x_{\infty}))
\right)
=
\nu_{j}(v).
$$
True 
in particular
for $v=3\psi_{2}(x_{\infty})$,
which implies
$\nu_{j}\left(\psi_{2}(x_{\infty})\right)
=0$.
Hence
$\psi_{2}(x_{\infty})=\ze_{x_{\infty}}$
because 
of $\mf{E}_{x_{\infty}}$ is a Hausdorff lcs
for which
$\{\nu_{j}^{x_{\infty}}\}_{j\in J}$
is a generating set of seminorms
of its topology.
Thus
$T_{x_{\infty}}$
is a 
well-defined
(necessarly linear) 
operator
in $\mf{E}_{x_{\infty}}$
and
consequently
$\lr{\mc{T}}{x_{\infty},\Phi}
\in
\ms{Gr}(\mf{V},\mf{V})$.
Finally
$(\forall j\in J)
(\forall\phi\in\Phi)
(\forall\lambda>0)$
\begin{alignat*}{1}
\nu_{j}
((\lambda-T_{x_{\infty}})\phi_{1}(x_{\infty}))
&=
\\
\nu_{j}
(\lambda\phi_{1}(x_{\infty})-
\phi_{2}(x_{\infty}))
&=
\qquad
\text{ by }
\eqref{15211512b},
\eqref{16051512b}
\\
\varlimsup_{x\to x_{\infty}}
\nu_{j}
(\lambda\phi_{1}(x)-
\phi_{2}(x))
&=
\qquad
\text{ by }
\eqref{15211512b}
\\
\varlimsup_{x\to x_{\infty}}
\nu_{j}
((\lambda-T_{x})\phi_{1}(x))
&\geq
\qquad
\text{ by hypoth. 
and 
\cite[Prop. $11$ $\S 5.6.$ Ch.$4$]{BourGT})
}
\\
\varlimsup_{x\to x_{\infty}}
\nu_{j}
(\lambda\phi_{1}(x))
&=
\nu_{j}
(\lambda\phi_{1}(x_{\infty})).
\end{alignat*}
\end{proof}
\begin{lemma}
\label{161172901}
In addition to
the hypotheses 
and notation 
of Lemma \ref{13001512b}
assume that
$(\forall x\in X_{0})
(\forall\lambda\in\R)
(\forall j\in J)
(\forall v_{x}\in Dom(T_{x}))
$
\begin{equation}
\label{16292901}
\nu_{j}
((\un-\lambda T_{x})v_{x})
\geq
\nu_{j}(v_{x}).
\end{equation}
Thus
$(\forall\lambda\in\R)
(\forall j\in J)
(\forall v_{x_{\infty}}\in 
Dom(T_{x_{\infty}}))$
\begin{equation}
\label{20022901}
\nu_{j}((\un-\lambda T_{x_{\infty}})
v_{x_{\infty}})
\geq
\nu_{j}(v_{x_{\infty}}).
\end{equation}
Moreover
$\forall\lambda\in\R$
\begin{equation}
\label{20102901}
\begin{cases}
\exists\,
(\un-\lambda T_{x_{\infty}})^{-1}
\in\mc{L}
(\mc{R}(\un-\lambda T_{x_{\infty}}),
\mf{E}_{x_{\infty}}),
\\
(\forall w\in
\mc{R}(\un-\lambda T_{x_{\infty}}))
(\forall j\in J)
\nu_{j}((\un-\lambda T_{x_{\infty}})^{-1}w)
\leq\nu_{j}(w).
\end{cases}
\end{equation}
Finally
\begin{equation}
\label{20342901}
\mc{R}(\un-\lambda T_{x_{\infty}})
\text{ is closed in $\mf{E}_{x_{\infty}}$.}
\end{equation}
\end{lemma}
\begin{proof}
$(\forall j\in J)
(\forall\phi\in\Phi)
(\forall\lambda\in\R)$
\begin{alignat*}{1}
\nu_{j}
((\un-
\lambda T_{x_{\infty}})\phi_{1}(x_{\infty}))
&=
\\
\nu_{j}
(\phi_{1}(x_{\infty})-
\lambda
\phi_{2}(x_{\infty}))
&=
\qquad
\text{ by }
\eqref{15211512b},
\eqref{16051512b}
\\
\varlimsup_{x\to x_{\infty}}
\nu_{j}
(\phi_{1}(x)-
\lambda\phi_{2}(x))
&=
\qquad
\text{ by }
\eqref{15211512b}
\\
\varlimsup_{x\to x_{\infty}}
\nu_{j}
((\un-\lambda T_{x})\phi_{1}(x))
&\geq
\qquad
\text{ by \eqref{16292901} 
and 
\cite[Prop. $11$ $\S 5.6.$ Ch.$4$]{BourGT})
}
\\
\varlimsup_{x\to x_{\infty}}
\nu_{j}
(\phi_{1}(x))
&=
\nu_{j}
(\phi_{1}(x_{\infty})).
\end{alignat*}
thus \eqref{20022901} follows.
Let $\lambda\in\R$,
by \eqref{20022901} 
we obtain
\eqref{20102901},
indeed $\forall f,g\in Dom(T_{x_{\infty}})$
such 
that
$(\un-\lambda T_{x_{\infty}})f
=
(\un-\lambda T_{x_{\infty}})g$
we have $\forall j\in J$
$$
0
=
\nu_{j}((\un-\lambda T_{x_{\infty}})
(f-g))
\geq
\nu_{j}(f-g),
$$
so
$f=g$ because of 
by construction
$\mf{E}_{x_{\infty}}$
is Hausdorff.
Thus
the following is a well-set map
$$
(\un-\lambda T_{x_{\infty}})^{-1}:
\mc{R}(\un-\lambda T_{x_{\infty}})
\ni
(\un-\lambda T_{x_{\infty}})f
\mapsto
f
\in
\mf{E}_{x_{\infty}},
$$
moreover by 
\eqref{20022901}
we obtain
the second sentence
of
\eqref{20102901},
hence the first one
follows
by the fact that the inverse map 
of any linear operator is linear.
By \eqref{20102901},
\cite[Prop. $3$ $\S 3.1.$ Ch.$3$]{BourGT}
and
\cite[Prop. $11$ $\S 3.6.$ Ch.$2$]{BourGT}
we deduce that
\begin{equation}
\label{20182901}
(\exists\,!\, 
B\in\mc{L}\left(
\ov{\mc{R}(\un-\lambda T_{x_{\infty}})},
\mf{E}_{x_{\infty}}\right))
(B\up \mc{R}(\un-\lambda T_{x_{\infty}})
=
(\un-\lambda T_{x_{\infty}})^{-1}).
\end{equation}
Let
$w\in
\ov{\mc{R}(\un-\lambda T_{x_{\infty}})}$
thus
$\exists\,\{f_{\alpha}\}_{\alpha\in D}$
net in $Dom(T_{x_{\infty}})$
such that
\begin{equation}
\label{19252901}
w=\lim_{\alpha\in D}
(\un-\lambda T_{x_{\infty}})
f_{\alpha},
\end{equation}
therefore
by
\eqref{20182901}
\begin{equation}
\label{19512901a}
Bw=
\lim_{\alpha\in D}
f_{\alpha},
\end{equation}
while
by
\eqref{19252901}
and
\eqref{19512901a}
\begin{alignat*}{2}
w-Bw
&=
\lim_{\alpha\in D}
\left(
(f_{\alpha}-\lambda T_{x_{\infty}}f_{\alpha})
-f_{\alpha}\right)
\\
&=
\lim_{\alpha\in D}
-\lambda T_{x_{\infty}}f_{\alpha}.
\end{alignat*}
So
\begin{equation}
\label{19512901b}
Bw-w
=
\lim_{\alpha\in D}
\lambda T_{x_{\infty}}f_{\alpha}.
\end{equation}
By
\eqref{19512901a},
\eqref{19512901b}
and the fact that
$\lambda T_{x_{\infty}}$
is closed, we obtain
$$
\begin{cases}
Bw\in Dom(T_{x_{\infty}}),
\\
\lambda T_{x_{\infty}}(Bw)
=
Bw-w,
\end{cases}
$$
which means
$w=(\un-\lambda T_{x_{\infty}})
Bw$,
so
$w\in \mc{R}(\un-\lambda T_{x_{\infty}})$
and \eqref{20342901}
follows.
\end{proof}
\begin{lemma}
\label{14253001}
Let us assume the hypotheses of
Lemma \ref{161172901},
moreover let
$\lambda\in
\R-\{0\}$,
$\{\lambda_{n}\}_{n\in\N}\subset\R-\{0\}$
such that
$\lim_{{n\in\N}}\lambda_{n}=\lambda$.
Thus
$$
\bigcap_{n\in\N}
\mc{R}(\un-\lambda_{n}T_{x_{\infty}})
\subseteq
\mc{R}(\un-\lambda T_{x_{\infty}}).
$$
\end{lemma}
\begin{proof}
Set only in this proof
$T\coloneqq T_{x_{\infty}}$.
Let $n\in\N$, by \eqref{20102901}
$\exists\,(\un-\lambda_{n}T)^{-1}:
\mc{R}(\un-\lambda_{n}T_{x_{\infty}})
\to
Dom(t)$
moreover
$$
\begin{cases}
\un-\lambda T
=
\lambda(\lambda^{-1}-T),
\\
(\un-\lambda_{n}T)^{-1}
=
\lambda_{n}^{-1}
(\lambda_{n}^{-1}-T)^{-1}.
\end{cases}
$$
Let
$g\in\bigcap_{n\in\N}
\mc{R}(\un-\lambda_{n}T_{x_{\infty}})$
thus
\begin{alignat*}{2}
(\un-\lambda T)
(\un-\lambda_{n} T)^{-1}g
-
g
&=
\frac{\lambda}{\lambda_{n}}
(\lambda^{-1}-T)
(\lambda_{n}^{-1}-T)^{-1}g
-
g
\\
&=
\frac{\lambda}{\lambda_{n}}
\left(
\lambda^{-1}
(\lambda_{n}^{-1}-T)^{-1}g
-
\lambda_{n}^{-1}
(\lambda_{n}^{-1}-T)^{-1}g
\right)
\\
&=
\frac{\lambda}{\lambda_{n}}
(\lambda^{-1}-
\lambda_{n}^{-1})
(\lambda_{n}^{-1}-T)^{-1}g,
\end{alignat*}
where
in the 
second equality we considered that
$-T
(\lambda_{n}^{-1}-T)^{-1}g-g
=
-\lambda_{n}^{-1}
(\lambda_{n}^{-1}-T)^{-1}
g$
obtained by
$(\lambda_{n}^{-1}-T)
(\lambda_{n}^{-1}-T)^{-1}g
=g$.
Thus
$\forall j\in J$
by \eqref{20102901}
$$
\nu_{j}
\left(
(\un-\lambda T)
(\un-\lambda_{n} T)^{-1}g
-
g\right)
\leq
\left|
\frac{\lambda}{\lambda_{n}}
\right|
|\lambda^{-1}-\lambda_{n}^{-1}|
\nu_{j}(g).
$$
But
$\lim_{n\in\N}
|\lambda^{-1}-\lambda_{n}^{-1}|
=1$
and
$\lim_{n\in\N}
|\lambda^{-1}-\lambda_{n}^{-1}|
=0$
so
$\nu_{j}
\left(
(\un-\lambda T)
(\un-\lambda_{n} T)^{-1}g
-
g
\right)=0$,
for all $j\in J$.
Therefore
$$
\lim_{n\in\N}
(\un-\lambda T)
(\un-\lambda_{n} T)^{-1}g
=
g,
$$
and the statement follows
by \eqref{20342901}.
\end{proof}
\begin{lemma}
\label{17552901}
Under the hypotheses 
and notation 
of Lemma \ref{13001512b}
we have that
$\un-\lambda T_{x_{\infty}}$
is a closed operator.
\end{lemma}
\begin{proof}
Let 
$(a,b)\in
\ov{Graph(\un-\lambda T_{x_{\infty}})}$
closure
in the space
$\mf{E}_{x_{\infty}}
\times
\mf{E}_{x_{\infty}}$
with the product
topology.
Thus
$(\forall\ep>0)
(\forall j\in J)
(\exists\,v_{(\ep,j)}\in Dom(T_{x_{\infty}}))$
$$
\begin{cases}
\nu_{j}
(a-v_{(\ep,j)})
<\frac{\ep}{2},
\\
\nu_{j}
(b-(\un-\lambda T_{x_{\infty}})v_{(\ep,j)})
<\frac{\ep}{2},
\end{cases}
$$
so
$$
\nu_{j}
((b-a)+\lambda T_{x_{\infty}}v_{(\ep,j)})
\leq
\nu_{j}(b-(\un-\lambda T_{x_{\infty}})
v_{(\ep,j)})
+
\nu_{j}(a-v_{(\ep,j)})
\leq
\ep.
$$
Therefore
$(\forall\ep>0)
(\forall j\in J)
(\exists\,v_{(\ep,j)}\in Dom(T_{x_{\infty}}))$
$$
\begin{cases}
\nu_{j}
(a-v_{(\ep,j)})
<\ep,
\\
\nu_{j}
\left(
(b-a)-(-\lambda T_{x_{\infty}}v_{(\ep,j)})
\right),
\end{cases}
$$
which means
$(a,(b-a))
\in\ov{Graph(-\lambda T_{x_{\infty}})}$.
Moreover $-\lambda T_{x_{\infty}}$
is a closed operator thus
$b-a=-\lambda T_{x_{\infty}}a$
or equivalently
$(a,b)\in 
Graph(\un-\lambda T_{x_{\infty}})$.
\end{proof}
\begin{remark}
\label{18581512b}
By \eqref{15211512b}
we have 
$\forall\phi\in\Phi$
that
$
\phi_{1}(x_{\infty})
=
\lim_{z\to x_{\infty}}
\phi_{1}(z)
$
and
$
\phi_{2}(x_{\infty})
=
\lim_{z\to x_{\infty}}
\phi_{2}(z)
=
\lim_{z\to x_{\infty}}
T_{x}\phi_{1}(z)
$,
hence
$$
\begin{cases}
\phi_{1}(x_{\infty})
=
\lim_{z\to x_{\infty}}
\phi_{1}(z)
\\
T_{x_{\infty}}
\phi_{1}(x_{\infty})
=
\lim_{z\to x_{\infty}}
T_{z}
\phi_{1}(z).
\end{cases}
$$
\end{remark}
\begin{definition}
\label{15062301}
Let $\lambda\in\R^{+}$
set
$$
\mu_{\lambda}:
\cc{cs}{\R^{+},\R}
\ni
f
\mapsto
\int_{\R^{+}}
e^{-s\lambda}
f(s)\,ds,
$$
where
the integral is with
respect to the Lebesgue
measure on $\R^{+}$.
\end{definition}
\begin{definition}
\label{16401812b}
Let
$\mf{W}
\coloneqq
\lr{\lr{\mf{M}}{\gamma}}{\rho,X,\mf{R}}$
and
$\lr{\mf{V},\mf{W}}{X,\R^{+}}$
be
a
$\left(\Theta,\mc{E}\right)-$structure
such that
\begin{equation}
\label{18470109}
\mf{M}_{x}
\subseteq
\bigcap_{\lambda>0}
\mf{L}_{1}(\R^{+},
\mc{L}_{S_{x}}(\mf{E}_{x});\mu_{\lambda}),\,
\forall x\in X.
\end{equation}
About $S_{x}$ 
and $\mf{E}_{x}$
see
Def.
\ref{10282712}.
Let
$x\in X$,
$\mc{O}\subseteq\Gamma(\rho)$.
and
$\mc{D}\subseteq\Gamma(\pi)$.
By recalling the notation in \eqref{15012602},
we 
say that 
$\lr{\mf{V},\mf{W}}{X,\R^{+}}$
has the
\emph
{Laplace duality property
on $\mc{O}$
and $\mc{D}$
at $x$},
shortly 
$\ms{LD}_{x}(\mc{O},\mc{D})$
if
\begin{equation*}
(\forall\lambda>0)
(\mf{L}(\Gamma_{\mc{O}}^{x}(\rho))_{\lambda}
\bullet
\Gamma_{\mc{D}}^{x}(\pi)
\subseteq
\Gamma^{x}(\pi)).
\end{equation*}
Moreover we
say that 
$\lr{\mf{V},\mf{W}}{X,\R^{+}}$
has the
\emph
{full Laplace duality property
on $\mc{O}$
and $\mc{D}$},
shortly 
$\ms{LD}(\mc{O},\mc{D})$
if
\begin{equation*}
(\forall\lambda>0)
(\mf{L}(\mc{O})_{\lambda}
\bullet
\mc{D}
\subseteq
\Gamma(\pi)).
\end{equation*}
Finally
$\ms{LD}$
is for
$\ms{LD}(\Gamma(\rho),\Gamma(\pi))$.
Here
$\mf{L}:
\prod_{x\in X}
\mf{M}_{x}
\to
\prod_{x\in X}
\mc{L}_{S_{x}}(\mf{E}_{x})^{\R^{+}}$
such that
$(\forall x\in X)
(\forall\lambda\in\R^{+})$
$$
\mf{L}(F)(x)(\lambda)
\coloneqq
\int_{0}^{\infty}
e^{-\lambda s}
F(x)(s)\,ds,
$$
where 
we recall that
the 
integration is with respect 
to the Lebesgue measure
on $\R^{+}$
and with respect to the 
locally convex topology
on
$\mc{L}_{S_{x}}(\mf{E}_{x})$.
Finally we used the notation in \eqref{15012602}.
\end{definition}
\begin{remark}
\label{10412505}
Under the notation of Def. \ref{16401812b} and by letting
$\mf{n}$ be the Lebesgue measure on $\R^{+}$, \eqref{18470109}
follows if the following holds
\begin{equation*}
\mf{M}_{x}
\subseteq
\mf{L}_{1}(\R^{+},\mc{L}_{S_{x}}(\mf{E}_{x});\mf{n}),\,
\forall x\in X.
\end{equation*}
Moreover under assumptions of Def. \ref{16401812b}
we have
\begin{equation*}
\begin{cases}
\mf{L}(\Gamma_{\mc{O}}^{x}(\rho))
\subseteq
\Gamma_{\mc{O}}^{x}(\rho)
\\
(\forall t>0)
(\Gamma_{\mc{O}}^{x}(\rho)_{t}
\bullet
\Gamma_{\mc{D}}^{x}(\pi)
\subseteq
\Gamma^{x}(\pi))
\end{cases}
\Rightarrow
\lr{\mf{V},\mf{W}}{X,\R^{+}}
\text{has the }
\ms{LD}_{x}(\mc{O},\mc{D}).
\end{equation*}
Similarly
\begin{equation*}
\begin{cases}
\mf{L}(\mc{O})
\subseteq
\mc{O}
\\
(\forall t>0)
(\mc{O}_{t}
\bullet
\mc{D}
\subseteq
\Gamma(\pi))
\end{cases}
\Rightarrow
\lr{\mf{V},\mf{W}}{X,\R^{+}}
\text{has the }
\ms{LD}(\mc{O},\mc{D}).
\end{equation*}
\end{remark}
A useful property is the following one
\begin{proposition}
\label{21120901}
Let
$\lr{\mf{V},\mf{W}}{X,\R^{+}}$
be
a
$\left(\Theta,\mc{E}\right)-$structure
satisfying
\eqref{18470109},
$x_{\infty}\in X$.
Set
$S_{z}=\{B_{l}^{z}\,\vert\, l\in L\}$,
then
$\forall z\in X$,
$\forall G\in\mf{L}_{1}(\R^{+},
\mc{L}_{S_{z}}(\mf{E}_{z});\mu_{\lambda})$
and
$\forall w_{z}\in
\bigcup_{l\in L}B_{l}^{z}$
\begin{equation}
\label{14481001}
\bigl(\int_{0}^{\infty}
e^{-\lambda s}
G(s)
\,ds\bigr)
w_{z}
=
\int_{0}^{\infty}
e^{-\lambda s}
G(s)w_{z}\,ds.
\end{equation}
Here
in the second member
the integration
is with respect to the locally convex topology
on $\mf{E}_{z}$,
while
in the first member 
the integration
is with respect to the locally convex topology
on $\mc{L}_{S_{z}}(\mf{E}_{z})$.
\end{proposition}
\begin{proof}
Let $z\in X$
and
$v\in\bigcup_{l\in L}B_{l}^{z}=\mf{E}_{z}$
then map
$\mc{L}_{S_{z}}(\mf{E}_{z})\ni A
\mapsto
Av\in\mf{E}_{z}$
is linear and continuous.
Indeed let
$l(v)\in L$
such that
$v\in B_{l(v)}^{z}$,
thus 
we have
$\nu_{j}^{z}(Av)
\leq
\sup_{w\in B_{l(v)}^{z}}
\nu_{j}^{z}(Aw)
\doteq
p_{j,l(v)}^{z}
(A)$.
Hence 
by a well-known result in
vector valued integration
we have
\eqref{14481001}.
\end{proof}
\begin{remark}
\label{21500412b}
Let
$\mf{V}
\coloneqq
\lr{\lr{\mf{E}}{\tau}}
{\pi,X,\n}
$
be a bundle of $\Omega-$spaces
and
$\mc{E}\subseteq\prod_{x\in X}\mf{E}_{x}$.
Set for all $v\in\prod_{x\in X}\mf{E}_{x}$
\begin{equation}
\label{11121419b}
\begin{cases}
B_{v}:X\ni x
\mapsto
\{v(x)\},
\\
\Theta
\coloneqq
\bigl\{
B_{w}
\,\vert\,
w\in\mc{E}
\bigr\}
\end{cases}
\end{equation}
Thus
$\Theta\subset
\prod_{x\in X}
Bounded(\mf{E}_{x})$
and
$\forall v\in\mc{E}$
\begin{equation}
\label{01592912}
\mc{E}
\cap
\prod_{x\in X}
B_{v}(x)
=
\{v\}.
\end{equation}
Therefore
for all $v\in\mc{E}$,
and
for all $x\in X$
with the notation of Def. \ref{10282712}
\begin{equation*}
\begin{cases}
\ms{D}(B_{v},\mc{E})
=
\{v\},
\\
\mc{B}_{B_{v}}^{x}
=\{v(x)\},\,
\\
S_{x}=\{\{w(x)\}
\,\vert\, w\in\mc{E}\},
\\
\mc{E}(\Theta)
=
\mc{E}.
\end{cases}
\end{equation*}
\end{remark}
Recall that since the Dupre' Thm.
any Banach bundle over a completely regular topological space
is full. 
\begin{definition}
\label{15210503}
Let
$\mf{V}
\coloneqq
\lr{\lr{\mf{E}}{\tau}}
{\pi,X,\|\cdot\|}$
be a full Banach bundle.
Let
$x_{\infty}\in X$
and
$
\mc{U}_{0}
\in\prod_{x\in X_{0}}
\cc{}{\R^{+},B_{s}(\mf{E}_{x})}$
be
such that
$\mc{U}_{0}(x)$
is a
$(C_{0})-$semigroup
of contractions
(respectively of
isometries)
on
$\mf{E}_{x}$
for all $x\in X_{0}$.
Moreover
let us denote
by
$T_{x}$
the infinitesimal generator
of the semigroup
$\mc{U}_{0}(x)$
for any $x\in X_{0}$
and
set
\begin{equation}
\label{15482601}
\begin{cases}
\mc{T}_{0}(x)
\coloneqq
Graph(T_{x}),\,
x\in X_{0}
\\
\Phi
\coloneqq
\{
\phi\in\Gamma^{x_{\infty}}
(\pi_{\ms{E}^{\oplus}})
\,\vert\,
(\forall x\in X_{0})
(\phi(x)\in\mc{T}_{0}(x))
\}
\\
\mc{E}
\coloneqq
\{
v
\in
\Gamma(\pi)
\,\vert\,
(\exists\,\phi\in\Phi)
(v(x_{\infty})=\phi_{1}(x_{\infty}))
\}
\\
\Theta
\coloneqq
\bigl\{
B_{w}
\,\vert\,
w\in\mc{E}
\bigr\},
\end{cases}
\end{equation}
where
$\lr{\lr{\mf{E}(\ms{E}^{\oplus})}
{\tau(\ms{E}^{\oplus},\mc{E}^{\oplus})}}
{\pi_{\ms{E}^{\oplus}},X,\mf{n}^{\oplus}}$
is
the 
bundle direct sum of the
family
$\{\mf{V},
\mf{V}\}$.
\end{definition}
The following is a direct generalization to our context 
of the definition given in \cite[Lm. $2.11$]{kurtz}
\begin{definition}
\label{14531903}
Let 
$
\mf{V}
\coloneqq
\lr{\lr{\mf{E}}{\tau}}
{\pi,X,\n}
$
be a 
bundle of $\Omega-$spaces,
where
$\n\coloneqq\{\nu_{j}\,\vert\, j\in J\}$.
Moreover
let
$Y$ be a topological space,
$s_{0}\in Y$,
$f\in
\prod_{x\in X}\mf{E}_{x}^{Y}$
and
$\{z_{n}\}_{n\in\N}\subset X$.
Then 
we say that
$\{f(z_{n})\}_{n\in\N}$
is 
\emph
{bounded} 
if
$\sup_{(n,s)\in\N\times Y}
\nu_{j}(f(z_{n})(s))<\infty$
for all $j\in J$.
$\{f(z_{n})\}_{n\in\N}$
is
\emph
{equicontinuous at $s_{0}$}
if
for all 
$j\in J$
and
for all 
$\ep>0$ there exists
a
neighbourhood 
$U$ 
of $s_{0}$
such that
for all 
$s\in U$
we have
$\sup_{n\in\N}\nu_{j}(f(z_{n})(s)
-f(z_{n})(s_{0}))\leq\ep$.
Finally
$\{f(z_{n})\}_{n\in\N}$
is
\emph
{equicontinuous}
if
$\{f(z_{n})\}_{n\in\N}$
is
equicontinuous at $s$
for every $s\in Y$.
\end{definition}
\begin{proposition}
\label{19492307}
Let
us assume the notation
of Def. \ref{15210503}.
Thus
$\{v(x_{\infty})\,\vert\, v\in\mc{E}\}
=
\{\phi_{1}(x_{\infty})\,\vert\,\phi\in\Phi\}$.
\end{proposition}
\begin{proof}
By definition follows the inclusion
$\subseteq$.
$\mf{V}$
being
full
we have
$(\forall\phi\in\Phi)
(\exists\,v\in\Gamma(\pi))
(v(x_{\infty})=\phi_{1}(x_{\infty}))$.
Thus
$(\forall\phi\in\Phi)
(\exists\,v\in\mc{E})
(v(x_{\infty})=\phi_{1}(x_{\infty}))$
hence the inclusion
$\supseteq$.
\end{proof}
\begin{theorem}[\textbf{MAIN}$\ms{1}$]
\label{17301812b}
Let
$\mf{V}
\coloneqq
\lr{\lr{\mf{E}}{\tau}}
{\pi,X,\|\cdot\|}$
be a 
Banach bundle
where $X$ is a completely regular space 
for which there exists
$x_{\infty}\in X$
such that 
its filter of neighbourhoods admits a countable basis.
Let
$
\mc{U}_{0}
\in\prod_{x\in X_{0}}
\cc{}{\R^{+},B_{s}(\mf{E}_{x})}$
be
such that
$\mc{U}_{0}(x)$
is 
a
$(C_{0})-$semigroup
of contractions
(respectively of
isometries)
on
$\mf{E}_{x}$
for all $x\in X_{0}$.
\par
\underline{If}
$D(T_{x_{\infty}})$
is dense in $\mf{E}_{x_{\infty}}$
and
$\exists\lambda_{0}>0$
(respectively
$\exists\lambda_{0}>0,
\lambda_{1}<0$)
such that
the range
$\mc{R}(\lambda_{0}-T_{x_{\infty}})$
is dense in $\mf{E}_{x_{\infty}}$,
(respectively
the ranges
$\mc{R}(\lambda_{0}-T_{x_{\infty}})$
and
$\mc{R}(\lambda_{1}-T_{x_{\infty}})$
are dense in $\mf{E}_{x_{\infty}}$),
\underline{then}
$$
\lr{\mc{T}}{x_{\infty},\Phi}
\in
\ms{Gr}(\mf{V},\mf{V}),
$$
and
$T_{x_{\infty}}$
in 
\eqref{172212312b}
is the generator
of
a $C_{0}-$semigroup
of contractions
(respectively of isometries)
on $\mf{E}_{x_{\infty}}$.
\par
Moreover 
\underline{assume}
that
$\{v(x)\,\vert\, v\in\mc{E}\}$
is dense in $\mf{E}_{x}$
for all $x\in X_{0}$,
by taking
the notation
in \eqref{15482601},
let
$\mf{W}
\coloneqq
\lr{\lr{\mf{M}}{\gamma}}{\rho,X,\mf{R}}$
and
$\lr{\mf{V},\mf{W}}{X,\R^{+}}$
be
a
$\left(\Theta,\mc{E}\right)-$structure
\footnote{
Well set indeed
by 
Prop. \ref{19492307},
the density assumptions
and
Rem. \ref{21500412b}
we have that
$S_{x}$
is dense in $\mf{E}_{x}$
for all $x\in X$.}
such that
\eqref{18470109} holds. 
Assume
$
\ms{U}_{\|\cdot\|_{B(\mf{E}_{z})}}
(\mc{L}_{S_{z}}(\mf{E}_{z}))
\subseteq
\mf{M}_{z}$
(respectively
$\ms{U}_{is}
(\mc{L}_{S_{z}}(\mf{E}_{z}))
\subseteq
\mf{M}_{z}$)
for all $z\in X$
\footnote{
See Prp.
\ref{18390901}
for models of $\mf{M}$
satisfying
\eqref{18470109}
and
$\ms{U}_{\|\cdot\|_{B(\mf{E}_{z})}}
(\mc{L}_{S_{z}}(\mf{E}_{z}))
\subseteq
\mf{M}_{z}$.
}
and that
there exists
$F\in\Gamma(\rho)$
such that
$F(x_{\infty})=\mc{U}(x_{\infty})$
and
\begin{description}
\item[i]
$\lr{\mf{V},\mf{W}}{X,\R^{+}}$
has
the
$\ms{LD}_{x_{\infty}}(\{F\},\mc{E})$;
\textbf{or}
it
has
the 
$\ms{LD}(\{F\},\mc{E})$;
\item[ii]
$(\forall v\in\mc{E})
(\exists\,\phi\in\Phi)$
s.t.
$\phi_{1}(x_{\infty})
=
v(x_{\infty})$
and
$(\forall\{z_{n}\}_{n\in\N}\subset X
\,\vert\,
\lim_{n\in\N}z_{n}=x_{\infty})$
we have
that
$\{
\mc{U}(z_{n})(\cdot)\phi_{1}(z_{n})
-
F(z_{n})(\cdot)v(z_{n})
\}_{n\in\N}$
is a 
bounded equicontinuous
sequence.
\end{description}
\underline{Then}
$(\forall v\in\mc{E})
(\forall K\in Comp(\R^{+}))$
\begin{equation}
\label{02502912}
\boxed{
\lim_{z\to x_{\infty}}
\sup_{s\in K}
\bigl\|
\mc{U}(z)(s)v(z)
-
F(z)(s)v(z)
\bigr\|=0,}
\end{equation}
and
\begin{equation}
\label{02512912p}
\boxed{
\mc{U}
\in\Gamma^{x_{\infty}}(\rho)
.}
\end{equation}
In particular
\begin{equation}
\label{02512912}
\{\lr{\mc{T}}{x_{\infty},\Phi}\}
\in
\Delta_{\Theta}\lr{\mf{V},\mf{W}}
{\mc{E},X,\R^{+}}.
\end{equation}
Here
$\mc{T}$
and
$D(T_{x_{\infty}})$
are defined as
in
Notation
\ref{15411512b}
with 
$\mc{T}_{0}$
and 
$\Phi$
given in
\eqref{15482601},
while
$\mc{U}
\in\prod_{x\in X}
\mf{M}_{x}$
such that
$\mc{U}\up X_{0}\coloneqq
\mc{U}_{0}$
and
$\mc{U}(x_{\infty})$
is the semigroup
on $\mf{E}_{x_{\infty}}$
generated by 
$T_{x_{\infty}}$.
\end{theorem}
\begin{proof}
Since the Dupre' Thm., see for example \cite[Cor. $2.10$]{gie}, 
we obtain that $\mf{V}$ is full.
By Lemma 
\ref{13001512b},
\cite[Lms.$(2.8)-(2.9)$]{kurtz}, 
and the Hille-Yosida
theorem, 
see \cite[Th.$(1.2)$]{kurtz}, 
we have the first
sentence of the statement
for the case of semigroup
of contractions.
By 
\cite[Cor. 
$3.1.19.$]{bra}
applied to $T_{x}$,
for any $x\in X_{0}$,
and by \eqref{20022901}
we have
$(\forall\lambda\in\R)
(\forall v_{x_{\infty}}\in 
Dom(T_{x_{\infty}}))$
\begin{equation}
\label{15363001}
\|(\un-\lambda T_{x_{\infty}})
v_{x_{\infty}})\|_{x_{\infty}}
\geq
\|v_{x_{\infty}}\|_{x_{\infty}}.
\end{equation}
Hence 
by 
\cite[Cor. 
$3.1.19.$]{bra},
$T_{x_{\infty}}$
will be a generator of a strongly
continuous semigroup
of isometries if we show that
$\forall\lambda\in\R-\{0\}$
\begin{equation}
\label{15443001}
\mc{R}(\un-\lambda T_{x_{\infty}})
=
\mf{E}_{x_{\infty}}.
\end{equation}
Let us set
$$
\rho_{0}(T_{x_{\infty}})
\coloneqq
\{
\lambda\in
\R-\{0\}
\,\vert\,
\mc{R}(\un-\lambda T)
=
\mf{E}_{x}\}.
$$ 
By \eqref{20102901}
$\rho_{0}(T_{x_{\infty}})
=
\rho(T_{x_{\infty}})
\cap
(\R-\{0\})$,
where
$\rho(T_{x_{\infty}})$
is
the resolvent
set of
$T_{x}$.
By \cite[Lemma $7.3.2$]{ds}
$\rho(T_{x_{\infty}})$
is open in $\C$ so 
$\rho_{0}(T_{x_{\infty}})$
is open in
$\R-\{0\}$
with respect to the 
topology
on
$\R-\{0\}$
induced by that on $\C$.
By 
Lemma \ref{14253001}
we deduce that
$\rho_{0}(T_{x_{\infty}})$
is also closed in
$\R-\{0\}$,
therefore 
$\rho_{0}(T_{x_{\infty}})=\R-\{0\}$
and \eqref{15443001} follows
as well
that
$T_{x_{\infty}}$
is
a generator of a strongly
continuous semigroup
of isometries.
\par
Now we shall apply
Lemma \ref{15482712}
in order to show the remaining part of the 
statement.
Let 
$v\in\mc{E}$
be fixed
then by
\eqref{15482601},
$(\exists\,\phi\in\Phi)
(v(x_{\infty})=\phi_{1}(x_{\infty}))$
thus
by
\eqref{15211512b}
and
Cor. \ref{28111707}
\begin{equation}
\label{20352312b}
\lim_{z\to x_{\infty}}
\|v(z)-\phi_{1}(z)\|
=0.
\end{equation}
Now
let
$F\in\Gamma(\rho)$
of which in
hypothesis
so in particular
\begin{equation}
\label{10182312b}
F(x_{\infty})=\mc{U}(x_{\infty}),
\end{equation}
moreover
$\forall s\in\R^{+}$
and $z\in X$
\begin{alignat}{2}
\label{20492312b}
\|
\mc{U}(z)(s)v(z)
-
F(z)(s)v(z)
\|
&
\leq
\notag
\\
\|
\mc{U}(z)(s)v(z)
-
\mc{U}(z)(s)\phi_{1}(z)
\|
+
\|\mc{U}(z)(s)\phi_{1}(z)-F(z)(s)v(z)\|
&
\leq
\notag
\\
\|v(z)-\phi_{1}(z)\|
+
\|\mc{U}(z)(s)\phi_{1}(z)-F(z)(s)v(z)\|.
\end{alignat}
For any $\lambda>0$
let us set
$$
g_{\infty}^{\lambda}
\coloneqq
(\lambda-T_{x_{\infty}})^{-1}
\phi_{1}(x_{\infty})
$$
thus
$g_{\infty}^{\lambda}
\in Dom(T_{x_{\infty}})$
hence by Rmk. \ref{18581512b}
and the construction
of $T_{x_{\infty}}$
$
\exists\,
\psi^{\lambda}
\in
\Phi
$
such that
\begin{equation}
\label{09252412b}
\begin{cases}
g_{\infty}^{\lambda}
=
\psi_{1}^{\lambda}(x_{\infty})
=
\lim_{z\in x_{\infty}}
\psi_{1}^{\lambda}(z)
\\
T_{x_{\infty}}
g_{\infty}^{\lambda}
=
\lim_{z\to x_{\infty}}
T_{z}
\psi_{1}^{\lambda}(z).
\end{cases}
\end{equation}
By \eqref{14481001}
and
\eqref{01592912}
for all $z\in X$
and for all
$w_{z}\in
\bigcup_{v\in\mc{E}}v(z)$
\begin{equation}
\label{15261001pre}
\bigl(\int_{0}^{\infty}
e^{-\lambda s}
F(z)(s)
\,ds\bigr)
w_{z}
=
\int_{0}^{\infty}
e^{-\lambda s}
F(z)(s)w_{z}\,ds.
\end{equation}
Moreover
by
the fact that $\mf{V}$ is full
we have that
for all 
$\phi\in\Phi$
there exists a $v\in\Gamma(\pi)$
such that
$v(x_{\infty})=\phi_{1}(x_{\infty})$,
thus
by construction
of
$\mc{E}$
\begin{equation}
\label{18390103}
(\forall\phi\in\Phi)
(\exists\,v\in\mc{E})
(v(x_{\infty})=\phi_{1}(x_{\infty})).
\end{equation}
Hence
by
\eqref{15261001pre},
\eqref{18390103}
and
\eqref{10182312b}
for all $\phi\in\Phi$
\begin{equation}
\label{15261001}
\bigl(\int_{0}^{\infty}
e^{-\lambda s}
F(x_{\infty})(s)
\,ds\bigr)
\phi_{1}(x_{\infty})
=
\int_{0}^{\infty}
e^{-\lambda s}
\mc{U}(x_{\infty})(s)
\phi_{1}(x_{\infty})
\,ds.
\end{equation}
Now
set
$$
\xi
\coloneqq
\mf{L}(F),
$$
thus by 
hypothesis
$(i)$
we have
for all $\lambda>0$
\begin{equation}
\label{11462412b}
\xi(\cdot)(\lambda)v(\cdot)
\in
\Gamma^{x_{\infty}}(\pi).
\end{equation}
Moreover
\begin{alignat}{2}
\label{10262412b}
\xi(x_{\infty})(\lambda)
v(x_{\infty})
&
=
\xi(x_{\infty})(\lambda)
\phi_{1}(x_{\infty})
\notag
\\
&
=
\int_{0}^{\infty}
e^{-\lambda s}
\mc{U}(x_{\infty})(s)
\phi_{1}(x_{\infty})
\,ds\,
\text{ by }
\eqref{15261001}
\notag
\\
&
=
(\lambda-T_{x_{\infty}})^{-1}
\phi_{1}(x_{\infty})\,
\text{ by \cite[$(1.3)$]{kurtz}}
\notag
\\
&
\doteq
g_{\infty}^{\lambda}
=
\psi_{1}^{\lambda}(x_{\infty})
\,
\text{ by } 
\eqref{09252412b}.
\end{alignat}
By 
the fact that
$\mf{V}$
is full,
by
\eqref{11462412b},
the fact that
$\psi_{1}^{\lambda}
\in\Gamma^{x_{\infty}}(\pi)$
by
\eqref{15211512b},
by
\eqref{10262412b}
and
by
Cor.
\ref{21492812}
we have
$\forall\lambda>0$
\begin{equation}
\label{10352412b}
\lim_{z\to x_{\infty}}
\|\psi_{1}^{\lambda}(z)-
\xi(z)(\lambda)v(z))\|
=0.
\end{equation}
Now
$(\forall\lambda>0)(\forall z\in X)$
set 
$$
w^{\lambda}(z)
\coloneqq
(\lambda\un-T_{z})
\psi_{1}^{\lambda}(z),
$$
thus
\begin{alignat}{2}
\label{12042412b}
\bigl\|
\int_{0}^{\infty}
e^{-\lambda s}
\left(
\mc{U}(z)(s)\phi_{1}(z)-F(z)(s)v(z)
\right)
\,ds
\bigr\|
&
\leq
\notag
\\
\bigl\|
\int_{0}^{\infty}
e^{-\lambda s}
\mc{U}(z)(s)(\phi_{1}(z)-w^{\lambda}(z))
\,ds
\bigr\|
+
\bigl\|
\int_{0}^{\infty}
e^{-\lambda s}
\bigl(
\mc{U}(z)(s)w^{\lambda}(z)-F(z)(s)v(z)
\bigr)
\,ds
\bigr\|
&
\leq
\notag
\\
\frac{1}{\lambda}
\|\phi_{1}(z)-w^{\lambda}(z)\|
+
\|\psi_{1}^{\lambda}(z)-
\xi(z)(\lambda)v(z))
&\|.
\end{alignat}
Here
we consider
that by hypothesis
and by the first part of the statemet
$\|\mc{U}(z)\|\leq 1$
for all $z\in X$,
moreover we 
applied
the
Hille-Yosida
formula 
\cite[$(1.3)$]{kurtz}.
Now
\begin{alignat}{1}
\label{11272412b}
\|\phi_{1}(z)-w^{\lambda}(z)\|
&
=
\\
\|
\phi_{1}(z)-
(\lambda\un-T_{z})
\psi_{1}^{\lambda}(z)
\|
&
\leq
\notag
\\
\|
\phi_{1}(z)-v(z)
\|
+
\|v(z)-\lambda\xi(z)(\lambda)v(z)
+
\lambda\xi(z)(\lambda)v(z)
-(\lambda\un-T_{z})
\psi_{1}^{\lambda}(z)
\|
&
\leq
\notag
\\
\|\phi_{1}(z)-v(z)\|
+
\lambda
\|\xi(z)(\lambda)v(z)-\psi_{1}^{\lambda}(z)\|
+
\|
T_{z}\psi_{1}^{\lambda}(z)
-(\lambda\xi(z)(\lambda)v(z)-v(z))
&
\|.
\notag
\end{alignat}
By \eqref{09252412b}
$
T_{x_{\infty}}
\psi_{1}^{\lambda}(x_{\infty})
=
T_{x_{\infty}}
g_{\infty}^{\lambda}
$
moreover
\begin{alignat}{1}
\label{11402412b}
T_{x_{\infty}}
g_{\infty}^{\lambda}
&
=
-(\lambda-T_{x_{\infty}})
g_{\infty}^{\lambda}
+\lambda g_{\infty}^{\lambda}
\notag
\\
&
=
-(\lambda-T_{x_{\infty}})
(\lambda-T_{x_{\infty}})^{-1}
\phi_{1}(x_{\infty})
+\lambda g_{\infty}^{\lambda}
\notag
\\
&
=
\lambda g_{\infty}^{\lambda}
-
\phi_{1}(x_{\infty})
=
\lambda\xi(x_{\infty})(\lambda)
v(x_{\infty})
-
v(x_{\infty}),
\end{alignat}
where
in the last equality
we used
\eqref{10262412b}
and the 
construction of $\phi$.
By 
\eqref{09252412b}
we have
that
$(X\ni z
\mapsto
T_{z}
\psi_{1}^{\lambda}(z))
\in
\Gamma^{x_{\infty}}(\pi)
$,
hence
by
\eqref{11402412b},
the fact that
$\lambda\xi(\cdot)(\lambda)v(\cdot)
-
v\in\Gamma^{x_{\infty}}(\pi)
$
by \eqref{11462412b},
we deduce 
by
the fact that 
$\mf{V}$ is full
and
by Cor.
\ref{21492812}
that $\forall\lambda>0$
\begin{equation}
\label{11542412b}
\lim_{z\to x_{\infty}}
\|T_{z}\psi_{1}^{\lambda}(z)
-(\lambda\xi(z)(\lambda)v(z)-v(z))\|=0.
\end{equation}
Therefore
by
\eqref{11272412b},
\eqref{20352312b},
\eqref{10352412b}
and
\eqref{11542412b}
$$
\lim_{z\to x_{\infty}}
\|\phi_{1}(z)-w^{\lambda}(z)\|
=0.
$$
By
this one
along with
\eqref{10352412b}
we can state by 
using
\eqref{12042412b}
that
$\forall\lambda>0$
\begin{equation*}
\lim_{z\to x_{\infty}}
\bigl\|
\int_{0}^{\infty}
e^{-\lambda s}
\bigl(
\mc{U}(z)(s)\phi_{1}(z)-F(z)(s)v(z)
\bigr)
\,ds
\bigr\|
=0.
\end{equation*}
Therefore
$\forall\lambda>0$
and 
$(\forall\{z_{n}\}_{n\in\N}\subset X
\,\vert\,
\lim_{n\in\N}z_{n}=x_{\infty})$
\begin{equation}
\label{12052412b}
\lim_{n\in\N}
\bigl\|
\int_{0}^{\infty}
e^{-\lambda s}
\bigl(
\mc{U}(z_{n})(s)\phi_{1}(z_{n})-
F(z_{n})(s)v(z_{n})
\bigr)
\,ds
\bigr\|
=0.
\end{equation}
By \eqref{12052412b},
hypothesis
$(ii)$
and 
\cite[Lemma $(2.11)$]{kurtz}
we have 
$(\forall\{z_{n}\}_{n\in\N}\subset X
\,\vert\,
\lim_{n\in\N}z_{n}=x_{\infty})$
and
$\forall K\in Comp(\R^{+})$
\begin{equation*}
\lim_{n\in\N}
\sup_{s\in K}
\bigl\|
\mc{U}(z_{n})(s)\phi_{1}(z_{n})
-
F(z_{n})(s)v(z_{n})
\bigr\|=0.
\end{equation*}
Therefore since the hypothesis on $x_{\infty}$
we obtain
$\forall K\in Comp(\R^{+})$
\begin{equation}
\label{12102412b}
\lim_{z\to x_{\infty}}
\sup_{s\in K}
\bigl\|
\mc{U}(z)(s)\phi_{1}(z)
-
F(z)(s)v(z)
\bigr\|=0.
\end{equation}
In conclusion
by
\eqref{12102412b},
\eqref{20352312b}
and
\eqref{20492312b}
we obtain
$\forall K\in Comp(\R^{+})$
\begin{equation}
\label{02012912}
\lim_{z\to x_{\infty}}
\sup_{s\in K}
\bigl\|
\mc{U}(z)(s)v(z)
-
F(z)(s)v(z)
\bigr\|=0,
\end{equation}
hence 
\eqref{02502912}.
By
\eqref{01592912}
and
\eqref{02012912}
we obtain
\eqref{02022912}.
Thus
\eqref{02512912p}
and
\eqref{02512912}
follow
by
Lemma
\ref{15482712},
by
\eqref{01592912}
and by the 
following
one
$\forall K\in Comp(\R^{+})$
and
$\forall v\in\mc{E}$
$$
\sup_{z\in X}
\sup_{s\in K}
\bigl\|
\mc{U}(z)(s)v(z)
\bigr\|
\leq
\sup_{z\in X}
\|v(z)\|
<\infty.
$$
where
we considered that
by 
construction
$\bigl\|
\mc{U}(z)(s)
\bigr\|\leq
1$,
for all $s\in\R^{+}$
and $z\in X$
and
that
$v\in\Gamma(\pi)$.
\end{proof}
\begin{remark}
\label{17341602}
If
$\mf{W}$
is full
$(\exists\,F\in\Gamma(\rho))
(F(x_{\infty})=\mc{U}(x_{\infty}))$,
so hypotheses reduce.
\end{remark}
\section{Corollary I. Construction of equicontinuous sequences}
By providing conditions ensuring the
bounded equicontinuity of which 
in hypothesis $(ii)$ of Thm. \ref{17301812b}
we obtain the following
\begin{corollary}
\label{21343012}
Let
us assume the hypotheses
of 
Thm. 
\ref{17301812b}
except
$(ii)$
replaced
by the following
one
$$
(\exists\,G\in\prod_{z\in X}
\mf{L}_{1}\left(\R^{+},
\mc{L}_{S_{x}}(\mf{E}_{x})\right)
(\exists\,H\in
\prod_{z\in X}^{b}\mc{L}(\mf{E}_{z}))
(\exists\,F\in\Gamma(\rho))
$$
such that
$F(x_{\infty})=\mc{U}(x_{\infty})$
and
$\forall s>0$
\begin{equation}
\label{22053012}
\begin{cases}
\sup_{x\in X}\sup_{s>0}\|F(x)(s)\|<\infty
\\
(\forall s_{1}>0)
(\exists\, a>0)
(\sup_{u\in[s_{1},s]}
\sup_{z\in X}
\|G(z)(u)\|\leq a |s-s_{1}|)
\\
(\forall z\in X)
(F(z)(s)
=
H(z)+
\int_{0}^{s}
G(z)(u)\,du),
\end{cases}
\end{equation}
where the 
integration is with respect 
to the Lebesgue measure
on $[0,s]$
and with respect to the 
locally convex topology on
$\mc{L}_{S_{z}}(\mf{E}_{z})$.
Then holds
the statement
of Thm. \ref{17301812b}.
\end{corollary}
\begin{proof}
Let
$v\in\mc{E}$
thus
$(\exists\,\phi\in\Phi)
(v(x_{\infty})=\phi_{1}(x_{\infty}))$
so
$(\forall\{z_{n}\}_{n\in\N}\subset X
\,\vert\,
\lim_{n\in\N}z_{n}=x_{\infty})$
we have
$$
\sup_{n\in\N}
\sup_{s>0}
\|\mc{U}(z_{n})(s)
\phi_{1}(z_{n})
-
F(z_{n})(s)v(z_{n})
\|
\leq
\sup_{n\in\N}
\|\phi_{1}(z_{n})\|
+
M
\sup_{n\in\N}
\|v(z_{n})\|
<\infty.
$$
Here
in the first
inequality we
used
$\|\mc{U}(z)(s)\|\leq 1$
for all $z\in X$ and $s>0$
by construction,
and
$M\coloneqq
\sup_{z\in X}\sup_{s>0}\|F(z)(s)\|<\infty$
by hypothesis,
while
in the second inequality
we used the fact
that
$v\in\prod_{x\in X}^{b}\mf{E}_{x}$,
by construction
and
that
$\sup_{n\in\N}
\|\phi_{1}(z_{n})\|<\infty$
because of
$\exists\,
\varlimsup_{n\in\N}\|\phi_{1}(z_{n})\|
\in\R$
by Rmk.
\ref{18581512b}
and by
construction
$\|\cdot\|$
is $u.s.c.$
Moreover
by
\cite[$(1.4)$]{kurtz},
\eqref{22053012}
and 
$S_{x}=
\{\{w(x)\}\,\vert\, w\in\mc{E}\}$
for all $x\in X$
we have
\begin{alignat*}{1}
\mc{U}(z_{n})(s)
\phi_{1}(z_{n})
-
F(z_{n})(s)v(z_{n})
&
=
\int_{0}^{s}
\left(
\mc{U}(z_{n})(u)
T_{z_{n}}
\phi_{1}(z_{n})
-
G(z_{n})(u)v(z_{n})
\right)
\,du
+
\\
&
+
\phi_{1}(z_{n})
-
H(z_{n})v(z_{n}).
\end{alignat*}
Thus
for any $s_{1},s_{2}\in\R^{+}$
\begin{alignat*}{1}
\sup_{n\in\N}
\|
(\mc{U}(z_{n})(s_{1})
\phi_{1}(z_{n})
-
F(z_{n})(s_{1})v(z_{n}))
-
(\mc{U}(z_{n})(s_{2})
\phi_{1}(z_{n})
-
F(z_{n})(s_{2})v(z_{n}))
\|
&
\leq
\\
|s_{1}-s_{2}|
\sup_{n\in\N}
\sup_{u\in[s_{1},s_{2}]}
\|\mc{U}(z_{n})(u)
T_{z_{n}}
\phi_{1}(z_{n})
-
G(z_{n})(u)v(z_{n})
\|
&
\leq
\\
|s_{1}-s_{2}|
\sup_{n\in\N}
(\|T_{z_{n}}\phi_{1}(z_{n})\|
-a\|v(z_{n})\|)
&
\leq
J|s_{1}-s_{2}|.
\end{alignat*}
Here
in the second inequality we used
$\|\mc{U}(z)(u)\|\leq 1$ by construction
and
the hypothesis,
in the third
one
the fact
that
$\sup_{n\in\N}
\|T_{z_{n}}\phi_{1}(z_{n})\|
<\infty$
as well
$\sup_{n\in\N}\|v(z_{n})\|<\infty$,
because of
$\exists\,
\varlimsup_{n\in\N}
\|T_{z_{n}}\phi_{1}(z_{n})\|
\in\R$
and
$
\exists\,
\varlimsup_{n\in\N}\|v(z_{n})\|
\in\R
$
due to 
the fact that
$\|\cdot\|$
is 
$u.s.c.$
by construction
and
Rmk. \ref{18581512b}
for the first
limit
and 
the continuity of 
$v$
for the second one.
Therefore
hypothesis
$(ii)$
of
Thm. 
\ref{17301812b}
is satisfied,
hence the statement
follows
by
Thm. 
\ref{17301812b}.
\end{proof}
\section{Corollaries II.
Construction of
$\lr{\mf{V},\mf{W}}{X,\R^{+}}$
with the 
$\ms{LD}$ 
}
\label{13320803}
In section \ref{05241150}
we develop a general strategy 
to establish the Laplace duality properties.
When this procedure is applied to fulfill 
hypothesis $(i)$ of Thm. \ref{17301812b},
we obtain Cor. \ref{12080805} and Thm. \ref{10581004}.
Let us start with the following simple
result about the relation among
full and Laplace duality property.
\begin{proposition}
\label{19030106}
Let
$\mf{W}
\coloneqq
\lr{\lr{\mf{M}}{\gamma}}{\rho,X,\mf{R}}$
and
$\lr{\mf{V},\mf{W}}{X,\R^{+}}$
be
a
$\left(\Theta,\mc{E}\right)-$structure
such that
$\mf{V}$
is a Banach bundle
and
$x_{\infty}\in X$.
Assume that
\begin{enumerate}
\item
$\mf{V}$
and
$\mf{W}$
are
full;
\item
$\mc{E}=\Gamma(\pi)$
and
$\Theta$ is given in \eqref{11121419b};
\item
$(\forall F\in\Gamma^{x_{\infty}}(\rho))
(M(F)\coloneqq
\sup_{x\in X}
\sup_{s\in\R^{+}}
\|F(x)(s)\|<\infty)$;
\item
$(\forall\sigma\in\Gamma(\rho))
(\sup_{x\in X}
\sup_{s\in\R^{+}}
\|\sigma(x)(s)\|<\infty)$;
\item
the filter of neighbourhoods of $x_{\infty}$ admits a 
countable basis.
\end{enumerate}
If
$\lr{\mf{V},\mf{W}}{X,\R^{+}}$
has the
$\ms{LD}$
then
it has the
$\ms{LD}_{x_{\infty}}$.
\end{proposition}
\begin{proof}
Let 
$F\in\Gamma^{x_{\infty}}(\rho)$
and
$w\in\Gamma^{x_{\infty}}(\pi)$
thus
by 
hypothesis $(2)$
and
Cor.
\ref{28111707}
there exist
$\sigma\in\Gamma(\rho)$
and
$v\in\Gamma(\pi)$
such that
$\sigma(x_{\infty})=F(x_{\infty})$,
$v(x_{\infty})=w(x_{\infty})$,
and $\forall K\in Comp(\R^{+})$,
$\forall v\in\mc{E}$
\begin{equation}
\label{19330106}
\begin{cases}
\lim_{z\to x_{\infty}}
\|w(z)-v(z)\|=0
\\
\lim_{z\to x_{\infty}}
\sup_{s\in K}
\|(F(x)(s)-\sigma(x)(s))v(x)\|
&=0.
\end{cases}
\end{equation}
Moreover
$\forall\lambda>0$
\begin{alignat}{1}
\label{20090106}
\bigl\|
\int_{0}^{\infty}
e^{\lambda s}
F(z)(s)w(z)\,ds
-
\int_{0}^{\infty}
e^{\lambda s}
\sigma(z)(s)v(z)\,ds
\bigr\|
&\leq
\notag
\\
\bigl\|
\int_{0}^{\infty}
e^{\lambda s}
F(z)(s)(w(z)-v(z))\,ds
\bigr\|
+
\bigl\|
\int_{0}^{\infty}
e^{\lambda s}
(F(z)(s)-\sigma(z)(s))v(z)\,ds
\bigr\|
&\leq
\notag
\\
\frac{1}{\lambda}M(F)\|v(z)-w(z)\|
+
\int_{0}^{\infty}
e^{\lambda s}
\bigl\|
(F(z)(s)-\sigma(z)(s))v(z)
\bigr\|\,ds.
\end{alignat}
By the hypotheses
$(3-4)$
$\sup_{z\in X}
\sup_{s\in\R^{+}}
\bigl\|(F(z)(s)-\sigma(z)(s))v(z)\bigr\|
<\infty$
hence
$\forall\{z_{n}\}_{n\in\N}\subset X$
such that
$\lim_{n\in\N}z_{n}=x_{\infty}$
we have
by
\eqref{19330106},
\eqref{20090106}
and a well-known theorem 
on convergence of sequences
of integrals
that
$\forall\lambda>0$
$$
\lim_{n\in\N}
\bigl\|
\int_{0}^{\infty}
e^{\lambda s}
F(z_{n})(s)w(z_{n})\,ds
-
\int_{0}^{\infty}
e^{\lambda s}
\sigma(z_{n})(s)v(z_{n})\,ds
\bigr\|
=0.
$$
Thus 
$\forall\lambda>0$
by 
hypothesis
$(5)$
$$
\lim_{z\to x_{\infty}}
\bigl\|
\int_{0}^{\infty}
e^{\lambda s}
F(z)(s)w(z)\,ds
-
\int_{0}^{\infty}
e^{\lambda s}
\sigma(z)(s)v(z)\,ds
\bigr\|
=0,
$$
hence
the statement
by 
Cor.
\ref{28111707}.
\end{proof}
Now we shall see that in the
case of a bundle of normed space
we can choose for all $x$
a simple 
space $\mf{M}_{x}$
satisfying
\eqref{18470109}.
\begin{proposition}
\label{18390901}
Let
$\mf{W}
\coloneqq
\lr{\lr{\mf{M}}{\gamma}}{\rho,X,\mf{R}}$
and
$\lr{\mf{V},\mf{W}}{X,\R^{+}}$
be
a
$\left(\Theta,\mc{E}\right)-$structure
such that 
for all $x\in X$,
$\mf{E}_{x}$
is a reflexive Banach
space,
$S_{x}\subseteq\p_{\omega}(\mf{E}_{x})$
and
$$
\mf{M}_{x}
\subseteq
\bigl\{
F\in
\cc{}{\R^{+},\mc{L}_{S_{x}}(\mf{E}_{x})}
\,\vert\,
(\forall\lambda>0)
\bigl(\int_{\R^{+}}^{*}
e^{-\lambda s}
\|F(s)\|_{B(\mf{E}_{x})}
\,ds
<\infty\bigr)
\bigr\}.
$$
Thus
\begin{equation}
\label{14441001}
\mf{M}_{x}
\subset
\bigcap_{\lambda>0}
\mf{L}_{1}
(\R^{+},\mc{L}_{S_{x}}(\mf{E}_{x});
\mu_{\lambda}).
\end{equation}
In particular
\eqref{14441001},
and
$\ms{U}_{\|\cdot\|_{B(\mf{E}_{x})}}
(\mc{L}_{S_{x}}(\mf{E}_{x}))
\subseteq
\mf{M}_{x}$
hold
if
for 
any $x\in X$
$$
\mf{M}_{x}
=
\bigl\{
F\in
\cc{c}{\R^{+},\mc{L}_{S_{x}}(\mf{E}_{x})}
\,\vert\,
\sup_{s\in\R^{+}}
\|F(s)\|_{B(\mf{E}_{x})}
<\infty
\bigr\}.
$$
\end{proposition}
\begin{proof}
The first sentence follows by
\cite[Cor. $2.6.$]{SilInt},
while the
second sentence comes by the first one.
\end{proof}
\subsection{$\ms{U}-$Spaces}
\label{05241150}
Aim of this section is to establish a procedure 
ensuring the full Laplace duality property, 
result achieved in Cor. \ref{18491004}.
The core concept is that of $\ms{U}-$Space 
provided 
in Def. \ref{14302503},
whose existence is established in Cor. \ref{15111901}
by mean of a special locally convex final topology
constructed in Def. \ref{10221801}.
Thm. \ref{15251401},
Thm. \ref{16291501} and Thm. \ref{15332203} 
represent the steps to obtain Cor. \ref{15111901}.
\par
Let us recall and introduce some notation.
For any $W,Z$ topological vector spaces
over $\K\in\{\R,\C\}$
we denote
by
$\mc{L}(W,Z)$
the 
$\K-$linear space 
of all continuous linear map on $W$
and with values in $Z$
and set
$\mc{L}(Z)
\coloneqq
\mc{L}(Z,Z)$
and
$Z^{*}
\coloneqq
\mc{L}(Z,\K)$.
If $Y$ is a topological space 
we let $\cc{cs}{Y,Z}$ denote 
the linear space
of all continuous maps
$f:Y\to Z$
with compact support.
If $Z\in Hlcs$ and $Y$ is locally compact 
we denote by 
$\mf{L}_{1}(Y,Z,\mu)$
the linear space 
of all maps on $Y$
and with values in $Z$
which are essentially 
$\mu-$integrable
in the sense described
in \cite[Ch. $6$]{IntBourb}. 
Moreover
for any family
$\{Z_{x}\}_{x\in X}$
of linear spaces
and
for all $x\in X$
set
$\Pr_{x}:
\prod_{y\in X}Z_{y}
\ni
f\mapsto f(x) 
\in
Z_{x}$
and
$\imath_{x}:
Z_{x}
\to
\prod_{y\in X}Z_{y}
$
such that
for all $x\neq y$
and $z_{x}\in Z_{x}$
$\Pr_{y}\circ\imath_{x}(z_{x})=\ze_{y}$,
while
$\Pr_{x}\circ\imath_{x}=Id_{x}$.
Let us set
\begin{equation*}
\lr{\cdot}{\cdot}:
End(\mc{H})
\times
\mc{H}
\ni
(A,v)
\mapsto
A(v)
\in
\mc{H},
\end{equation*}
and for all $x\in X$
$$
\lr{\cdot}{\cdot}_{x}:
End\left(\mf{E}_{x}\right)
\times
\mf{E}_{x}
\ni
(A,v)
\mapsto
A(v)\in
\mf{E}_{x}.
$$
\begin{definition}
\label{12042101}
We call $\mf{Q}$ a consistent class of data 
if 
$\mf{Q}=\lr{X,Y,\mu,\{\mf{E}_{x}\}_{x\in X}}
{\{\tau_{x}\}_{x\in X},\{\n_{x}\}_{x\in X},
\{Q_{x}\}_{x\in X},\lr{\mc{H}}{\mf{T}}}$
where
\begin{enumerate}
\item
$X$ is a set,  
$Y$ is a locally compact space
and $\mu$ is a Radon measure on $Y$;
\item
$\{\mf{E}_{x}\}_{x\in X}$ 
is a family of $Hlcs$;
\item
$\{\tau_{x}\}_{x\in X}$ is a family of topologies
such that
$\lr{\mc{L}(\mf{E}_{x})}{\tau_{x}}
\in Hlcs$,
$\forall x\in X$;
\item
$\{\n_{x}\}_{x\in X}$ is a family 
such that
$\n_{x}\coloneqq\{\nu_{j_{x}}^{x}
\,\vert\, j_{x}\in J_{x}\}$
is a fundamental set of seminorms
of $\mf{E}_{x}$,
$\forall x\in X$;
\item
$\{Q_{x}\}_{x\in X}$ is a family
such that
$Q_{x}\coloneqq\{q_{\alpha_{x}}^{x}
\,\vert\, \alpha_{x}\in A_{x}\}$
is a fundamental set of seminorms
of
$\lr{\mc{L}(\mf{E}_{x})}{\tau_{x}}$,
$\forall x\in X$;
\item
$\lr{\mc{H}}{\mf{T}}\in Hlcs$
such that
\begin{itemize}
\item
$\mc{H}
\subseteq
\prod_{x\in X}\mf{E}_{x}$
as linear spaces;
\item
$\imath_{x}(\mf{E}_{x})
\subset\mc{H}$,
for all $x\in X$;
\item
$\Pr_{x}\in
\mc{L}\left(
\lr{\mc{H}}{\mf{T}},
\mf{E}_{x}
\right)$
and
$\imath_{x}\in
\mc{L}\left(\mf{E}_{x},
\lr{\mc{H}}{\mf{T}}\right)$,
for all
$x\in X$;
\item
$\exists\,\A\subseteq
\prod_{x\in X}\mc{L}(\mf{E}_{x})$
linear space
such that
\begin{enumerate}
\item
$\theta(\A)\up\mc{H}
\subseteq
\mc{L}(\lr{\mc{H}}{\mf{T}})$,
\item
$\imath_{x}
(\mc{L}(\mf{E}_{x}))
\subseteq\A$
for all $x\in X$;
\end{enumerate}
\end{itemize}
\end{enumerate}
where $\theta$ is defined in Def. \ref{13021401}.
We call $X$ the base of $\mf{Q}$, 
$Y$ the locally compact space of $\mf{Q}$ 
and $\mu$ the Radon measure of $\mf{Q}$. 
Moreover we call $\{\mf{E}_{x}\}_{x\in X}$ 
the primary family underlying $\mf{Q}$,
while we call $\{\tau_{x}\}_{x\in X}$
the secondary family underlying $\mf{Q}$.
We call $\mf{Q}$ entire if $\mc{H}=\prod_{x\in X}\mf{E}_{x}$.
\end{definition}
In the present section let 
$\mf{Q}=\lr{X,Y,\mu,\{\mf{E}_{x}\}_{x\in X}}
{\{\tau_{x}\}_{x\in X},\{\n_{x}\}_{x\in X},
\{Q_{x}\}_{x\in X},\lr{\mc{H}}{\mf{T}}}$
be a fixed consistent class of data.
\begin{definition}
\label{12182503}
Let
$\mf{W}
\coloneqq
\lr{\lr{\mf{M}}{\gamma}}{\rho,X,\mf{R}}$
be
a bundle of 
$\Omega-$spaces
such that
for all $x\in X$
$$
\mf{M}_{x}
\subseteq
\mf{L}_{1}(Y,
\lr{\mc{L}(\mf{E}_{x})}{\tau_{x}};\mu).
$$
Set
\begin{equation}
\label{14162503}
\begin{cases}
\blacksquare_{\mu}:
\prod_{x\in X}
\mf{L}_{1}(Y,
\lr{\mc{L}(\mf{E}_{x})}{\tau_{x}};\mu)
\times
\prod_{x\in X}
\mf{E}_{x}
\to
\prod_{x\in X}
\mf{E}_{x}
\\
\blacksquare_{\mu}
(H,v)(x)
\coloneqq
\lr{\int_{\R^{+}}
H(x)(s)\,
d\mu(s)}{v(x)}_{x}
\in
\mf{E}_{x}.
\end{cases}
\end{equation}
\end{definition}
\begin{remark}
\label{14092503}
Let
$\lr{\mf{V},\mf{W}}{X,\R^{+}}$
be
a
$\left(\Theta,\mc{E}\right)-$structure
satisfying
\eqref{18470109}
and
$\mc{O}\subseteq\Gamma(\rho)$,
$\mc{D}\subseteq\Gamma(\pi)$.
Then
\begin{equation}
\label{LD}
\ms{LD}(\mc{O},\mc{D})
\Leftrightarrow
(\forall\lambda>0)
\left(
\blacksquare_{\mu_{\lambda}}
\left(\mc{O},\mc{D}\right)
\subseteq
\Gamma(\pi)
\right).
\end{equation}
Similarly for all $x\in X$
\begin{equation}
\label{LDx}
\ms{LD}_{x}(\mc{O},\mc{D})
\Leftrightarrow
(\forall\lambda>0)
\left(
\blacksquare_{\mu_{\lambda}}
\left(\Gamma_{\mc{O}}^{x}(\rho)
,\Gamma_{\mc{D}}^{x}(\pi)
\right)
\subseteq
\Gamma^{x}(\pi)
\right).
\end{equation}
\end{remark}
\begin{definition}
[
$\ms{U}-$Spaces
]
\label{14302503}
$\mf{G}$
is a 
$\ms{U}-$space
with respect to
$\{\lr{\mc{L}(\mf{E}_{x})}
{\tau_{x}}\}_{x\in X}$,
$\mf{T}$
and
$D$
iff
\begin{enumerate}
\item
$\mf{G}\in Hlcs$;
\item
$\mf{G}
\subset
\mc{L}\left(\lr{\mc{H}}{\mf{T}}\right)$
as linear spaces;
\item
$D\subseteq\mc{H}$;
\item
$(\forall T\in lcp)
\bigl(\exists\,
\Psi_{T}\in
End\bigl(
End(\mc{H})^{T},
\prod_{x\in X}
End(\mf{E}_{x})^{Y}
\bigr)
\bigr)
(\forall\nu\in Radon(T))$
$$
\Psi_{T}\bigl(
\mf{L}_{1}(T,\mf{G},\nu)
\bigr)
\subseteq
\prod_{x\in X}
\mf{L}_{1}\left(T,
\lr{\mc{L}(\mf{E}_{x})}{\tau_{x}};\nu\right),
$$
and
$\forall\ov{F}
\in
\mf{L}_{1}(T,\mf{G},\nu)
$,
$\forall v\in D$.
$\forall x\in X$
\begin{equation}
\label{17432203}
\boxed{
\lr{\int
\Psi_{T}(\ov{F})(x)(s)
\,d\nu(s)}{v(x)}_{x}
=
\lr{\int\ov{F}(s)\,d\nu(s)}{v}(x)
}
\end{equation}
\end{enumerate}
\end{definition}
The reason of introducing the 
concept of
$\ms{U}-$spaces
will be clarified by the following
\begin{proposition}
\label{14492503}
Let
$\lr{\mf{V},\mf{W}}{X,\R^{+}}$
be
a
$\left(\Theta,\mc{E}\right)-$structure
satisfying
\eqref{18470109},
and
let
$\mf{G}$
be
a 
$\ms{U}-$space
with respect to
$\{\mc{L}_{S_{x}}(\mf{E}_{x})\}_{x\in X}$,
$\mf{T}$
and $\mc{D}$.
Then
$\forall\lambda>0$,
$\ov{F}
\in
\mf{L}_{1}(\R^{+},\mf{G},\mu_{\lambda})$,
$v\in\mc{D}$
\begin{equation}
\label{15102503}
\blacksquare_{\mu_{\lambda}}
(\Psi_{\R^{+}}(\ov{F}),v)
=
\lr{\int\ov{F}(s)\,d\mu_{\lambda}(s)}{v}.
\end{equation}
Moreover
if
$\exists\,\mc{F}\subset
\bigcap_{\lambda>0}
\mf{L}_{1}(\R^{+},\mf{G},\mu_{\lambda})$
such that
$\Psi_{\R^{+}}(\mc{F})=\mc{O}$
then
\begin{equation}
\label{21002103}
\boxed{
\ms{LD}(\mc{O},\mc{D})
\Leftrightarrow
(\forall\lambda>0)
(\lr{\B_{\lambda}}{\mc{D}}
\subseteq\Gamma(\pi)).}
\end{equation}
Here
$$
\B_{\lambda}
\coloneqq
\bigl\{
\int\ov{F}(s)\,d\mu_{\lambda}(s)
\,\vert\,\ov{F}\in\mc{F}
\bigr\}.
$$
\end{proposition}
\begin{proof}
\eqref{15102503}
follows by
\eqref{17432203},
while
\eqref{21002103}
follows
by
\eqref{15102503}
and Rmk.
\ref{14092503}.
\end{proof}
\begin{remark}
\label{13390104}
In particular
if
$\exists\,\mc{F}\subset
\bigcap_{\lambda>0}
\mf{L}_{1}(\R^{+},\mf{G},\mu_{\lambda})$
such that
$\Psi_{\R^{+}}(\mc{F})=\mc{O}$
then
$$
\lr{\mf{G}}{\mc{D}}
\subseteq
\Gamma(\pi)
\Rightarrow
\ms{LD}(\mc{O},\mc{D}).
$$
More in general
if
$\exists\,\mf{G}_{0}$
complete
subspace
of $\mf{G}$
and
$\exists\,\mc{F}\subset
\bigl\{
\ov{F}\in
\bigcap_{\lambda>0}
\mf{L}_{1}(\R^{+},\mf{G},\mu_{\lambda})
\,\vert\,
\ov{F}(\R^{+})\subseteq\mf{G}_{0}
\bigr\}$
such that
$\Psi_{\R^{+}}(\mc{F})=\mc{O}$
then
$$
\lr{\mf{G}_{0}}{\mc{D}}
\subseteq
\Gamma(\pi)
\Rightarrow
\ms{LD}(\mc{O},\mc{D}).
$$
\end{remark}
Thus the $\ms{U}$
property
expressed by
\eqref{17432203}
is an important tool 
for ensuring
the 
satisfaction
of the $\ms{LD}$. 
For this reason
the 
remaining of the present section
will be dedicated to the construction
of a space $\mf{G}$,
Def. \eqref{10221801},
which is a 
$\ms{U}-$space,
see
Thm. \ref{15332203}
and
Cor. \ref{18491004}
for the
$\ms{LD}(\mc{O},\mc{D})$.
\begin{definition}
\label{13021401}
Set
$$
\begin{cases}
\chi_{\mc{H}}:
End(\mc{H})
\to
\prod_{x\in X}End(\mf{E}_{x}),
\\
(\forall x\in X)
(\forall w\in
End(\mc{H}))
((\Pr_{x}\circ\chi_{\mc{H}})(w)
=\Pr_{x}\circ w\circ\imath_{x}),
\\
\chi\coloneqq
\chi_{\prod_{x\in X}\mf{E}_{x}}.
\end{cases}
$$
Well defined indeed by construction
$\imath_{x}(\mf{E}_{x})
\subset\mc{H}$,
for all $x\in X$.
Finally set
$$
\begin{cases}
\theta:
\prod_{x\in X}End(\mf{E}_{x})
\to
End\left(\prod_{x\in X}\mf{E}_{x}\right),
\\
(\forall x\in X)
(\forall u\in
\prod_{x\in X}End(\mf{E}_{x}))
(\Pr_{x}\circ\theta(u)=\Pr_{x}(u)\circ\Pr_{x}),
\\
\theta_{\mc{H}}:
Im(\chi_{\mc{H}})
\ni
u\mapsto
\theta(u)\up\mc{H}.
\end{cases}
$$
Well-posed by applying 
\cite[Prop. $4$, $n^{\circ}5$,
$\S 1$,Ch. $2$]{BourA1}.
\end{definition}
\begin{remark}
\label{10211801}
$(\forall x\in X)
(\forall u\in
\prod_{x\in X}End(\mf{E}_{x}))$
we have
$(\Pr_{x}\circ\theta(u)\circ\imath_{x}
=
\Pr_{x}(u))$.
\end{remark}
\begin{proposition}
\label{14152703}
The space
$\prod_{x\in X}\mf{E}_{x}$
with the
product topology
satisfies
the request
for the space
$\lr{\mc{H}}{\mf{T}}$
in Def. \ref{12042101}
with the choice
$\A=\prod_{x\in X}\mc{L}(\mf{E}_{x})$.
\end{proposition}
\begin{proof}
$\Pr_{x}\in\mc{L}\left(
\prod_{y\in X}\mf{E}_{y},
\mf{E}_{x}
\right)$
by definition
of the product topology,
moreover
$\imath_{x}\in\mc{L}\left(
\mf{E}_{x},
\prod_{y\in X}\mf{E}_{y}
\right)$.
Indeed 
$\imath_{x}$ is clearly linear
and
by considering that
for any net
$\{f^{\alpha}\}_{\alpha\in D}$
and any
$f$
in
$\prod_{y\in X}\mf{E}_{y}$,
$\lim_{\alpha\in D}f^{\alpha}=f$
if and only if
$\lim_{\alpha\in D}f^{\alpha}(y)=f(y)$
for all $y\in X$,
we deduce 
that
for any net
$\{f_{x}^{\alpha}\}_{\alpha\in D}$
and any
$f_{x}$
in
$\mf{E}_{x}$
such that
$\lim_{\alpha\in D}f_{x}^{\alpha}=f_{x}$
we have
$\lim_{\alpha\in D}
\imath_{x}(f_{x}^{\alpha})
=\imath_{x}(f_{x})$,
so 
$\imath_{x}$
is continuous.
Let $x\in X$
and
$u\in\prod_{x\in X}\mc{L}(\mf{E}_{x})$
so
$
\Pr_{x}(u)\circ\Pr_{x}
\in\mc{L}\left(
\prod_{y\in X}\mf{E}_{y},
\mf{E}_{x}
\right)$,
so $(6a)$
follows
by the definition
of $\theta$ and 
\cite[Prp. $4$, $No 3$, $\S 2$]{BourGT}.
Finally $(6b)$
is trivial.
\end{proof}
The following is the main structure of the
present section.
For the definition and properties
of
locally convex final
topologies 
see 
\cite[$No 4$, $\S 4$]{BourTVS}.
\begin{definition}
\label{10221801}
Set for all
$x\in X$
$$
\begin{cases}
G\coloneqq
\theta(\A)
\up\mc{H},
\\
g_{x}:
\mc{L}(\mf{E}_{x})
\ni f_{x}\mapsto
\imath_{x}\circ f_{x}\circ\Pr_{x}
\in
End\left(
\prod_{y\in X}\mf{E}_{y}
\right)
\\
h_{x}:
\mc{L}(\mf{E}_{x})
\ni f_{x}\mapsto
g_{x}(f_{x})\up\mc{H}.
\end{cases}
$$
We shall denote by
$\mf{G}$
and call the locally convex space relative to the 
consistent class of data $\mf{Q}$,
the lcs
$G$ provided with the 
locally convex final
topology 
of the family
of topologies
$\{\tau_{x}\}_{x\in X}$
of the family
$\{\mc{L}(\mf{E}_{x})\}_{x\in X}$,
for the family
of linear mappings
$\{h_{x}\}_{x\in X}$.
\end{definition}
\begin{definition}
\label{11441302}
Set in 
$
\prod_{x\in X}End(\mf{E}_{x})
$
the following 
binary operation
$\circ$.
For all $x\in X$ we set
$\Pr_{x}(f\circ h)\coloneqq f(x)\circ h(x)$.
\end{definition}
It is easy to verify that
$\lr{\prod_{x\in X}End(\mf{E}_{x})}{+,\circ}$
is an algebra over $\K$
as well as
$\lr{\prod_{x\in X}
\mc{L}(\mf{E}_{x})}{+,\circ}$.
\begin{lemma}
\label{17141401}
$G\subset
\mc{L}\left(\lr{\mc{H}}{\mf{T}}\right)$,
moreover 
$\theta$
is a morphism of algebras.
Finally
if 
$\A$ is a subalgebra
of
$\prod_{x\in X}\mc{L}(\mf{E}_{x})$
then
$G$
is a subalgebra
of
$\mc{L}\left(\lr{\mc{H}}{\mf{T}}\right)$.
\end{lemma}
\begin{proof}
The first sentence
is immediate 
by
$(6a)$
in Def. \ref{12042101}.
Let $u,v\in\prod_{x\in X}\mc{L}(\mf{E}_{x})$
thus for all $x\in X$
\begin{alignat*}{1}
\Pr_{x}\circ\theta(u\circ v)
&=
(u(x)\circ v(x))\circ\Pr_{x}
\\
&=
u(x)\circ\Pr_{x}\circ\theta(v)
\\
&=
\Pr_{x}\circ\theta(u)\circ\theta(v),
\end{alignat*}
so
$\theta(u\circ v)=\theta(u)\circ\theta(v)$,
similarly 
we can show that $\theta$ is linear
by
the linearity
of $\Pr_{x}$ for all $x\in X$.
Thus
$\theta$ is a morphism of algebras,
so
the last sentence of the
statement
follows
by the first one and
the fact that
$\A$
is an algebra.
\end{proof}
\begin{proposition}
\label{14271401}
$\theta_{\mc{H}}
\circ
\chi_{\mc{H}}
(w)
=
w\circ\imath_{x}\circ\Pr_{x}
\up\mc{H}$
for all $w\in End(\mc{H})$,
Moreover
$\theta_{\mc{H}}
(Im(\chi_{\mc{H}}))
\subset 
Dom(\chi_{\mc{H}})$
and
$\chi_{\mc{H}}\circ\theta_{\mc{H}}
=Id\up
Im(\chi_{\mc{H}})$.
\end{proposition}
\begin{proof}
Let
$w\in End(\mc{H})$
thus
for all $x\in X$
we have
$(\Pr_{x}\circ\theta_{\mc{H}}
\circ\chi_{\mc{H}})(w)
=
\Pr_{x}(\chi_{\mc{H}}(w))\circ\Pr_{x}
\up\mc{H}
=
\Pr_{x}\circ w\circ
\imath_{x}\circ\Pr_{x}\up\mc{H}$
and the first 
sentence of the statement follows.
By the first sentence 
and the assumption that 
$\imath_{x}(\mf{E}_{x})\subset\mc{H}$
we have 
$\theta(Im(\chi_{\mc{H}}))\up\mc{H}
\subset End(\mc{H})$
so $\chi_{\mc{H}}\circ\theta_{\mc{H}}$
is well set.
Moreover
for all
$x\in X$
and
$u\in Im(\chi_{\mc{H}})$
we have
$\Pr_{x}\left(
\chi_{\mc{H}}\left(\theta(u)\up\mc{H}\right)
\right)
=
\Pr_{x}\circ\theta(u)\circ\imath_{x}
=
\Pr_{x}(u)
\circ\Pr_{x}\circ\imath_{x}
=
\Pr_{x}(u)$.
\end{proof}
\begin{proposition}
\label{10231801}
Let $x\in X$, 
then
\begin{enumerate}
\item
$g_{x}=\theta\circ\imath_{x}$
so
$Im(h_{x})\subseteq G$;
\item
$h_{x}\in End(\mc{L}(\mf{E}_{x}),G)$;
\item
$\exists\,
h_{x}^{-1}:
Im(h_{x})\to\mc{L}(\mf{E}_{x})$
and
$$
\begin{cases}
h_{x}^{-1}=\Pr_{x}\circ\chi_{\mc{H}}
\up Im(h_{x}),
\\
Im(h_{x})=\{\theta(\imath_{x}(f_{x}))
\up\mc{H}
\,\vert\, f_{x}\in
\mc{L}(\mf{E}_{x})\}.
\end{cases}
$$
\end{enumerate}
\end{proposition}
\begin{proof}
$\forall y\in X$
we have
$$
Pr_{y}\circ\theta(\imath_{x}(f_{x}))
=
\Pr_{y}(\imath_{x}(f_{x}))\circ\Pr_{y}
=
\begin{cases}
\ze_{y},x\ne y
\\
f_{x}\circ\Pr_{x},x=y.
\end{cases}
$$
Moreover
$$
Pr_{y}\circ g_{x}(f_{x})
=
\Pr_{y}\circ\imath_{x}
\circ 
f_{x}
\circ\Pr_{x}
=
\begin{cases}
\ze_{y},x\ne y
\\
f_{x}\circ\Pr_{x},x=y.
\end{cases}
$$
So the 
first sentence of
statement $(1)$ follows.
Thus
$h_{x}\left(\mc{L}(\mf{E}_{x})\right)
=
g_{x}\left(\mc{L}(\mf{E}_{x})\right)
\up\mc{H}
=
\theta\left(\imath_{x}\left(
\mc{L}(\mf{E}_{x})\right)
\right)
\up\mc{H}$
so by
$(6b)$ of Def. \ref{12042101}
the
second
sentence of
statement $(1)$ follows.
Statement $(2)$
follows
by the trivial linearity
of $g_{x}$
and by the second sentence of statement
$(1)$.
Let
$f_{x}\in\mc{L}(\mf{E}_{x})$
and
$w=\imath_{x}\circ f_{x}\circ\Pr_{x}\up\mc{H}$.
Then
by the
assumption $(6)$
we have that
$w\in End(\mc{H})$,
and
$\chi_{\mc{H}}(w)=\imath_{x}(f_{x})$,
indeed
$\Pr_{x}(\chi_{\mc{H}}(w))
=
\Pr_{x}\circ\imath_{x}
\circ
f_{x}
\circ
\Pr_{x}\circ\imath_{x}
=
f_{x}
=
\Pr_{x}(\imath(f_{x}))$.
Thus
$\imath_{x}(f_{x})\in Im(\chi_{\mc{H}})$
so by Pr. \ref{14271401}
$\theta(\imath_{x}(f_{x}))\up\mc{H}
\in 
Dom(\chi_{\mc{H}})$
and 
$h_{x}^{-1}$ is well set.
Moreover
\begin{alignat*}{2}
(\Pr_{x}\circ\chi_{\mc{H}})\circ 
h_{x}(f_{x})
&
=
\Pr_{x}\circ\chi_{\mc{H}}
\circ\theta_{\mc{H}}(\imath_{x}(f_{x}))
\\
&=
\Pr_{x}(\imath_{x}(f_{x}))
=f_{x},
\end{alignat*}
where the first equality comes
by stat. $(1)$
and by
$\imath_{x}(f_{x})\in Im(\chi_{\mc{H}})$, 
while
the second by Prop. \ref{14271401}.
Finally
\begin{alignat*}{2}
g_{x}\circ\Pr_{x}\circ
\chi_{\mc{H}}(\theta(\imath_{x}(f_{x})))
&
=
g_{x}\circ\Pr_{x}(\imath_{x}(f_{x}))
\\
&=
g_{x}(f_{x})
=
\theta(\imath_{x}(f_{x})).
\end{alignat*}
Thus stat. $(3)$ follows.
\end{proof}
\begin{lemma}
\label{16141302}
If $\lr{\mc{L}(\mf{E}_{x})}{\tau_{x}}$
is a 
topological algebra
for all
$x\in X$
and
$\A$ is an algebra
then
$\mf{G}$ 
is a 
topological algebra.
\end{lemma}
\begin{proof}
Let us set for all $F\in\mf{G}$
$L_{F}:\mf{G}\ni H\mapsto F\circ H\in\mf{G}$,
well set $\mf{G}$ being an algebra
by Lemma \ref{17141401}.
Thus for all
$x\in X$,
$f\in\A$
and
$l_{x}\in\mc{L}(\mf{E}_{x})$
\begin{alignat*}{2}
(L_{\theta(f)}\circ h_{x})
l_{x}
&=
L_{\theta(f)}
(\theta(\imath_{x}(l_{x}))
\up\mc{H})
=
\theta(f\circ\imath_{x}(l_{x}))
\up\mc{H}
\\
&=
\bigl(
\theta\circ\imath_{x}(f(x)\circ l_{x})
\bigr)
\up\mc{H}
=
\bigl(
g_{x}(f(x)\circ l_{x})
\bigr)
\up\mc{H}
\\
&=
h_{x}(f(x)\circ l_{x})
=
(h_{x}\circ L_{f(x)})l_{x},
\end{alignat*}
where
$L_{f_{x}}:\mc{L}(\mf{E}_{x})\ni 
s_{x}\mapsto f_{x}\circ s_{x}
\in\mc{L}(\mf{E}_{x})$
for all $f_{x}\in\mc{L}(\mf{E}_{x})$.
Here
the first and fourth equality follow
by Prop. \ref{10231801},
the second one by
Lemma \ref{17141401}.
Moreover by hypothesis
$L_{f(x)}$ is continuous,
while 
$h_{x}$ is continuous
by 
\cite[Prop.$5$, $No 4$, $\S 4$ Ch $2$]
{BourTVS},
so
$L_{\theta(f)}\circ h_{x}$
is linear and continuous.
Therefore
$L_{\theta(f)}$
is linear and continous
by 
\cite[Prop. $5$, $No 4$, $\S 4$ Ch $2$]
{BourTVS}.
Similarly 
$R_{F}$
is linear and continuous,
where
$R_{F}:\mf{G}\ni H\mapsto H\circ F\in\mf{G}$,
thus the statement.
\end{proof}
\begin{definition}
\label{14321401}
Set
$$
\begin{cases}
\Psi_{Y}^{\mc{H}}:
End(\mc{H})^{Y}
\to
\prod_{x\in X}
End(\mf{E}_{x})^{Y},
\\
(\Pr_{x}\circ\Psi_{Y}^{\mc{H}})(\ov{F})(s)
=
(\Pr_{x}\circ\chi_{\mc{H}})(\ov{F}(s)).
\end{cases}
$$
Moreover
set
$$
\begin{cases}
\Lambda:
\prod_{x\in X}
End(\mf{E}_{x})^{Y}
\to
\left(End\left(\prod_{x\in X}\mf{E}_{x}\right)
\right)^{Y},
\\
\Lambda(F)(s)
=
\theta(F(\cdot)(s)).
\end{cases}
$$
$\forall\ov{F}\in End(\mc{H})^{Y}$,
$\forall 
F\in
\prod_{x\in X}
End(\mf{E}_{x})^{Y}
$,
$\forall x\in X$
and $\forall s\in Y$,
where
$F(\cdot)(s)
\in\prod_{y\in X}End(\mf{E}_{x})$
such that
$\Pr_{x}(F(\cdot)(s))
=
F(x)(s)$.
\par
Finally
set
$$
\Lambda_{\A}^{Y}
\coloneqq
\Lambda\up
\bigl\{
F\in
\prod_{x\in X}
\mc{L}(\mf{E}_{x})^{Y}
\,\vert\,
(\forall s\in Y)
(F(\cdot)(s)\in\A)
\bigr\}.
$$
\end{definition}
\begin{proposition}
\label{15111401}
Let $x\in X$ and $s\in Y$,
then
for all $\ov{F}\in End(\mc{H})^{Y}$
\begin{enumerate}
\item
$(\Pr_{x}\circ\Psi_{Y}^{\mc{H}})(\ov{F})(s)
=
\Pr_{x}\circ\ov{F}(s)\circ\imath_{x}$;
\item
$\Psi_{Y}^{\mc{H}}
\circ\Lambda_{\A}^{Y}=Id$;
\item
$Im(\Lambda_{\A}^{Y})
\subset
G^{Y}$.
\end{enumerate}
\end{proposition}
\begin{proof}
Stats. $(1)$
and $(3)$ 
are trivial.
Let
$F\in Dom(\Lambda_{\A}^{Y})$
so
\begin{alignat*}{2}
(\Pr_{x}\circ
\Psi_{Y}^{\mc{H}}
\circ
\Lambda_{\A}^{Y}
)(F)(s)
&=
(\Pr_{x}\circ\chi_{\mc{H}})(\Lambda_{\A}^{Y}(F)(s))
=
\Pr_{x}\circ\Lambda_{\A}^{Y}(F)(s)\circ\imath_{x}
\\
&=
\Pr_{x}\circ
\theta(F(\cdot)(s))
\circ\imath_{x}
=
\Pr_{x}(F(\cdot)(s))
\circ\Pr_{x}
\circ\imath_{x}
\\
&=
F(x)(s)
=
\Pr_{x}(F)(s),
\end{alignat*}
and stat. $(2)$ follows.
\end{proof}
\begin{proposition}
\label{16141501}
$(\forall x\in X)
(\forall s\in Y)
(\forall\ov{F}\in G^{Y})$
we have
$$
(\Pr_{x}\circ\Psi_{Y}^{\mc{H}})
(\ov{F})(s)
\circ\Pr_{x}
=
\Pr_{x}\circ
(\ov{F}(s))
$$
\end{proposition}
\begin{proof}
Let
$\ov{F}\in G^{Y}$
thus
$\exists\,U\in\A^{Y}$
such that
$\ov{F}(s)=\theta(U(s))
\up\mc{H}$,
hence
for all
$x\in X,
s\in Y$
\begin{alignat*}{2}
(\Pr_{x}\circ\Psi_{Y}^{\mc{H}})(\ov{F})(s)
\circ\Pr_{x}
&=
\Pr_{x}(\Psi_{Y}^{\mc{H}}(\ov{F}))(s)
\circ
\Pr_{x}
\\
&=
\Pr_{x}\circ\ov{F}(s)\circ\imath_{x}
\circ
\Pr_{x},&
\text{ by Prop. \ref{15111401}}
\\
&=
(\Pr_{x}\circ
\theta(U(s)))
\up\mc{H}
\circ\imath_{x}
\circ
\Pr_{x}
\\
&\doteq
(\Pr_{x}(U(s))
\circ\Pr_{x})
\up\mc{H}
\circ\imath_{x}
\circ
\Pr_{x}
\\
&=
\Pr_{x}(U(s))
\circ\Pr_{x}
\up\mc{H}
\\
&\doteq
\Pr_{x}
\circ
\theta(U(s))
\up\mc{H}
\\
&=
\Pr_{x}
\circ
(\ov{F}(s)).
\end{alignat*}
\end{proof}
\begin{definition}
\label{11221801}
Let $x\in X$
$$
\begin{cases}
I_{x}:Hom(\mc{L}(\mf{E}_{x}),\K)
\to
Hom\left(
\prod_{y\in X}\mc{L}(\mf{E}_{y}),\K\right),
\\
I_{x}(t_{x})\coloneqq t_{x}\circ\Pr_{x}.
\end{cases}
$$
\end{definition}
\begin{lemma}
\label{12001801}
Let $x\in X$
thus
\begin{enumerate}
\item
$(\forall t_{x}\in Hom(\mc{L}(\mf{E}_{x}),\K))
(\forall y\in X)$
we have
$$
\begin{cases}
I_{x}(t_{x})\circ\chi_{\mc{H}}\circ h_{y}
=
t_{x},
x=y
\\
I_{x}(t_{x})\circ\chi_{\mc{H}}\circ h_{y}
=\ze,
x\neq y;
\end{cases}
$$
\item
$(\forall t_{x}\in
\lr{\mc{L}(\mf{E}_{x})}{\tau_{x}}^{*})
(I_{x}(t_{x})\circ\chi_{\mc{H}}\up\mf{G}
\in\mf{G}^{*})$
\end{enumerate}
\end{lemma}
\begin{proof}
Let $x\in X$ and $t_{x}\in
Hom(\mc{L}(\mf{E}_{x}),\K)$
thus for all
$y\in X$
and
$f_{y}\in\mc{L}(\mf{E}_{y})$
we have
\begin{alignat*}{1}
I_{x}(t_{x})\circ\chi_{\mc{H}}\circ h_{y}
(f_{y})
&=
t_{x}\circ\Pr_{x}\circ\chi_{\mc{H}}
(\imath_{y}\circ f_{y}\circ\Pr_{y}\up\mc{H})
\\
&=
t_{x}\circ
(\Pr_{x}\circ\imath_{y}\circ f_{y}
\circ\Pr_{y}\circ\imath_{x}),
\end{alignat*}
and stat. $(1)$ follows.
Stat. $(2)$
follows by
stat. $(1)$ and
\cite[Prop. $5$, $No 4$, $\S 4$ Ch $2$]
{BourTVS}.
\end{proof}
The following is the 
first
main result of
this section.
\begin{theorem}
\label{15251401}
We have
\begin{enumerate}
\item
$\Psi_{Y}^{\mc{H}}
\in
Hom(\mf{L}_{1}(Y,\mf{G},\mu),
\prod_{x\in X}
\mf{L}_{1}(Y,
\lr{\mc{L}(\mf{E}_{x})}{\tau_{x}},\mu))$;
\item
$(\forall x\in X)(\forall s\in Y)
(\forall\ov{F}\in\mf{L}_{1}(Y,\mf{G},\mu))
$
$$
\int\Pr_{x}(\Psi_{Y}^{\mc{H}}(\ov{F}))(s)
\,d\mu(s)
=
\Pr_{x}
\circ
\bigl(
\int\ov{F}(s)\,d\mu(s)
\bigr)
\circ\imath_{x}.
$$
\end{enumerate}
\end{theorem}
\begin{proof}
Let $x\in X$, 
set
$$
\Delta_{x}:G\ni f
\mapsto
\Pr_{x}\circ f\circ\imath_{x}
\in
\mc{L}(\mf{E}_{x}).
$$
$\Delta_{x}$ is well-defined
by Lemma \ref{17141401}.
By applying
\cite[Prop. $5$, $No 3$, $\S 4$ Ch $2$]
{BourTVS}
$\Delta_{x}\in\mc{L}
(\mf{G},
\lr{\mc{L}(\mf{E}_{x})}{\tau_{x}})$
if and only if
$(\forall y\in X)
(\Delta_{x}\circ h_{y}\in
\mc{L}(\lr{\mc{L}(\mf{E}_{y})}{\tau_{y}},
\lr{\mc{L}(\mf{E}_{x})}{\tau_{x}}))$.
Moreover $\forall y\in X$ 
and $\forall f_{y}\in\mc{L}(\mf{E}_{y})$
we have
$$
(\Delta_{x}\circ h_{y})(f_{y})
=
\Pr_{x}\circ\imath_{y}\circ f_{y}
\circ\Pr_{y}\circ\imath_{x},
$$
so
$$
\begin{cases}
\Delta_{x}\circ h_{y}=Id, x=y
\\
\Delta_{x}\circ h_{y}=\ze,
x\neq y.
\end{cases}
$$
In any case
$\Delta_{x}\circ h_{y}\in
\mc{L}(\lr{\mc{L}(\mf{E}_{y})}{\tau_{y}},
\lr{\mc{L}(\mf{E}_{x})}{\tau_{x}})$,
thus
\begin{equation*}
\Delta_{x}\in\mc{L}
(\mf{G},
\lr{\mc{L}(\mf{E}_{x})}{\tau_{x}})
\end{equation*}
hence
\begin{equation}
\label{13351801}
(\forall t_{x}\in
\lr{\mc{L}(\mf{E}_{x})}{\tau_{x}}^{*})
(t_{x}\circ
\Delta_{x}
\in\mf{G}^{*}).
\end{equation}
Therefore
\begin{alignat*}{1}
t_{x}\left(
\Pr_{x}\circ\bigl(
\int\ov{F}(s)\,d\mu(s)
\bigr)
\circ\imath_{x}
\right)
&=
(t_{x}\circ\Delta)
\bigl(\int\ov{F}(s)\,d\mu(s)\bigr)
\\
&=
\int(t_{x}\circ\Delta)(\ov{F}(s))\,d\mu(s)
\\
&=
\int t_{x}\bigl(\Pr_{x}\circ
\ov{F}(s)\circ\imath_{x}\bigr)\,d\mu(s)
\\
&=
\int
t_{x}\left(
(\Pr_{x}\circ\Psi_{Y}^{\mc{H}})(\ov{F})(s)
\right)
\,d\mu(s),
\end{alignat*}
where
the second equality comes by
\eqref{13351801}
and
\cite[Prop.$1$, $No 1$, $\S 1$, Ch. $6$]
{IntBourb}, 
while the last one
comes by
Prp.
\ref{15111401}. 
\end{proof}
\begin{definition}
\label{11121501}
Let $Z$ be a topological vector space
set
$$
\begin{cases}
\ms{ev}_{Z}
\in 
Hom(Z,Hom(\mc{L}(Z),Z),
\\
(\forall v\in Z)
(\forall f\in\mc{L}(Z))
(\ms{ev}_{Z}(v)(f))
\coloneqq
f(v)).
\end{cases}
$$
Moreover
set
$\eta
\coloneqq
\ms{ev}_{\mc{H}}$
and
$\forall x\in X$
set
$\ep_{x}
\coloneqq
\ms{ev}_{\mf{E}_{x}}$.
\end{definition}
\begin{lemma}
\label{11011501}
Let $D\subseteq\mc{H}$
thus
$(A)\Rightarrow(B)$, 
where
\begin{description}
\item[(A)]
$(\forall x\in X)(\forall v_{x}\in\Pr_{x}(D))
(\ep_{x}(v_{x})\in\mc{L}
(\lr{\mc{L}(\mf{E}_{x})}
{\tau_{x}},\mf{E}_{x}))$;
\item[(B)]
$(\forall v\in D)
(\eta(v)\in\mc{L}
(\mf{G},\lr{\mc{H}}{\mf{T}}))$.
\end{description}
\end{lemma}
\begin{proof}
Let
$y\in X$
thus 
for all
$v\in\mc{H}$
$$
\eta(v)\circ h_{y}
=
\imath_{y}\circ\ep_{y}(\Pr_{y}(v)).
$$
Hence by $(A)$
and the fact that by construction
$\imath_{y}$ is continuous with respect to 
the topology
$\mf{T}$
we have
for all
$v\in D$
$$
\eta(v)\circ g_{y}
\in
\mc{L}\left(
\lr{\mc{L}(\mf{E}_{y})}{\tau_{y}},
\lr{\mc{H}}{\mf{T}}
\right).
$$
Thus 
$(B)$ 
follows
by the universal property
of any
locally final topology, 
cf.
\cite[$(ii)$ of Prop. 
$5$, $N\,4$, $\S 4$ Ch $2$]{BourTVS}.
\end{proof}
The following is 
the second
main result
of the section
\begin{theorem}
\label{16291501}
Let $D\subseteq\mc{H}$
and
assume
$(A)$ of Lemma \ref{11011501}.
Then
$(\forall\ov{F}
\in\mf{L}_{1}(Y,\mf{G},\mu))
(\forall x\in X)
(\forall v\in D)$
\begin{equation}
\label{15281901}
\int
\lr{\Pr_{x}(\Psi_{Y}^{\mc{H}}
(\ov{F}))(s)}{v(x)}_{x}
\,d\mu(s)
=
\lr{\int\ov{F}(s)\,d\mu(s)}{v}(x).
\end{equation}
Here
the integral in the left-side
is with respect to the $\mu$
and the locally convex topology
on $\mf{E}_{x}$,
while
the integral in the right-side
is with respect to the $\mu$
and the locally convex topology
on $\mf{G}$.
\end{theorem}
\begin{proof}
$(\forall\ov{F}
\in\mf{L}_{1}(Y,\mf{G},\mu))
(\forall x\in X)
(\forall v\in D)$
we have
\begin{alignat*}{1}
\Pr_{x}\circ
\bigl(
\int\ov{F}(s)\,d\mu(s)
\bigr)(v)
&=
(\Pr_{x}\circ\eta(v))
\bigl(\int\ov{F}(s)\,d\mu(s)\bigr)
\\
&=
\int
(\Pr_{x}\circ\eta(v))
(\ov{F}(s))
\,d\mu(s)
\\
&=
\int
(\Pr_{x}\circ\ov{F}(s))(v)\,
d\mu(s)
\\
&=
\int
\Pr_{x}(\Psi_{Y}^{\mc{H}}(\ov{F}))(s)
(v(x))
\,d\mu(s).
\end{alignat*}
Here
in the second equality we applied
\cite[Prop.$1$, $No 1$, $\S 1$, Ch. $6$]
{IntBourb}
and the fact that
$\Pr_{x}\circ\eta(v)
\in\mc{L}(\mf{G},\mf{E}_{x})$
because of Lemma \ref{11011501}
and the linearity and continuity 
of $\Pr_{x}$
with respect to the 
topology
$\mf{T}$.
Finally in the last equality
we used Prop. \ref{16141501}.
\end{proof}
The following is the main result of this section
\begin{theorem}
\label{15332203}
Let $D\subseteq\mc{H}$
and
assume
$(A)$ of Lemma \ref{11011501}.
Then
$(\forall\ov{F}
\in\mf{L}_{1}(Y,\mf{G},\mu))
(\forall x\in X)
(\forall v\in D)$
\begin{equation}
\label{14342203}
\boxed{
\lr{\int
\Pr_{x}(\Psi_{Y}^{\mc{H}}(\ov{F}))(s)
\,d\mu(s)}{v(x)}_{x}
=
\lr{\int\ov{F}(s)\,d\mu(s)}{v}(x).
}
\end{equation}
Equivalently
$\mf{G}$
is a 
$\ms{U}-$space
with respect to
$\{\lr{\mc{L}(\mf{E}_{x})}
{\tau_{x}}\}_{x\in X}$,
$\mf{T}$
and
$D$.
Here
the integral in the left-side
is with respect to the $\mu$
and the locally convex topology on
$\lr{\mc{L}(\mf{E}_{x})}{\tau_{x}}$.
\end{theorem}
\begin{proof}
By $(A)$ of Lemma \ref{11011501},
stat.$(1)$ of Th. \ref{15251401}
and
\cite[Prop.$1$, $No 1$, $\S 1$, Ch. $6$]
{IntBourb}
we have
$(\forall\ov{F}
\in\mf{L}_{1}(Y,\mf{G},\mu))
(\forall x\in X)
(\forall v\in D)$
$$
\int
\lr{\Pr_{x}
(\Psi_{Y}^{\mc{H}}(\ov{F}))(s)}{v(x)}_{x}
\,d\mu(s)
=
\lr{\int
\Pr_{x}(\Psi_{Y}^{\mc{H}}(\ov{F}))(s)
\,d\mu(s)}{v(x)}_{x},
$$
hence the statement
follows
by Thm. \ref{16291501}.
\end{proof}
\begin{remark}
\label{16082203}
By \eqref{14342203}
and
stat.$(2)$ of Th. \ref{15251401}
$(\forall\ov{F}
\in\mf{L}_{1}(Y,\mf{G},\mu))
(\forall x\in X)
(\forall v\in D)
$
$$
\lr{\int\ov{F}(s)\,d\mu(s)}{v}(x)
=
\lr{\int\ov{F}(s)\,d\mu(s)}
{\imath_{x}(v(x))}(x).
$$
Thus for all
$v,w\in D$
and $x\in X$
$$
v(x)=w(x)
\Rightarrow
\lr{\int\ov{F}(s)\,d\mu(s)}{v}(x)
=
\lr{\int\ov{F}(s)\,d\mu(s)}{w}(x).
$$
\end{remark}
\begin{corollary}
\label{15111901}
Let $\mc{S}\in\prod_{x\in X}
\mc{P}(Bounded(\mf{E}_{x}))$
and
$\mc{D}$
such that
\begin{equation}
\label{15391901a}
\begin{cases}
N(x)
\coloneqq
\bigcup_{l_{x}\in L_{x}}
B_{l_{x}}^{x}
\text{ is total in $\mf{E}_{x}$},
\forall x\in X,
\\
\mc{D}\subseteq
\mc{H}
\cap
\prod_{x\in X}N(x),
\end{cases}
\end{equation}
where
$\mc{S}(x)
=\{B_{l_{x}}^{x}\,\vert\, l_{x}\in L_{x}\}$.
Assume 
that
for all $x\in X$
the topology
$\tau_{x}$
is generated by the set of seminorms
$\{
p_{(l_{x},j_{x})}^{x}
\,\vert\,
(l_{x},j_{x})\in
L_{x}\times J_{x}\}$,
where
\footnote{In others words 
$\lr{\mc{L}(\mf{E}_{x})}{\tau_{x}}
=
\mc{L}_{\mc{S}_{x}}(\mf{E}_{x})$,
see 
Notation \ref{notat}
and 
Def.
\ref{17471910A}.}
\begin{equation}
\label{15391901b}
p_{(l_{x},j_{x})}^{x}:
\mc{L}(\mf{E}_{x})\ni
f_{x}
\mapsto
\sup_{w\in B_{l_{x}}^{x}}
\nu_{j_{x}}^{x}(f_{x}w)
\in\R^{+}.
\end{equation}
Then
\begin{enumerate}
\item
$(A)$ of Lemma \ref{11011501}
for $D=\mc{D}$;
\item
\eqref{15281901} holds
and
$\mf{G}$
is a 
$\ms{U}-$space
with respect to
$\{\mc{L}_{\mc{S}(x)}(\mf{E}_{x})\}_{x\in X}$,
$\mf{T}$
and
$\mc{D}$.
\end{enumerate}
\end{corollary}
\begin{proof}
By request \eqref{15391901a}
we have that the lcs
$\lr{\mc{L}(\mf{E}_{x})}{\tau_{x}}$
is Hausdorff
so the position is well-set.
By construction
$(\forall x\in X)
(\forall v_{x}\in D(x))
(\exists\,\ov{l}_{x}\in L_{x})
(v_{x}\in B_{\ov{l}_{x}}^{x})$,
so 
$(\forall f_{x}\in\mc{L}(\mf{E}_{x}))
(\forall j_{x}\in J_{x})$
\begin{alignat*}{1}
\nu_{j_{x}}^{x}
(\ep_{x}(v_{x})f_{x})
&=
\nu_{j_{x}}^{x}(f_{x}(v_{x}))
\\
&\leq
p_{(\ov{l}_{x},j_{x})}^{x}(f_{x}),
\end{alignat*}
hence
statement
$(1)$
by 
\cite[Prop. $5$, $No 4$, $\S 1$ Ch $2$]
{BourTVS}.
Statement $(2)$
follows by 
statement $(1)$,
Thm. \ref{16291501}
and
Thm. 
\ref{15332203}
respectively.
\end{proof}
\begin{corollary}
[$\ms{LD}(\mc{O},\mc{D})$]
\label{18491004}
Let
$\lr{\mf{V},\mf{W}}{X,\R^{+}}$
be
a
$\left(\Theta,\mc{E}\right)-$structure
satisfying
\eqref{18470109}
and
$
\Gamma(\pi)
\cap\mc{H}
\cap
\prod_{x\in X}\mc{B}_{B}^{x}
\ne\emptyset
$.
Set
\begin{equation}
\label{21241004}
\begin{cases}
\mc{O}\subseteq\Gamma(\rho)
\\
\mc{D}\subseteq
\Gamma(\pi)\cap\mc{H}
\cap
\prod_{x\in X}\mc{B}_{B}^{x}
\end{cases}
\end{equation}
If
$\exists\,
\mc{F}\subset
\bigcap_{\lambda>0}
\mf{L}_{1}(\R^{+},\mf{G},\mu_{\lambda})$
such that
$\Psi_{\R^{+}}^{\mc{H}}(\mc{F})=\mc{O}$
then
\eqref{21002103}
holds.
\par
In particular
if
$\exists\,\mc{F}\subset
\bigcap_{\lambda>0}
\mf{L}_{1}(\R^{+},\mf{G},\mu_{\lambda})$
such that
$\Psi_{\R^{+}}^{\mc{H}}(\mc{F})=\mc{O}$
then
$$
\lr{\mf{G}}{\mc{D}}
\subseteq
\Gamma(\pi)
\Rightarrow
\ms{LD}(\mc{O},\mc{D}).
$$
Here
$\mc{B}_{B}^{x}$,
for all $x\in X$,
is defined in
\eqref{11232712}.
\end{corollary}
\begin{proof}
By statement $(2)$
of
Cor.
\ref{15111901},
Pr. \ref{14492503}
and
Rm. \ref{13390104}.
\end{proof}
\begin{remark}
\label{21211004}
Note that
if
$\mc{E}\subset\Theta$,
as for example for the positions taken in
Rmk. \ref{21500412b},
we have
$\mc{E}\subset\prod_{x\in X}\mc{B}_{B}^{x}$.
Hence
if
$\mc{E}\subseteq\mc{H}$
we have
$\mc{E}
\subseteq
\Gamma(\pi)
\cap
\mc{H}
\cap
\prod_{x\in X}\mc{B}_{B}^{x}$.
\end{remark}
\begin{corollary}
\label{12080805}
Let us assume 
the hypotheses of Thm. \ref{17301812b}
made exception for the $(i)$
replaced by the following one:
$\mc{E}\subseteq\mc{H}$
and
$\exists\,
\mc{F}\subset
\bigcap_{\lambda>0}
\mf{L}_{1}(\R^{+},\mf{G},\mu_{\lambda})$
such that
$\Psi_{\R^{+}}^{\mc{H}}(\mc{F})=\Gamma(\rho)$
and
$$
\lr{\mf{G}}{\mc{E}}
\subseteq
\Gamma(\pi).
$$
Then all the statements of Thm. \ref{17301812b} hold true.
\end{corollary}
\begin{proof}
Since Rmk. \ref{21211004}, Cor. \ref{18491004}
and Thm. \ref{17301812b}.
\end{proof} 
\subsection{
Uniform Convergence
over 
$\mc{K}\in Comp(\lr{\mc{H}}{\mf{T}})$.
}
In this section
we assume
given
the following
data
\begin{enumerate}
\item
a Banach bundle $\mf{V}$,
a $\left(\Theta,\mc{E}\right)-$structure
$\lr{\mf{V},\mf{M}}{X,Y}$
where
$\Theta$ is
defined in
\eqref{11121419b},
where
we denote
$\mf{W}
\coloneqq
\lr{\lr{\mf{M}}{\gamma}}{\rho,X,\mf{R}}$
and
$\mf{V}
\coloneqq
\lr{\lr{\mf{E}}{\tau}}{\pi,X,\{\|\cdot\|\}}$;
\item
a Banach space
$\lr{\mc{H}}{\|\cdot\|_{\mc{H}}}$
such that
$\lr{\mc{H}}{\mf{T}}$
satisfies
$(6)$ of Def. \ref{12042101},
where
$\mf{T}$ is
the topology induced by the norm
$\|\cdot\|_{\mc{H}}$
and
$\tau_{x}$
is such that
$\lr{\mc{L}(\mf{E}_{x})}{\tau_{x}}
=
\mc{L}_{S_{x}}(\mf{E}_{x})$
for every $x\in X$;
\item
$\A$
as
in
$(6)$ of Def. \ref{12042101};
\item
$\mf{G}$,
$\Psi_{Y}^{\mc{H}}$
and
$\Lambda_{\A}^{Y}$
as defined in Def. \ref{10221801}
and
Def.
\ref{14321401}
respectively.
\end{enumerate}
The proof of the following Lemma is an
adaptation to the present framework
of the proof of 
\cite[Prop. $5.13$]{cho}.
\begin{lemma}
\label{19260304}
Let 
$\mc{U}\in\prod_{x\in X}\mf{M}_{x}$
and
$x_{\infty}\in X$
moreover
assume that
\begin{enumerate}
\item
$\mc{E}\subseteq\mc{H}
\subseteq\prod_{x\in X}^{b}\mf{E}_{x}$
such that
$(\exists\,a>0)(\forall f\in\mc{H})
(\|f\|_{\sup}\leq a\|f\|_{\mc{H}})$,
where
$\|f\|_{\sup}\coloneqq
\sup_{x\in X}\|f(x)\|_{x}$;
\item
$\exists\,F\in\Gamma(\rho)$
such that
$F(x_{\infty})=\mc{U}(x_{\infty})$
and
$\{F(\cdot)(s)\,\vert\, s\in Y\}
\subseteq\A$
\item
$\{\mc{U}(\cdot)(s)\,\vert\, s\in Y\}
\subseteq\A$;
\item
$\{\ov{F}(s)\,\vert\, s\in Y\}$
and
$\{\ov{\mc{U}}(s)\,\vert\, s\in Y\}$
are
equicontinuous
as subsets of 
$\mc{L}(\lr{\mc{H}}{\|\cdot\|_{\mc{H}}})$,
where
$\ov{\mc{U}}\coloneqq\Lambda_{\A}^{Y}(\mc{U})$.
and
$\ov{F}\coloneqq\Lambda_{\A}^{Y}(F)$.
\end{enumerate}
Then
$(A)\Leftrightarrow(B)$
where
\begin{description}
\item[\normalfont{(A)}]
$\mc{U}\in\Gamma^{x_{\infty}}(\rho)$;
\item[\normalfont{(B)}]
For all $\mc{K}\in Comp(\mc{H})$
such that $\mc{K}\subseteq\mc{E}$
and for all $K\in Comp(Y)$
$$
\lim_{z\to x_{\infty}}
\sup_{s\in K}
\sup_{v\in\mc{K}}
\bigl\|
\mc{U}(z)(s)v(z)
-
F(z)(s)v(z)
\bigr\|
=0.
$$
\end{description}
\end{lemma}
\begin{proof}
We shall prove only
$(A)\Rightarrow(B)$,
indeed the other implication
follows
by
$(3)\Rightarrow(4)$
of
Lemma
\ref{15482712}.
So assume $(A)$ to be true.
In this proof let
us set
$B(\mc{H})\coloneqq
\mc{L}(\lr{\mc{H}}{\|\cdot\|_{\mc{H}}})$,
moreover
$\Psi\coloneqq\Psi_{Y}^{\mc{H}}$
and $\Lambda\coloneqq\Lambda_{\A}^{Y}$,
moreover
set
$\ov{F}\coloneqq\Lambda_{\A}^{Y}(F)$
for every $F\in\Gamma(\rho)$;
thus
by stat. $(2)$ of
Pr. \ref{15111401}
$\Psi(\ov{F})=F$
and
$\Psi(\ov{\mc{U}})=\mc{U}$.
Hence by Pr. \ref{16141501}
for all $v\in\mc{E}$
$F\in\Gamma(\rho)$,
$z\in X$ and $s\in Y$
\begin{equation}
\label{20160304}
\mc{U}(z)(s)v(z)
=
(\ov{\mc{U}}v)(z),\,
F(z)(s)v(z)
=
(\ov{F}v)(z).
\end{equation}
By $(A)$
and
implication
$(4)\Rightarrow(3)$
of
Lemma
\ref{15482712}
we have
for all $K\in Comp(Y)$
and $v\in\mc{E}$
\begin{equation}
\label{20150304}
\lim_{z\to x_{\infty}}
\sup_{s\in K}
\bigl\|
\mc{U}(z)(s)v(z)
-
F(z)(s)v(z)
\bigr\|
=0.
\end{equation}
Fix
$\mc{K}\in Comp(\mc{H})$
such that $\mc{K}\subseteq\mc{E}$,
$f\in\mc{K}$
and $\ep>0$,
thus by 
\eqref{20150304}
and
\eqref{20160304}
there exists $U$
neighbourhood of $x_{\infty}$
such that
\begin{equation}
\label{20170304}
\sup_{s\in K}
\sup_{z\in U}
\bigl\|
\bigl((\ov{\mc{U}}(s)-\ov{F}(s))f\bigr)(z)
\bigr\|
\leq\ep/2.
\end{equation}
Define
\begin{equation*}
\begin{cases}
M
\coloneqq
\max\{
\sup_{s\in Y}
\|\ov{F}(s)\|_{B(\mc{H})},
\,
\sup_{s\in Y}
\|\ov{\mc{U}}(s)\|_{B(\mc{H})}
\}
\\
\eta\coloneqq\ep/4aM
\\
\mf{U}(f)
\coloneqq
\{g\in\mc{K}\,\vert\,\|f-g\|_{\mc{H}}<\eta\}.
\end{cases}
\end{equation*}
Thus
for all $g\in\mf{U}(f)$
\begin{alignat*}{1}
\sup_{z\in U}
\sup_{s\in K}
\bigl\|
\mc{U}(z)(s)g(z)
-
F(z)(s)g(z)
\bigr\|
&
=
\\
\sup_{s\in K}
\sup_{z\in U}
\bigl\|
\bigl((\ov{\mc{U}}(s)-\ov{F}(s))g\bigr)(z)
\bigr\|
&
\leq
\\
\sup_{s\in K}
\sup_{z\in U}
\bigl\|
\bigl((\ov{\mc{U}}(s)-\ov{F}(s))f\bigr)(z)
\bigr\|
+
\sup_{s\in K}
\sup_{z\in U}
\bigl\|
\ov{\mc{U}}(s)(g-f)(z)
\bigr\|
+
\sup_{s\in K}
\sup_{z\in U}
\bigl\|
F(s)(g-f)(z)
\bigr\|
&
\leq
\\
\ep/2
+
a
\sup_{s\in K}
\bigl\|
\ov{\mc{U}}(s)(g-f)
\bigr\|_{\mc{H}}
+
a
\sup_{s\in K}
\bigl\|
F(s)(g-f)
\bigr\|_{\mc{H}}
&
\leq
\\
\ep/2
+
2aM
\bigl\|g-f\bigr\|_{\mc{H}}
&
<
\ep.
\end{alignat*}
Therefore $(B)$
follows by considering
that 
$\{\mf{U}(f)\,\vert\, f\in\mc{K}\}$
is an open cover of 
the compact
$\mc{K}$.
Indeed let for example
$\{\mf{U}(f_{i})\,\vert\, i=1,...,n\}$
a finite subcover of $\mc{K}$
thus
by setting
$W\coloneqq\bigcap_{i=1}^{n}U_{n}$
with obvious meaning of $U_{i}$,
we have
\begin{equation*}
\sup_{z\in W}
\sup_{s\in K}
\sup_{g\in\mc{K}}
\bigl\|
\mc{U}(z)(s)g(z)
-
F(z)(s)g(z)
\bigr\|
<
\ep.
\end{equation*}
\end{proof}
\begin{remark}
We can set 
$\mc{H}=\prod_{x\in X}^{b}\mf{E}_{x}$
with the usual 
norm
$\|\cdot\|_{\sup}$.
\end{remark}
\begin{theorem}
[$\mc{K}-$Uniform Convergence]
\label{10581004}
Let
$\mf{V}
\coloneqq
\lr{\lr{\mf{E}}{\tau}}
{\pi,X,\|\cdot\|}$
be a 
Banach bundle.
Let
$x_{\infty}\in X$
and
$
\mc{U}_{0}
\in\prod_{x\in X_{0}}
\cc{}{\R^{+},B_{s}(\mf{E}_{x})}$
be
such that
$\mc{U}_{0}(x)$
is 
a
$(C_{0})-$semigroup
of contractions
(respectively of
isometries)
on
$\mf{E}_{x}$
for all $x\in X_{0}$.
Assume that
\begin{enumerate}
\item
$D(T_{x_{\infty}})$
is dense in $\mf{E}_{x_{\infty}}$;
\item
$\mf{V}$ and $\mf{W}$
satisfy
\eqref{15482601};
\item
$\exists\lambda_{0}>0$
(respectively
$\exists\lambda_{0}>0,
\lambda_{1}<0$)
such that
the range
$\mc{R}(\lambda_{0}-T_{x_{\infty}})$
is dense in $\mf{E}_{x_{\infty}}$,
(respectively
the ranges
$\mc{R}(\lambda_{0}-T_{x_{\infty}})$
and
$\mc{R}(\lambda_{1}-T_{x_{\infty}})$
are dense in $\mf{E}_{x_{\infty}}$);
\item
$
\ms{U}_{\|\cdot\|_{B(\mf{E}_{z})}}
(\mc{L}_{S_{z}}(\mf{E}_{z}))
\subseteq
\mf{M}_{z}$
(respectively
$\ms{U}_{is}
(\mc{L}_{S_{z}}(\mf{E}_{z}))
\subseteq
\mf{M}_{z}$)
for all $z\in X$;
\item
$\mc{E}\subseteq\mc{H}
\subseteq\prod_{x\in X}^{b}\mf{E}_{x}$
\item 
$X$ is completely regular
and
the filter of neighbourhoods of $x_{\infty}$ 
admits a countable basis;
\item
$\exists\,\mc{F}\subset
\bigcap_{\lambda>0}
\mf{L}_{1}(\R^{+},\mf{G},\mu_{\lambda})$
such that
$\Psi_{\R^{+}}^{\mc{H}}(\mc{F})=\Gamma(\rho)$;
\item
$(\exists\,F\in\Gamma(\rho))
(F(x_{\infty})=\mc{U}(x_{\infty}))$
such that
\begin{enumerate}
\item
$\lr{\int\ov{F}(s)\,d\mu_{\lambda}(s)}
{\mc{E}}\subseteq\Gamma(\pi)$,
for all $\lambda>0$;
\item
$(\forall v\in\mc{E})
(\exists\,\phi\in\Phi)$
s.t.
$\phi_{1}(x_{\infty})
=
v(x_{\infty})$
and
$(\forall\{z_{n}\}_{n\in\N}\subset X
\,\vert\,
\lim_{n\in\N}z_{n}=x_{\infty})$
we have
that
$\{
\mc{U}(z_{n})(\cdot)\phi_{1}(z_{n})
-
F(z_{n})(\cdot)v(z_{n})
\}_{n\in\N}$
is a 
bounded equicontinuous
sequence.
\end{enumerate}
\end{enumerate}
Then
\begin{equation}
\label{19351004}
\mc{U}
\in\Gamma^{x_{\infty}}(\rho).
\end{equation}
Furthermore
if
\begin{enumerate}
\item
$(\exists\,a>0)(\forall f\in\mc{H})
(\|f\|_{\sup}\leq a\|f\|_{\mc{H}})$,
\item
$\{F(\cdot)(s)\,\vert\, s\in\R^{+}\}
\subseteq\A$
and
$\{\mc{U}(\cdot)(s)\,\vert\, s\in\R^{+}\}
\subseteq\A$;
\item
$\{\ov{F}(s)\,\vert\, s\in\R^{+}\}$
and
$\{\ov{\mc{U}}(s)\,\vert\, s\in\R^{+}\}$
are
equicontinuous
as subsets of 
$\mc{L}(\lr{\mc{H}}{\|\cdot\|_{\mc{H}}})$.
\end{enumerate}
Then
for all $\mc{K}\in Comp(\mc{H})$
such that $\mc{K}\subseteq\mc{E}$
and for all $K\in Comp(\R^{+})$
\begin{equation}
\label{19371004}
\lim_{z\to x_{\infty}}
\sup_{s\in K}
\sup_{v\in\mc{K}}
\bigl\|
\mc{U}(z)(s)v(z)
-
F(z)(s)v(z)
\bigr\|
=0.
\end{equation}
Here
$D(T_{x_{\infty}})$
is
defined as
in
Notation
\ref{15411512b}
with 
$\mc{T}_{0}$
and 
$\Phi$
given in
\eqref{15482601}.
While
$\mc{U}
\in\prod_{x\in X}
\mf{M}_{x}$
such that
$\mc{U}\up X_{0}\coloneqq
\mc{U}_{0}$
and
$\mc{U}(x_{\infty})$
is the semigroup
on $\mf{E}_{x_{\infty}}$
generated by 
$T_{x_{\infty}}$
operator defined
in \eqref{172212312b}.
Moreover
$\|f\|_{\sup}\coloneqq
\sup_{x\in X}\|f(x)\|_{x}$,
while
$\ov{\mc{U}}\coloneqq\Lambda_{\A}^{Y}(\mc{U})$
and
$\ov{F}\coloneqq\Lambda_{\A}^{Y}(F)$.
\end{theorem}
\begin{proof}
By hyp. $(7)$ and
st. $(1)$ of
Th. \ref{15251401},
\eqref{18470109}
follows.
Moreover
\eqref{21241004}
follows
by 
hyp. $(5)$, 
and
Rm. \ref{21211004}.
Hence
by
hyps. $(7-8a)$,
and
Cor. \ref{18491004}
follows
the
$\ms{LD}(\{F\},\mc{E})$.
Then
\eqref{19351004}
follows by
Th. \ref{17301812b}.
\eqref{19371004}
follows
by
\eqref{19351004}
and
Lm. \ref{19260304}.
\end{proof}
\begin{remark}
\label{23461004}
By 
st.$(2)$ of
Pr. \ref{15111401}
hyp. $(7)$
is equivalent to the following one
$\Lambda_{\A}^{\R^{+}}(\Gamma(\rho))
\subseteq
\bigcap_{\lambda>0}
\mf{L}_{1}(\R^{+},\mf{G},\mu_{\lambda})$.
In any case the form in hyp. $(7)$ has the
advantage to be considered as a tool for
constructing $\Gamma(\rho)$.
Finally note that
$$
\lr{\mf{G}}{\mc{E}}\subseteq\Gamma(\pi)
\Rightarrow
(8a).
$$
\end{remark}
\subsection{
$\lr{\mc{H}}{\mf{T}}$
as 
Direct Integral
of a 
Continuous Field
of
left-Hilbert 
and associated
left-von Neumann 
Algebras.}
Assume that
$\mf{V}
=
\lr{\lr{\mf{E}}{\tau}}{\pi,X,\n}$
is a
continuous field
of
left-Hilbert algebras
on $X$.
Let
$\mc{H}$
be the
direct integral
of
$\mf{V}$
with respect to some 
finite
Radon
measure
on $X$
and
$\B\subset\mc{H}$
a linear space,
set
$$
\A(\B)
\coloneqq
\bigl\{
X\ni x\mapsto L_{a(x)}
\,\vert\,
a\in\B
\bigr\},
$$
where
$L_{a_{x}}\in B(\mf{E}_{x})$
for any $a_{x}\in\mf{E}_{x}$,
is the left multiplication
on the left-Hilbert algebra
$\mf{E}_{x}$.
Then
$\mc{H}$ and 
$\A(\B)$
satisfies the requirements in
Def. \ref{12042101},
moreover
\begin{equation}
\label{14502904}
G(\B)
\coloneqq
\theta(\A(\B))\up\mc{H}
=
L_{\B}
\end{equation}
where
$L_{a}\in B(\mc{H})$
for any $a\in\mc{H}$,
is the left multiplication
on the left-Hilbert algebra
$\mc{H}$.
If every $\mf{E}_{x}$
is unital
then
$\mc{H}$
is unital,
thus
$L_{(\cdot)}$
is an 
injective
(isometric)
map
of $\mc{H}$
into
$B(\mc{H})$.
Therefore
under this additional requirement
we can take the 
following
identification
$$
G(\B)\simeq\B
\text{ as linear spaces}.
$$
Let
$\ms{H}\coloneqq
\{\ms{H}^{i}\in\prod_{x\in X}\mf{E}_{x}
\}_{i=0}^{2}$
such that
$\ms{H}_{x}^{0}$
is a left Hilbert subalgebra of $\mf{E}_{x}$,
while
$\ms{H}_{x}^{k}$
is a linear subspace of $\ms{H}_{x}^{0}$,
for all $k=1,2$ and $x\in X$.
Set
\begin{equation}
\label{14582904}
\begin{cases}
\Gamma(\pi,\ms{H})
\coloneqq
\bigl\{
\sigma\in\mc{H}
\,\vert\,
(\forall x\in X)
(\sigma(x)\in\ms{H}_{x}^{0})
\bigr\}
\\
\mc{D}_{\ms{H}}
\coloneqq
\bigl\{
\sigma\in\mc{H}
\,\vert\,
(\forall x\in X)
(\sigma(x)\in\ms{H}_{x}^{1})
\bigr\}
\\
\B_{\ms{H}}
\coloneqq
\bigl\{
\sigma\in\mc{H}
\,\vert\,
(\forall x\in X)
(\sigma(x)\in\ms{H}_{x}^{2})
\bigr\}.
\end{cases}
\end{equation}
Thus
$\Gamma(\pi,\ms{H})$
is a left Hilbert subalgebra of $\mc{H}$
and
$\B_{\ms{H}}, \mc{D}_{\ms{H}}$
are linear subspaces of 
$\Gamma(\pi,\ms{H})$,
so
\begin{equation}
\label{14512904}
L_{\B_{\ms{H}}}(\mc{D}_{\ms{H}})
\subseteq
\Gamma(\pi,\ms{H}).
\end{equation}
By \eqref{14512904}
and
\eqref{14502904}
follows 
that for all
$\sigma\in\B_{\ms{H}}$,
$\eta\in\mc{D}_{\ms{H}}$
and $y\in X$
\begin{equation}
\label{15172904}
\begin{cases}
\lr{G(\B_{\ms{H}})}{\mc{D}_{\ms{H}}}
\subseteq
\Gamma(\pi,\ms{H}),
\\
\lr{\theta\left(
x\mapsto L_{\sigma(x)}\right)}
{\eta}(y)
=
\sigma(y)\eta(y).
\end{cases}
\end{equation}
Let us consider now the 
continuous field
of left-von Neumann algebras
associated with the fixed 
field of Hilbert algebras,
and by abusing of language,
let us denote it with the
same symbol
$\mf{V}
=
\lr{\lr{\mf{E}}{\tau}}{\pi,X,\n}$,
as well as
$\mc{H}$
will denote the associated 
direct integral with respect to some
finite Radon measure on $X$.
Let $\Delta_{x}$ be the modular operator
associated with the Hilbert algebra
$\mf{E}_{x}$ and $\sigma_{x}$ the corresponding
modular group.
Thus we can set
$$
\begin{cases}
\A_{\Delta}
\coloneqq\{
S_{t}:
X\ni x\mapsto
\sigma_{x}(t)\in 
Aut(\mf{E}_{x})
\,\vert\, t\in\R
\}
\\
G_{\Delta}\coloneqq
\theta(\A_{\Delta})\up\mc{H}
\\
\Sigma_{t}
\coloneqq
\theta(S_{t})\up\mc{H},\,
t\in\R.
\end{cases}
$$
Note that for every $t\in\R$,
$v\in\mc{H}$
and $x\in X$
$$
\Sigma_{t}(v)(x)
=\sigma_{x}(t)(v(x)).
$$ 
Now
if we set
$$
\Gamma(\pi)
\coloneqq\mc{H}
$$
for any linear subspace $\mc{D}$ of 
$\mc{H}$
we have 
$$
\lr{G_{\Delta}}{\mc{D}}
\subseteq
\Gamma(\pi).
$$
Finally
note that
to $\A_{\Delta}$ we can associate 
the following
map
$$
\ov{\Sigma}:\R^{+}\ni t
\mapsto
\Sigma_{t}
\in
G_{\Delta},
$$
for which we have
for 
all $x\in X$
$$
\Psi_{\R}^{\mc{H}}(\ov{\Sigma})
(x)
=
\sigma_{x}.
$$
In the previous
example we consider 
the extreme case in which 
$\Gamma(\pi)=\mc{H}$.
In order to have a model where
$\Gamma(\pi)\subset\mc{H}$
we have to get a more detailed structure,
namely the 
half-side modular inclusion.
So for any $x\in X$ let 
$\lr{\n_{x}\subset\mf{E}_{x}}{\Omega_{x}}$
be a $hsmi^{+}$ and 
$V_{x}$ the Wiesbrock
one-parameter semigroup of unitarities
associated with it
so $V_{x}\in Hstr(\mf{E}_{x})^{+}$
such that
$\n_{x}=Ad(V_{x}(1))\mf{E}_{x}$.
Therefore what we are interested in 
is that for all $t\in\R^{+}$
\begin{equation}
\label{20290405}
\begin{cases}
Ad(V_{x}(t))(\mf{E}_{x})
\subseteq\mf{E}_{x},
\\
Ad(V_{x}(t))(\n_{x})
\subseteq\n_{x}.
\end{cases}
\end{equation}
By using the first inclusion
in \eqref{20290405}
we can set
$$
\begin{cases}
\A_{V}
\coloneqq\{
V_{t}:
X\ni x\mapsto
Ad(V_{x}(t))
\up
\mf{E}_{x}
\in 
Aut(\mf{E}_{x})
\,\vert\, t\in\R
\}
\\
G_{V}\coloneqq
\theta(\A_{V})\up\mc{H}
\\
\ov{\mc{V}}_{t}
\coloneqq
\theta(V_{t})\up\mc{H},\,
t\in\R.
\end{cases}
$$
Hence for all $x\in X$
and $t\in\R$
$$
\begin{cases}
\ov{\mc{V}}_{t}(v)(x)
=
Ad(V_{x}(t))v(x)
\\
\Psi_{\R}^{\mc{H}}(\ov{\mc{V}})
(x)(t)
=
Ad(V_{x}(t))
\end{cases}
$$
Therefore
if we set
$\mc{D}$
and
$\Gamma(\pi)$
such that
$$
\mc{D}
\subseteq
\Gamma(\pi)
\coloneqq
\int^{\oplus}
\n_{x}\,
d\mu(x)
\subset
\mc{H}
$$
then
by
using the second
inclusion
in
\eqref{20290405}
we have
$$
\lr{G_{V}}{\mc{D}}
\subseteq
\Gamma(\pi).
$$
For
any 
semi-finite
von Neumann algebra $\n$
and any $\phi\in\ms{N}_{\n}$ faithful
we have that
the Tomita-Takesaki modular group
$\sigma_{\n}^{\phi}$ is inner
(see \cite[Thm. 8.3.14]{tak2})
i.e. it
is implemented by a strongly continuous
group morphism
$V:\R\to U(\n)$,
where
$U(\n)\coloneqq\{U\in\n\,\vert\,U^{-1}=U^{*}\}$,
so in particular
\begin{equation}
\label{23452904a}
V(\R)\subset\n.
\end{equation}
Now let 
$\lr{H_{\phi},\pi_{\phi}}{\Omega_{\phi}}$
be a cyclic representation associated with $\phi$
and $\n_{\phi}\coloneqq\pi(\n_{\phi})$ which is
a von Neumann algebra $\phi$ being normal,
then
by
\eqref{23452904a}
immediatedly we have
\begin{equation}
\label{23452904b}
\pi_{\phi}(V(\R))\subset\n_{\phi}.
\end{equation}
By
the invariance
$\phi=\phi\circ\sigma_{\n}^{\phi}$,
and the cited unitary
implementation
we obtain
that there exists
$W_{\phi}$ unitary action on $H_{\phi}$
such that
\begin{equation}
\label{23372904}
\begin{cases}
Ad(W_{\phi}(t))\circ\pi_{\phi}
=
Ad(\pi_{\phi}(V(t)))\circ\pi_{\phi},
\\
W_{\phi}(t)=\Delta_{\phi}^{it},
\end{cases}
\end{equation}
where
the second sentence comes by
\cite[Thm. 8.1.2]{tak2},
with
$\Delta_{\phi}$
the modular operator associated with 
$\lr{\n_{\phi}}{\Omega_{\phi}}$.

\chapter{Sections of Projectors}
\section*{Introduction}
In this final part of the work we accomplish in 
\textbf{Thm. \ref{13020103}} the claim of extending 
the classical stability problem to the framework of bundles of $\Omega-$spaces
and consequently to obtain stability results for 
operators acting in different Banach spaces.
Let us describe the principal steps required for this result.
\par
$\left(\Theta,\mc{E}\right)-$structures are central in constructing 
in 
Thm. \ref{17301812b} 
the section of $C_{0}-$semigroups $\mc{U}$ continuous at $x_{\infty}$.
However the presence in their definition 
of the uniform convergence over compact subsets of a topological space $Y$
rather than the \emph{pointwise} convergence,
drastically restricts in general the fulfillment of the property \eqref{14451103}
(with $\Gamma(\xi)$ replaced by $\Gamma(\rho)$)
characterizing invariant structures.
Possible exceptions are those where the base space $X$ is compact and under suitable hypothesis
the above property is used to determine $\Gamma(\rho)$, see Rmk. \ref{17150312}.
\par
Now first of all the property \eqref{14451103} is basic to establish in \textbf{Cor. \ref{21152602}} 
the essential step \eqref{17401003} toward the main result \eqref{17321002bis}.
Here $\mc{P}$ is a suitable section of spectral projectors 
of the infinitesimal generators of the semigroups $\{\mc{U}(x)\}_{x\in X}$.
For instance for obtaining \eqref{18132502} we apply Lemma \ref{14452602}. 
Secondly in order to prove \eqref{17401003} 
we need the concept of $\mu-$relatedness provided in \textbf{Def. \ref{15492502}}
requiring $\mu-$\emph{integrable Hlcs valued maps}.
Indeed see the technical \textbf{Cor. \ref{15262502}} resulting 
by the bundle type generalization of the Lebesgue Theorem we establish in \textbf{Thm. \ref{15101701}}
and by Lemma \ref{12151902}, remarkable results by themself.
Finally for defining $\mc{P}(x)$ we apply to $\mc{U}(x)$ 
the well-known integral formula \eqref{16171102} for any $x\in X$.
\par
Therefore it appears natural in the present context to replace
$\left(\Theta,\mc{E}\right)-$structures based on
function spaces of continuous maps provided with the topology of compact convergence,
with those based on function spaces, provided with the topology of pointwise convergence,
of Hausdorff locally convex space valued integrable maps defined on a locally compact space 
provided with a Radon measure.
\par
To this end we introduce for any Radon measure $\mu$ on a locally compact space $Y$
the
concept
of
$\left(\Theta,\mc{E},\mu\right)-$structure \textbf{Def. \ref{17161902}}.
Roughly a $\left(\Theta,\mc{E},\mu\right)-$structure $\lr{\mf{V},\mf{Q}}{X,Y}$
where $\mf{Q}=\lr{\lr{\mf{H}}{\gamma}}{\xi,X,\mf{Y}}$
and $\mf{V}=\lr{\lr{\mf{E}}{\tau}}{\pi,X,\n}$,
is defined in the same way as a $\left(\Theta,\mc{E}\right)-$structure except that
\begin{equation}
\label{20081003}
\begin{cases}
\mf{H}_{x}
\subseteq
\mf{L}_{1}
\left(Y,\mc{L}_{S_{x}}
(\mf{E}_{x}),\mu\right),
\\
\text{$\mf{Y}_{x}$ induces the topology on $\mf{H}_{x}$ of pointwise convergence on $Y$.}
\end{cases}
\end{equation}
Here for any $x\in X$ we recall that $\mc{L}_{S_{x}}(\mf{E}_{x})$ is the Hlcs of 
continuous linear maps on $\mf{E}_{x}$
provided with the topology of uniform convergence over the sets belonging to $S_{x}$.
$\lr{\mf{V},\mf{Q}}{X,Y}$ is \emph{invariant} if in addition 
\begin{equation}
\label{14451103}
\left\{
F\in\prod_{z\in X}^{b}\mf{H}_{z}
\,\vert\,
(\forall t\in Y)
(F_{t}
\bullet
\mc{E}(\Theta)
\subseteq
\Gamma(\pi))
\right\}
=
\Gamma(\xi).
\end{equation}
\par
The main reason behind the intoduction of this concept is represented by the following result.
Let $\mf{V}$ be a Banach bundle, 
$T$ be a suitable section of closed densely defined linear operators 
satisfying the property of separation of the spectrum,
$\Gamma$ be a curve associated with $T$ and $(K,A,\phi)$ be a triplet associated with $\Gamma$, 
a straight extension to the bundle case of the separation of the spectrum introduced by Kato
(Def. \ref{13361311biss}).
If $\lr{\mf{V},\mf{Q}}{X,\R^{+}}$ is an invariant $\left(\Theta,\mc{E},\mf{q}\right)-$structure
for all $\mf{q}\in\{\nu^{\phi},\eta_{s}^{\phi}\,\vert\, s\in K\}$
(Def. \ref{16420402})
and $\lr{\mf{V}}{\mf{Z},\ms{H}}$ is $\mf{q}-$related 
such that $\ms{H}_{z}=\mf{H}_{z}$ for all $z\in X$,
then under additional hypothesis Cor. \ref{21152602} 
\begin{equation}
\label{17401003}
\mc{W}_{T}
\in
\Gamma^{x_{\infty}}(\xi)
\Rightarrow
\mc{P}
\bullet
\Gamma_{\mc{E}(\Theta)}^{x_{\infty}}(\pi)
\subseteq
\Gamma^{x_{\infty}}(\pi).
\end{equation}
As we describe below, 
by letting $\mf{n}$ be the Lebesgue measure on $\R^{+}$
we shall apply this result to a $\mf{n}-$related set and to
the $\left(\Theta,\mc{E},\mf{n}\right)-$structure underlying the
$\left(\Theta,\mc{E}\right)-$structure used in Thm. \ref{17301812b}.
\par
Now central in proving \eqref{17401003} 
is the fact that the global relation \eqref{14451103}
implies the pointwise one.
More exactly for any invariant 
$\left(\Theta,\mc{E},\mu\right)-$structure
$\lr{\mf{V},\mf{D}}{X,Y}$ with a suitable $\Theta$,
by denoting
$\mf{D}=\lr{\lr{\mf{B}}{\gamma_{3}}}{\eta,X,\mf{L}}$,
we have Lemma \ref{14452602}
\begin{equation}
\label{17481003loc}
\left\{
H\in
\left[
\left(
\prod_{z\in X}^{b}
\mf{B}_{z}
\right)_{\diamond}^{x_{\infty}}
\right]_{peq}
\,\vert\,
(\forall t\in Y)
(H_{t}
\bullet
\mc{E}(\Theta)
\subseteq
\Gamma^{x_{\infty}}(\pi))
\right\}
\subseteq
\Gamma_{\diamond}^{x_{\infty}}(\eta).
\end{equation}
To derive \eqref{17401003} we use this inclusion for the case $Y=\R^{+}$
and multiple times Cor. \ref{15262502}.
\eqref{17481003loc} can be extended to invariant $\left(\Theta,\mc{E}\right)-$structures
whenever $Comp(Y)=\mc{P}_{\omega}(Y)$,
we shall use this remark in the case $Y=\{pt\}$.
\par
Now if $\mc{E}(\Theta)\subseteq\Gamma(\pi)$ and if we show that
\begin{equation}
\label{18271003}
\mc{P}
\bullet
\Gamma_{\mc{E}(\Theta)}^{x_{\infty}}(\pi)
\subseteq
\Gamma^{x_{\infty}}(\pi),
\end{equation}
then 
\begin{equation}
\label{17321002bis}
\mc{P}
\in
\Gamma^{x_{\infty}}(\eta),
\end{equation}
follows by \eqref{17481003loc} for the special case $Y=\{pt\}$.
It is worthwhile remarking that \eqref{18271003} represents a satisfactory result
for all practical purposes concerning the stability of $\mc{P}$.
However in order to interpret $\mc{P}$ satisfying \eqref{18271003}
as a bounded section continuous at $x_{\infty}$ we need to employ \eqref{17481003loc}. 
\par
Next let $Y=\R^{+}$.
To prove \eqref{18271003} we apply \eqref{17401003}
to the section of contractions $\mc{U}$ 
continuous at $x_{\infty}$ obtained in Thm. \ref{17301812b}
and to the $\left(\Theta,\mc{E},\mf{n}\right)-$structure
underlying the $\left(\Theta,\mc{E}\right)-$structure used in 
Thm. \ref{17301812b}.
More exactly given a suitable $\left(\Theta,\mc{E}\right)-$structure 
$\lr{\mf{V},\mf{W}}{X,Y}$
which in general is not a $\left(\Theta,\mc{E},\mu\right)-$ structure
and by letting $\mf{W}=\lr{\lr{\mf{M}}{\gamma}}{\rho,X,\mf{R}}$,
Thm. \ref{17301812b}
states the existence of a section $T$ of closed operators
such that 
$
\mc{U}
=
\mc{W}_{T}
\in
\Gamma^{x_{\infty}}(\rho)$.
So under the additional hypothesis that there exists an $F\in\Gamma(\rho)$ 
such that $F(x_{\infty})=\mc{U}(x_{\infty})$ we obtain
\begin{equation}
\label{05311507}
\mc{U}
\in
\Gamma_{\diamond}^{x_{\infty}}(\rho).
\end{equation}
In order to apply \eqref{17401003} we need a procedure to extract 
from the initial $\left(\Theta,\mc{E}\right)-$structure
a structure $\lr{\mf{V},\mf{Q}}{X,Y}$
which 
is a $\left(\Theta,\mc{E},\mf{q}\right)-$structure
for all $\mf{q}\in\{\nu^{\phi},\eta_{s}^{\phi}\,\vert\, s\in K\}$
and such that 
\begin{equation}
\label{05301121}
\Gamma_{\diamond}^{x_{\infty}}(\rho)\subseteq\Gamma^{x_{\infty}}(\xi).
\end{equation}
This is performed by applying a general construction called
the $\left(\Theta,\mc{E},\mu\right)-$structure
$\lr{\mf{V},\mf{V}(\ms{M}^{\mu},\Gamma(\rho))}{X,Y}$
underlying $\lr{\mf{V},\mf{W}}{X,Y}$
and defined in \textbf{Def. \ref{18072802}}.
\par
In view of the property \eqref{05301121}
we have to maintain the 
vicinity of the initial and the underlying
structure.
This is performed by applying the by now 
usual general result \cite[Thm. 5.8]{gie}
for constructing bundles with a given 
subspace of bounded continuous sections.
Thus the choice we perform in Def. \ref{18072802}
for the stalks $\ms{M}_{x}^{\mu}$
is essentially the weakest one in order 
to satisfy \eqref{20081003} 
and to allow the space 
$\Gamma(\rho)$
of bounded continuous sections
of
$\mf{W}$
to be a 
subspace
of
the space 
$\Gamma
(\pi_{\ms{M}^{\mu}})$
of bounded continuous sections of the underlying 
bundle $\mf{V}(\ms{M}^{\mu},\Gamma(\rho))$.
\par
Prp. \ref{23332802} shows that our choice is the right one indeed
\begin{equation}
\label{17501003}
\Gamma_{\diamond}^{x_{\infty}}(\rho)
\subseteq
\Gamma^{x_{\infty}}
(\pi_{\ms{M}^{\mu}}).
\end{equation}
It remains only to select the right measure $\mu$ in order to satisfy the hypothesis of Cor. \ref{21152602} 
in particular such that the $\left(\Theta,\mc{E},\mu\right)-$structure
underlying $\lr{\mf{V},\mf{W}}{X,Y}$
is a $\left(\Theta,\mc{E},\mf{q}\right)-$structure
for all $\mf{q}\in\{\nu^{\phi},\eta_{s}^{\phi}\,\vert\, s\in K\}$.
To this end let us say that $(\mf{W},\mf{Z})$ satisfies the $\mf{n}-$hypothesis,
if 
the $\left(\Theta,\mc{E},\mf{n}\right)-$structure
underlying $\lr{\mf{V},\mf{W}}{X,Y}$
is 
invariant,
$\lr{\mf{V}}{\mf{Z},\ms{M}^{\mf{n}}}$
is
$\mf{n}-$related
and
$\mf{L}^{\infty}(Y,\mf{n})\blacktriangleright\Gamma(\zeta)\subseteq\Gamma(\zeta)$
(where $\Gamma(\zeta)$ is the space of bounded continuous sections of $\mf{Z}$
and $\blacktriangleright$ is defined in Def. \ref{15361702}).
Now it is possible to show that if $(\mf{W},\mf{Z})$ satisfies the $\mf{n}-$hypothesis,
then $\lr{\mf{V},\mf{V}(\ms{M}^{\mf{n}},\Gamma(\rho))}{X,Y}$
is an invariant $\left(\Theta,\mc{E},\mf{q}\right)-$structure
and
$\lr{\mf{V}}{\mf{Z},\ms{M}^{\mf{n}}}$
is
$\mf{q}-$related
for all $\mf{q}\in\{\nu^{\phi},\eta_{s}^{\phi}\,\vert\, s\in K\}$.
In other words the hypothesis in Cor. \ref{21152602} 
for obtaining \eqref{17401003} holds true
with the position $\mf{Q}=\mf{V}(\ms{M}^{\mf{n}},\Gamma(\rho))$.
\par
Finally provided that $(\mf{W},\mf{Z})$ satisfies the $\mf{n}-$hypothesis
we conclude that 
\eqref{18271003} and consequently \eqref{17321002bis},
follow
by
\eqref{05311507},
\eqref{17501003} with the position $\mu=\mf{n}$
and
\eqref{17401003}.
\par
Except when explicitly stated, we assume all the notations set in 
Ch. \ref{05250734} and Ch. \ref{05251221},
in particular all the vector spaces are over $\C$.
\section
{$\lr{\nu,\eta}{E,Z,T}$
invariant set
with respect to 
$\mc{F}$}
\label{15491702A}
In the present section \ref{15491702A} 
let us fix a consistent class of data $\mf{O}$
(Def. \ref{12042101}) 
and let $\mf{G}$ denote the locally convex 
space relative to $\mf{O}$
(Def. \ref{10221801}).
\begin{definition}
\label{19560202}
Let 
$Z,T$ be two
locally compact spaces,
$E\in Hlcs$, 
$\nu\in Radon(Z)$
and
$\eta\in Radon(T)^{Z}$.
Set
$$
\mf{L}_{(1,1)}
(T,E,\eta,\nu)
\coloneqq
\left\{
\ov{F}\in\bigcap_{\lambda\in Z}
\mf{L}_{1}(T,E,\eta_{\lambda})
\,\vert\,
\left(
Z\ni\lambda
\mapsto
\int\ov{F}(s)
d\eta_{\lambda}(s)
\in E
\right)
\in
\mf{L}_{1}(Z,E,\nu)
\right\}
$$
\end{definition}
\begin{corollary}
\label{19510202}
Let $Z$ be a locally compact space, $\nu\in Radon(Z)$,
$\eta\in Radon(Y)^{Z}$, 
$\mf{D}\in\prod_{x\in X}\mc{P}(\mf{E}_{x})$,
$D=\prod_{x\in X}\mf{D}_{x}$
and assume $(A)$ of 
Lemma \ref{11011501}.
Thus
$(\forall\ov{F}\in
\mf{L}_{(1,1)}
(Y,\mf{G},\eta,\nu))
(\forall x\in X)
(\forall v\in D)$
$$
\Pr_{x}
\circ
\left[
\int
\left(
\int\ov{F}(s)\,
d\eta_{\lambda}(s)
\right)\,
d\nu(\lambda)
\right]
(v)
=
\left[
\int
\left(
\int
\Pr_{x}(\Psi(\ov{F}))(s)
v(x)\,
d\eta_{\lambda}(s)
\right)\,
d\nu(\lambda)
\right].
$$
\end{corollary}
\begin{proof}
Let
$\ov{F}\in
\mf{L}_{(1,1)}
(Y,\mf{G},\eta,\nu)$,
$x\in X$
and
$v\in D$.
By Thm. 
\ref{16291501}
$$
\Pr_{x}
\circ
\left[
\int
\left(
\int\ov{F}(s)\,
d\eta_{\lambda}(s)
\right)\,
d\nu(\lambda)
\right]
(v)
=
\int
\Pr_{x}\circ\Psi
\left(
\int\ov{F}(s)\,d\eta_{(\cdot)}(s)
\right)
(\lambda)
(v(x))
\,d\nu(\lambda).
$$
Moreover $\forall\lambda\in Z$
\begin{alignat*}{2}
\Pr_{x}\circ\Psi
\left(
\int\ov{F}(s)\,d\eta_{(\cdot)}(s)
\right)
(\lambda)
(v(x))
&=
\Pr_{x}
\circ
\left(
\int
\ov{F}(s)\,
d\eta_{\lambda}(s)
\right)
\circ
\imath_{x}
(v(x))
\\
&=
\int
\Pr_{x}(\Psi(\ov{F}))(s)
\circ
\Pr_{x}\circ\imath_{x}
(v(x))
\,d\eta_{\lambda}(s)
\\
&=
\int
\Pr_{x}(\Psi(\ov{F}))(s)
v(x)
\,d\eta_{\lambda}(s),
\end{alignat*}
where in the first equality we used 
Prop. \ref{15111401},
while in the second one 
Thm. \ref{16291501}.
Then the statement follows.
\end{proof}
\begin{definition}
\label{20030202}
$V$ is a 
\emph
{$\lr{\nu,\eta}{E,Z,T}$
invariant set
with respect to 
$\mc{F}$}
if
\begin{enumerate}
\item
$T,Z$ are two locally compact spaces;
\item
$E\in Hlcs$
\item
there exists $M\in Hlcs$ and $\mf{b}:E\times M\to M$ bilinear;
\item
$V\subseteq M$ linear subspace;
\item
$\nu\in Radon(Z)$
and
$\eta\in Radon(T)^{Z}$;
\item
$\mc{F}
\subseteq
\mf{L}_{(1,1)}
(T,E,\eta,\nu)$
\item
$\forall\ov{F}\in\mc{F}$
$$
\left[
\int
\left(
\int\ov{F}(s)\,
d\eta_{\lambda}(s)
\right)\,
d\nu(\lambda)
\right]
V
\subseteq
V,
$$
where we denote $\mf{b}(e,m)$ by $e m$,
for any $e\in E$ and $m\in M$.
\end{enumerate}
\end{definition}
\begin{proposition}
\label{18310302}
Let us assume the hypotheses of Cor. \ref{19510202} and 
$V$ be a
$\lr{\nu,\eta}{\mf{G},Z,Y}$
invariant set
with respect to 
$\mc{F}$
such that
$V\cap D\neq\emptyset$.
Then
$\forall v\in V\cap D$
and
$\forall\ov{F}\in\mc{F}$
$$
\left(
X\ni x
\mapsto
\left[
\int
\left(
\int
\Pr_{x}(\Psi(\ov{F}))(s)
v(x)\,
d\eta_{\lambda}(s)
\right)\,
d\nu(\lambda)
\right]
\in\mf{E}_{x}
\right)
\in
V.
$$
\end{proposition}
\begin{proof}
By Cor.
\ref{19510202}.
\end{proof}
\section
{
Construction of sets in
$
\Delta_{\Theta}\lr{\mf{V},\mf{D},\mf{W}}
{\mc{E},X,\R^{+}}
$
through
invariant sets
}
\label{15491702B}
In the present section \ref{15491702B}
let us fix an entire consistent class of data $\mf{O}$
(Def. \ref{12042101}) 
such that $\R^{+}$ is its locally compact space,
$\nu$ defined in Def. \ref{16420402} is its Radon measure,
the primary family $\ms{E}\coloneqq\{\mf{E}_{x}\}_{x\in X}$
underlying $\mf{O}$
is a family of Banach spaces
and such that for any $x\in X$ the corresponding 
element of the secondary family underlying $\mf{O}$
is the strong operator topology on $\mc{L}(\mf{E}_{x})$.
Let $\mf{G}$ denote the locally convex 
space relative to $\mf{O}$
(Def. \ref{10221801}).
For any Banach space $C$
let $Cld(C)$ denote,
the set of all
closed linear operators 
densely defined in $C$ and at values in $C$.
For any $T\in Cld(C)$ let $P(T)$ denote 
the resolvent set of $T$
and $\Sigma(T)$ be the spectrum of $T$.
Let $R(T;\cdot):P(T)\ni \zeta\mapsto (T-\zeta)^{-1}\in B(C)$ 
be the resolvent map of $T$.
In the next definition we adapt to our framework a definition 
provided in 
\cite[Ch. $9$, $\S1$, $n^{\circ}4$]{kato}.
\begin{definition}
\label{13411311biss}
Let $M>1$ and $\beta\in\R$.
Let $\mc{G}(M,\beta,\ms{E})$ be the set
of all 
$T\in\prod_{x\in X}Cld(\mf{E}_{x})$
such that
$]\beta,\infty[\subseteq P(-T(x))$ 
and
$(\forall\xi>\beta)
(\forall k\in\N)
(\forall x\in X)$
$$
\|(T(x)+\xi)^{-k}\|_{B(\mf{E}_{x})}
\leq
M
(\xi-\beta)^{-k}.
$$
Moreover
let
us
denote
by
$
\{e^{-tT(x)}\}_{t\in\R^{+}}
$
the strongly continuous
semigroup
generated
by 
$-T(x)$.
\end{definition}
In the following definition
we adapt to our framework the definition of separation
of the spectrum for a closed operator
provided in \cite[$n^{\circ}4$,$\S 6$, Ch. $3$]{kato}.
\begin{definition}
\label{13361311biss}
Let $M>1$ and $\beta\in\R$.
We say that
$T\in\mc{G}(M,\beta,\ms{E})$
satisfies
the
property 
of separation of the spectrum
if
$(\exists\,\Gamma)
(\forall x\in X)
(\exists\,\Sigma_{T(x)}'
\subseteq\Sigma(T(x)))
(\exists\,A_{T(x)}\in Op(\C))$
such that $\Gamma$ is a regular closed curve in $\C$,
$\Sigma_{T(x)}'$
is bounded
and
$$
\Sigma_{T(x)}'
\subset
A_{T(x)}
\subset
\mc{O}_{i}(\Gamma),
\,
\Sigma_{T(x)}''
\subset
\mc{O}_{e}(\Gamma).
$$
Here
$\mc{O}_{i}(\Gamma)$
is the interior of $\Gamma$,
namely
the compact set
of $\C$ whose frontier
is $\Gamma$,
$\mc{O}_{e}(\Gamma)
\coloneqq
\complement
\mc{O}_{i}(\Gamma)$
is the exterior of $\Gamma$,
finally
$\Sigma_{T(x)}''
\coloneqq
\Sigma(T(x))
\cap
\complement
\Sigma_{T(x)}'$.
We call any curve $\Gamma$ with the above property a curve associated with $T$,
while we call $(K,A,\phi)$ a triplet associated with $\Gamma$
iff 
$K\subset\R^{+}$
is compact,
$A$ is an open neighbourhood
of $K$
and
$\phi:A\to\C$
is such that
$\phi\in C_{1}(A,\R^{2})$\footnote{By identifying
$\C$
with 
$\R^{2}$,
so
$\phi$
is
derivable with
contiuous
derivative}
and
$\phi(K)=\Gamma$.
\par
Let
$T\in\mc{G}(M,\beta,\ms{E})$
satisfy
the
property 
of separation of the spectrum and $\Gamma$ a curve associated with $T$,
then
$\forall x\in X$
we set
\begin{equation}
\label{16171102}
P(x)
\coloneqq
-
\frac{1}{2\pi i}
\int_{\Gamma}
R(T(x);\zeta)\,
d\zeta
\in B(\mf{E}_{x}),
\end{equation}
where
the integration
is with respect
to the 
norm topology
on $B(\mf{E}_{x})$.
Moreover for any triplet $(K,A,\phi)$ 
associated with $\Gamma$ set
$R_{T}^{\phi}\in\prod_{x\in X}\mc{L}(\mf{E}_{x})^{\R^{+}}$
such that
$R_{T}^{\phi}(x)(s)\coloneqq R(T(x);\phi(s))$,
for all $x\in X$ and $s\in K$,
while 
$R_{T}^{\phi}(x)(s)\coloneqq\ze$,
if $s\in\R^{+}-K$.
\end{definition}
\begin{remark}
\label{15221311biss}
Let
$M>1$, $\beta\in\R$, $T\in\mc{G}(M,\beta,\ms{E})$
satisfy
the
property 
of separation of the spectrum
and $\Gamma$ a curve associated with $T$.
Then
for all
$x\in X$
by 
\cite[Th.
$6.17.$, Ch. $3$]{kato},
$P(x)\in\Pr(\mf{E}_{x})$
and
$\mf{E}_{x}=
M_{x}'
\oplus
M_{x}''$
direct sum of two
closed
subspaces of $\mf{E}_{x}$,
where
$M_{x}'=P(x)\mf{E}_{x}$
and
$M_{x}''
=
(\un_{x}-P(x))
\mf{E}_{x}$.
Moreover
$T(x)$
decomposes
according
the previous
decomposition,
namely
$T(x)\up M_{x}'
\in B(M_{x}')$
such that
$\Sigma(T_{x}\up M_{x}')
=\Sigma_{T_{x}}'$
and
$T_{x}\up M_{x}''$
is a
closed operator
in
$M_{x}''$
such that
$\Sigma(T_{x}\up M_{x}'')=\Sigma_{T_{x}}''$.
\end{remark}
\begin{definition}
\label{16420402}
Let
$K\subset\R^{+}$
be a compact set,
$A$ an open neighbourhood
of $K$
and
$\phi:A\to\C$
be
such that
$\phi\in C_{1}(A,\R^{2})$
and
$\phi(K)=\Gamma$.
For any  $s\in K$ define $\eta_{s}^{\phi}\in Radon(\R^{+})$
such that
$$
\eta_{s}^{\phi}:
\cc{cs}{\R^{+}}\ni f
\mapsto
\int_{\R^{+}}e^{\phi(s)t}f(t)\,dt.
$$
Moreover
let
$\nu^{\phi}\in Radon(\R^{+})$
be the $\ze-$extension
of 
$\nu_{0}^{\phi}\in Radon(K)$
such that
$$
\nu_{0}^{\phi}:\cc{cs}{K}
\ni g
\mapsto
\int_{K}
\frac{- g(s)}{2\pi i}
\frac{d\phi}{ds}(s)\,
ds.
$$
Finally
let
$M>1$,
$\beta\in\R$
and
$T\in\mc{G}(M,\beta,\ms{E})$,
then we set
$\mc{W}_{T}
\in
\prod_{x\in X}\ms{U}(B_{s}(\mf{E}_{x}))$
such that
$(\forall x\in X)(\forall t\in\R^{+})$
$$
\begin{cases}
\mc{W}_{T}(x)(t)
\coloneqq
e^{-T(x)t},\\
\ov{F}_{T}\coloneqq\Lambda(\mc{W}_{T}),
\end{cases}
$$
where $\Lambda$ has been defined 
in Def. \ref{14321401}.
\end{definition}
\begin{lemma}
\label{15251311biss}
Let
$M>1$,
$\beta\in\R$
and
$T\in\mc{G}(M,\beta,\ms{E})$
satisfy
the
property 
of separation of the spectrum.
Assume that
there exists a curve $\Gamma$ associated with $T$
such that
\begin{equation}
\label{19281311biss}
Re(\Gamma)
\subseteq
\R^{-}.
\end{equation}
Then for any triple $(K,A,\phi)$ associated with 
$\Gamma$
we have that
$\forall x\in X$
and
$\forall v_{x}\in\mf{E}_{x}$
\begin{equation}
\label{17551311biss}
P(x)v_{x}
=
-
\frac{1}{2\pi i}
\int_{K}
\frac{d\phi}{ds}(s)
R(T(x);\phi(s))v_{x}
\,ds,
\end{equation}
and
$\forall s\in K$,
\begin{equation}
\label{18150402}
R(T(x),\phi(s))v_{x}
=
\int_{0}^{\infty}
e^{\phi(s) t}
e^{-tT(x)}
v_{x}\,dt
=
\int_{\R^{+}}
\mc{W}_{T}(x)(t)v_{x}\,
d\eta_{s}^{\phi}(t).
\end{equation}
Here
the integration
is with respect
to the 
norm topology
on $\mf{E}_{x}$.
If
$\ov{F}_{T}\in
\mf{L}_{(1,1)}
(\R^{+},\mf{G},\eta^{\phi},\nu^{\phi})$
and
$V$ is a
$\lr{\nu^{\phi},\eta^{\phi}}{\mf{G},K,\R^{+}}$
invariant set
with respect to 
$\{\ov{F}_{T}\}$,
then
\begin{equation}
\label{15351311biss}
P\bullet V\subseteq V.
\end{equation}
\end{lemma}
\begin{proof}
By \eqref{16171102},
\cite[
$IV.35$
Th. $1$]{IntBourb}, 
and
by 
the 
norm
continuity
of the map
$
B(\mf{E}_{x})\ni A\mapsto Aw\in
\mf{E}_{x}$
for any $w\in\mf{E}_{x}$,
we have
\eqref{17551311biss}.
Moreover
by
\eqref{19281311biss}
we can
apply
\cite[eq.
$1.28$,
$n^{\circ}3$,
$\S 1$, Ch. $9$]{kato}
and 
\eqref{18150402}
follows
by
Def. \ref{16420402}.
Fix $v\in V$
so
$\forall x\in X$
\begin{alignat}{1}
\label{18590402}
P(x)v(x)
&=
-
\frac{1}{2\pi i}
\int_{K}
\frac{d\phi}{ds}(s)
R(T(x);\phi(s))v(x)
\,ds
\notag
\\
&=
-
\frac{1}{2\pi i}
\int_{K}
\frac{d\phi}{ds}(s)
\left(
\int_{\R^{+}}
\mc{W}_{T}(x)(t)v(x)\,
d\eta_{s}^{\phi}(t)
\right)
\,ds
\notag
\\
&=\int_{K}
\left(
\int_{\R^{+}}
\Pr_{x}
\left(
\Psi(\ov{F}_{T})
\right)(t)
v(x)\,
d\eta_{s}^{\phi}(t)\right)\,
d\nu^{\phi}(s).
\end{alignat}
Here the first equality
comes by \eqref{17551311biss},
the second one by
\eqref{18150402}
and
the third one
by
Prp. \ref{15111401}
and Def.
\ref{16420402}.
Next with the notation in 
Cor. \ref{15111901}
we choose
$(\forall x\in X)(S_{x}=\p_{\omega}(\mf{E}_{x}))$,
and since $\mf{O}$ is entire we can select 
$\mc{D}(x)=\mf{E}_{x}$, for all $x\in X$,
in other words $p_{l_{x},j_{x}}^{x}$ is the 
strong operator topology on $\mc{L}(\mf{E}_{x})$.
Thus $(A)$ of 
Lemma \ref{11011501} 
is satisfied since
Cor. \ref{15111901}, 
so the statement follows
by \eqref{18590402}
and
Prp.
\ref{18310302}.
\end{proof}
\begin{corollary}
\label{13511102}
Under the hypothesis of 
Thm. \ref{17301812b}
let the primary family $\ms{E}$ of $\mf{O}$
be the family of stalks of the Banach Bundle
$\mf{V}$ of which in 
Thm. \ref{17301812b}.
Assume that the hypothesis of Lemma \ref{15251311biss} are satisfied,
where $T$ is such that
$-T(x)$ is the infinitesimal generator of $\mc{U}(x)$ for all $x\in X$,
where $\mc{U}$ is the section of semigroups construcuted in Thm. \ref{17301812b}.
Moreover
let
$\lr{\mf{V},\mf{D}}{X,\{pt\}}$
be
an invariant 
$\left(\Theta,\mc{E}\right)-$structure
such that
$\mc{E}(\Theta)\subset\Gamma(\pi)$
and let $(K,A,\phi)$ be a triplet 
associated with a curve associated with $T$.
If $\mc{E}(\Theta)$
is a
$\lr{\nu^{\phi},\eta^{\phi}}{\mf{G},K,\R^{+}}$
invariant set
with respect to 
$\{\ov{F}_{T}\}$
and 
$\ov{F}_{T}\in
\mf{L}_{(1,1)}
(\R^{+},\mf{G},\eta^{\phi},\nu^{\phi})$,
then
$\{\mc{U}\}
\in
\Delta_{\Theta}\lr{\mf{V},\mf{D},\mf{W}}
{\mc{E},X,\R^{+}}$.
\end{corollary}
\begin{proof}
Since \eqref{15351311biss} and the definition of 
invariant $(\Theta,\mc{E})-$structures.
\end{proof}
\section
{Construction of sets in $\Delta\lr{\mf{V},\mf{D}}{\Theta,\mc{E}}$}
\label{15511702}
\begin{assumptions}
In this section 
$X$ is a topological space,
$Y$ is a locally compact space 
$\mu$ is a Radon measure on $Y$.
Let $\mf{L}^{\infty}(Y,\mu)$ denote the linear space of all complex valued maps defined on $Y$ 
which are $\mu-$measurable and bounded in measure for the measure $\mu$.
Let
$\mf{V}=\lr{\lr{\mf{E}}{\tau}}{\pi,X,\n}$
be
a bundle
of $\Omega-$spaces,
we indicate
with
$
\n\coloneqq
\{
\nu_{j}\,\vert\,
j\in J
\}
$
the directed
set of 
seminorms
on
$\mf{E}$.
\end{assumptions}
\begin{definition}
\label{22132802}
Let $Z\in Hlcs$
and $\{\psi_{i}\,\vert\, i\in I\}$
a fundamental set of seminorms on $Z$.
We denote 
by
$$
\left(Z^{Y}\right)_{s}
$$
the $Hlcs$ whose underlying linear space 
is 
$Z^{Y}$ 
and whose locally convex topology 
is generated by the 
following set of seminorms
\begin{equation}
\label{20531902}
\begin{cases}
\{q_{s}^{i}\,\vert\, s\in Y, i\in I\},
\\
q_{s}^{i}:
Z^{Y}
\ni 
f
\mapsto
\psi_{i}(f(s)).
\end{cases}
\end{equation}
Moreover
for any $B\subseteq Z^{Y}$
we shall denote by 
$B_{s}$ the $Hlc$ subspace of
$\left(Z^{Y}\right)_{s}$.
Notice that this definition is well-set being
independent by the choice of the fundamental
set of seminorms, indeed the topology is that
of uniform convergence over the finite subsets
of $Y$.
\end{definition}
\begin{definition}
\label{15361702}
Set
$$
\begin{cases}
\overset{\mu}{\blacklozenge}:
\prod_{x\in X}
\mf{L}_{1}(Y,\mf{E}_{x},\mu)
\to
\prod_{x\in X}
\mf{E}_{x},
\\
\overset{\mu}{\blacklozenge}(H)(x)
\coloneqq
\int H(x)(s)\,
d\mu(s)
\in
\mf{E}_{x},
\end{cases}
$$
for all
$H\in
\prod_{x\in X}
\mf{L}_{1}(Y,\mf{E}_{x},\mu)$
and
for all $x\in X$.
Moreover define
\begin{equation*}
\begin{aligned}
\blacktriangleright:
\C^{Y}\times\prod_{x\in X}(\mf{E}_{x})^{Y}&\to\prod_{x\in X}(\mf{E}_{x})^{Y},
\\
(f\blacktriangleright H)(x)(s)&\coloneqq f(s)H(x)(s),
\\
\forall f\in\C^{Y}, H&\in\prod_{x\in X}(\mf{E}_{x})^{Y}, x\in X, s\in Y.
\end{aligned}
\end{equation*}
\end{definition}
Notice that 
\begin{equation*}
\mf{L}^{\infty}(Y,\mu)
\blacktriangleright
\prod_{x\in X}\mf{L}_{1}(Y,\mf{E}_{x},\mu)
\subseteq
\prod_{x\in X}\mf{L}_{1}(Y,\mf{E}_{x},\mu).
\end{equation*}
\begin{definition}
\label{12121902}
Set
$$
\begin{cases}
\bigstar:
\prod_{x\in X}
\mc{L}(\mf{E}_{x})^{Y}
\times
\prod_{x\in X}
\mf{E}_{x}
\to
\prod_{x\in X}\mf{E}_{x}^{Y},
\\
(\forall x\in X)(\forall s\in Y)
(F\bigstar v)(x)(s)
\coloneqq
F(x)(s)(v(x)).
\end{cases}
$$
\end{definition}
\begin{definition}
\label{15492502}
\emph{$\lr{\mf{V}}{\mf{Z}}$
are $\mu-$related}
if
\begin{enumerate}
\item
$\mf{Z}
\coloneqq
\lr{\lr{\mf{T}}{\gamma}}{\zeta,X,\mf{K}}$
ia a bundle of $\Omega-$spaces;
\label{15492502st1}
\item
for all $x\in X$
\footnote{
In case
$\mf{E}_{x}$
is a Banach space
$Meas(Y,\mf{E}_{x},\mu)
\bigcap
\mf{L}_{1}(Y,\mf{E}_{x},\mu)
=
\mf{L}_{1}(Y,\mf{E}_{x},\mu)$.
}
\begin{equation}
\label{20242906}
\begin{cases}
\mf{T}_{x}
\subseteq
Meas(Y,\mf{E}_{x},\mu)
\bigcap
\mf{L}_{1}(Y,\mf{E}_{x},\mu),
\\
\mf{K}_{x}
=
\left\{
\sup_{(s,j)\in O}
q_{(s,j)}^{x}
\,\vert\,
O\in\p_{\omega}(Y\times J)
\right\},
\\
q_{(s,j)}^{x}:
\mf{T}_{x}
\ni
f_{x}\mapsto
\nu_{j}(f_{x}(s)),
\forall
s\in Y,
j\in J;
\end{cases}
\end{equation}
\item
$\Gamma(\zeta)
\subset
\left[
\prod_{x\in X}^{b}
\mf{T}_{x}
\right]_{ui}$;
\label{15492502st3}
\item
$\overset{\mu}{\blacklozenge}(\Gamma(\zeta))
\subseteq
\Gamma(\pi)$.
\label{15492502st4}
\end{enumerate}
Here we set
for all
$\A\subseteq
\prod_{x\in X}^{b}
\mf{T}_{x}$
\begin{equation}
\label{17180207}
\left[
\A
\right]_{ui}
\coloneqq
\left\{
H\in\A
\,\vert\,
(\forall j\in J)
\left(
\int_{Y}^{\bullet}
\sup_{x\in X}
\nu_{j}(H(x)(s))
\,d|\mu|(s)
<\infty
\right)
\right\}
\end{equation}
Finally
\emph{$\lr{\mf{V}}{\mf{Z},\ms{H}}$ are $\mu-$related}
if
\begin{enumerate}
\item
$\ms{H}=\{\ms{H}_{x}\}_{x\in X}$
such that
$\ms{H}_{x}\subseteq
\mc{L}(\mf{E}_{x})^{Y}$
for all $x\in X$,
\item
$\lr{\mf{V}}{\mf{Z}}$
are $\mu-$related
\item
\begin{equation}
\label{19090907}
\left(
\prod_{x\in X}\ms{H}_{x}\right)
\bigstar
\left(
\prod_{x\in X}\mf{E}_{x}\right)
\subseteq
\prod_{x\in X}\mf{T}_{x}.
\end{equation}
\end{enumerate}
\end{definition}
\begin{theorem}
[\textbf{GLT}]
\label{15101701}
Let 
$\lr{\mf{V}}{\mf{Z}}$
be
$\mu-$related
and let $x\in X$ such that 
its filter of neighbourhoods admits a countable basis. 
Thus
$$
\overset{\mu}{\blacklozenge}
\left(
\left[\Gamma_{\diamond}^{x}(\zeta)
\right]_{ui}
\right)
\subseteq
\Gamma_{\diamond}^{x}(\pi).
$$
\end{theorem}
\begin{proof}
Let $x\in X$ and 
$F\in\left[\Gamma_{\diamond}^{x}
(\zeta)\right]_{ui}$
thus by Cor. \ref{28111707}
there exists 
$\eta\in\Gamma(\zeta)$
such that
for all $j\in J,s\in Y$
\begin{equation}
\label{12221702}
\begin{cases}
F(x)=\eta(x)
\\
\lim_{z\to x}
\nu_{j}(F(z)(s)-\eta(z)(s))
=0.
\end{cases}
\end{equation}
Fix $j\in J$ thus 
by
\cite[Prop.$6$, $No 2$, 
$\S 1$, Ch. $6$]{IntBourb}
for all $z\in X$
\begin{equation}
\label{14550907}
\nu_{j}
\left(
\int
(F(z)(s)-\eta(z)(s))
\,d\mu(s)
\right)
\leq
\int^{\bullet}
\nu_{j}(F(z)(s)-\eta(z)(s))\,
d|\mu|(s)
\end{equation}
Moreover
$\nu_{j}^{z}$
is continuous by 
definition of bundles
of $\Omega-$spaces,
while
$F(z)$
and
$\eta(z)$
are by construction
$\mu-$measurable,
hence
by
\cite[Thm. $1$; Cor. $3$, 
$n^{\circ}3$, 
$\S\,5$,Ch. $4$]{IntBourb} 
the map
$Y\ni s\mapsto
\nu_{j}(F(z)(s)-\eta(z)(s))$
is 
$\mu-$measurable
thus
$|\mu|-$measurable.
Moreover
by the hypothesis on
$F$
and
by 
Def. \ref{15492502}\eqref{15492502st3}
\begin{equation}
\label{15370907}
\int^{\bullet}
\nu_{j}(F(z)(s)-\eta(z)(s))\,
d|\mu|(s)
\leq
\int^{\bullet}
\left(
\sup_{x\in X}
\nu_{j}(F(x)(s))
+
\sup_{x\in X}
\nu_{j}(\eta(x)(s))
\right)
d|\mu|(s)
<\infty.
\end{equation}
Therefore
by
\cite[Prp..$9$, $No 3$, $\S 1$, 
Ch. $5$]{IntBourb}
the map
$Y\ni s\mapsto
\nu_{j}(F(z)(s)-\eta(z)(s))$
is 
$|\mu|-$ essentially 
integrable
hence by 
the fact that
$\int_{Y}^{\bullet}f\,d|\mu|
=\int_{Y}f\,d|\mu|$
for all $|\mu|-$essentially
integrable map $f$,
we have
by
\eqref{14550907}
\begin{equation}
\label{12541702}
\nu_{j}
\left(
\int
(F(z)(s)-\eta(z)(s))
\,d\mu(s)
\right)
\leq
\int
\nu_{j}(F(z)(s)-\eta(z)(s))
\,
d|\mu|(s).
\end{equation}
Let $\{z_{n}\}_{n}\subset X$
be
such that
$\lim_{n\in\N}z_{n}=x$
thus
by \eqref{12221702}
\begin{equation}
\label{12491702A}
\lim_{n\in\N}
\nu_{j}(F(z_{n})(s)-\eta(z_{n})(s))
=0.
\end{equation}
For all
$s\in Y$
\begin{equation*}
\nu_{j}(F(z)(s)-\eta(z)(s))\,
\leq
\sup_{x\in X}
\nu_{j}(F(x)(s))
+
\sup_{x\in X}
\nu_{j}(\eta(x)(s))
\end{equation*}
thus
by the hypothesis on
$F$,
by 
Def. \ref{15492502}\eqref{15492502st3},
by the fact that
$\int_{Y}^{\bullet}\leq\int_{Y}^{*}$,
by
\eqref{12491702A},
and
by
the Lebesgue Theorem
\cite[Th.$6$, $No 7$, $\S 3$, 
Ch. $4$]{IntBourb}
we have
\begin{equation}
\label{12531702}
\lim_{n\in\N}
\int
\nu_{j}(F(z_{n})(s)-\eta(z_{n})(s))\,
d|\mu|(s)
=0.
\end{equation}
Finally by
\eqref{12541702},
\eqref{12531702}
and the hypothesis on $x$
we obtain
$$
\lim_{z\to x}
\nu_{j}
\left(
\int
F(z)(s)
\,
d\mu(s)
-
\int
\eta(z)(s)
\,
d \mu(s)
\right)
=0,
$$
thus the statement follows
by 
Def. \ref{15492502}\eqref{15492502st4}
and
Cor.
\ref{28111707}.
\end{proof}
\begin{definition}
[\textbf{$\left(\Theta,\mc{E},\mu\right)-$structure}]
\label{17161902}
We say
that
$\lr{\mf{V},\mf{Q}}{X,Y}$
is
a
\emph
{
$\left(\Theta,\mc{E},\mu\right)-$
structure
}
if
\begin{enumerate}
\item
$\mc{E}\subseteq\Gamma(\pi)$;
\item
$\Theta\subseteq
\prod_{x\in X}
Bounded(\mf{E}_{x})$;
\item
$\forall B\in\Theta$
\begin{enumerate}
\item
$\ms{D}(B,\mc{E})
\ne\emptyset$;
\item
$\bigcup_{B\in\Theta}\mc{B}_{B}^{x}$
is total
in $\mf{E}_{x}$
for all $x\in X$;
\end{enumerate}
\item
$\mf{Q}=\lr{\lr{\mf{H}}{\delta}}{\xi,X,\mf{Y}}$
is a
bundle of $\Omega-$spaces
such that
for all $x\in X$
\begin{equation}
\label{22512906}
\boxed{
\begin{cases}
\mf{H}_{x}
\subseteq
\mf{L}_{1}\left(Y,\mc{L}_{S_{x}}(\mf{E}_{x}),\mu\right),
\\
\mf{Y}_{x}
=
\left\{
\sup_{(t,j,B)\in O}
P_{(t,j,B)}^{x}\up\mf{H}_{x}
\,\vert\,
O\in\p_{\omega}
\left(Y\times J\times\Theta\right)
\right\}
\\
P_{(t,j,B)}^{x}:
\mf{L}_{1}\left(Y,\mc{L}_{S_{x}}(\mf{E}_{x}),\mu\right)
\ni
F\mapsto
\sup_{v\in\ms{D}(B,\mc{E})}
\nu_{j}(F(t)v(x)),\,
\forall
t\in Y,
B\in\Theta,
j\in J.
\end{cases}
}
\end{equation}
\end{enumerate}
Here $S_{x}$, $\mc{B}_{B}^{x}$
and $\ms{D}(B,\mc{E})$
are defined in \eqref{11232712}.
Moreover
$\lr{\mf{V},\mf{Q}}{X,Y}$
is
an
\emph
{invariant
$\left(\Theta,\mc{E},\mu\right)-$
structure}
if it is
a
$\left(\Theta,\mc{E},\mu\right)-$
structure
such that
\begin{equation}
\label{18112502}
\left\{
F\in\prod_{z\in X}^{b}
\mf{H}_{z}
\,\vert\,
(\forall t\in Y)
(F_{t}
\bullet
\mc{E}(\Theta)
\subseteq
\Gamma(\pi))
\right\}
=
\Gamma(\xi).
\end{equation}
\end{definition}
\begin{definition}
\label{14350303}
Let 
$\mu_{\lambda}$
for all $\lambda>0$
be
defined as in Def. \ref{15062301},
let
$\lr{\mf{V},\mf{Q}}{X,\R^{+}}$
be
a
$\left(\Theta,\mc{E},\mu\right)-$
structure
and 
denote
$\mf{Q}
=
\lr{\lr{\mf{H}}{\delta}}{\xi,X,\mf{S}}$,
moreover let 
$x\in X$,
$\mc{O}\subseteq\Gamma(\xi)$
and
$\mc{D}\subseteq\Gamma(\pi)$.
Set 
\begin{equation*}
\ms{Lap}(\mf{V})(x)
\coloneqq
\bigcap_{\lambda>0}
\mf{L}_{1}(\R^{+},
\mc{L}_{S_{x}}(\mf{E}_{x});\mu_{\lambda}).
\end{equation*}
Assume that
\begin{equation}
\label{12480707}
\Gamma_{\mc{O}}^{x}(\xi)
\bigcap
\ms{Lap}(\mf{V})(x)
\neq
\emptyset
\end{equation}
We 
say that 
$\lr{\mf{V},\mf{Q}}{X,\R^{+}}$
has the
\emph
{weak-Laplace duality property
on $\mc{O}$
and $\mc{D}$
at $x$},
shortly 
$w-\ms{LD}_{x}(\mc{O},\mc{D})$
if 
$\forall\lambda>0$
\begin{equation*}
\blacksquare_{\mu_{\lambda}}
\left(
\Gamma_{\mc{O}}^{x}(\xi)
\bigcap
\ms{Lap}(\mf{V})(x),
\Gamma_{\mc{D}}^{x}(\pi)
\right)
\subseteq
\Gamma^{x}(\pi).
\end{equation*}
\end{definition}
\begin{definition}
\label{18072802}
Let
$\lr{\mf{V},\mf{W}}{X,Y}$
be
a
$\left(\Theta,\mc{E}\right)-$structure
and denote
$\mf{W}=\lr{\lr{\mf{M}}{\gamma}}{\rho,X,\mf{R}}$.
Assume that
for all $x\in X$
\begin{equation}
\label{22372906}
\mf{M}_{x}
\subseteq
\mf{L}_{1}\left(Y,\mc{L}_{S_{x}}(\mf{E}_{x}),\mu\right).
\end{equation}
Set 
$\ms{M}^{\mu}
\coloneqq
\bigl\{\lr{\ms{M}_{x}^{\mu}}{\mf{Y}_{x}}\bigr\}_{x\in X}$
where for all $x\in X$
\begin{equation*}
\begin{aligned}
\ms{M}_{x}^{\mu}
&\coloneqq
\ov{\left\{
\sigma(x)
\,\vert\,
\sigma\in
\Gamma(\rho)
\right\}}
\text{ closure in the space
$\mf{L}_{1}
\left(Y,\mc{L}_{S_{x}}
(\mf{E}_{x}),\mu\right)_{s}$};
\\
\mf{Y}_{x}
&\coloneqq
\{
\sup_{(t,j,B)\in O}
P_{(t,j,B)}^{x}
\up
\ms{M}_{x}^{\mu}
\,\vert\,
O\in\mc{P}_{\omega}(Y\times J
\times\Theta)
\}.
\end{aligned}
\end{equation*}
Notice that
$\ms{M}^{\mu}$
is a nice
family of Hlcs,
and that
$\Gamma(\rho)$
satisfies by construction
$FM(3)$ with respect to 
$\ms{M}^{\mu}$.
Moreover 
by
\eqref{22542906}
and
the fact that
$\{t\}\in Comp(Y)$
for all $t\in Y$
\begin{equation}
\label{22582906}
P_{(t,j,B)}^{x}
=q_{(\{t\},j,B)}^{x}.
\end{equation}
By
\cite[Cor.$1.6.(iii)$]{gie}
we deduce that
$\Gamma(\rho)$
satisfies
$FM(4)$ 
with respect to 
$\{\lr{\mf{M}_{x}}
{\mf{R}_{x}}\}_{x\in X}$.
Therefore
we obtain by
\eqref{15073006}
and
\eqref{22582906}
that
for all 
$t\in Y$, $j\in J$, $B\in\Theta$
and for all
$\sigma\in\Gamma(\rho)$
\begin{equation*}
X\ni x\mapsto P_{(t,j,B)}^{x}(\sigma(x))
\text{ is $u.s.c.$}
\end{equation*}
Moreover
the upper envelope 
of 
a finite set of 
$u.s.c.$ maps
is an 
$u.s.c.$ map,
see \cite[Thm.$4$,$\S 6.2.$,Ch.$4$]{BourGT},
therefore
for all
$O\in\mc{P}_{\omega}(Y\times J
\times\Theta)$
\begin{equation}
\label{16392906}
X\ni x\mapsto 
\sup_{(t,j,B)\in O}
P_{(t,j,B)}^{x}(\sigma(x))
\text{ is $u.s.c.$}
\end{equation}
Hence
$\Gamma(\rho)$
satisfies
$FM(4)$ 
with respect to 
$\ms{M}^{\mu}$.
Finally
by the boundedness 
of $\Gamma(\rho)$
by definition
and
by \eqref{22582906}
we have also that
for all
$\sigma\in\Gamma(\rho)$
and
$O\in\mc{P}_{\omega}(Comp(Y)\times J
\times\Theta)$
\begin{equation*}
\sup_{x\in X}
\sup_{(t,j,B)\in O}
P_{(t,j,B)}^{x}(\sigma(x))
<\infty.
\end{equation*}
Therefore
we can
construct
the 
bundle
generated
by the couple
$\lr{\ms{M}^{\mu}}{\Gamma(\rho)}$
(Def. \ref{17471910Ba})
\begin{equation*}
\mf{V}(\ms{M}^{\mu},\Gamma(\rho)).
\end{equation*}
Clearly 
$\lr{\mf{V},\mf{V}(\ms{M}^{\mu},\Gamma(\rho))}{X,Y}$
is a $\left(\Theta,\mc{E},\mu\right)-$structure
that we call the
\emph
{$\left(\Theta,\mc{E},\mu\right)-$structure
underlying
$\lr{\mf{V},\mf{W}}{X,Y}$}.
\end{definition}
\begin{definition}
\label{19421902}
Let
$\lr{\mf{V},\mf{Q}}{X,Y}$
be
a
$\left(\Theta,\mc{E},\mu\right)-$
structure
and
$A\subset\prod_{x\in X}\mf{H}_{x}$.
Define
$A_{peq}$
as
the set of
all pointwise
equicontinuous
elements
in 
$A$,
and
$A_{ceq}$
as
the set of
all compactly equicontinuous
elements
in 
$A$,
see Def. \ref{21031238}.
\end{definition}
\begin{remark}
\label{23222906}
Lemma \ref{15482712}
holds
by replacing
a
$\left(\Theta,\mc{E}\right)-$structure
$\lr{\mf{V},\mf{W}}{X,Y}$
with
a
$\left(\Theta,\mc{E},\mu\right)-$
structure
$\lr{\mf{V},\mf{Q}}{X,Y}$
and
$K\in Comp(Y)$
with 
$t\in Y$.
In what follows when referring
to
Lemma \ref{15482712}
for a
$\left(\Theta,\mc{E},\mu\right)-$
structure
we shall mean 
the corresponding result
with
the replacements
described here.
\end{remark}
\begin{lemma}
\label{12151902}
Let 
$\lr{\mf{V},\mf{Q}}{X,Y}$
be
a
$\left(\Theta,\mc{E},\mu\right)-$
structure
and
$\lr{\mf{V}}{\mf{Z},\ms{H}}$
be
$\mu-$related,
where
$\mf{Q}
\coloneqq
\lr{\lr{\mf{H}}{\gamma}}{\xi,X,\mf{Y}}$
and
$\ms{H}_{x}\coloneqq\mf{H}_{x}$
for all $x\in X$.
Thus 
$$
\Gamma(\xi)
\bigstar
\mc{E}(\Theta)
\subseteq
\Gamma(\zeta)
\Rightarrow
(\forall x\in X)
\left(
\Gamma_{\diamond}^{x}(\xi)_{peq}
\bigstar
\Gamma_{\mc{E}(\Theta)}^{x}(\pi)
\subseteq
\Gamma_{\diamond}^{x}(\zeta)\right).
$$
\end{lemma}
\begin{proof}
Let 
$j\in J$,
$x\in X$ 
and
$w\in\Gamma_{\mc{E}(\Theta)}^{x}(\pi)$, 
so
there exists $v\in\mc{E}(\Theta)$
such that $v(x)=w(x)$ 
then
by Cor. \ref{28111707}
\begin{equation}
\label{14311902}
\lim_{z\to x}
\nu_{j}(w(z)-v(z))
=0.
\end{equation}
Moreover
let
$F\in\Gamma_{\diamond}^{x}(\xi)$,
so 
by
Lemma \ref{15482712}
$\exists\,\sigma\in\Gamma(\xi)$
such that $F(x)=\sigma(x)$ and 
for all $t\in Y$
\begin{equation}
\label{14301902}
\lim_{z\to x}
\nu_{j}\left(F(z)(t)v(z)-\sigma(z)(t)v(z)
\right)=0.
\end{equation}
Moreover 
$(\forall t\in Y)
(\exists\,M_{(t,j)}>0)
(\exists\,j_{1}\in J)
(\forall z\in X)$
\begin{alignat*}{1}
\nu_{j}\left((F\bigstar w)(z)(t)-
(\sigma\bigstar v)(z)(t)\right)
&=
\nu_{j}\left(F(z)(t)w(z)-
\sigma(z)(t)v(z)\right)
\\
&\leq
\nu_{j}\left(F(z)(t)(w(z)-v(z))\right)
+
\nu_{j}\left(F(z)(t)v(z)-\sigma(z)(t)v(z)
\right)
\\
&\leq
M_{(t,j)}
\nu_{j_{1}}(w(z)-v(z))
+
\nu_{j}\left(F(z)(t)v(z)-\sigma(z)(t)v(z)
\right).
\end{alignat*}
Therefore
by \eqref{14311902}
and
\eqref{14301902}
for all $t\in Y$
$$
\lim_{z\to x}
\nu_{j}\left((F\bigstar w)(z)(t)-
(\sigma\bigstar v)(z)(t)\right)
=0.
$$
Moreover 
$(\forall t\in Y)
(\exists\,M_{(t,j)}>0)
(\exists\,j_{1}\in J)$
\begin{equation}
\label{18590907}
\sup_{z\in X}
\nu_{j}((F\bigstar w)(z)(t))
\leq
M_{(t,j)}
\sup_{z\in X}
\nu_{j_{1}}(w(z))
<\infty.
\end{equation}
By the antecedent 
of the implication of the statement
we deduce that
$\sigma\bigstar v\in\Gamma(\zeta)$
hence
the statement 
follows
by 
Cor. \ref{28111707},
\eqref{20242906},
by the fact that
by \eqref{19090907}
$F\bigstar w\in\prod_{x\in X}\mf{T}_{x}$
and by
\eqref{18590907}.
\end{proof}
\begin{proposition}
\label{17441202bis}
Let
$\lr{\mf{V},\mf{W}}{X,Y}$
be
a
compatible 
$\left(\Theta,\mc{E}\right)-$structure.
Then for all $x\in X$
\begin{equation}
(\Gamma_{\diamond}^{x}(\rho)_{peq})_{t}
\bullet 
\Gamma_{\mc{E}(\Theta)}^{x}(\pi)
\subseteq
\Gamma_{\diamond}^{x}(\pi)
\end{equation}
\end{proposition}
\begin{proof}
Notice that 
$(F\bigstar v)(t)=F_{t}\bullet v$,
thus if we set $Y=\{pt\}$ 
the statement follows by
Lemma \ref{12151902}.
\end{proof}
\begin{corollary}
\label{15262502}
Let 
$\lr{\mf{V},\mf{Q}}{X,Y}$
be
a
$\left(\Theta,\mc{E},\mu\right)-$
structure
and
$\lr{\mf{V}}{\mf{Z},\ms{H}}$
be
$\mu-$related.
If $x\in X$ is such that its filter of neighbourhoods admits a countable basis, 
then
\begin{equation*}
\Gamma(\xi)
\bigstar
\mc{E}(\Theta)
\subseteq
\Gamma(\zeta)
\Rightarrow
\overset{\mu}{\blacklozenge}
\left(
\left[
\Gamma_{\diamond}^{x}(\xi)_{peq}
\bigstar
\Gamma_{\mc{E}(\Theta)}^{x}(\pi)
\right]_{ui}
\right)
\subseteq
\Gamma_{\diamond}^{x}(\pi).
\end{equation*}
Here
$\mf{Q}
\coloneqq
\lr{\lr{\mf{H}}{\gamma}}{\xi,X,\mf{Y}}$,
$\mf{Z}
\coloneqq
\lr{\lr{\mf{T}}{\delta}}{\zeta,X,\mf{K}}$
and
$\ms{H}_{x}\coloneqq\mf{H}_{x}$
for all $x\in X$.
\end{corollary}
\begin{proof}
By Thm. \ref{15101701}
and
Lemma \ref{12151902}.
\end{proof}
\begin{lemma}
\label{14452602}
Let
$\lr{\mf{V},\mf{Q}}{X,Y}$
be
an
invariant
$\left(\Theta,\mc{E},\mu\right)-$
structure
where
$\Theta$
defined in
\eqref{11121419b}.
Then
for all
$x\in X$
\begin{equation}
\label{15232602}
\left\{
H\in
\left[
\left(
\prod_{z\in X}^{b}
\mf{H}_{z}
\right)_{\diamond}^{x}
\right]_{peq}
\,\vert\,
(\forall t\in Y)
(H_{t}
\bullet
\mc{E}(\Theta)
\subseteq
\Gamma^{x}(\pi))
\right\}
\subseteq
\Gamma_{\diamond}^{x}(\xi).
\end{equation}
\end{lemma}
\begin{proof}
Let
$v\in\mc{E}(\Theta)$,
$t\in Y$ 
and
$H$ belong to the set in the left side
of \eqref{15232602}. 
Thus
by \eqref{18112502}
$\exists\,F\in\Gamma(\xi)$
such that
$F_{t}\bullet v\in\Gamma(\pi)$,
$F(x)=H(x)$
and
$H_{t}\bullet v\in\Gamma^{x}(\pi)$
by construction.
Then 
by 
Cor. \ref{28111707}
we obtain
for all $j\in J$ 
$$
\lim_{z\to x}
\nu_{j}(H(z)(t)v(z)-F(z)(t)v(z))=0.
$$
Therefore
the statement
follows
by
Lemma \ref{15482712}
and
\eqref{01592912}.
\end{proof}
\begin{remark}
\label{12012802}
Notice that Lemma \ref{14452602}
holds
if we replace
invariant
$\left(\Theta,\mc{E},\mu\right)-$
structure
with
invariant
$\left(\Theta,\mc{E}\right)-$
structure,
see Def. \ref{10282712},
and assume that
$Comp(Y)=\mc{P}_{\omega}(Y)$.
\end{remark}
By recalling Def. \ref{16420402} and Def. \ref{13361311biss}
we can state the following significant 
\begin{corollary}
\label{21152602}
Let $\mf{V}$ be a Banach bundle and set $\ms{E}\coloneqq\{\mf{E}_{z}\}_{z\in X}$.
Let $M>1$, $\beta\in\R$ and let $T\in\mc{G}(M,\beta,\ms{E})$ 
satisfy the property of separation of the spectrum.
Let $x\in X$ admitting a countable basis of its filter of neighbourhoods.
Let 
$\Gamma$ be a curve associated with $T$ and $(K,A,\phi)$ be a triplet associated with $\Gamma$ 
(Def. \ref{13361311biss})
such that
\begin{equation}
\label{19281311}
\begin{cases}
Re(\Gamma)
\subseteq
\R^{-},
\\
\beta>0
\Rightarrow
-\beta
\notin
Re(\Gamma).
\end{cases}
\end{equation}
Moreover let 
$\mf{Q}=\lr{\lr{\mf{H}}{\gamma_{1}}}{\xi,X,\mf{R}}$
and
$\mf{Z}=\lr{\lr{\mf{T}}{\gamma_{2}}}{\zeta,X,\mf{K}}$
be such that
\begin{enumerate}
\item 
$\lr{\mf{V},\mf{Q}}{X,\R^{+}}$ is an
invariant $\left(\Theta,\mc{E},\mu\right)-$structure
for all
$\mu\in\{\nu^{\phi},\eta_{s}^{\phi}\,\vert\, s\in K\}$,
with $\Theta$ determined by $\mc{E}$ by \eqref{11121419b};
\label{21152602hp1}
\item
$\lr{\mf{V}}{\mf{Z},\ms{H}}$ is $\mu-$related 
for all $\mu\in\{\nu^{\phi},\eta_{s}^{\phi}\,\vert\, s\in K\}$,
moreover $\ms{H}_{z}\coloneqq\mf{H}_{z}$, for all $z\in X$;
\label{21152602hp2}
\item
there exist $F,G\in\Gamma(\xi)$ such that 
$F(x)=\mc{W}_{T}(x)$
and
$G(x)=R^{\phi}(x)$;
\label{21152602hp3}
\item
for all $z\in X$
\begin{equation}
\label{11130307}
\cc{cs}{\R^{+},\mc{L}_{S_{z}}(\mf{E}_{z})}
\subseteq
\mf{H}_{z};
\end{equation}
\item
$\Gamma(\xi)
\bigstar
\mc{E}(\Theta)
\subseteq
\Gamma(\zeta)$.
\label{21152602hp4}
\end{enumerate}
Thus
\begin{equation}
\label{12042802}
\mc{W}_{T}\in
\Gamma^{x}(\xi)
\Rightarrow
P\bullet
\Gamma_{\mc{E}(\Theta)}^{x}(\pi)
\subseteq
\Gamma^{x}(\pi).
\end{equation}
Moreover let $\mf{D}=\lr{\lr{\mf{B}}{\gamma_{3}}}{\eta,X,\mf{L}}$ be
such that 
$\lr{\mf{V},\mf{D}}{X,\{pt\}}$ is an invariant $\left(\Theta,\mc{E}\right)-$structure.
If $\Pr\left(\mf{E}_{z}\right)\subset\mf{B}_{z}$
for all $z\in X$
and
there exists $N\in\Gamma(\eta)$ such that $N(x)=P(x)$,
then
\begin{equation}
\label{12102802}
\mc{W}_{T}
\in
\Gamma^{x}(\xi)
\Rightarrow
P
\in\Gamma^{x}(\eta).
\end{equation}
\end{corollary}
\begin{proof}
In this proof we denote
$R_{T}^{\phi}$ 
simply by $R^{\phi}$.
$R^{\phi}$ is $K-$supported
by construction,
moreover
the resolvent map $R^{\phi}(z)$ being analytic
is 
$\|\cdot\|_{B(\mf{E}_{z})}-$
continuous 
hence
continuous
as a map valued in
$\mc{L}_{S_{z}}(\mf{E}_{z})$ 
for any $z\in X$.
So
$$
R^{\phi}
\in\prod_{z\in X}
\cc{cs}{\R^{+},\mc{L}_{S_{z}}(\mf{E}_{z})},
$$
hence by \eqref{11130307}
follows 
\begin{equation}
\label{17500107}
R^{\phi}
\in
\prod_{z\in X}
\mf{H}_{z}.
\end{equation}
By \eqref{18150402}
for all $s\in K$, $z\in X$
and
$\forall v_{z}\in \mf{E}_{z}$
\begin{equation}
\label{17051311}
\begin{aligned}
\|
R(T(z),\phi(s))v_{z}
\|
&
\leq
\int_{\R^{+}}^{*}
e^{-|Re(\phi(s))|t}
\|e^{-tT(z)}v_{z}\|
dt
\\
&
\leq
M\|v_{z}\|
\int_{\R^{+}}^{*}
e^{\left(\beta-|Re(\phi(s))|\right) t}
dt
=
\frac{M\|v_{z}\|}{\beta-|Re(\phi(s))|},
\end{aligned}
\end{equation}
where
$\int_{\R^{+}}^{*}$
is the upper integral
on $\R^{+}$
with respect to the
Lebesgue measure.
We considered
in the first
inequality
\cite[Prop.
$6$,
$n^{\circ}$,
$\S 1$,
Ch. $6$]{IntBourb},
in the second
one
the 
inequality
\cite[$(1.37)$,$n^{\circ}4$,$\S 1$, Ch. $9$]{kato},
finally
in the equality
the Laplace transform
of the 
map
$\exp(\beta t)$.
Therefore 
by
\eqref{17500107}
and
\eqref{17051311}
$$
R^{\phi}
\in\left(
\prod_{z\in X}
\mf{H}_{z}
\right)_{peq}.
$$
Thus
by 
\eqref{17051311},
\eqref{11121419b}
and
\eqref{22512906}
\begin{equation}
\label{10312802}
R^{\phi}
\in\left(
\prod_{z\in X}^{b}
\mf{H}_{z}
\right)_{peq}.
\end{equation}
By 
\eqref{17500107} and \eqref{19090907}
we have that
$R^{\phi}\bigstar v
\in
\prod_{z\in X}
\mf{T}_{z}$
for all 
$v\in\prod_{z\in X}^{b}\mf{E}_{z}$.
By
hypothesis
\eqref{19281311}
we deduce that
$\frac{1}
{\beta-|Re(\phi(s))|}$
is 
defined
on
$K$,
hence
continuous
and integrable
in
it,
thus
by
\eqref{17051311}
\begin{equation}
\label{17590207}
R^{\phi}\bigstar v
\in
\left[
\prod_{z\in X}
\mf{T}_{z}
\right]_{ui}.
\end{equation}
By the continuity
of
$\frac{1}
{\beta-|Re(\phi(s))|}$
on $K$
we deduce 
that
the
map
$
\frac{|\frac{d\phi}{ds}(s)|}
{\beta-|Re(\phi(s))|}
$
is
integrable
in 
$K$.
Hence
by
\eqref{17051311}
and
\eqref{17551311biss}
\begin{equation}
\label{17483006}
\sup_{z\in X}
\|P(z)\|_{B(\mf{E}_{z})}
\leq
D
\coloneqq
\frac{1}{2\pi i}
\int_{K}
\frac{M\left|\frac{d\phi}{ds}(s)\right|}
{\beta-|Re(\phi(s))|}
\,ds.
\end{equation}
Therefore
for all 
$v\in\mc{E}$
by considering
that
$\mc{E}\subset\prod_{z\in X}^{b}\mf{E}_{z}$
$$
\sup_{z\in X}
\|P(z)v(z)\|_{B(\mf{E}_{z})}
\leq
D\sup_{z\in X}
\|v(z)\|_{\mf{E}_{z}}
<\infty.
$$
Thus
\begin{equation}
\label{15181303}
P\in\prod_{z\in X}^{b}\mf{B}_{z}.
\end{equation}
Let $z\in X$
and $v\in\Gamma_{\mc{E}(\Theta)}^{z}(\pi)$.
By \eqref{18150402}
for all $s\in K$
\begin{equation}
\label{17542502}
(R^{\phi}\bigstar v)(z)(s)
=
\ro{\eta_{s}^{\phi}}
(\mc{W}_{T}\bigstar v)(z).
\end{equation}
Moreover by \eqref{17551311biss}
\begin{equation}
\label{17562502}
P\bullet v
=
\ro{\nu^{\phi}}
(R^{\phi}\bigstar v).
\end{equation}
Notice that
$(R^{\phi}\bigstar v)(z)(s)
=
\left(R^{\phi}(\cdot)(s)\bullet v\right)(z)$
so by \eqref{17542502}
for all $s\in K$
\begin{equation}
\label{18022502}
R^{\phi}(\cdot)(s)\bullet v
=
\ro{\eta_{s}^{\phi}}
(\mc{W}_{T}\bigstar v).
\end{equation}
If 
$\mc{W}_{T}\in\Gamma^{x}(\xi)$
then by
\cite[Ch. $9$, $\S1$, $n^{\circ}4$,$(1.37)$]{kato}
and hypothesis $(3)$
follows that
$\mc{W}_{T}
\in
\Gamma_{\diamond}^{x}(\xi)_{peq}$.
Therefore for all $w\in\Gamma_{\mc{E}(\Theta)}^{x}(\pi)$
by using \cite[Ch. $9$, $\S1$, $n^{\circ}4$,$(1.37)$]{kato}
we can apply 
Cor. \ref{15262502}
to $\mc{W}_{T}\bigstar w$
and then since 
\eqref{18022502}
we obtain
for all $s\in K$ 
\begin{equation}
\label{13540303}
R^{\phi}(\cdot)(s)
\bullet w
\in\Gamma^{x}(\pi).
\end{equation}
By construction
$\mc{E}(\Theta)
\subseteq\Gamma(\pi)$
so 
$\mc{E}(\Theta)
\subseteq
\Gamma_{\mc{E}(\Theta)}^{x}(\pi)$,
hence
\begin{equation}
\label{18042502}
R^{\phi}(\cdot)(s)
\bullet 
\mc{E}(\Theta)
\subseteq
\Gamma^{x}(\pi).
\end{equation}
Moreover by \eqref{10312802} and hypothesis $(3)$
follows that
$$
R^{\phi}
\in
\left[
\left(
\prod_{z\in X}^{b}
\mf{H}_{z}
\right)_{\diamond}^{x}
\right]_{peq}.
$$
Hence
by 
Lemma \ref{14452602}
and
\eqref{18042502}
\begin{equation}
\label{18132502}
R^{\phi}
\in\Gamma_{\diamond}^{x}(\xi)_{peq}.
\end{equation}
Finally
\eqref{12042802}
follows
by
\eqref{18132502},
\eqref{17562502},
\eqref{17590207}
and
Cor.
\ref{15262502}.
Next since
\eqref{15181303},
\eqref{17483006}
and the hypothesis that 
there exists $N\in\Gamma(\eta)$ such that $N(x)=P(x)$,
we obtain that
\begin{equation*}
P\in
\left[
\left(
\prod_{z\in X}^{b}
\mf{B}_{z}
\right)_{\diamond}^{x}
\right]_{peq}.
\end{equation*}
Thus
\eqref{12102802}
follows by 
\eqref{12042802},
Rmk.
\ref{12012802}
and by
$\mc{E}(\Theta)
\subseteq
\Gamma_{\mc{E}(\Theta)}^{x}(\pi)$.
\end{proof}
\begin{remark}
\label{12500303}
Prp. \ref{18390901} 
can be used 
in order to verify part in
hypothesis $(1)$ of Cor. \ref{21152602}.
\end{remark}
The following general result shall permit to apply 
Cor. \ref{21152602} to 
Thm. \ref{17301812b}.
\begin{proposition}
\label{23332802}
Let
$\lr{\mf{V},\mf{W}}{X,Y}$
be
a
$\left(\Theta,\mc{E}\right)-$structure
and let us denote
$\mf{W}=\lr{\lr{\mf{M}}{\delta}}{\rho,X,\mf{R}}$.
Assume that
\eqref{22372906}
holds
for all $x\in X$
and
let
$\lr{\mf{V},\mf{V}
(\ms{M}^{\mu},\Gamma(\rho))}{X,Y}$
be
the
$\left(\Theta,\mc{E},\mu\right)-$
structure
underlying
$\lr{\mf{V},\mf{W}}{X,Y}$.
Then for all $x\in X$
\begin{equation}
\label{17350303}
\Gamma_{\diamond}^{x}(\rho)
\subseteq
\Gamma_{\Gamma(\rho)}^{x}
(\pi_{\ms{M}^{\mu}}),
\end{equation}
inclusion to be considered modulo canonical isomorphism.
\end{proposition}
\begin{proof}
By construction it results that
$\Gamma(\rho)
\subseteq
\Gamma(\pi_{\ms{M}^{\mu}})$
modulo
the canonical
isomorphism,
thus the statement follows since 
Lemma \ref{15482712} 
and \eqref{22582906}.
\end{proof}
\begin{definition}
\label{14460503}
We call 
$\mf{X}=\lr{\mf{V},x_{\infty}}{\mc{U}_{0}}$ 
a quasi-appropriate set of contractions (isometries)
if the following holds.
$\mf{V}=\lr{\lr{\mf{E}}{\tau}}{\pi,X,\|\cdot\|}$
is a 
Banach bundle
where $X$ is a completely regular space, 
$x_{\infty}\in X$
such that 
its filter of neighbourhoods admits a countable basis
and 
$\mc{U}_{0}
\in\prod_{x\in X_{0}}
\cc{}{\R^{+},B_{s}(\mf{E}_{x})}$
such that
$\mc{U}_{0}(x)$
is 
a
$C_{0}-$semigroup
of contractions
(isometries)
on
$\mf{E}_{x}$
for all 
$x\in X_{0}\coloneqq X-\{x_{\infty}\}$.
Moreover let $T_{x}$ be the infinitesimal generator of the semigroup
$\mc{U}_{0}(x)$ for any $x\in X_{0}$,
let
$\lr{\lr{\mf{E}(\ms{E}^{\oplus})}
{\tau(\ms{E}^{\oplus},\mc{E}^{\oplus})}}
{\pi_{\ms{E}^{\oplus}},X,\mf{n}^{\oplus}}$
be the bundle direct sum of the family $\{\mf{V},\mf{V}\}$\footnote{well-set since
$\mf{B}$ is full by the Dupre' Thm.} 
and set
\begin{equation}
\label{16290303}
\begin{cases}
\mc{T}_{0}\text{ is the map on $X_{0}$ such that} 
\\
\mc{T}_{0}(x)
\coloneqq
Graph(T_{x}),\,
x\in X_{0},
\\
\Phi
\coloneqq
\{
\phi\in\Gamma^{x_{\infty}}
(\pi_{\ms{E}^{\oplus}})
\,\vert\,
(\forall x\in X_{0})
(\phi(x)\in\mc{T}_{0}(x))
\},
\\
\mc{E}
\coloneqq
\{
v
\in
\Gamma(\pi)
\,\vert\,
(\exists\,\phi\in\Phi)
(v(x_{\infty})=\phi_{1}(x_{\infty}))\},
\\
B_{v}:X\ni x\mapsto\{v(x)\},\,
\forall v\in\prod_{x\in X}\mf{E}_{x},
\\
\Theta
\coloneqq
\left\{
B_{w}
\,\vert\,
w\in\mc{E}
\right\},
\\
\mc{T}\text{ is the map on $X$ extending $\mc{T}_{0}$ and such that} 
\\
\mc{T}(x_{\infty})
\coloneqq
\{\phi(x_{\infty})\,\vert\,\phi\in\Phi\},
\\
D(T_{x_{\infty}})
\coloneqq
\Pr_{1}^{x_{\infty}}(\mc{T}(x_{\infty}))
=
\{\phi_{1}(x_{\infty})
\,\vert\,\phi\in\Phi\}.
\end{cases}
\end{equation}
We call $\mf{X}=\lr{\mf{V},x_{\infty}}{\mc{U}_{0}}$ 
an appropriate set of contractions (isometries)
if it is a quasi-appropriate set of contractions (isometries)
and all the following holds.
\begin{enumerate}
\item
$D(T_{x_{\infty}})$
is dense in 
$\mf{E}_{x_{\infty}}$,
\item
$\{v(x)\,\vert\, v\in\mc{E}\}$
is dense in $\mf{E}_{x}$
for all $x\in X_{0}$;
\end{enumerate}
hence according to Thm. \ref{17301812b} what follows
\begin{equation*}
T_{x_{\infty}}:
D(T_{x_{\infty}})
\ni
\phi_{1}(x_{\infty})
\mapsto
\phi_{2}(x_{\infty}),
\end{equation*}
defines a linear operator.
We require that 
$\exists\lambda_{0}>0$
($\exists\lambda_{0}>0,
\lambda_{1}<0$)
such that
the range
$\mc{R}(\lambda_{0}-T_{x_{\infty}})$
is dense in $\mf{E}_{x_{\infty}}$,
(the ranges
$\mc{R}(\lambda_{0}-T_{x_{\infty}})$
and
$\mc{R}(\lambda_{1}-T_{x_{\infty}})$
are dense in $\mf{E}_{x_{\infty}}$).
Thus according to Thm. \ref{17301812b}
$T_{x_{\infty}}$ is the infinitesimal generator
of a $C_{0}-$semigroup of contractions
(isometries)
on $\mf{E}_{x_{\infty}}$.
Therefore we can define 
\emph{section of semigroups associated with $\mf{X}$} 
the following map
\begin{equation*}
\mc{U}\in\prod_{x\in X}
\ms{U}_{\|\cdot\|_{B(\mf{E}_{x})}}
(\mc{L}_{S_{x}}(\mf{E}_{x})),
\end{equation*}
($\mc{U}\in\prod_{x\in X}
\ms{U}_{is}
(\mc{L}_{S_{x}}(\mf{E}_{x}))$)
such that
$\mc{U}(x)$
is the $C_{0}-$semigroup of contractions
(isometries)
on $\mf{E}_{x}$
whose
infinitesimal generator
is $T_{x}$
for all $x\in X$.
Finally set
\begin{equation*}
\ms{T}:X\ni x
\mapsto 
-T_{x}\in
Cld(\mf{E}_{x}).
\end{equation*}
We require that $\ms{T}$ satisfies the property of separation of the spectrum
and that there exists a curve $\Gamma$ associated with $\ms{T}$
such that
\begin{equation*}
Re(\Gamma)\subseteq\R^{-}.
\end{equation*}
We call $\ms{T}$ \emph{section of generators associated with $\mf{X}$.}
Finally for any curve $\Gamma$ associated with $\ms{T}$ such that $Re(\Gamma)\subseteq\R^{-}$
we define 
\emph{section of projectors associated with $\mf{X}$ and $\Gamma$} 
the map $\mc{P}\in\prod_{x\in X}\Pr(\mf{E}_{x})$
such that for all $x\in X$
\begin{equation*}
\mc{P}(x)
\coloneqq
-
\frac{1}{2\pi i}
\int_{\Gamma}
R(-T_{x};\zeta)\,
d\zeta
\in B(\mf{E}_{x}).
\end{equation*}
Here we recall that
$R(-T_{x};\cdot):
P(-T_{x})
\ni \zeta
\mapsto
(-T_{x}-\zeta)^{-1}
\in B(\mf{E}_{x})$
is the resolvent
map
of
$-T_{x}$
and
$P(-T_{x})$
is its resolvent
set,
while
the integration
is with respect
to the 
norm topology
on $B(\mf{E}_{x})$.
\end{definition}
\begin{theorem}
[\textbf{MAIN}$\ms{2}$]
\label{13020103}
Let 
$\mf{X}=\lr{\mf{V},x_{\infty}}{\mc{U}_{0}}$ 
be a quasi-appropriate set of contractions (isometries),
let us denote 
$\mf{V}=\lr{\lr{\mf{E}}{\tau}}{\pi,X,\|\cdot\|}$
and use the notation in \eqref{16290303}.
Assume that $\{v(x_{\infty})\,\vert\, v\in\mc{E}\}$ is dense in $\mf{E}_{x_{\infty}}$.
Then $Dom(T_{x_{\infty}})$ is dense in $\mf{E}_{x_{\infty}}$.
Next assume that $\mf{X}$ satisfies all the remaining requests in order to be 
an appropriate set of contractions (isometries).
Let $\mc{U}$ be the section of semigroups associated with $\mf{X}$,
$\Gamma$ be a curve associated with the section of generators associated with $\mf{X}$
such that $Re(\Gamma)\subseteq\R^{-}$,
$(K,A,\upphi)$ be a triplet associated with $\Gamma$
and $\mc{P}$ be the section of projectors associated with $\mf{X}$ and $\Gamma$.
Let $\mf{n}$ denote the Lebesgue measure on $\R^{+}$.
We assume that there exist 
$\mf{W}=\lr{\lr{\mf{M}}{\delta}}{\rho,X,\mf{R}}$ and 
$\mf{Z}=\lr{\lr{\mf{T}}{\gamma_{2}}}{\zeta,X,\mf{K}}$
with the following properties.
\begin{enumerate}
\item
$\lr{\mf{V},\mf{W}}{X,\R^{+}}$ is a $\left(\Theta,\mc{E}\right)-$structure;
\label{13020103hp1}
\item
for all $x\in X$
\begin{equation*}
\cc{cs}{\R^{+},\mc{L}_{S_{x}}(\mf{E}_{x})}
\subseteq
\mf{M}_{x}
\subseteq
\mf{L}_{1}(\R^{+},
\mc{L}_{S_{x}}(\mf{E}_{x}),\mf{n});
\end{equation*}
\label{13020103hp2}
\item
$\lr{\mf{V},\mf{V}(\ms{M}^{\mf{n}},\Gamma(\rho))}{X,\R^{+}}$\footnote{Def. \ref{18072802}}
is 
invariant
and
$\lr{\mf{V}}{\mf{Z},\ms{M}^{\mf{n}}}$
is
$\mf{n}-$related
such that 
$\mf{L}^{\infty}(\R^{+},\mf{n})\blacktriangleright\Gamma(\zeta)\subseteq\Gamma(\zeta)$;
\label{13020103hp3}
\item
$\Gamma(\rho)
\bigstar
\mc{E}(\Theta)
\subseteq
\Gamma(\zeta)$;
\label{13020103hp4}
\item
$\ms{U}_{\|\cdot\|_{B(\mf{E}_{x})}}
(\mc{L}_{S_{x}}(\mf{E}_{x}))
\subseteq
\mf{M}_{x}$
($\ms{U}_{is}
(\mc{L}_{S_{x}}(\mf{E}_{x}))
\subseteq
\mf{M}_{x}$),
for all $x\in X$;
\label{13020103hp5}
\item
$\exists\,F\in\Gamma(\rho)$
such that
$F(x_{\infty})=\mc{U}(x_{\infty})$
and
\begin{description}
\item[i]
$\lr{\mf{V},\mf{W}}{X,\R^{+}}$
has
the
$\ms{LD}_{x_{\infty}}(\{F\},\mc{E})$
\textbf{or}
it
has
the 
$\ms{LD}(\{F\},\mc{E})$;
\item[ii]
for any $v\in\mc{E}$ there exists $\phi\in\Phi$ such that 
$v(x_{\infty})=\phi_{1}(x_{\infty})$
and
$(\forall\{z_{n}\}_{n\in\N}\subset X
\,\vert\,
\lim_{n\in\N}z_{n}=x_{\infty})$
we have
that
$\{
\mc{U}(z_{n})(\cdot)\phi_{1}(z_{n})
-
F(z_{n})(\cdot)v(z_{n})
\}_{n\in\N}$
is a bounded equicontinuous sequence.
\end{description}
\label{13020103hp6}
\end{enumerate}
Thus we can state what follows.
\begin{enumerate}
\item
If $\exists\,G\in\Gamma(\rho)$
such that
$G(x_{\infty})=R^{\upphi}(x_{\infty})$,
then
$\mc{P}
\bullet
\Gamma_{\mc{E}(\Theta)}^{x_{\infty}}(\pi)
\subseteq
\Gamma^{x_{\infty}}(\pi)$.
\label{13020103st1}
\item
Let $\mf{D}=\lr{\lr{\mf{B}}{\gamma_{3}}}{\eta,X,\mf{L}}$ be
such that $\lr{\mf{V},\mf{D}}{X,\{pt\}}$ is an invariant
$\left(\Theta,\mc{E}\right)-$structure.
If $\Pr\left(\mf{E}_{x}\right)\subset\mf{B}_{x}$ for all $x\in X$
and if there exists $N\in\Gamma(\eta)$ such that $N(x_{\infty})=\mc{P}(x_{\infty})$,
then 
\begin{equation}
\label{19000107}
\mc{P}\in\Gamma^{x_{\infty}}(\eta),
\end{equation}
and
\begin{equation}
\label{12102802bis}
\boxed{
\{\lr{\mc{T}}{x_{\infty},\Phi}\}
\in
\Delta\lr{\mf{V},\mf{D}}{\Theta,\mc{E}}.
}
\end{equation}
\end{enumerate}
\end{theorem}
\begin{proof}
$\mf{V}$ is full since
the Dupre' Thm. \cite[Cor. $2.10$]{gie}.
So by Prop. \ref{19492307} and the density assumption
follows that $Dom(T_{x_{\infty}})$ is dense in $\mf{E}_{x_{\infty}}$.
Since \cite[2.2]{gie} we deduce that 
the set of all bounded continuous sections of any bundle of $\Omega-$spaces 
over a completely regular space satisfies the property $FM(3)$.
Therefore $\mf{M}_{x}\subset\ms{M}_{x}^{\mf{n}}$ for all $x\in X$,
since the immersion $\lr{\mf{M}_{x}}{\mf{R}_{x}}\hookrightarrow 
\mf{L}_{1}\left(Y,\mc{L}_{S_{x}}(\mf{E}_{x}),\mf{n}\right)_{s}$
is continuous.
Thus
\begin{equation}
\label{20273006}
\cc{cs}{\R^{+},\mc{L}_{S_{x}}(\mf{E}_{x})}
\subset
\ms{M}_{x}^{\mf{n}}.
\end{equation}
Now since $Re(\Gamma)\subseteq\R^{-}$ and since $d\upphi/ds$ is continuous and then bounded on $K$,
we deduce by hypothesis \eqref{13020103hp3}
that 
\begin{equation}
\label{05272109}
\begin{aligned}
&\text{
$\lr{\mf{V},\mf{V}(\ms{M}^{\mf{n}},\Gamma(\rho))}{X,\R^{+}}$
is an invariant $\left(\Theta,\mc{E},\mu\right)-$structure and}
\\
&
\text{
$\lr{\mf{V}}{\mf{Z},\ms{M}^{\mf{n}}}$
is
$\mu-$related,}
\,\forall\mu\in\{\nu^{\upphi},\eta_{s}^{\upphi}\,\vert\,s\in K\}.
\end{aligned}
\end{equation}
In particular \eqref{18470109} holds, so we can apply 
Thm. 
\ref{17301812b}
to obtain 
$\mc{U}\in\Gamma^{x_{\infty}}(\rho)$
and in virtue of hypothesis \eqref{13020103hp6}
that
$\mc{U}\in\Gamma_{\diamond}^{x_{\infty}}(\rho)$.
Thus by Prp. \ref{23332802} we have
\begin{equation}
\label{20180207}
\mc{U}
\in
\Gamma_{\diamond}^{x_{\infty}}
(\pi_{\ms{M}^{\mf{n}}}).
\end{equation}
Now for the position $\mf{Q}=\mf{V}(\ms{M}^{\mf{n}},\Gamma(\rho))$
the hypotheses \eqref{21152602hp1} and \eqref{21152602hp2} 
of Cor. \ref{21152602} are satisfied
since \eqref{05272109}.
Moreover $F,G\in\Gamma(\pi_{\ms{M}^{\mf{n}}})$
indeed
$\Gamma(\rho)\subseteq\Gamma(\pi_{\ms{M}^{\mf{n}}})$
modulo the canonical isomorphism,
so hypothesis \eqref{21152602hp3} of Cor. \ref{21152602} is satisfied.
Hence statement \eqref{13020103st1} follows by \eqref{20180207} and \eqref{12042802},
while \eqref{19000107} follows by \eqref{20180207} and \eqref{12102802}.
Next
$\lr{\mc{T}}{x_{\infty},\Phi}
\in
\ms{Gr}(\mf{V},\mf{V})$
since 
Thm. \ref{17301812b},
while 
$\mc{P}(x)T_{x}\subseteq T_{x}\mc{P}(x)$ for all $x\in X$
since the resolvent map of any operator commutes with its operator
see for example \cite[$\S$ 6.1. Ch. 3]{kato},
thus
\eqref{12102802bis}
follows by 
\eqref{19000107}. 
\end{proof}
\begin{remark}
By \eqref{02022912} follows that \eqref{19000107}
is equivalent to say that for all $v\in\mc{E}$
\begin{equation*}
\lim_{z\to x_{\infty}}\left\|\left(\mc{P}(z)-N(z)\right)v(z)\right\|=0.
\end{equation*}
\end{remark}
\section{
Kurtz Bundle Construction
}
\label{17572301}
In this section we construct a special
bundle $\mf{E}$
of Banach space
such that
for it the
Main 
Thm. \ref{17301812b}
reduces to 
the \cite[Th. $2.1.$]{kurtz}
showing in this way that
(a particular case) of the
construction of Kurtz
falls into the general setting
of bundle of $\Omega-$spaces.
\begin{notation}
In this section we shall use the notation
of \cite{kurtz} 
with the additional
specification of denoting with
$\|\cdot\|_{n}$
the norm in the Banach space $L_{n}$.
Moreover we denote by
$X$ 
the Alexandrov (one-point) compactification
of the locally compact space $\N$ with the
discrete topology.
Here we recall some definitions given in
\cite{kurtz}.
$\lr{L}{\|\cdot\|}$ is a Banach space and
$\{\lr{L_{n}}{\|\cdot\|_{n}}\}_{n\in\N}$ 
is a sequence of Banach 
spaces, moreover 
$\{P_{n}\in B(L,L_{n})\}_{n\in\N}$
is a sequence of bounded linear operators
such that $\forall f\in L$
\begin{equation}
\label{10572601}
\lim_{n\to\infty}
\|P_{n}f\|_{n}=\|f\|.
\end{equation}
Given an element 
$f\in L$ and
a sequence
$\{f_{n}\}_{n\in\N}$ such that
$f_{n}\in L_{n}$ for all $n\in\N$
we set
\begin{equation}
\label{11012601}
\lim_{n\to\infty}f_{n}
\cu
f\overset{def}{\Leftrightarrow}
\lim_{n\to\infty}
\|f_{n}-P_{n}f\|_{n}=0.
\end{equation}
In addition to the requirements
of \cite{kurtz}
we assume also that
\begin{equation}
\label{11212601}
(\forall n\in\N)
(\ov{P_{n}(L)}=L_{n})
\end{equation}
We shall set here
$L_{\infty}\coloneqq L$,
$\|\cdot\|
\coloneqq
\|\cdot\|_{\infty}$,
where
$\|\cdot\|$
is the norm on $L$.
Finally 
for all $Z$
we recall that
$B_{s}(Z)$
is the locally convex space
of all linear bounded operators
on $Z$ with the strong operator
topology. 
\end{notation}
\begin{lemma}
\label{unique}
Let
$f,g\in L$ 
and
$\{f_{n}\}_{n\in\N}$ 
such that
$f_{n}\in L_{n}$ for all $n\in\N$.
Then
$
(\lim_{n\to\infty}f_{n}
\cu
f)
\wedge
(\lim_{n\to\infty}f_{n}
\cu
g)
\Rightarrow
f=g
$
\end{lemma}
\begin{proof}
Let
$(\lim_{n\to\infty}f_{n}
\cu
f)$
and
$(\lim_{n\to\infty}f_{n}
\cu
g)$
thus
$$
\lim_{n\in\N}
\|
P_{n}(f-g)
\|
\leq
\lim_{n\in\N}
\|P_{n}f-f_{n}\|
+
\lim_{n\in\N}
\|P_{n}g-f_{n}\|
=0,
$$
so the statement follows by
\eqref{10572601}.
\end{proof}
\begin{definition}
\label{10452601}
Set
$$
\begin{cases}
\ms{L}\coloneqq
\{
\lr{L_{x}}{\|\cdot\|_{x}}
\}_{x\in X},
\\
\mc{E}(L)\coloneqq
\{\sigma^{f}\,\vert\, f\in L\},
\end{cases}
$$
where
$\sigma^{f}\in\prod_{x\in X}L_{x}$
such that
$\sigma^{f}(n)\coloneqq P_{n}f$
for all $n\in\N$
and
$\sigma^{f}(\infty)\coloneqq f$.
\end{definition}
\begin{definition}
\label{11322601}
By \eqref{10572601}
the sequence
$\{\|P_{n}f\|_{n}\}_{n\in\N}$
is bounded for all $f\in L$ so
$\sigma^{f}\in\prod_{x\in X}^{b}L_{x}$.
Moreover by \eqref{10572601}
$\mc{E}(L)$ satisfies $FM(4)$,
finally by the request
\eqref{11212601}
it satisfies also
$FM(3)$.
Therefore 
we can 
define the
\emph{Kurtz bundle}
the following
bundle
$$
\mf{V}(\ms{L},\mc{E}(L))
$$
generated
by
the couple
$\lr{\ms{L}}{\mc{E}(L)}$,
see
in Def. \ref{17471910Ba}.
\end{definition}
\begin{remark}
\label{16022601}
By Rmk. \ref{17150312} we have that
\begin{equation}
\label{16022601b}
\mc{E}(L)
\subseteq
\Gamma(\pi_{\ms{L}})
\text{ modulo the canonical isomorphism.}
\end{equation}
Finally
by applying the principle 
of uniform boundedness,
\cite[Th. $1.29$, $No 3$, Ch.$3$]{kato},
we deduce that
the sequence 
$\{\|P_{n}\|_{B(L,L_{n})}\}_{n\in\N}$
is bounded.
\end{remark}
\begin{definition}
\label{18151602}
Fix
$
\mc{U}_{0}
\in\prod_{n\in\N}
\cc{}{\R^{+},B_{s}(L_{n})}
$
such that
$\mc{U}_{0}(x)$
is a
$(C_{0})-$semigroup
of isometries
on
$L_{n}$
for all $n\in\N$.
Denote
by
$T_{n}$
the infinitesimal generator
of the semigroup
$\mc{U}_{0}(n)$
for any $n\in\N$.
Let us take 
the positions
\eqref{15482601},
where
$\lr{\lr{\mf{E}(\ms{E}^{\oplus})}
{\tau(\ms{E}^{\oplus},\mc{E}^{\oplus})}}
{\pi_{\ms{E}^{\oplus}},X,\mf{n}^{\oplus}}$
is
the 
bundle direct sum of the
family
$\{\mf{V}(\ms{L},\mc{E}(L)),
\mf{V}(\ms{L},\mc{E}(L))\}$.
In addition
we maintain the Notation
\ref{15411512b}
where $\mf{V}$ has to be 
considered 
the Kurtz bundle
and 
$x_{\infty}\coloneqq\infty$,
thus
$
\mc{T}
\in
\prod_{x\in X}
Graph(L_{x}
\times L_{x})
$
so that
$
\mc{T}\up X-\{\infty\}
\coloneqq
\mc{T}_{0}
$
and
$$
\mc{T}(\infty)
\coloneqq
\{\phi(\infty)\,\vert\,\phi\in\Phi\},
$$
and
$
D(T_{\infty})
\coloneqq
\Pr_{1}^{\infty}(\mc{T}(\infty))
=
\{\phi_{1}(\infty)
\,\vert\,\phi\in\Phi\}.
$
Finally
$\mc{S}
\coloneqq
\{S_{x}\}_{x\in X}$
where
$(\forall B\in\Theta)(\forall x\in X)$
\begin{equation}
\label{11232712bis}
\begin{cases}
\ms{D}(B,\mc{E})
\coloneqq
\mc{E}
\cap
\left(\prod_{x\in X}B_{x}\right)
\\
\mc{B}_{B}^{x}
\coloneqq
\{v(x)\,\vert\, v\in\ms{D}(B,\mc{E})\}
\}
\\
S_{x}
\coloneqq
\{\mc{B}_{B}^{x}\,\vert\, B\in\Theta\}.
\end{cases}
\end{equation}
\end{definition}
\begin{proposition}
\label{11422601}
Let 
$\ov{f}\in\prod_{x\in X}L_{x}$
Thus
$$
\lim_{n\to\infty}\ov{f}(n)\cu\ov{f}(\infty)
\Leftrightarrow
\ov{f}\in\Gamma^{\infty}(\pi_{\ms{L}}).
$$
\end{proposition}
\begin{proof}
By \eqref{16022601b} 
and
implication $(3)\Rightarrow(1)$
of Cor. \ref{28111707}
we have that
$\lim_{n\to\infty}\ov{f}(n)\cu\ov{f}(\infty)$
implies that
$$
\ov{f}
\text{
is continuous
at $\infty$,}
$$
indeed 
$\sigma^{\ov{f}(\infty)}\in
\Gamma(\pi_{\ms{L}})$
modulo isomorphism.
By the upper semicontinuity
of $\|\cdot\|:\mf{E}\to\R^{+}$,
due to the
construction of the
bundle
$\mf{V}(\ms{L},\mc{E}(L))$
and to
\cite[$1.6.(ii)$]{kurtz},
and by the fact that
the composition of 
any u.s.c. map with any continuous one
at a point is an u.s.c. map 
in the same point,
we deduce that
$\|\cdot\|\circ\ov{f}$
is u.s.c. at $\infty$.
Thus
$\sup_{x\in X}\|\ov{f}(x)\|_{x}<\infty$,
indeed
we applied
to the 
u.s.c. map
$\|\cdot\|\circ\ov{f}$
the fact that 
$X$ is compact (so quasi compact),
$-\|\cdot\|\circ\ov{f}$
is l.s.c,
the
\cite[Th. $3$, $\S 6.2.$, Ch. $4$]{BourGT}
and
\cite[form.$(2)$, $\S 5.4.$, Ch. $4$]{BourGT}.
Therefore
$$
\ov{f}\in\prod_{x\in X}^{b}L_{x}.
$$
Then
$\ov{f}\in\Gamma^{\infty}(\pi_{\ms{L}})$.
The remaining implication
follows 
by Cor. \ref{28111707}
and
by the fact that
$\mf{V}(\ms{L},\mc{E}(L))$
is full
since $X$ is compact so completely regular 
and since the Dupre' theorem see for example
\cite[Cor. 2.10]{gie}.
\end{proof}
\begin{proposition}
\label{15512601}
We have 
$$
\Gamma^{\infty}(\pi_{\mathbf{L}^{\oplus}})
=
\left\{
\sigma_{1}\oplus\sigma_{2}
\,\vert\,
\sigma_{i}\in\prod_{x\in X}
L_{x},
\lim_{n\to\infty}\sigma_{i}(n)
\cu
\sigma(\infty),
i=1,2
\right\}.
$$
Here,
we used the
Convention
\ref{16392601}
and
set
$(\sigma_{1}\oplus\sigma_{2})(x)
\coloneqq
\sigma_{1}(x)\oplus\sigma_{2}(x)$.
\end{proposition}
\begin{proof}
By
Convention
\ref{16392601}
and
Cor.
\ref{17571212}
$\sigma_{1}\oplus\sigma_{2}$
is continuous at $\infty$
if and only if
$\sigma_{i}$
is continuous at $\infty$
for all $i=1,2$.
Thus the statement
by
Prp.
\ref{11422601}.
\end{proof}
\begin{proposition}
\label{18542601}
Let
$
\mc{U}_{0}
\in\prod_{n\in\N}
\cc{}{\R^{+},B_{s}(L_{n})}
$
be
such that
$\mc{U}_{0}(x)$
is a
$(C_{0})-$semigroup
of contractions
on
$L_{n}$
for all $n\in\N$.
Moreover
let us denote
by
$T_{n}$
the infinitesimal generator
of the semigroup
$\mc{U}_{0}(n)$
for any $n\in\N$.
Thus with the positions
\eqref{15482601}
where
$\mf{V}$
is
the
Kurtz
bundle
we have
\begin{equation}
\label{18452801}
\begin{cases}
\Phi
=
\left\{
\sigma_{1}\oplus\sigma_{2}
\,\vert\,
(\forall i\in\{1,2\})
(\sigma_{i}\in\prod_{x\in X}
L_{x})(1-2)
\right\}
\\
(1)
\lim_{n\to\infty}\sigma_{i}(n)
\cu
\sigma_{i}(\infty)
\\
(2)
(\forall n\in\N)
(\sigma_{1}(n),\sigma_{2}(n))
\in
Graph(T_{n}),
\end{cases}
\end{equation}
and
\begin{equation}
\label{17472801}
\begin{cases}
\mc{E}
=
\left\{
\sigma^{\sigma_{1}(\infty)}
\,\vert\,
\sigma_{1}\in\prod_{x\in X}
L_{x}
(1-2-3)
\right\}
\\
(1)
\lim_{n\to\infty}\sigma_{1}(n)
\cu
\sigma_{1}(\infty)
\\
(2)
(\forall n\in\N)
(\sigma_{1}(n)
\in
Dom(T_{n}))
\\
(3)
(\exists\,f\in L_{\infty})
(\lim_{n\to\infty}T_{n}
\sigma_{1}(n)
\cu
f).
\end{cases}
\end{equation}
Moreover
$\exists\,!\,f$ 
satisfying $(3)$ in
\eqref{17472801}
and
$(\forall\sigma_{1}\in\mc{E})
((\sigma_{1},\sigma_{2})\in\Phi)$,
where
$\sigma_{2}\in\prod_{x\in X}L_{x}$
such that
$(\forall n\in\N)
(\sigma_{2}(n)\coloneqq T_{n}\sigma_{1}(n))$
and
$\sigma_{2}(\infty)
\coloneqq f$.
\end{proposition}
\begin{proof}
The first sentence 
follows
by Prp.
\ref{15512601},
while
the second 
comes
by the first one
and Lemma \ref{unique}.
\end{proof}
\begin{assumptions}
We assume 
$\exists\,
\{I_{n}\in B(L_{n},L)\}_{n\in\N}$
such that
\begin{equation}
\label{15052601}
\begin{cases}
\sup_{n\in\N}\|I_{n}\|_{B(L_{n},L)}
<\infty,
\\
(\forall f\in L)
(\forall n\in\N)
(I_{n}\circ P_{n}=Id).
\end{cases}
\end{equation}
Moreover
we assume that
\begin{equation}
\label{11442701}
\varlimsup_{n\to\infty}
\|P_{n}\|\leq 1.
\end{equation}
In addition we assume that
$
(\forall g\in L)
(\exists\,
\sigma_{1}\in\prod_{x\in X}
L_{x})
$
such that
\begin{equation}
\label{17202701}
\begin{cases}
(1)
\lim_{n\to\infty}
\sigma_{1}(n)
\cu
\sigma_{1}(\infty)
\\
(2)
(\forall n\in\N)
(\sigma_{1}(n)
\in
Dom(T_{n}))
\\
(3)
(\exists\,f\in L_{\infty})
(\lim_{n\to\infty}T_{n}
\sigma_{1}(n)
\cu
f)
\\
(4)
g=\sigma_{1}(\infty).
\end{cases}
\end{equation}
Set
\begin{equation}
\label{16582801}
\mf{U}
\coloneqq
\left\{
F\in
\cc{}{\R^{+},B_{s}(L)}
\,\vert\,
(\forall s\in\R^{+})
(\forall v\in L)
(\|F(s)v\|=\|v\|)
\right\}.
\end{equation}
\end{assumptions}
In the following definition
we shall give the data for constructing
a bundle
$\mf{W}$
such that
$\lr{\mf{V}(\ms{L},\mc{E}(L)),
\mf{W}}{X,\R^{+}}$
would
be
a
$\left(\Theta,\mc{E}\right)-$structure.
\begin{definition}
\label{13352701}
Set 
$P_{\infty}\coloneqq 
I_{\infty}
\coloneqq
Id:L\to L$,
moreover 
$\forall U\in\mf{U}$
set
$F_{U}\in\prod_{x\in X}
\cc{c}{\R^{+},\mc{L}_{S_{x}}(L_{x})}$
such that
$\forall x\in X$
$$
\begin{cases}
F_{U}(x)
\coloneqq
P_{x}\circ 
U(\cdot)
\circ I_{x},
\\
P_{x}\circ 
U(\cdot)
\circ I_{x}:
\R^{+}\ni s
\mapsto
P_{x}\circ 
U(s)
\circ I_{x}
\in B(L_{x}).
\end{cases}
$$
Now we can define
$\forall x\in X$
$$
\ms{M}_{x}
\coloneqq
\textrm{span}
\left
\{
F_{U}(x)
\,\vert\,
U\in\mf{U}
\right
\}.
$$
$\ms{M}_{x}$
has to be considered as
Hlcs with the topology
induced by that on
$\cc{c}{\R^{+},\mc{L}_{S_{x}}(L_{x})}$.
\footnote{
$\cc{c}{\R^{+},\mc{L}_{S_{x}}(L_{x})}$
is Hausdorff
for all $x\in X$
by 
the fact that
$\bigcup_{B\in\Theta}\mc{B}_{B}^{x}=L_{x}$,
see later Prop.
\ref{16332701}.}
Moreover set
$$
\mc{M}
\coloneqq
\textrm{span}
\left
\{
F_{U}
\,\vert\,
U\in\mf{U}
\right
\}.
$$
\end{definition}
\begin{theorem}
\label{14282701}
$\ms{M}_{x}$ 
as Hlcs 
is well-defined
for any $x\in X$,
moreover
$\mc{M}\subset
\prod_{x\in X}^{b}
\ms{M}_{x}$
and
$\ms{M}_{x}=
\{F(x)\,\vert\, F\in\mc{M}\}$.
Finally
$\mc{M}$ satisfies
$FM(3)-FM(4)$
with respect to 
$\ms{M}$.
\end{theorem}
\begin{proof}
By Rmk.
\ref{21500412b}
we have that
$\cc{c}{\R^{+},B_{s}(L_{x})}
\subset
\cc{c}{\R^{+},\mc{L}_{S_{x}}(L_{x})}$
hence 
for the first sentence of the statement
it is sufficient 
to show
that
$
P_{x}
\circ 
U(\cdot)
\circ I_{x}
\in
\cc{c}{\R^{+},B_{s}(L_{x})}$
for any $U\in\mf{U}$.
For $x=\infty$ is trivial so 
let $n\in\N$ and $f_{n}\in L_{n}$
thus for all $s\in\R^{+}$
and all net $\{s_{\alpha}\}_{\alpha\in D}$
in $\R^{+}$
converging at $s$
we have 
$$
\lim_{\alpha\in D}
\|
P_{n}
\circ 
U(s_{\alpha})
\circ 
I_{n}
(f_{n})
-
P_{n}
\circ 
U(s)
\circ 
I_{n}
(f_{n})
\|_{n}
=
\lim_{\alpha\in D}
\|
P_{n}
(U(s_{\alpha})-U(s))
I_{n}f_{n}
\|_{n}
=0,
$$
where
we used the fact that
$U$ is strongly continuous 
and $P_{n}$ is norm 
continuous by construction.
Thus the first sentence of the statement
follows.
Let $v\in\mc{E}$ and
$U\in\mf{U}$
thus $\forall K\in Comp(\R^{+})$
\begin{alignat*}{1}
\sup_{n\in\N}
\sup_{s\in K}
\|
P_{n}
U(s)
I_{n}
v(n)
\|_{n}
&\leq
M
\sup_{n\in\N}
\sup_{s\in K}
\|
U(s)
I_{n}
v(n)
\|_{\infty}
\\
&=
M
\sup_{n\in\N}
\|
I_{n}
v(n)
\|_{\infty}
\\
&\leq
M
\sup_{n\in\N}
\|
I_{n}
\|
\sup_{n\in\N}
\|
v(n)
\|_{\infty}<\infty.
\end{alignat*}
Here
$M\coloneqq\sup_{n\in\N}\|P_{n}\|$,
in the second
one the hypothesis that 
$U(s)$ is an 
isometry
for all $s\in\R^{+}$, in the final
inequality 
we considered
\eqref{15052601},
$\mc{E}\subset\prod_{x\in X}^{b}L_{x}$
and that
$M<\infty$ by
Rmk. \ref{16022601}.
Therefore
by Rmk.
\ref{21500412b}
$\mc{M}\subset
\prod_{x\in X}^{b}
\ms{M}_{x}$.
The equality
$\ms{M}_{x}=
\{F(x)\,\vert\, F\in\mc{M}\}$
comes by construction,
in particular
$\mc{M}$ 
satisfies
the $FM(3)$ with respect to the $\ms{M}$.
$\forall v\in\mc{E}$
\begin{alignat*}{2}
\varlimsup_{n\to\infty}
\sup_{s\in K}
\|
P_{n}U(s)I_{n}
v(n)
\|_{n}
&\leq
\varlimsup_{n\to\infty}
\left(\|
P_{n}
\|
\sup_{s\in K}
\|
U(s)I_{n}
v(n)
\|_{n}
\right),
&
\text{\cite[Prop. $11$, $\S 5.6.$
Ch. $4$]{BourGT}}
\\
&\leq
\varlimsup_{n\to\infty}
\|P_{n}\|
\varlimsup_{n\to\infty}
\sup_{s\in K}
\|U(s)I_{n}v(n)\|_{n},
&
\text{\cite[Prop. $13$, $\S 5.7.$
Ch. $4$]{BourGT}}
\\
&\leq
\varlimsup_{n\to\infty}
\|I_{n}v(n)\|_{\infty},
&
\text{
\eqref{11442701},
\eqref{16582801}}
\\
&=
\varlimsup_{n\to\infty}
\|
I_{n}P_{n}f
\|_{\infty},
&
\text{
$v\in\mc{E}\subset\Gamma(\pi)
\simeq\mc{E}(L)$
}
\\
&=
\|f\|_{\infty},
&
\text{
\eqref{15052601}
}
\\
&=
\|v(\infty)\|_{\infty}.
\end{alignat*}
Thus by considering that $U$ is a map
of isometries
we have
$$
\varlimsup_{n\to\infty}
\sup_{s\in K}
\|
P_{n}U(s)I_{n}
v(n)
\|_{n}
\leq
\sup_{s\in K}
\|
P_{\infty}
U(s)I_{\infty}
v(\infty)
\|_{\infty}.
$$
Hence
by 
\cite[Prop. $3$, $\S 7.1.$
Ch. $4$]{BourGT}
and
\cite[$(13)$, $\S 5.6.$
Ch. $4$]{BourGT}
we deduce that
$$
X\ni x\mapsto
\sup_{s\in K}
\|
P_{x}U(s)I_{x}
v(x)
\|_{x}
\text{ is $u.s.c.$ at $\infty$},
$$
therefore it is $u.s.c.$ on $X$
because of it is continuous in each
point in $\N$ due to the fact that
the topology induced on
$\N$ by that on $X$ is the discrete topology.
So 
$\mc{M}$ 
satisfies
the $FM(4)$ with respect to the $\ms{M}$.
\end{proof}
\begin{definition}
\label{16232701}
Thm. \ref{14282701}
allows us to
construct a bundle of $\Omega-$space
namely
the bundle
$\mf{V}(\ms{M},\mc{M})$
generated by the couple
$\lr{\ms{M}}{\mc{M}}$,
see Def \ref{17471910Ba}.
\end{definition}
\begin{remark}
\label{16282701}
By Rmk. \ref{17150312} we have
\begin{equation}
\label{14222801}
\mc{M}
\subseteq
\Gamma(\pi_{\ms{M}})
\text{ modulo the canonical isomorphism.}
\end{equation}
Hence
by 
$\ms{M}_{x}=
\{F(x)\,\vert\, F\in\mc{M}\}$
we have that
$\mf{V}(\ms{M},\mc{M})$
is full.
\end{remark}
\begin{proposition}
\label{16332701}
We have that
$\ov{\bigcup_{B\in\Theta}\mc{B}_{B}^{x}}
=L_{x}$
for all $x\in X$
moreover
$\lr{\mf{V}(\ms{L},\mc{E}(L)),
\mf{V}(\ms{M},\mc{M})}{X,\R^{+}}$
is
a
$\left(\Theta,\mc{E}\right)-$structure.
\end{proposition}
\begin{proof}
By assumptions
\eqref{17202701},
\eqref{11212601},
Prp.
\ref{18542601}
and Rmk.
\ref{21500412b}
we obtain
that
$\ov{\bigcup_{B\in\Theta}\mc{B}_{B}^{x}}
=
L_{x}$
for all $x\in X$.
The remaining
requests for the second sentence
of the statement
come by the construction
of $\mc{M}$ and $\ms{M}$.
\end{proof}
\begin{corollary}
\label{13560202}
If
$D(T_{x_{\infty}})$
is dense in $\mf{E}_{x_{\infty}}$,
and
$\exists\lambda_{0}>0,
\lambda_{1}<0$
such that
the ranges
$\mc{R}(\lambda_{0}-T_{x_{\infty}})$
and
$\mc{R}(\lambda_{1}-T_{x_{\infty}})$
are dense in $\mf{E}_{x_{\infty}}$),
then
$\lr{\mc{T}}{\infty,\Phi}
\in
\ms{Gr}(\mf{V}(\ms{L},\mc{E}(L)),\mf{V}(\ms{L},\mc{E}(L)))$
and
the following
$$
T_{\infty}:D(T_{\infty})
\ni\phi_{1}(\infty)
\mapsto
\phi_{2}(\infty)
$$
is a well-defined operator
which 
is the generator
of
a $C_{0}-$semigroup
of 
isometries
on $\mf{E}_{\infty}$.
\end{corollary}
\begin{proof}
By
Prp.s
\ref{16332701}
and
\ref{19272701}
we have that
the first part of hypotheses
of Thm. 
\ref{17301812b}
is satisfied
so the 
statement 
by the first sentence of the statement
of
Thm. 
\ref{17301812b}.
\end{proof}
\begin{definition}
\label{16002801}
Let us denote by
$\mc{U}_{\infty}$
the 
$C_{0}-$semigroup
of isometries
on $L_{\infty}$.
Moreover
set
$\mc{U}\in
\prod_{x\in X}\ms{U}_{is}(B_{s}(L_{x}))$
such that
$\mc{U}\up\N=\mc{U}_{0}$
and
$\mc{U}(\infty)
=
\mc{U}_{\infty}$.
\end{definition}
\begin{theorem}
\label{17332801}
$(\exists\,F\in\Gamma(\pi_{\ms{M}}))
(F(\infty)=\mc{U}(\infty))$
such that
$(\forall v\in\mc{E})
(\exists\,\phi\in\Phi)$
s.t.
$\phi_{1}(x_{\infty})
=
v(x_{\infty})$
and
$(\forall\{z_{n}\}_{n\in\N}\subset X
\,\vert\,
\lim_{n\in\N}z_{n}=x_{\infty})$
we have
that
$\{
\mc{U}(z_{n})(\cdot)\phi_{1}(z_{n})
-
F(z_{n})(\cdot)v(z_{n})
\}_{n\in\N}$
is a 
bounded equicontinuous
sequence.
Moreover we can choose
$F$
such that
$F=
F_{\mc{U}_{\infty}}$.
\end{theorem}
\begin{proof}
By 
Prop.
\ref{18542601}
and 
\eqref{14222801}
the statement
is equivalent
to show
that
$
\forall
\sigma_{1}\in\prod_{x\in X}
L_{x}$
satisfying
$(1-2-3)$
of 
\eqref{17472801}
and 
$(\forall\{z_{n}\}_{n\in\N}\subset X
\,\vert\,
\lim_{n\in\N}z_{n}=\infty)$
we have
that
\begin{equation}
\label{17592801}
\{
\mc{U}(z_{n})(\cdot)\sigma_{1}(z_{n})
-
F_{\mc{U}_{\infty}}
(z_{n})(\cdot)
\sigma^{\sigma_{1}(\infty)}(z_{n})
\}_{n\in\N}
\end{equation}
is a 
bounded equicontinuous
sequence.
Moreover by 
the second assumption
\eqref{15052601}
and
\eqref{17592801}
\begin{equation}
\label{18022801}
\{
\mc{U}(z_{n})(\cdot)
\sigma_{1}(z_{n})
-
P_{z_{n}}
\mc{U}_{\infty}
(z_{n})(\cdot)
\sigma_{1}(\infty)
\}_{n\in\N}
\end{equation}
is a 
bounded equicontinuous
sequence.
Set
$\sigma_{2}\in\prod_{x\in X}L_{x}$
such that
$\sigma_{2}(x)\coloneqq T_{x}\sigma_{1}(x)$,
for all $x\in X$,
thus
$$
\sigma_{i}\in\Gamma^{\infty}(\pi_{\ms{L}}),
$$
for all $i=1,2$, indeed
for
$i=1$ follows by 
$(1)$ of
\eqref{17472801}
and 
Prop.
\ref{11422601},
while for
$i=2$ 
follows by 
construction of
of $T_{\infty}$,
the second sentence of Prop.
\ref{18542601},
the fact that
by construction
$\Phi\subseteq\Gamma(\pi_{\ms{E}^{\oplus}})$,
see \eqref{15482601},
and finally by Cor.
\ref{17571212}.
Therefore 
in particular
$\sigma_{i}$ is continuous 
at $\infty$.
Thus
by considering
that
$\sigma^{\sigma_{i}(\infty)}
\in\Gamma(\pi_{\ms{L}})$
modulo isomorphism
by
\eqref{16022601b}
and that $\mf{V}(\ms{L},\mc{E}(L))$ is full
we deduce 
by Prop. \ref{28111555}
$$
\lim_{n\in\N}
\|
\sigma_{i}(z_{n})
-
\sigma^{\sigma_{i}(\infty)}
(\pi\circ\sigma_{i}(z_{n}))
\|_{\pi\circ\sigma_{i}(z_{n})}
=0.
$$
Then by considering that
$\pi\circ\sigma_{i}=Id$
because of $\sigma_{i}$ is a section,
we have
\begin{equation}
\label{19172801}
\lim_{n\in\N}
\|
\sigma_{i}(z_{n})
-
P_{z_{n}}\sigma_{i}(\infty)
\|_{z_{n}}
=0.
\end{equation}
The statement now
follows
by \eqref{19172801},
\eqref{18022801}
and
by using the same argumentation
used in proof of \cite[Th. $1.2$]{kurtz}
for proving a similar result.
\end{proof}
\begin{proposition}
\label{19272701}
With the notation
of Def. \ref{15062301}
we have that
\begin{equation*}
\ms{M}_{x}
\subset
\bigcap_{\lambda>0}
\mf{L}_{1}
(\R^{+},\mc{L}_{S_{x}}(L_{x});
\mu_{\lambda}),
\end{equation*}
and
\eqref{18470109}
holds.
\end{proposition}
\begin{proof}
By Prp.
\ref{18390901}.
\end{proof}
\begin{theorem}
\label{13512801}
$
\lr{\mf{V}(\ms{L},\mc{E}(L)),
\mf{V}(\ms{M},\mc{M})}{X,\R^{+}}$
has the
full Laplace duality property,
moreover 
$\forall U\in\ms{U}_{1,\|\cdot\|}(B_{s}(L))$,
$\forall\lambda>0$
and 
$\forall f\in L$
we have that
$$
\mf{L}(F_{U})(\cdot)(\lambda)
\sigma^{f}(\cdot)
=
\sigma^{(\lambda-T_{U})^{-1}f}.
$$
Here
$T_{U}$ is the infinitesimal generator
of the semigroup $U$.
\end{theorem}
\begin{proof}
Let $f\in L$ and $U\in\mf{U}$
thus for all $x\in X$ and $\lambda>0$
we have
\begin{alignat}{2}
\label{14242801}
\int_{0}^{\infty}
e^{-\lambda s}
P_{x}U(s)I_{x}\sigma^{f}(x)\,
ds
&=
\int_{0}^{\infty}
e^{-\lambda s}
P_{x}U(s)f\,ds
\notag
\\
&=
P_{x}
\int_{0}^{\infty}
e^{-\lambda s}
U(s)f\,ds,
\end{alignat}
where the first equality
follows
by the second assumption
\ref{15052601},
while the second one by the linearity
and continuity of $P_{x}$
and by 
\cite[Prop.$1$, $No 1$, $\S 1$, Ch. $6$]
{IntBourb}.
Thus the first sentence of the statement
by \eqref{16022601b}
and
\eqref{14222801}.
The second sentence of the statement
folllows by the \eqref{14242801}
and by Hille-Yosida Thm. , see
\cite[Th. $1.2.$]{kurtz}.
\end{proof}
\begin{corollary}
\label{12560202}
Let us assume the hypotheses of Cor.
\ref{13560202}. 
Then
$(\forall g\in L)
(\forall K\in Comp(\R^{+}))$
\begin{equation}
\label{16180202}
\lim_{z\to\infty}
\sup_{s\in K}
\left\|
\left(\mc{U}(z)(s)\circ P_{z}
-
P_{z}\circ\mc{U}_{\infty}(s)\right)
g
\right\|=0.
\end{equation}
Moreover
\begin{equation}
\label{02512912pbis}
\mc{U}
\in\Gamma^{\infty}(\rho).
\end{equation}
In particular
\begin{equation}
\label{02512912bis}
\{\lr{\mc{T}}{\infty,\Phi}\}
\in
\Delta_{\Theta}\lr{\mf{V}(\ms{L},\mc{E}),
\mf{V}(\ms{M},\mc{M})}
{\mc{E},X,\R^{+}}.
\end{equation}
\end{corollary}
\begin{proof}
By
Prp.
\ref{19272701}
follows
\eqref{18470109},
hypothesis 
$(i)$
of
Thm. 
\ref{17301812b}
follows
by 
Thm. 
\ref{13512801},
$(ii)$
by 
Thm. 
\eqref{17332801},
finally
$(iii)$
follows
by
\cite[Coroll. of Prop.$16$, $\S2.9$, Ch. $9$]
{BourGT}
and by
the fact that
$\{\{n\}\,\vert\, n\in\N\}$
is a base 
for the topology
on $\N$.
Thus
by
Thm. 
\ref{17301812b}
we obtain
\eqref{02512912pbis},
\eqref{02512912bis}
and
$(\forall v\in\mc{E})
(\forall K\in Comp(\R^{+}))$
\begin{equation}
\label{16170202}
\lim_{z\to\infty}
\sup_{s\in K}
\left\|
\mc{U}(z)(s)v(z)
-
F(z)(s)v(z)
\right\|=0,
\end{equation}
where
$F$ is any map
of which
in 
Thm. \ref{17332801}.
Now by Thm. 
\ref{17332801}
we can take
in the previous equation
$F=F_{\mc{U}_{\infty}}$,
moreover
by 
\eqref{17472801}
and assumption
\eqref{17202701}
we have
$$
\mc{E}=\{\sigma^{g}\,\vert\, g\in L\},
$$
therefore by 
\eqref{15052601}
$\forall s\in\R^{+}$,
$\forall z\in X$
and
$\forall g\in L$
$$
F_{\mc{U}_{\infty}}(z)
\sigma^{g}(z)
=
(P_{z}\circ\mc{U}_{\infty}(s))
g.
$$
Hence
by
\eqref{16170202}
follows
\eqref{16180202}.
\end{proof}

\section*{Acknowledgments}
I am grateful to Professor Victor I. Burenkov 
for many useful discussions during my postdoctoral position held in Padova
under his supervision.
I would like to thank the Editorial Board of ``Eurasian Mathematical Journal''
for the publication of the entire work in three parts. 

\bibliography{references}
\bibliographystyle{amsplain}

\end{document}